\documentclass[11pt]{article}
\usepackage{amssymb,amsmath,amsfonts,amsthm,array,bm,bbm,color,arydshln,eufrak,verbatim}%
\usepackage{array,cite}
\usepackage[dvipsnames]{xcolor}
\usepackage{graphicx}
\providecommand{\U}[1]{\protect\rule{.1in}{.1in}}

\numberwithin{equation}{section}

\newtheorem{theorem}{Theorem}[section]
\newtheorem{lemma}{Lemma}[section]
\newtheorem{corollary}{Corollary}[section]
\newtheorem{proposition}{Proposition}[section]
\newtheorem{remark}{Remark}[section]

\newtheorem{definition}{Definition}[section]
\newtheorem{hypothesis}{Hypothesis}[section]

\def\<{\langle}
\def\>{\rangle}
\def\d{{\rm d}}

\def\div{{\rm div}}
\def\E{\mathbb{E}}

\def\P{\mathbb{P}}

\def\xxi{\boldsymbol{\xi}}

\begin{document}

\title{Global Martingale Entropy Solutions to\\  the Stochastic Isentropic Euler Equations}

\author{Gui-Qiang G. Chen\footnote{ Mathematical Institute, University of Oxford, Oxford, OX2 6GG, UK. Email: chengq@maths.ox.ac.uk}
\quad
Feimin Huang\footnote{ School of Mathematical Sciences, University of Chinese Academy of Sciences, Beijing 100049, China; Academy of Mathematics and Systems Science, Chinese Academy of Sciences, Beijing 100190, China. Email: fhuang@amt.ac.cn}
\quad
Danli Wang\footnote{School of Mathematical Sciences, University of Chinese Academy of Sciences, Beijing 100049, China; Academy of Mathematics and Systems Science, Chinese Academy of Sciences, Beijing 100190, China. Email: wangdanli19@amss.ac.cn}}

\maketitle
\begin{abstract}
We establish the existence and compactness of global martingale entropy solutions with finite relative-energy 
for the stochastically forced system of isentropic Euler equations governed by a general pressure law.
To achieve these, a stochastic compensated compactness framework in $L^p$ is developed 
to overcome the difficulty that the uniform $L^{\infty}$ bound for the stochastic approximate
solutions is unavailable, owing to the stochastic forcing term. 
The convergence of the vanishing viscosity method is established by employing the stochastic compactness framework, 
along with careful uniform estimates of the stochastic approximate solutions, to obtain 
the existence of global martingale entropy solutions with finite relative-energy.
In particular, in the polytropic pressure case for all adiabatic exponents,
we prove that the global solutions satisfy the local mechanical energy inequality 
when the initial data are only required to have finite relative-energy 
(while the higher moment estimates for entropy are not required here, as needed in the earlier work).
Higher-order relative energy estimates for approximate solutions are also derived to establish 
the entropy inequality for more convex entropy pairs and to then prove the compactness of solutions 
to the stochastic isentropic Euler system.
The stochastic compensated compactness framework and the uniform estimate techniques for approximate solutions 
developed in this paper should be useful in the study of other similar problems.
\end{abstract}



\tableofcontents

\section{Introduction}
We are concerned with the global existence and dynamic properties 
of solutions of the Cauchy problem for 
the stochastically forced system of isentropic Euler equations:
  \begin{equation}\label{eq-stochastic-Euler-system}
  \begin{cases}
  \d \rho + \partial_x m\, \d t =0,\\[1mm]
  \d m + \partial_x\big(\frac{m^2}{\rho} +P(\rho)\big)\,\d t=\Phi(\rho,m)\,\d W(t),\\[1mm]
  (\rho, m)|_{t=0}=(\rho_0,m_0),
  \end{cases}
  \end{equation}
where $\rho$ is the density, $m$ is the momentum, $P(\rho)$ is the pressure,
$W$ is a cylindrical Wiener process, 
the noise coefficient $\Phi$ satisfies a suitable growth condition, 
and $x\in \mathbb{R}$.
When $\rho>0$, $u=\frac{m}{\rho}$ is the velocity.

The pressure-density relation $P=P(\rho)$ under consideration in this paper 
is governed by a general pressure law,
obeying conditions \eqref{2.2a}--\eqref{eq-general-pressure-law-2b} as described in \S2, which ensures that
system \eqref{eq-stochastic-Euler-system} is strictly hyperbolic and genuinely nonlinear when $\rho>0$,
among others. A particular case is the $\gamma$-law, which is the typical pressure law 
for the class of polytropic gases:
  \begin{equation}\label{eq-gamma-law}
  P(\rho)=\kappa \rho^{\gamma}
  \end{equation}
with adiabatic exponent $\gamma >1$ for some constant $\kappa>0$.

\smallskip
When $\Phi(\rho,m)\equiv 0$, system \eqref{eq-stochastic-Euler-system} reduces 
to the deterministic isentropic Euler equations:
  \begin{equation}\label{eq-deterministic-Euler-system}
  \begin{cases}
  \d \rho + \partial_x m\,\d t =0,\\[1mm]
  \d m + \partial_x\big(\frac{m^2}{\rho} +P(\rho)\big)\,\d t=0,\\[1mm]
  (\rho,m)|_{t=0}=(\rho_0,m_0).
  \end{cases}
  \end{equation}
System \eqref{eq-deterministic-Euler-system} governed by a general pressure law was 
first analyzed 
by Chen-LeFloch \cite{chenlefloch00ARMA}, which was extended by 
Chen-LeFloch \cite{chenlefloch03ARMA} to deal with more general pressure functions. 
Most recently, Chen et al. \cite{chenhuangLiWangWang24CMP} studied the compressible Euler-Poisson equations 
for a general pressure law allowing different power laws near the vacuum and the far field, in which
a special weak entropy pair is constructed to obtain the uniform higher integrability 
of the velocity such that the $L^p$ compensated compactness framework for the general pressure law can be established.
For the polytropic gas, more results have been obtained.
The global existence of $L^{\infty}$ entropy solutions 
for \eqref{eq-deterministic-Euler-system} was established 
in DiPerna \cite{Diperna83} for the case $\gamma=\frac{2N+3}{2N+1}$ for integer 
$N\ge 2$, Chen \cite{chen1986convergence} and Ding-Chen-Luo \cite{ding1985convergence} 
for the general case $\gamma \in \left(1,\frac{5}{3}\right]$, 
Lions-Perthame-Tadmor \cite{LPT94} for the case $\gamma\ge 3$,
and Lions-Perthame-Souganidis \cite{lions1996existence} 
to close the gap $\gamma\in (\frac{5}{3}, 3)$ and simplify
the proof for $\gamma>1$. 
The isothermal case $\gamma=1$ was treated by Huang-Wang \cite{huang2002isothermal}. 
The above results are based on the $L^{\infty}$
compensated compactness method, 
in which the invariant region method is used 
to obtain the uniform $L^{\infty}$ bound for the parabolic approximation. 
Chen-Perepelitsa \cite{chenperepelitsa10CPAM} constructed a general compensated compactness 
framework in  $L^p$ to prove the vanishing viscosity 
limit of the Navier-Stokes equations to the isentropic Euler equations 
with finite-energy initial data relative to the different end-states at infinity 
for all $\gamma>1$. 
In this $L^p$ compensated compactness framework, 
the $L^{\infty}$ uniform bound provided by the invariant region method 
was replaced by some uniform higher-integrability estimates 
since the Navier-Stokes equations do not admit natural invariant regions. 
Instead of $L^{\infty}$ entropy solutions, the finite-energy solutions of the isentropic Euler equations 
for the case $\gamma>1$ was obtained in \cite{chenperepelitsa10CPAM}, 
even for global spherically symmetric solutions for the multidimensional
case for $1<\gamma\le 3$ in \cite{chenperepelitsa15CMP},
which extended the result in LeFloch-Westdickenberg \cite{lefloch2007finiteenergy} for $1<\gamma\le \frac{5}{3}$. 
Recently, the compensated compactness framework established 
in Chen-Perepelitsa \cite{chenperepelitsa10CPAM} was applied 
by Chen-Schrecker \cite{chenschrecker18ARMA} to study the existence 
of globally-defined entropy solutions to the Euler equations 
for transonic nozzle flows with general cross-sectional areas 
that may have different end-states for all the case $\gamma>1$.

For stochastically forced systems of hyperbolic conservation laws, 
there exist only few results. 
Kim \cite{Kim11} studied strong solutions of stochastic quasilinear 
symmetric hyperbolic systems.
Recently, measured-valued solutions of the stochastic compressible Euler equations 
were investigated in Hofmanov\'a-Koley-Sarkar \cite{hofmanova2022measure-valued}. 
To our knowledge, the only paper that considered entropy solutions to 
the stochastic Euler equations
is Berthelin-Vovelle \cite{BV19}, in which they established the 
global existence of martingale entropy solutions to the stochastically 
forced system of isentropic Euler 
equations \eqref{eq-stochastic-Euler-system} 
for the polytropic pressure case
on the one-dimensional torus. 
In this paper \cite{BV19}, the compensated compactness technique of the deterministic 
case \cite{lions1996existence} was adapted to 
the stochastic case as in Feng-Nualart \cite{feng-Nualart08} 
for the stochastic scalar conservation laws. 
Because of the stochastic forcing term $\Phi(\rho,m)\d W(t)$, 
the uniform $L^{\infty}$ bounds for the parabolic approximate solutions 
cannot be obtained. 
To overcome this difficulty, they established some higher moment estimates based on 
careful analysis of entropies, thanks to the explicit formula 
of the entropy kernel. However, the price to pay is that the initial data 
are required to satisfy corresponding higher moment estimates.

In this paper, our aim is to establish 
the existence and compactness of global
martingale entropy 
solutions with finite relative-energy to the stochastic isentropic Euler 
system \eqref{eq-stochastic-Euler-system} governed by the general pressure law,
when the initial data allow for a positive far-field density.
In the case that the initial density tends to zero at the far-field, 
martingale entropy solutions with finite energy are obtained. 
Compared to the result by Berthelin-Vovelle \cite{BV19}, 
our initial data are only required to satisfy the finite relative-energy. 
The main difficulty is that the uniform $L^{\infty}$ bound 
for the parabolic approximate solutions is unavailable. 
Thus, we develop the stochastic version of compensated compactness framework 
in $L^p$, 
which is the extension of the deterministic ones 
in Chen-Perepelitsa \cite{chenperepelitsa10CPAM} and Chen et al. \cite{chenhuangLiWangWang24CMP}. 
Moreover, we also apply this stochastic 
compensated compactness framework in $L^p$ to study 
the stochastic isentropic Euler system \eqref{eq-stochastic-Euler-system} 
governed by the $\gamma$-law,
for which we obtain the global martingale entropy solutions
that satisfy the local mechanical energy inequality in the sense of distributions.

Furthermore, when the initial data have a finite higher-order relative energy, 
we show that the corresponding solutions also retain the finite higher-order relative energy. 
This property allows us to derive the entropy inequality 
for a broader class of convex entropy pairs, 
including the mechanical energy and its associated flux in the general pressure law case. 
These results are crucial, as they also enable us to establish the compactness of the sequence of martingale entropy solutions to the stochastic isentropic Euler 
equations \eqref{eq-stochastic-Euler-system}.

Some new difficulties and strategies to handle them are as follows:
To overcome the difficulty that the uniform $L^{\infty}$ bound for the parabolic 
approximation is unavailable, 
we develop the stochastic version of $L^p$ compensated compactness framework 
as in the deterministic case, the basis of which is the uniform energy estimate 
and some uniform higher integrability estimates for the parabolic 
approximation (see \S \ref{sec-uniform-estmates-for-parabolic-approximation}). 
However, to establish these uniform higher integrability estimates, 
we cannot use the infinite-dimensional 
It\^o formula, since the transformations are not in $C^2$, not even usual 
functions of the underlying stochastic process.
We deal with this problem by obtaining strong solutions with high regularity 
(see Theorem \ref{thm-wellposedness-parabolic-approximation-R}) 
such that the parabolic approximation 
equations \eqref{eq-parabolic-approximation} 
hold in the classical strong sense. 
Then we can fix any space variable $x$ 
and use the finite-dimensional It\^o formula 
to establish the uniform higher integrability estimates. 
Additionally, to obtain the global strong solutions of 
the parabolic approximation 
equations \eqref{eq-parabolic-approximation}, we need to derive the 
$L^{\infty}$ bound. 
In the deterministic case, the invariant region method is a powerful tool 
to derive the $L^{\infty}$ bound. 
To make the invariant region method applicable in the stochastic case, 
on one hand, inspired by Berthelin-Vovelle \cite{BV19}, 
we assume that the noise coefficient is compactly supported 
(see \eqref{condition-compact-support-of-noise-coefficient}); 
on the other hand, it is also required that the parabolic approximation
equations \eqref{eq-parabolic-approximation} hold in the classical strong sense
and the finite-dimensional It\^o formula
can be applied (see Proposition \ref{prop-L-infty-estimates-by-invariant-region}).

After establishing the well-posedness of the parabolic approximation equations \eqref{eq-parabolic-approximation}, 
as in the deterministic case, we employ the stochastic compactness method
to pass to the limit.
A usual way is to apply the Jakubowski-Skorokhod representation theorem, 
which is different from the deterministic case, 
gives us the almost sure Skorokhod representation, {\it i.e.}, 
a sequence of new random variables on a new probability space that have the same laws 
as the original ones and almost surely 
converge (see Proposition \ref{prop-apply-jakubowski-skorokhod-representation-to-take-limit}).
This method is powerful; 
however, it also brings us additional difficulties, {\it i.e.}, 
we need to be able to identify those new random variables when necessary. 
More specifically, to establish the stochastic version of $L^p$ compensated compactness framework, 
we need to introduce a random Young measure, 
as in the deterministic case, to pass to the limit 
in the nonlinear terms. 
Therefore, it is significant to identify that the Skorokhod representation of
the random Young measure associated with the parabolic approximation
is a Dirac mass and also a random Young measure associated
with the corresponding Skorokhod representation of the parabolic approximation.
Thanks to the integrability of the parabolic approximation, the random Young measure 
here is $L^r$-random Dirac mass with $r\ge 1$. 
The fact that the random Young measure is an $L^r$-random Dirac mass 
is uniquely determined by its law (see \cite[Proposition 4.5]{BV19}). 
Thus, we can identify the Skorokhod representation of the random Young 
measure (see Lemma \ref{prop-identify-tilde-mu-varepsilon-delta-mass}).

In the procedure of establishing the stochastic version of $L^p$ compensated compactness framework, 
a key step is
to derive the tightness of random entropy dissipation measures 
in some negative-order Sobolev space (see Lemma \ref{prop-tartar-commutation}). 
Here the tightness should be derived under the Skorokhod representation of the parabolic 
approximation (see Proposition \ref{prop-apply-jakubowski-skorokhod-representation-to-take-limit}).
However, this can not be directly done and is difficult, since the Skorokhod representation 
is a weak solution of the parabolic approximation equations \eqref{eq-parabolic-approximation}. 
Instead, we first obtain the tightness of random entropy dissipation measures under the original 
parabolic approximation in some negative-order Sobolev space  (see Proposition \ref{prop-H^-1-compactness}), 
which implies the corresponding tightness of random entropy dissipation measures 
under the Skorokhod representation by using the equality of laws 
(see the proof of Lemma \ref{prop-tartar-commutation}).

For the general pressure law case, the explicit formula of entropy is unavailable
so that the following additional difficulties arise:

\smallskip
(i) To establish suitable uniform estimates used to obtain enough compactness, 
the method used in the deterministic polytropic pressure case, {\it i.e.}, 
deriving the uniform higher integrability estimates for the density and the velocity 
by utilizing an appropriate (non-convex) entropy with a special generating function, 
cannot be directly applied, since there is no precise formula for the entropy kernel. 
Furthermore, it is not clear how the higher moment estimates can be established 
for the parabolic approximations and the entropy pairs 
as in Berthelin-Vovelle \cite{BV19} 
by using the asymptotic expansion of the entropy kernel.
In fact, these estimates in \cite{BV19} are based on the fact that the entropy kernel
has an explicit formula, which allows for careful analysis and calculation by
choosing different generating functions. 
To overcome this difficulty, we employ the special entropy pair 
constructed in \cite{chenhuangLiWangWang24CMP} to establish the higher integrability 
of the velocity. 

(ii) In the step of deriving the tightness of random entropy dissipation measures 
in some negative-order Sobolev space,
the compactness that we can obtain is too weak to apply 
the classical div-curl lemma in \cite{murat78}. 
However, thanks to the higher integrability of the density and the velocity established above, 
we are able to show the equi-integrability, 
so that we are able to use the generalized 
div-curl lemma (see Lemma \ref{lem-generalized-div-curl-lemma}). 
In fact, due to the generality of the pressure and the lack of an explicit formula for the entropy kernel, 
only some rough bounds on the entropy pairs are obtained. 

(iii) For the reduction of the Young measure, since the uniform $L^{\infty}$ bounds 
for the parabolic approximation
are unavailable, the reduction technique for the bounded supported Young measure 
introduced in \cite{chen1986convergence,ding1985convergence,Diperna83,lions1996existence} 
can not be applied. 
In fact, in the deterministic case, \cite{chenperepelitsa10CPAM} already encountered such a problem;
thanks to the precise structure of the entropy kernel for the polytropic pressure, 
it can be shown first that every connected subset in the support of the Young measure 
is bounded and  then apply the reduction argument in \cite{chen1986convergence,ding1985convergence,Diperna83,lions1996existence}, 
which is the key idea in  \cite{chenperepelitsa10CPAM}. 
For the general pressure law case, this idea in \cite{chenperepelitsa10CPAM} 
cannot be used directly. 
We employ the reduction framework established in \cite{chenhuangLiWangWang24CMP}
for the Young measure with unbounded support and for the general pressure law, 
which is based on a careful analysis of
the singularities of the fractional derivatives of the entropy kernel.

We also stress that, compared to the result in \cite{BV19}, 
the initial data here are only required to have finite (relative) energy
so that we do not need to establish the higher moment estimates for entropy pairs. 
These improvements are essential since
we can expect to extend the one-dimensional result in this paper to 
the multidimensional spherically symmetric case, 
utilizing the so-established stochastic version of 
compensated compactness framework in $L^p$. 
This is out of the scope of this paper and will be addressed in our further work.

The organization of this paper is as follows: 
In \S \ref{sec-well-posedness-of-parabolic-approximation}, 
the main result is Theorem \ref{thm-wellposedness-parabolic-approximation-R}, 
which shows that there exists a unique strong solution to the parabolic approximate 
equations \eqref{eq-parabolic-approximation}. 
We postpone the proof of Theorem \ref{thm-wellposedness-parabolic-approximation-R} 
to \S \ref{sec-proof-of-well-posedness-of-parabolic-approximation}, in which we use 
a different method to prove the well-posedness of the parabolic approximate 
equations \eqref{eq-parabolic-approximation} 
and obtain slightly better results (compared to \cite{BV19}), {\it i.e.}, almost surely, 
the approximate density has a positive lower bound and the parabolic approximate solution 
has an $L^{\infty}$ bound. 
The proof of Theorem \ref{thm-wellposedness-parabolic-approximation-R} 
is based on the cut-off technique, {\it i.e.}, we first prove the existence 
and uniqueness for the cut-off 
equations (see 
Theorem \ref{thm-wellposedness-cut-off-parabolic-approximation-R-relax-initial-data}) 
and then remove the cut-off operator to obtain the unique global solution. 
In \S \ref{sec-uniform-estmates-for-parabolic-approximation}, 
the higher integrability estimates for the parabolic approximate solution are established.

In \S \ref{sec-tightness}, we pass to the limit from the parabolic approximation 
based on the compactness method.
In the stochastic setting, a usual way is to apply the Skorokhod representation theorem. 
However, this is restricted to the Polish space. 
Some of the spaces that we work with are not metric spaces, 
thus we need to apply the Jakubowski-Skorokhod representation theorem \cite{jakubowski97}, 
which provides a generalization of the classical Skorokhod representation theorem 
to the sub-polish spaces that are not metrizable.

In \S \ref{sec-compensated-compactness-and-reduction-of-young-measure}, 
we establish the stochastic version of compensated compactness framework in $L^p$. 
The uniform higher integrability estimates obtained 
in \S \ref{sec-uniform-estmates-for-parabolic-approximation} 
enable us to use the stochastic-version generalized Murat's lemma 
(see Lemma \ref{lem-stochastic-version-murat-lemma}) to establish the tightness 
in some negative-order Sobolev space. 
Note that, in this step, it does not require the higher moment estimates 
as in \cite{BV19} (where they required the generating 
function $\psi=\xi^{2p}$ in $\eta^{\psi}$ with $p\ge 4+\frac{1}{2\theta}$), 
but only need the estimates in Propositions 
\ref{prop-uniform-estimates-rho-gamma+1}--\ref{prop-higher-integrability-of-velocity-general-pressure-law}.
Having this tightness at hand, we apply the
generalized div-curl lemma
to derive the Tartar commutation relation, under which we can prove that, 
almost surely, 
the limit random Young measure is either concentrated on the vacuum region 
or reduced to a Dirac mass.

In \S \ref{sec-martingale-solution}, we perform the limit to obtain the martingale entropy solutions
with finite relative-energy of the stochastic isentropic Euler 
system \eqref{eq-stochastic-Euler-system} with
the general pressure law.
The uniform higher integrability estimates obtained 
in \S \ref{sec-uniform-estmates-for-parabolic-approximation} 
provide enough equi-integrability, which enables us to perform 
the limit and derive the energy estimates 
and
the entropy inequality for the stochastic isentropic Euler system \eqref{eq-stochastic-Euler-system}. 
Furthermore, we establish the higher-order relative energy estimate for the solutions,
under the condition that the initial data have the finite higher-order relative energy, 
which enables us to derive the entropy inequality for the solutions
satisfied by a broader class of entropy pairs, 
including the mechanical energy-energy flux.
Then, as a direct application of the stochastic-version compensated compactness 
framework in $L^p$, we obtain the compactness of the global martingale entropy 
solutions to the stochastic isentropic Euler 
system \eqref{eq-stochastic-Euler-system} governed by 
the general 
pressure law.

In \S \ref{sec-polytropic-gas-case}, we study the stochastic isentropic Euler 
system \eqref{eq-stochastic-Euler-system} 
in the polytropic pressure case for all $\gamma>1$, 
the results of which for $\gamma\le 3$ are covered by the general pressure law case. 
Moreover, thanks to the explicit formulas of entropy pairs, we obtain better results 
in this case.
In particular, constructing suitable approximate cut-off 
generating functions, 
we succeed in deriving the local entropy inequality 
corresponding to the mechanical energy;
see Theorems \ref{thm-well-posedness-for-euler-on-whole-space-deterministic-initial-data} 
and  \ref{thm-better-well-posedness-for-euler-on-whole-space}. 
When the initial data have the finite higher-order relative energy, 
we 
obtain the entropy inequality for 
the convex entropy pairs with generating functions of sub-cubic growth
and establish the compactness of the martingale entropy solution sequence 
to the stochastic isentropic Euler 
system \eqref{eq-stochastic-Euler-system} 
for all adiabatic exponents $\gamma>1$.

\section{Mathematical Problems and Main Theorems}
In this section, we discuss some basic properties of the stochastic isentropic Euler equations, 
set up our mathematical problems, and present the main theorems of this paper. 

\subsection{Stochastic Isentropic Euler Equations}

System \eqref{eq-stochastic-Euler-system} can be written as
  \begin{equation}\label{eq-stochastic-Euler-system-matrix-form}
  \d U + \partial_x F(U) \d t = \Psi(U) \d W(t),
  \end{equation}
where
\begin{equation}\label{2.2a}
 U=
  \begin{pmatrix}
  \rho\\
  m
  \end{pmatrix},\quad
  F(U)=\begin{pmatrix}
  m\\
  \frac{m^2}{\rho}+P(\rho)
  \end{pmatrix},\quad
  \Psi(U)=\begin{pmatrix}
  0\\
  \Phi(U)
  \end{pmatrix}.
\end{equation}
The pressure-density relation under consideration in this paper is governed by the following general pressure law 
obeying the two conditions:
\begin{itemize}
\item[\rm (i)] 
The strict hyperbolicity and genuine nonlinearity of
system 
\eqref{eq-stochastic-Euler-system} 
are required when $\rho>0$:
  \begin{equation}\label{2.2b}
  P^{\prime}(\rho)>0, \quad 2P^{\prime}(\rho)+\rho P^{\prime\prime}(\rho) >0\qquad \text{ for \ $\rho>0$}
  \end{equation}
and $P(\rho)\in C^1(\left[0,\infty\right))\cap C^4(0,\infty)$.

\item[\rm (ii)] There exist constants $\rho^{\star}>\rho_{\star}>0$ such that, for some $1< \gamma_2\le \gamma_1 <3$,
    \begin{align}
    &P(\rho)=\kappa_1 \rho^{\gamma_1}(1+P_1(\rho))\qquad \text{ for \ $0\le \rho < \rho_{\star}$},\label{eq-general-pressure-law-1}\\
    &P(\rho)=\kappa_2 \rho^{\gamma_2}(1+P_2(\rho))\qquad \text{ for \ $\rho^{\star}\le \rho < \infty$},\label{eq-general-pressure-law-2}
    \end{align}
with $\kappa_1, \kappa_2 >0$ and $P_1(\rho), P_2(\rho)\in C^4(0,\infty)$ 
satisfying
  \begin{align}
  &|P_1^{(n)}(\rho)|\le C_{\star}\rho^{\gamma_1-1-n}\qquad 
  \text{for $n=0,1,\cdots,4$, and $0<\rho<\rho_{\star}$},
   \label{eq-general-pressure-law-1b}\\
  &|P_2^{(n)}(\rho)|\le C^{\star}\rho^{-\flat-n}\qquad\,\,\,\,
  \text{for $n=0,1,\cdots,4$, and $\rho^{\star}\le \rho < \infty$}, 
  \label{eq-general-pressure-law-2b}
  \end{align}
where $\flat>0$, $C_{\star}>0$ depends only on $\rho_{\star}$, 
and $C^{\star}>0$ depends only on $\rho^{\star}$.
\end{itemize}

A particular case is the $\gamma$-law \eqref{eq-gamma-law}
with adiabatic exponent $\gamma >1$ for some constant $\kappa>0$.
Without loss of generality, constant $\kappa$ in \eqref{eq-gamma-law} 
can be chosen, by scaling, 
as $\kappa=\frac{(\gamma-1)^2}{4\gamma}$.

\smallskip
Denote the sound speed by
\begin{equation}\label{sound-speed}
c(\rho)=\sqrt{P'(\rho)}.
\end{equation}
Condition \eqref{2.2b} ensures that, away from the vacuum, system 
\eqref{eq-stochastic-Euler-system} 
is strictly hyperbolic and
admits two genuinely nonlinear characteristic fields associated with two distinct
wave speeds:
\begin{equation}\label{wave-speed}
\lambda_1=\frac{m}{\rho}-c(\rho), \qquad \lambda_2=\frac{m}{\rho}+c(\rho).
\end{equation}
At the vacuum $\rho=0$, $c(0)=0$, and the wave speeds coincide
so that the strict hyperbolicity of system \eqref{eq-stochastic-Euler-system-matrix-form} fails.
Corresponding two Riemann invariants are
\begin{equation}\label{eq-riemann-invariants-representation-formula}
w_1=\frac{m}{\rho}-\mathcal{K}(\rho), \qquad w_2=\frac{m}{\rho}+\mathcal{K}(\rho),
\end{equation}
where 
\begin{equation}\label{2.9a}
\mathcal{K}(\rho):=\int_0^\rho \frac{\sqrt{P^{\prime}(y)}}{y}\,\d y. 
\end{equation}

\subsection{Entropy Functions}
An entropy-entropy flux pair (or entropy pair, for short) of 
system \eqref{eq-stochastic-Euler-system}, 
equivalently system \eqref{eq-stochastic-Euler-system-matrix-form},
is a pair of functions $(\eta, q):\mathbb{R}_{+}\times \mathbb{R}\to \mathbb{R}^2$ such that
  \begin{equation}\label{eq-definition-equ-for-entropy-pair}
  \nabla q(U)=\nabla \eta(U)\nabla F(U),
  \end{equation}
where $\nabla=(\partial_{\rho}, \partial_{m})$. 
Furthermore, $(\eta, q)(\rho,m)$ is called a weak entropy pair 
if $\eta$ vanishes at the vacuum: $\rho=0$.

When $\rho>0$, we denote $u=\frac{m}{\rho}$ as the velocity 
and rewrite equations \eqref{eq-definition-equ-for-entropy-pair} 
as the following wave equation for entropy $\eta$:
  $$
  \partial_{\rho}^2 \eta-\mathcal{K}^{\prime}(\rho)^2 \partial_{u}^2 \eta =0,
  $$
which is degenerate when $\rho\to 0$, where $\mathcal{K}$ is defined in \eqref{2.9a} above.

As indicated in \cite{Diperna83,LPT94,chenlefloch00ARMA,chenlefloch03ARMA}, 
any weak entropy pair can be expressed by a corresponding generating function $\psi$:
  \begin{equation}\label{eq-entropy-flux-representation}
  \begin{aligned}
  \eta^{\psi}(\rho,u)=\int_{\mathbb{R}} \chi(\rho,u,s)\psi(s)\, \d s,\quad
  q^{\psi}(\rho,u)=\int_{\mathbb{R}} \sigma(\rho,u,s)\psi(s)\, \d s,
  \end{aligned}
  \end{equation}
where the entropy kernel $\chi(\rho,u,s)$ is the solution of the singular Cauchy problem:
  \begin{equation}\label{eq-entropy-kernel}
  \begin{cases}
  \partial_{\rho}^2 \chi -\mathcal{K}^{\prime}(\rho)^2 \partial_{u}^2 \chi=0,\\
  \chi|_{\rho=0}=0,\quad
  \partial_{\rho}\chi|_{\rho=0}=\delta_{u=s},
  \end{cases}
  \end{equation}
and the corresponding entropy flux kernel $\sigma(\rho,u,s)$ satisfies
  $$
  \begin{cases}
  \partial_{\rho}^2 (\sigma-u\chi) -\mathcal{K}^{\prime}(\rho)^2 \partial_{u}^2 (\sigma-u\chi)=0,\\
  (\sigma-u\chi)|_{\rho=0}=0,\quad
  \partial_{\rho}(\sigma-u\chi)|_{\rho=0}=0.
  \end{cases}
  $$
Note that equation \eqref{eq-entropy-kernel} is invariant under the Galilean transformation
({\it cf.} \cite{chenlefloch00ARMA}), which implies that
$\chi(\rho,u,s)=\chi(\rho,u-s,0)=\chi(\rho,0,s-u)$. 
Therefore, we write $\chi(\rho,u,s)=\chi(\rho;u-s)$ when no confusion arises.

In the polytropic pressure case \eqref{eq-gamma-law}, any weak entropy pair can be expressed 
by the explicit formula via a generating function $\psi$
({\it cf.} \cite{ding1985convergence,Diperna83,LPT94}):
  $$
  \begin{aligned}
  \eta^{\psi}(\rho, m)&=\int_{\mathbb{R}} [\rho^{2\theta}-(u-s)^2]^{\lambda}_{+}\psi(s)\,\d s\\
  &=\rho \int_{-1}^1 \psi(u+\rho^{\theta}s)[1-s^2]^{\lambda}_{+}\,\d s,\\
  q^{\psi}(\rho, m)&=\int_{\mathbb{R}} (\theta s+(1-\theta)u)[\rho^{2\theta}-(u-s)^2]^{\lambda}_{+}\psi(s)\,\d s\\
  &=\rho \int_{-1}^1 (u+\theta\rho^{\theta}s)\psi(u+\rho^{\theta}s)[1-s^2]^{\lambda}_{+}\,\d s,
  \end{aligned}
  $$
through the entropy kernal $\chi(\rho, u, s)$ and entropy flux kernel
$\sigma(\rho, u,s)$:
  $$
  \chi(\rho, u,s)=\chi(\rho;s-u)=[\rho^{2\theta}-(s-u)^2]^{\lambda}_{+},\qquad \sigma(\rho,u,s)= (\theta s+(1-\theta)u)[\rho^{2\theta}-(u-s)^2]^{\lambda}_{+}
  $$
where $\lambda=\frac{3-\gamma}{2(\gamma-1)}>-\frac12, \theta=\frac{\gamma-1}{2}$, and $[f]_+:=\max \{f,0\}$.

\smallskip
In particular, when $\psi(s)=\frac12 s^2$, the explicit entropy pair 
is given by the mechanical energy and mechanical energy flux:
  \begin{equation}\label{eq-mechanical-energy}
  \begin{aligned}
  \eta_E(\rho,m)=\frac12 \frac{m^2}{\rho}+\rho e(\rho),\qquad
  q_E(\rho,m)=\frac12 \frac{m^3}{\rho^2}+m\big(\rho e(\rho)\big)^{\prime},
  \end{aligned}
  \end{equation}
where the internal energy $e(\rho)$ is given by 
relation $P(\rho)=\rho^2 e^{\prime}(\rho)$ with $e(0)=0$.
The relative mechanical energy with respect to the far-field state $(\rho_{\infty}, 0)$ with $\rho_\infty> 0$
({\it i.e.}, $(\rho_0, m_0) \to (\rho_{\infty}, 0)$ as $|x|\to \infty$) is
  \begin{equation}\label{eq-relative-mechanical-energy}
  \eta_E^{*}(\rho,m)=\frac12 \frac{m^2}{\rho}+e^*(\rho, \rho_{\infty}),
  \end{equation}
where the relative internal energy is defined by
  \begin{equation}\label{eq-relative-internal-energy}
  e^*(\rho, \rho_{\infty}):=\rho e(\rho)-\rho_{\infty}e(\rho_{\infty}) 
  - \big(\rho e(\rho) \big)^{\prime}|_{\rho=\rho_{\infty}}(\rho-\rho_{\infty}).
  \end{equation}
Notice that we have assumed that the far-field velocity $u_\infty=\frac{m_\infty}{\rho_\infty}=0$ 
without loss of generality, due to the Galilean invariance.

\smallskip
There is also a high-order energy:
  $
  \eta_{\diamond}(\rho,m)=\frac{1}{12}\frac{m^4}{\rho^3}+\frac{e(\rho)}{\rho}m^2+g(\rho),
  $
with $g(\rho)$ determined by
$
g(0)=g^{\prime}(0)=0$
and $g^{\prime \prime}(\rho)=\frac{2P^{\prime}(\rho)e(\rho)}{\rho}$,
for which the associated energy flux is denoted by $q_{\diamond}(\rho,m)$. 
As above, the relative high-order energy is defined by
  $$
  \eta_{\diamond}^*(\rho,m)=\frac{1}{12}\frac{m^4}{\rho^3}+\frac{e(\rho)}{\rho}m^2+e_{\diamond}^*(\rho,\rho_{\infty}),
  $$
where $e_{\diamond}^*(\rho,\rho_{\infty})=g(\rho)-g(\rho_{\infty})-g^{\prime}(\rho_{\infty})(\rho-\rho_{\infty})$, 
the associated entropy flux is denoted by $q_{\diamond}^*(\rho,m)$.

\subsection{Properties of the General Pressure Function and 
the Relative Internal Energy Function}
We now present some properties of the pressure function $p=P(\rho)$ determined by the general pressure law
\eqref{2.2b}--\eqref{eq-general-pressure-law-2b} and the associated relative internal energy function.

\begin{lemma}[\hspace{-0.2mm}\cite{chenhuangLiWangWang24CMP}, Lemma 3.1]\label{lem-properties-for-general-pressure-law}
For sufficiently small $\rho_{\star}$ in \eqref{eq-general-pressure-law-1} and sufficiently large $\rho^{\star}$ in \eqref{eq-general-pressure-law-2}, the following properties hold{\rm :}
\begin{enumerate}
\item[\rm (i)] When $\rho\in \left(0,\rho_{\star}\right]$,
\begin{align}
  &\underline{C}_1 \rho^{\gamma_1}\le P(\rho) \le \overline{C}_1 \rho^{\gamma_1},\quad \underline{C}_1\gamma_1 \rho^{\gamma_1-1}\le P^{\prime}(\rho) \le \overline{C}_1 \gamma_1 \rho^{\gamma_1-1},
  \label{iq-lower-upper-bound-for-general-pressure-1}\\
  &C^{-1} \rho^{\gamma_1-1}\le e(\rho) \le C \rho^{\gamma_1-1},\quad C^{-1} \rho^{\gamma_1-2}\le e^{\prime}(\rho) \le C \rho^{\gamma_1-2},
  \label{iq-lower-upper-bound-for-general-internal-energy-1}\\
  &C^{-1} \rho^{\theta_1}\le \mathcal{K}(\rho) \le C \rho^{\theta_1};
  \label{iq-lower-upper-bound-for-k(rho)-1}
  \end{align}
\item[\rm (ii)]  When $\rho\in \left[\rho^{\star},\infty \right)$,
\begin{align}
  &\underline{C}_2 \rho^{\gamma_2}\le P(\rho) \le \overline{C}_2 \rho^{\gamma_2},\quad \underline{C}_2\gamma_2 \rho^{\gamma_2-1}\le P^{\prime}(\rho) \le \overline{C}_2 \gamma_2 \rho^{\gamma_2-1},
  \label{iq-lower-upper-bound-for-general-pressure-2}\\
  &C^{-1} \rho^{\gamma_2-1}\le e(\rho) \le C \rho^{\gamma_2-1},\quad C^{-1} \rho^{\gamma_2-2}\le e^{\prime}(\rho) \le C \rho^{\gamma_2-2},
  \label{iq-lower-upper-bound-for-general-internal-energy-2}\\
  &C^{-1} \rho^{\theta_2}\le \mathcal{K}(\rho) \le C \rho^{\theta_2},
  \label{iq-lower-upper-bound-for-k(rho)-2}
  \end{align}
\end{enumerate}
where, for $i=1,2$, $\,\,\theta_i=\frac{\gamma_i-1}{2}$, 
the positive constants $\underline{C}_i$ and $\overline{C}_i$ 
depend only on $\gamma_i$ respectively, 
and $C=C(\gamma_1,\gamma_2,\kappa_1,\kappa_2,\rho_{\star},\rho^{\star})>0$.
\end{lemma}

\begin{lemma}\label{lem-properties-for-relative-internal-energy}
The relative internal energy associated with the general pressure 
law  
\eqref{2.2b}--\eqref{eq-general-pressure-law-2b}
satisfies the following properties{\rm :}
 \begin{align}
  &e^*(\rho, \rho_{\infty}) \ge C_{\gamma_1}\rho(\rho^{\theta_1}-\rho_{\infty}^{\theta_1})^2\qquad\,\,\,\,\, 
  \text{ when $\rho\in \left(0,\rho_{\star}\right]$}, \label{iq-relative-internal-energy-control-general-pressure-law-1}\\
  &e^*(\rho, \rho_{\infty}) \ge C_{\gamma_2}\rho(\rho^{\theta_2}-\rho_{\infty}^{\theta_2})^2\qquad\,\,\,\,\, 
  \text{ when $\rho\in \left[\rho^{\star},\infty \right)$}, \label{iq-relative-internal-energy-control-general-pressure-law-2}\\
  &\rho^{\gamma_1}\le \bar C_{\gamma_1} \big(e^*(\rho, \rho_{\infty}) + \rho_{\infty}^{\gamma_1}\big)
    \qquad\qquad \text{when $\rho\in \left(0,\rho_{\star}\right]$},\label{iq-density-control-by-relative-internal-energy-general-pressure-law-1}\\
  &\rho^{\gamma_2}\le \bar C_{\gamma_2} \big(e^*(\rho, \rho_{\infty}) + \rho_{\infty}^{\gamma_2}\big)
  \qquad\qquad \text{when $\rho\in \left[\rho^{\star},\infty \right)$},\label{iq-density-control-by-relative-internal-energy-general-pressure-law-2}\\
  &\rho e(\rho) \le  \bar C \big(e^*(\rho, \rho_{\infty})+ \rho_{\infty}e(\rho_{\infty})\big)\quad \text{ when $\rho\in [\rho_{\star}, \rho^{\star}]$},\label{iq-density-control-by-relative-internal-energy-general-pressure-law-3}
  \end{align}
where, for $i=1,2$, 
$\,\,\theta_i=\frac{\gamma_i-1}{2}$,
$C_{\gamma_i} >0$
are some constants, $\bar C_{\gamma_i}$ are constants depending on $C_{\gamma_i}$, 
and constant $\bar C=C(\rho_{\star},\rho^{\star},\gamma_1,\kappa_1)>0$.
\end{lemma}

These properties can be proved by direct calculation, so we omit the proof.

\subsection{Stochastic Forcing}
In this section, we present precise conditions on the stochastic forcing terms under consideration. 

Let $(\Omega,\mathcal{F},(\mathcal{F}_t)_{t\ge 0}, \mathbb{P})$ be a stochastic basis 
with complete and right-continuous filtrations. 
Let $W(t)=\sum_{k\ge 1}\beta_k(t) e_k$, where $(\beta_k(t))_{k\ge 1}$ are independent Brownian motions 
relative to $(\mathcal{F}_t)_{t\ge 0}$ and $(e_k)_{k\ge 1}$ denotes a complete orthonormal system 
in a separable Hilbert space $\mathfrak{A}$. The noise coefficient $\Phi(\rho,m):\mathfrak{A}\to L^1(\mathbb{R})$ 
satisfying $\Phi(0,m)=0$ is defined by
  \begin{equation}\label{eq-def-noise-coefficient}
  \Phi(\rho,m)e_k=a_k\zeta_k(\cdot,\rho,m),
  \end{equation}
where $a_k$ are constants such that $|a_k|$ is monotone in $k$ and $|a_k|\to 0$ as $k\to \infty$, 
and $\zeta_k$ are continuous in $(\rho,m)$ so that they satisfy the following condition uniformly 
in $x\in \mathbb{R}$:
  \begin{equation}\label{iq-noise-coefficience-growth-condition-on-R-compact-supp}
  \begin{aligned}
  &\mathfrak{G}_1(x,\rho,m):=\Big(\sum_{k\ge 1}|a_k\zeta_k(x,\rho,m)|^2\Big)^{\frac12}\le C \Big(\sum_{k\ge 1}|\zeta_k(x,\rho,m)|^2\Big)^{\frac12}
  \le B_0\rho \Big(1+\big(\frac{m}{\rho}\big)^2 +e(\rho)\Big)^{\frac12},\\
  &\text{supp}_x (\zeta_k)\subset\mathbb{K}\Subset \mathbb{R},
  \end{aligned}
  \end{equation}
or
  \begin{equation}\label{iq-noise-coefficience-growth-condition-on-R}
  \begin{aligned}
  \mathfrak{G}_2(x,\rho,m):=\Big(\sum_{k\ge 1}|a_k\zeta_k(x,\rho,m)|^2\Big)^{\frac12}&\le C \Big(\sum_{k\ge 1}|\zeta_k(x,\rho,m)|^2\Big)^{\frac12}\\
  &\le B_0\rho \Big(\big(\frac{m}{\rho}\big)^2 +e(\rho)\Big)^{\frac12}\qquad \text{for any $x\in \mathbb{R}$},
  \end{aligned}
  \end{equation}
where $C>0$ and $B_0>0$ are constants.

We define the auxiliary space $\mathfrak{A} \subset \mathfrak{A}_0$ by
  $
  \mathfrak{A}_0=\big\{v=\sum_{k\ge 1}\alpha_ke_k\, : \, \sum_{k\ge 1}\frac{\alpha_k^2}{k^2} < \infty\big\},
  $
with the norm:
  $$
  \|v\|_{\mathfrak{A}_0}^2 =\sum_{k\ge 1}\frac{\alpha_k^2}{k^2}\qquad\quad \mbox{for } \,v=\sum_{k\ge1} \alpha_k e_k.
  $$
The embedding: $\mathfrak{A} \subset \mathfrak{A}_0$ is Hilbert-Schmidt.
The trajectories of $W(t)=\sum_{k\ge 1}\beta_k(t) e_k$ 
are $\mathbb{P}$-{\it a.s.} in $C([0,T];\mathfrak{A}_0)$ (see \cite{liuRockner15}).

\begin{remark}
A simple example of the noise coefficient $\Phi(\rho,m)$ satisfying the conditions above 
is given by 
  $$
  \Phi(\rho,m)e_k=\zeta_k(\cdot,\rho,m)=\begin{cases}
  a_1 \alpha(x)\rho(x)\quad &\mbox{for $k=1$},\\
  0,\quad &\mbox{for $k>1$},
  \end{cases}
  $$
where $\alpha(x)\in C_{\rm c}^{\infty}(\mathbb{R})$ is a given function.
\end{remark}

\begin{remark}\label{remark-simplified-model-for-stochastic-forcing}
To understand the definition of the stochastic forcing terms, 
we provide a simplified model of stochastic forcing terms here{\rm :}  
Let the noise be a one-dimensional {\rm (}$1$-{\rm D}{\rm )} Wiener 
process $($Brownian motion$)$,  {\it i.e.},
$W(t)=\beta_1(t)$.
In this case, the Hilbert space $\mathfrak{A}=\mathbb{R}$ is a $1$-{\rm D} space,
with the basis $e_1=1$. Then the noise coefficient is a function
  $
  \Phi(\rho,m, x):=a_1 \alpha(x)\rho,
  $
where $\alpha(x)\in C_{\rm c}^{\infty}(\mathbb{R})$ is a given function.
\end{remark}

\subsection{Martingale Solutions and Main Theorems}\label{sec-notation-solution-and-main-result}
In this subsection, we introduce the notion of martingale solutions with finite relative-energy
and the main theorems (Theorems 2.1--2.4) of this paper.

\subsubsection{Martingale solutions with finite relative-energy}
For the general pressure law 
\eqref{2.2b}--\eqref{eq-general-pressure-law-2b},
when the initial density has a positive far-field $\rho_{\infty}>0$, 
we consider the following martingale 
solutions.

\begin{definition}\label{def-martingale-entropy-solutions-with-relative-finite-energy-general-pressure}
We call
  $$
  (\Omega,\mathcal{F},(\mathcal{F}_t),\mathbb{P},\rho,m,W)
  $$
a \textit{martingale 
solution with finite relative-energy} of the Cauchy problem \eqref{eq-stochastic-Euler-system}
if
\begin{itemize}
\item[\rm (i)] $(\Omega,\mathcal{F},(\mathcal{F}_t),\mathbb{P})$ is a stochastic basis 
with filtration $(\mathcal{F}_t)$ satisfying the usual conditions{\rm ;}
\item[\rm (ii)] $W$ is a $\mathcal{F}_t$-cylindrical Wiener process{\rm ;}

\item[\rm (iii)] 
The density $\rho\ge 0$ is predictable and 
$\rho \in C_{\rm w}([0,T],L^{\gamma_2}_{\rm loc}(\mathbb{R}))$ $\mathbb{P}$-{\it a.s.}{\rm ;}

\item[\rm (iv)] The momentum $m$ is predictable and 
$m \in C_{\rm w}([0,T],L^{\frac{2\gamma_2}{\gamma_2+1}}_{\rm loc}(\mathbb{R}))$ $\mathbb{P}$-{\it a.s.}{\rm ;}
\item[\rm (v)] The relative-energy is bounded:
    \begin{equation}\label{iq-bdd-relative-energy-for-Euler}
    \mathbb{E}\big[ \big\|\int_{\mathbb{R}} \big( \frac12 \frac{m^2}{\rho} +e^*(\rho, \rho_{\infty})\big)\, \d x  \big\|_{L^{\infty}([0,T])}\big]
    \le C(T,E_{0,1}),
    \end{equation}
    where $E_{0,1}$ is the finite initial relative-energy{\rm ;}
\item[\rm (vi)] 
For any $\varphi \in C_{\rm c}^{\infty}((0,T)\times\mathbb{R})$,
the mass and momentum equations hold almost surely{\rm :}
$$
\begin{aligned}
&\int_0^T \int_{\mathbb{R}}  \rho \partial_t\varphi \, \d x \d t
+\int_0^T \int_{\mathbb{R}}  m \partial_x \varphi \, \d x \d t=0,\\
&\int_0^T \int_{\mathbb{R}}   m \partial_t \varphi\, \d x \d t
+ \int_0^T \int_{\mathbb{R}} \big( \frac{ m^2}{ \rho}+P( \rho) \big) \partial_x \varphi \,\d x \d t
+\int_0^T \int_{\mathbb{R}} \Phi( U)\,\varphi \, \d x \d  W(t)= 0.
\end{aligned}
$$
\end{itemize}
\end{definition}

When the initial data satisfy
  $
  (\rho_0,m_0) \to (0,0)\ \text{ as $|x|\to \infty$},
  $
we can formulate the definitions as above with apparent change.

\smallskip
\begin{remark}
In {\rm Definition \ref{def-martingale-entropy-solutions-with-relative-finite-energy-general-pressure}},
the solutions are required to belong to the space $C_{\rm w}([0,T];X)$ with $X$ being a Banach space, 
{\it i.e.}, the solutions are continuous under the weak topology of $X$. 
This ensures that the solutions are stochastic processes in the classical sense, 
which is different from the deterministic case where we usually obtain
that the solutions of the Euler system belong to $L^p_{\rm loc}(\mathbb{R}^2_+)$ for some $p>1$.
\end{remark}

\begin{remark}\label{remark-analysis-for-well-defined-stochastic-integral}
Using the embedding{\rm :} $L^1(\mathbb{R}) \hookrightarrow H^{-\ell}(\mathbb{R})$ 
for $\ell > \frac12$, 
we can check that, by conditions {\rm (iii)} and {\rm (v)} in
{\rm Definition \ref{def-martingale-entropy-solutions-with-relative-finite-energy-general-pressure}}, 
under assumption 
\eqref{iq-noise-coefficience-growth-condition-on-R-compact-supp}
or \eqref{iq-noise-coefficience-growth-condition-on-R} 
on the noise coefficient function $\Phi(\rho,m)\in L_2(\mathfrak{A}, H^{-\ell}(\mathbb{R}))$, 
{\it i.e.}, the Hilbert-Schmidt operator from $\mathfrak{A}$ to $H^{-\ell}(\mathbb{R})$,
the stochastic integral $\int_0^{\cdot} \Phi(\rho,m)\, \d W$ 
is a well-defined $(\mathcal{F}_t)$-martingale taking value in $H^{-\ell}(\mathbb{R})$. 
Indeed, take \eqref{iq-noise-coefficience-growth-condition-on-R-compact-supp} for example,
$$
\begin{aligned}
\|\Phi(\rho,m)\|_{L_2(\mathfrak{A},H^{-\ell}(\mathbb{R}))}^2 
&=\sum_{k\ge 1}\|a_k \zeta_k(\rho,m)\|_{H^{-\ell}(\mathbb{R})}^2
\\&
\le \sum_k \|a_k \zeta_k(\rho,m)\|_{L^1(\mathbb{R})}^2
\le B_0\int_{\mathbb{K}} \rho\, \d x \int_{\mathbb{K}} \big(\rho+\frac{m^2}{\rho} +\rho e(\rho)\big)\, \d x.
\end{aligned}
$$
For the general pressure law \eqref{2.2b}--\eqref{eq-general-pressure-law-2b},
if $\rho\in L^{\gamma_2}_{\rm loc}(\mathbb{R})$
and $\frac{m}{\sqrt{\rho}} \in L^2_{\rm loc}(\mathbb{R})$, then the claim follows
by using \eqref{iq-lower-upper-bound-for-general-internal-energy-2} 
and the fact that $\rho \le 1+ \rho^{\gamma_2}$.
\end{remark}

\begin{remark}
For the simplified model of stochastic forcing terms given 
in {\rm Remark \ref{remark-simplified-model-for-stochastic-forcing}}, we directly obtain that
$$
\begin{aligned}
  \int_0^t \|\Phi(\rho,m,\cdot)\|_{L^1}^2 \, \d \tau 
  =\int_0^t \Big(\int_{\text{\rm supp}_x (\alpha)}|a_1 \alpha
  |\rho\, \d x\Big)^2 \d \tau
  \le C(\alpha,t)\sup_{\tau\in[0,t]} \Big(\int_{\text{\rm supp}_x(\alpha)} \rho  \, \d x\Big)^2.
  \end{aligned}
$$
Then, by similar argument to that in {\rm Remark \ref{remark-analysis-for-well-defined-stochastic-integral}}, 
we see that the stochastic integral $\int_0^{\cdot} \Phi(\rho,m,x)\, \d W$ is well-defined.
\end{remark}

\subsubsection{Main Theorems: 
Theorem \ref{thm-well-posedness-for-euler-on-whole-space-deterministic-initial-data-general-pressure-law} 
and Theorem 
\ref{thm-better-well-posedness-for-euler-on-whole-space-deterministic-initial-data-general-pressure-law}}
For the general pressure law \eqref{2.2b}--\eqref{eq-general-pressure-law-2b},
when the initial density has a positive far-field $\rho_{\infty}>0$, 
we have the following main result:

\begin{theorem}\label{thm-well-posedness-for-euler-on-whole-space-deterministic-initial-data-general-pressure-law}
Assume that the initial data $(\rho_0, m_0)\in L^{\gamma_2}_{\rm loc}(\mathbb{R})\times L^1_{\rm loc}(\mathbb{R})$ 
have the finite relative-energy{\rm :}
$$
\begin{aligned}
\mathcal{E}(\rho_0,m_0):=\int_{\mathbb{R}} \big(\frac12 \frac{m_0^2}{\rho_0} +e^*(\rho_0,\rho_{\infty})\big)\, \d x <\infty.
\end{aligned}
$$
Then there exists a martingale 
solution with finite relative-energy to \eqref{eq-stochastic-Euler-system} in the sense of 
{\rm Definition \ref{def-martingale-entropy-solutions-with-relative-finite-energy-general-pressure}}.
Moreover, for any entropy pair $(\eta,q)$ defined in \eqref{eq-entropy-flux-representation} 
with generating function $\psi\in C^2(\mathbb{R})$ that is convex and satisfies
\begin{equation}\label{eq-definition-of-growth-for-weight-function-general-pressure}
\lim_{s\to \pm\infty } 
\frac{\psi(s)}{|s|^{2-\delta}}=0,\quad 
\lim_{s\to\pm\infty } 
\frac{\psi^{\prime}(s)}{|s|^{1-\delta}}=0
\qquad\,\, \text{for $\,2>\delta> \frac{3}{2}(1- \frac{1}{\gamma_2})$,}
\end{equation}
the following entropy inequality is almost surely satisfied for any nonnegative 
$\varphi \in C_{\rm c}^{\infty}((0,T)\times\mathbb{R})${\rm :}
\begin{equation}\label{iq-entropy-inequality-for-Euler-general-pressure}
\begin{aligned}
  &\int_0^T \int_{\mathbb{R}} \big( \eta(  U)\varphi_t+ q( U) \varphi_x \big) \,\d x \d t
  + \int_0^T \int_{\mathbb{R}}  \partial_{ m} \eta(U) \Phi(U)\, \varphi\,\d x \d W(t) \\[0.5mm]
  &\,+\frac12 \int_0^T \int_{\mathbb{R}} \partial_{ m}^2 \eta( U) \sum_k a_k^2 ( \zeta_k)^2( U) 
  \,\varphi\, \d x \d t
  \ge  0.
\end{aligned}
\end{equation}
\end{theorem}

\begin{remark}
For {\rm Theorem \ref{thm-well-posedness-for-euler-on-whole-space-deterministic-initial-data-general-pressure-law}}, 
$E_{0,1}$ in \eqref{iq-bdd-relative-energy-for-Euler} satisfies $E_{0,1}\equiv \mathcal{E}(\rho_0,m_0)$.
\end{remark}

\begin{remark}
For the polytropic pressure case, because of \eqref{iq-lower-growth-for-entropy-in-entropy-inequality}, 
we can obtain a better result than that in 
\eqref{iq-entropy-inequality-for-Euler-general-pressure}{\rm ;} 
see {\rm Definition \ref{def-martingale-weak-entropy-solution-with-relative-finite-energy-polytropic-gas}} 
for the details.
\end{remark}

Furthermore, we have the following main theorem:
\begin{theorem}\label{thm-better-well-posedness-for-euler-on-whole-space-deterministic-initial-data-general-pressure-law}
Assume that the initial data $(\rho_0, m_0)\in L^{\gamma_2}_{\rm loc}(\mathbb{R})\times L^1_{\rm loc}(\mathbb{R})$ 
satisfy $\mathcal{E}(\rho_0,m_0)<\infty$ and the following additional
finite relative higher-order energy condition{\rm :}
  $$\begin{aligned}
  &\mathcal{E}_{\diamond}(\rho_0,m_0):=\int_{\mathbb{R}} \big(\frac{1}{12}\frac{m_0^4}{\rho_0^3}+\frac{e(\rho_0)}{\rho_0}m_0^2+e_{\diamond}^*(\rho_0,\rho_{\infty}) \big)\, \d x <\infty.
  \end{aligned}$$
Then 
\begin{itemize}
\item[\rm (i)] There exists a martingale 
solution with finite relative-energy to \eqref{eq-stochastic-Euler-system} in the sense of 
{\rm Definition \ref{def-martingale-entropy-solutions-with-relative-finite-energy-general-pressure}}. 
Additionally, the entropy inequality \eqref{iq-entropy-inequality-for-Euler-general-pressure} is almost surely 
satisfied with \eqref{eq-definition-of-growth-for-weight-function-general-pressure} replaced by
  \begin{equation}\label{eq-definition-of-higher-growth-for-weight-function-general-pressure}
 \lim_{s\to \pm\infty } 
 \frac{\psi(s)}{|s|^{3-\delta}}=0,\quad 
 \lim_{s\to \pm\infty } 
 \frac{\psi^{\prime}(s)}{|s|^{2-\delta}}=0\,\,\,\qquad \mbox{for } \begin{cases}
  3>\delta> 2\big(1- \frac{1}{\gamma_2}\big) \quad &\text{for $\gamma_2 \le 2$},\\
  3> \delta > 1 \quad &\text{for $\gamma_2>2$}.
  \end{cases}
  \end{equation}
\item[\rm (ii)] Let $U^{\delta}:=(\rho^{\delta},m^{\delta})$ be a sequence of 
martingale entropy solutions to \eqref{eq-stochastic-Euler-system} obtained in {\rm (i)} 
corresponding to a sequence of initial data $(\rho_0^{\delta}, m_0^\delta)$ satisfying 
$\mathcal{E}(\rho_0^{\delta}, m_0^\delta)+\mathcal{E}_{\diamond}(\rho_0^{\delta}, m_0^\delta) \le C$,
where constant $C$ is independent of $\delta$. Then 
\begin{itemize}
\item[\rm (a)] $U^{\delta}$ satisfies
that, for any $K\Subset \mathbb{R}$, there exists $C(K,p,T)>0$ 
independent of $\delta$ such that
$$
  \begin{aligned}
  \E \big[\Big( \int_0^t \int_{K} \big(\rho^{\delta} P(\rho^{\delta})+ \frac{|m^{\delta}|^3}{(\rho^{\delta})^2} + \eta_{\diamond}(\rho^{\delta},m^{\delta})
  \big) \,\d x\d \tau \Big)^p\big]
  \le  C(K,p,T) \qquad \mbox{for any $t\in[0,T]$},
  \end{aligned}
$$
so that $U^{\delta}:=(\rho^{\delta},m^{\delta})$ are tight in 
$L^{\gamma_2+1}_{\rm w,loc}(\mathbb{R}^2_+) \times L^{\frac{3(\gamma_2+1)}{\gamma_2+3}}_{\rm w,loc}(\mathbb{R}^2_+)$ 
and admit the Skorokhod representation $(\tilde \rho^{\delta}, \tilde m^{\delta})$.
\item[\rm (b)]
When $\gamma_2 <2$, there exist random variables $(\tilde\rho,\tilde m)$ such that 
almost surely $(\tilde \rho^{\delta},\,\tilde m^{\delta}) \to (\tilde \rho,\, \tilde m)$ 
almost everywhere and $($up to a subsequence$)$ almost surely{\rm :}
$$
(\tilde \rho^{\delta},\,\tilde m^{\delta}) \to (\tilde \rho,\, \tilde m)\qquad 
\text{in $L^{\bar p}_{\rm loc}(\mathbb{R}^2_+)\times L^{\bar q}_{\rm loc}(\mathbb{R}^2_+)$}
$$
for $\bar p\in [1, \gamma_2+1 )$ and $\bar q\in [ 1, \frac{3(\gamma_2+1)}{\gamma_2+3} )$, and 
$(\tilde\rho,\tilde m)$ is also a martingale entropy solution to \eqref{eq-stochastic-Euler-system}, 
{\it i.e.}, almost surely satisfies the entropy inequality \eqref{iq-entropy-inequality-for-Euler-general-pressure} with \eqref{eq-definition-of-growth-for-weight-function-general-pressure} replaced by \eqref{eq-definition-of-higher-growth-for-weight-function-general-pressure}.
\end{itemize}
\end{itemize}
  
\end{theorem}
\begin{remark}
When $\gamma_2 <2$, we can choose $1>\delta > 2(1-\frac{1}{\gamma_2})$. In this case, we obtain 
the entropy inequality for the mechanical energy, {\it i.e.}, the physical entropy.
\end{remark}

\medskip
In fact, the two results above can be more general, extending to the case of random initial data.

\begin{theorem}\label{thm-well-posedness-for-euler-on-whole-space-general-pressure-law}
Assume that $\Im$ is a Borel probability measure on $L^{\gamma_2}_{\rm loc}(\mathbb{R})\times L^1_{\rm loc}(\mathbb{R})$ satisfying
\begin{align*}
  &{\rm supp}\, \Im=\big\{(\rho,m)\in \mathbb{R}^2_+\,:\, \rho \ge 0\big\},\\[1mm]
  &E_{p_0}(\Im):=
  \int_{L^{\gamma_2}_{\rm loc}(\mathbb{R})\times L^1_{\rm loc}(\mathbb{R})} \big| 
  \mathcal{E}(\rho,m)\big|^{p_0}\, \d \Im(\rho,m) <\infty  \qquad\mbox{for some $p_0\ge 6$}.
\end{align*}
Then there exists a martingale
solution with finite relative-energy to \eqref{eq-stochastic-Euler-system} in the sense of {\rm Definition \ref{def-martingale-entropy-solutions-with-relative-finite-energy-general-pressure}} 
with initial law $\Im$, {\it i.e.}, there exists a $\mathcal{F}_0$-measurable random variable $(\rho_0,m_0)$ such that $\Im=\mathbb{P}\circ (\rho_0,m_0)^{-1}$,
that is, $\Im(\cdot)=\mathbb{P}\{\omega\,:\, (\rho_0,m_0)(\omega)\in \cdot\}$, 
for which we denote $E_{0,p_0}:=E_{p_0}(\Im)$.
Moreover, the same entropy inequality as \eqref{iq-entropy-inequality-for-Euler-general-pressure} is almost surely satisfied.
\end{theorem}

\begin{theorem}\label{thm-better-well-posedness-for-euler-on-whole-space-general-pressure-law}
Assume that $\Im$ is a Borel probability measure on $L^{\gamma_2}_{\rm loc}(\mathbb{R})\times L^1_{\rm loc}(\mathbb{R})$ satisfying
  $$\begin{aligned}
  &{\rm supp}\, \Im=\big\{(\rho,m)\in \mathbb{R}^2_+\,:\, \rho \ge 0\big\},\\[1mm]
  &E_{2}(\Im)=\int_{L^{\gamma_2}_{\rm loc}(\mathbb{R})\times L^1_{\rm loc}(\mathbb{R})} 
  \big| \mathcal{E}(\rho,m)\big|^{2}\, \d \Im(\rho,m) <\infty,\\
  &\int_{L^{\gamma_2}_{\rm loc}(\mathbb{R})\times L^1_{\rm loc}(\mathbb{R})} 
  \big| \mathcal{E}_{\diamond}(\rho,m)\big|^{2}\, \d \Im(\rho,m) <\infty.
  \end{aligned}$$
Then 
\begin{itemize}
\item[\rm (i)] There exists a martingale 
solution with finite relative-energy to \eqref{eq-stochastic-Euler-system} in the sense of {\rm Definition \ref{def-martingale-entropy-solutions-with-relative-finite-energy-general-pressure}} with initial law $\Im$. 
Additionally, the entropy inequality \eqref{iq-entropy-inequality-for-Euler-general-pressure} is almost surely satisfied with \eqref{eq-definition-of-growth-for-weight-function-general-pressure} replaced by \eqref{eq-definition-of-higher-growth-for-weight-function-general-pressure}{\rm ;}

\item[\rm (ii)] Let $U^{\delta}:=(\rho^{\delta},m^{\delta})$ be a sequence of martingale entropy solution to
\eqref{eq-stochastic-Euler-system} obtained in {\rm (i)} corresponding to a sequence of initial law 
$\Im^{\delta}$ satisfying 
$$
\begin{aligned}
  \int_{L^{\gamma_2}_{\rm loc}(\mathbb{R})\times L^1_{\rm loc}(\mathbb{R})} \big| \mathcal{E}(\rho,m)\big|^{2}\, \d \Im^{\delta}(\rho,m)+\int_{L^{\gamma_2}_{\rm loc}(\mathbb{R})\times L^1_{\rm loc}(\mathbb{R})} \big| \mathcal{E}_{\diamond}(\rho,m)\big|^{2}\, \d \Im^{\delta}(\rho,m) \le C,
\end{aligned}
$$
where constant $C>0$ is independent of $\delta$. 
Then $U^{\delta}$ satisfies that, for any $K\Subset \mathbb{R}$, 
there exists $C(K,p,T)>0$ independent of $\delta$ such that
$$
  \begin{aligned}
  \E \big[\Big( \int_0^t \int_{K} \big(\rho^{\delta} P(\rho^{\delta})+ \frac{|m^{\delta}|^3}{(\rho^{\delta})^2} + \eta_{\diamond}(\rho^{\delta},m^{\delta})
  \big) \,\d x\d \tau \Big)^p\big]
  \le  C(K,p,T) \qquad\mbox{ for any $t\in[0,T]$},
  \end{aligned}
$$
so that $U^{\delta}:=(\rho^{\delta},m^{\delta})$ are tight in 
$L^{\gamma_2+1}_{\rm w,loc}(\mathbb{R}^2_+) \times L^{\frac{3(\gamma_2+1)}{\gamma_2+3}}_{\rm w,loc}(\mathbb{R}^2_+)$ 
and admit the Skorokhod representation $(\tilde \rho^{\delta}, \tilde m^{\delta})$. Moreover, when $\gamma_2 <2$, there exist random variables $(\tilde\rho,\tilde m)$ such that 
almost surely $(\tilde \rho^{\delta},\,\tilde m^{\delta}) \to (\tilde \rho,\, \tilde m)$ 
almost everywhere and $($up to a subsequence$)$ almost surely,
$$
(\tilde \rho^{\delta},\,\tilde m^{\delta}) \to (\tilde \rho,\, \tilde m)\qquad\,\, 
\text{in $L^{\bar p}_{\rm loc}(\mathbb{R}^2_+)\times L^{\bar q}_{\rm loc}(\mathbb{R}^2_+)$}
$$
for $\bar p\in [1, \gamma_2+1 )$ and $\bar q\in [ 1, \frac{3(\gamma_2+1)}{\gamma_2+3} )$. 
In addition, $(\tilde\rho,\tilde m)$ is also an entropy solution to \eqref{eq-stochastic-Euler-system}, 
{\it i.e.}, almost surely satisfies the entropy inequality \eqref{iq-entropy-inequality-for-Euler-general-pressure} 
with \eqref{eq-definition-of-growth-for-weight-function-general-pressure} replaced by \eqref{eq-definition-of-higher-growth-for-weight-function-general-pressure}.
\end{itemize}
\end{theorem}

\begin{remark}
When $p_0=1$, $E_1(\Im):=E_{0,1}$ has appeared in \eqref{iq-bdd-relative-energy-for-Euler}.
\end{remark}

\begin{remark}
Since {\rm Theorem \ref{thm-well-posedness-for-euler-on-whole-space-deterministic-initial-data-general-pressure-law}} 
and {\rm Theorem \ref{thm-better-well-posedness-for-euler-on-whole-space-deterministic-initial-data-general-pressure-law}} are the direct corollaries of
{\rm Theorem \ref{thm-well-posedness-for-euler-on-whole-space-general-pressure-law}} 
and {\rm Theorem \ref{thm-better-well-posedness-for-euler-on-whole-space-general-pressure-law}}, respectively,
it suffices 
to prove
{\rm Theorem \ref{thm-well-posedness-for-euler-on-whole-space-general-pressure-law}} and {\rm Theorem \ref{thm-better-well-posedness-for-euler-on-whole-space-general-pressure-law}}.
\end{remark}

\begin{remark}
In {\rm Theorem \ref{thm-well-posedness-for-euler-on-whole-space-general-pressure-law}}, we require $p_0\ge 6$ in order to pass to the limit in the entropy flux and derive the entropy inequality{\rm ;} 
see {\rm Propositions \ref{prop-uniform-estimates-rho-gamma+1}} and
{\rm \ref{prop-entropy-inequality-general-pressure-law}}--{\rm \ref{prop-weak-solution-for-general-pressure-law}}.
\end{remark}

When the initial data satisfy
$(\rho_0,m_0) \to (0,0)\ \text{ as  $|x|\to \infty$}$,
we can also formulate the theorems as above with apparent changes.

\medskip
In this paper, we require additional notation for subsequent development.

From now on, 
we define $\theta(\rho)=\theta_1$ when $\rho\le \rho_{\star}$ and $\theta(\rho)=\theta_2$ when $\rho\ge \rho^{\star}$. 
We use the notation $(\rho,m)\in H$ for convenience instead of $(\rho,m)\in H\times H$, 
where $H$ is any function space. 
We use $\lfloor a \rfloor$ to denote the largest integer that no more than $a$. 
The transpose of a vector $U$ is denoted by $U^{\top}$. 
$\hat K \Subset \mathbb{R}$ means that $\hat K$ is a compact subset of $\mathbb{R}$. 
Denote $\,\overset{*}{\rightharpoonup}\,$ as the weak-* convergence and $\,\rightharpoonup\,$ as the weak convergence. 
$\mathcal{D}^{\prime}$ denotes the space of distributions. We use $B_r \subset \mathbb{R}^d$ 
to denote an open ball with radius $r$ and centered at the origin.
$\mathcal{L}(f)$ denotes the law of $f$. 
The vacuum region is denoted by 
$\mathcal{V}:=\{(\rho,m)\in \mathbb{R}^2_+\, :\, \rho=0\}$, 
and the non-vacuum set is denoted by 
$\mathfrak{T}=\{(\rho,m)\in \mathbb{R}^2_+\, :\, \rho>0\}$. 
The Sobolev space $W^{k,2}$ is denoted by $H^k$. 
We denote the dual of $W^{1,\infty}_0(\mathcal{O})$ by $W^{-1,1}(\mathcal{O})$, where $\mathcal{O}$ is an open bounded set. 
We use $C_0(\bar{\mathcal{O}})$ to denote the space of functions that 
are continuous on $\bar{\mathcal{O}}$ and vanish on the boundary $\partial \mathcal{O}$. The norm of $C_0(\bar{\mathcal{O}})$ is $\|\cdot\|_{L^{\infty}}$.
Denote $X^*$ as the dual space of $X$, and
$X_{\rm w}$
as the topological space $X$ equipped with the weak topology. 

If $X$ be a Banach space, then $C_{\rm w}([0,T],X)$ denotes the set of functions $f\in L^{\infty}([0,T];X)$ 
that are continuous in $t\in [0,T]$ under the weak topology, 
{\it i.e.}, for any sequence $(t_k)_k \subset [0,T]$ satisfying $t_k\to t$,
  $$
  \lim_{t_k \to t} \langle g, f(t_k) \rangle =\langle g, f(t)\rangle 
  \qquad \mbox{for any $g\in X^*\,$, {\it a.e.} $t\in [0,T]$}.
  $$
Also $f_n \to f$ in $C_{\rm w}([0,T],X)$ means that $\sup_{t\in [0,T]}|\langle g, f_n-f \rangle |\to 0$
for any $g \in X^*$ (see also \cite[Definition 1.8.1]{BFHbook18}). Denote $C^{\alpha}(0,T;X)$ as
the set of functions $f:[0,T]\to X$ that are $\alpha$-H\"older continuous under the norm topology.

For $0<s_1 <1$ and $1\le r_1 < \infty$, denote
  $$
  W^{s_1,r_1}(0,T; X):=\Big\{ f\,:\,f\in L^{r_1}(0,T;X)\ \text{and}\ 
  \int_0^T\int_0^T \frac{\|f(t)-f(\tau)\|_X^{r_1}}{|t-\tau|^{r_1s_1+1}}\, \d t \d \tau < \infty \Big\},
  $$
where $X$ is a Banach space. The norm of this space is defined by
  $$
  \|f\|_{W^{s_1,r_1}(0,T; X)}:=\big(\|f\|_{L^{r_1}}^{r_1}+\|f\|_{\dot{W}^{s_1,r_1}}^{r_1}\big)^{\frac{1}{r_1}},
  $$
with semi-norm 
$\|f\|_{\dot{W}^{s_1,r_1}}:=\Big(\int_0^T\int_0^T \frac{\|f(t)-f(\tau)\|_X^{r_1}}{|t-\tau|^{r_1s_1+1}}\,
 \d t \d \tau \Big)^{\frac{1}{r_1}}$.

\section{Stochastic Parabolic Approximation}\label{sec-parabolic-approximation}
For any fixed $\varepsilon\in (0,1]$, we consider the following approximation equations 
to \eqref{eq-stochastic-Euler-system}:
  \begin{equation}\label{eq-parabolic-approximation}
  \begin{cases}
  \d U^{\varepsilon} + \partial_x F(U^{\varepsilon}) \d t =\varepsilon \partial_{xx} U^{\varepsilon}\d t
  + \Psi^{\varepsilon}(U^{\varepsilon}) \d W(t),\\
  U^{\varepsilon}|_{t=0}=U^{\varepsilon}_0,
  \end{cases}
  \end{equation}
with
  $
  \Psi^{\varepsilon}(U^{\varepsilon})
  =(0\, ,\, \Phi^{\varepsilon}(U^{\varepsilon}))^{\top},
  $
where $\Phi^{\varepsilon}(U^{\varepsilon}):\mathfrak{A}\to L^1(\mathbb{R})$ is defined as
  $$
  \Phi^{\varepsilon}(\rho^{\varepsilon},m^{\varepsilon})e_k
  =a_k \zeta^{\varepsilon}_k(\,\cdot\,,\rho^{\varepsilon},m^{\varepsilon})
  $$
with
  $$
  \zeta_k^{\varepsilon}=\begin{cases}
  (\zeta_k\jmath^{\varepsilon})*\Upsilon^{\varepsilon}\qquad\quad &\text{if}\ k\le \lfloor \frac{1}{\varepsilon} \rfloor, \\
  0\quad &\text{if}\ k> \lfloor \frac{1}{\varepsilon} \rfloor,
  \end{cases}
   $$
such that
$\Upsilon^{\varepsilon}$ is the approximation to the identity and $\jmath^{\varepsilon}$ is a truncation. 
More precisely,
$$
\Upsilon^{\varepsilon}(x,\rho,m)=\frac{1}{\varepsilon^3}I_1(\frac{x}{\varepsilon}) 
I_2(\frac{\rho}{\varepsilon})I_3(\frac{m}{\varepsilon}),
$$
where $I_1$ and $I_3$ are the approximations to the identity on $\mathbb{R}$, $I_2$ is the approximation 
to the identity on $\mathbb{R}$ with $\text{supp}\, I_2 \subset [0,1]$. 
Here we set $\zeta_k=0$ for $\rho\le 0$ to define the convolution with respect to $\rho$, which is consistent with 
assumption \eqref{iq-noise-coefficience-growth-condition-on-R-compact-supp} 
or \eqref{iq-noise-coefficience-growth-condition-on-R}:
  \begin{itemize}
  \item when $\zeta_k$ satisfies \eqref{iq-noise-coefficience-growth-condition-on-R-compact-supp},
    $
    \jmath^{\varepsilon}:=\jmath_1^{\varepsilon}(\rho,m),
    $
  \item when $\zeta_k$ satisfies \eqref{iq-noise-coefficience-growth-condition-on-R},
    $
    \jmath^{\varepsilon}:=\jmath_1^{\varepsilon}(\rho,m)\jmath_2^{\varepsilon}(x),
    $
  \end{itemize}
where 
$\text{supp}(\jmath_1^{\varepsilon}(\rho,m))\subset \Gamma_{\mathcal{H}^{\varepsilon}}:=\{(\rho,m)\,:\, -\mathcal{H}^{\varepsilon}\le w_1 \le w_2 \le \mathcal{H}^{\varepsilon} \}$ 
for the Riemann invariants $(w_1,w_2)$ as defined in
\eqref{eq-riemann-invariants-representation-formula}, and
$\jmath_2^{\varepsilon}(x)= \jmath(\varepsilon x)$ with $\jmath(x)=1$ if $|x|\le \frac12$ and $\jmath(x)=0$ if $|x|> 1$. 
We set $\mathcal{H}^{\varepsilon}=c_1 \varepsilon^{-\alpha_1}$, where $c_1>0$ is a fixed constant, and $\alpha_1$ satisfies (see Lemma \ref{prop-tightness-of-partial-x-eta})
$$
  \begin{cases}
  \alpha_1 < \theta_2 
  \qquad\quad &\text{when $\gamma_1\le 2$},\\
  \alpha_1 < \min\{\frac12, \theta_2\}\quad &\text{when $\gamma_1> 2$}
  \end{cases}
$$
for the general pressure law case. 
Letting $\varepsilon$ be sufficiently small, we can always choose
  \begin{equation}\label{eq-invariant-region-H-varepsilon}
  \mathcal{H}^{\varepsilon}\ge 1+(\rho_{\infty}+1)^{\theta_2}
  \end{equation}
such that the initial 
data $U^{\varepsilon}_0\in \Gamma_{\mathcal{H}^{\varepsilon}}$
(see Theorem \ref{thm-wellposedness-parabolic-approximation-R} 
and Proposition \ref{prop-L-infty-estimates-by-invariant-region}). 
Additionally, we also choose $\varepsilon$ sufficiently small such that (see Lemma \ref{prop-tightness-of-partial-x-eta})
  \begin{equation}\label{eq-invariant-region-H-varepsilon-control-rho-star}
  \mathcal{H}^{\varepsilon} \ge \rho_{\star}.
  \end{equation}
Denote $\Lambda^{\varepsilon}=\{x\,:\, |x|< \frac{1}{\varepsilon}\}$
and 
$\mathbb{K}^{\varepsilon}=\{y\,:\, |y-x|<\varepsilon\,\,\,\mbox{for any $x\in \mathbb{K}$}\} 
\subset \mathbb{K}^1$. 
Then, defining $\mathfrak{G}_1^{\varepsilon}$ and $\mathfrak{G}_2^{\varepsilon}$ 
as in 
\eqref{iq-noise-coefficience-growth-condition-on-R-compact-supp}--\eqref{iq-noise-coefficience-growth-condition-on-R} with $\zeta_k$ replaced by $\zeta_k^{\varepsilon}$, we have
  \begin{equation}\label{condition-compact-support-of-noise-coefficient}
  \begin{aligned}
  \text{supp}\, \mathfrak{G}_1^{\varepsilon} \subset 
  \mathbb{K}^1 \times \Gamma_{\mathcal{H}^{\varepsilon}},
  \qquad
  \text{supp}\, \mathfrak{G}_2^{\varepsilon} \subset \Lambda^{\varepsilon} \times \Gamma_{\mathcal{H}^{\varepsilon}},
  \end{aligned}
  \end{equation}
  \begin{equation}\label{iq-noise-coefficience-growth-conditions-for-parabolic-approximation}
  \begin{aligned}
  \mathfrak{G}_1^{\varepsilon}(x,\rho,m)\le B_0\rho\Big(1+\varepsilon^2+\big(\frac{m}{\rho}\big)^2 +e(\rho)\Big)^{\frac12},\,\,
  \mathfrak{G}_2^{\varepsilon}(x,\rho,m)\le B_0\rho\Big(\varepsilon^2+\big(\frac{m}{\rho}\big)^2 +e(\rho)\Big)^{\frac12}.
  \end{aligned}
  \end{equation}
Additionally, we have the following Lipschitz condition:
  $$
  \sum_{k}|\zeta_k^{\varepsilon}(U_1)-\zeta_k^{\varepsilon}(U_2)|^2 \le C(\varepsilon)|U_1-U_2|^2
  \qquad\,\, \mbox{for any $x\in \mathbb{R}$ and $U_1,U_2 \in \mathbb{R}_{+}\times \mathbb{R}$}.
  $$
From now on, to be clearer, we use the following notation:
  $$
  \begin{aligned}
  D_x \zeta_k^{\varepsilon}(x, \rho,m):=\partial_x \zeta_k^{\varepsilon}
  + \nabla_{\rho,m} \zeta_k^{\varepsilon}\cdot (\partial_x \rho, \partial_x m),\qquad
  D_x^j \zeta_k^{\varepsilon}(x, \rho,m):=D_x^{j-1}(D_x \zeta_k^{\varepsilon}).
  \end{aligned}
  $$

\subsection{Well-Posedness of the Stochastic Parabolic 
Approximation}\label{sec-well-posedness-of-parabolic-approximation}
We first state our hypotheses for the far-field condition of the initial data 
and the corresponding growth condition for the noise coefficient function.

\begin{hypothesis}\label{hypo-for-far-field-of-initial-data-and-noise-coefficient}
There are three cases for the initial data and the noise coefficient functions 
of \eqref{eq-stochastic-Euler-system}{\rm :}
  \begin{itemize}
  \item {\rm Case 1:}  $(\rho_0,u_0) \to (0,0)$ as $|x|\to \infty$, and $\zeta_k$ satisfy \eqref{iq-noise-coefficience-growth-condition-on-R}.
  \item {\rm Case 2:}  $(\rho_0,u_0) \to (0,0)$ as $|x|\to \infty$, and $\zeta_k$ satisfy \eqref{iq-noise-coefficience-growth-condition-on-R-compact-supp}.
  \item {\rm Case 3:} $(\rho_0,u_0) \to (\rho_{\infty},0)$ as $|x|\to \infty$ with $\rho_\infty>0$, 
  and $\zeta_k$ satisfy \eqref{iq-noise-coefficience-growth-condition-on-R-compact-supp}.
  \end{itemize}
\end{hypothesis}
Correspondingly, we have the following hypotheses:
\begin{hypothesis}\label{hypo-for-far-field-of-initial-data-and-noise-coefficient-parabolic-approximation}
We approximate the initial data of parabolic approximation \eqref{eq-parabolic-approximation} 
as follows{\rm :}
  \begin{itemize}
  \item $\rho^{\varepsilon}_0 \to \rho_{\infty}(\varepsilon)$ as $|x|\to \infty$ for {\rm Cases 1--2},
  so that $\rho_{\infty}(\varepsilon)=\varepsilon^{\alpha_0}$ for some $\alpha_0>0$ satisfying
      $\alpha_0 \gamma_2 >1$, and
    \begin{equation}\label{iq-rho-infinity-varepsilon-sufficiently-small}
      \rho_{\infty}(\varepsilon)<\rho_{\star},
      \end{equation}
      by letting $\varepsilon$ sufficiently small.
   
  \item $\rho^{\varepsilon}_0 \to \rho_{\infty}$ as $|x|\to \infty$ for {\rm Case 3}.
  \end{itemize}
\end{hypothesis}
\noindent
Then we have the following existence and uniqueness of solutions 
to the parabolic approximation equations \eqref{eq-parabolic-approximation},
where $\rho_{\infty}$ is used to denote constant $\rho_{\infty}$ or $\rho_{\infty}(\varepsilon)$ 
for simplicity.

\begin{theorem}\label{thm-wellposedness-parabolic-approximation-R}
Assume that $(\rho^{\varepsilon}_0-\rho_{\infty},u^{\varepsilon}_0)\in H^5({\mathbb{R}})$ 
$\mathbb{P}$-{\it a.s.}$, \rho^{\varepsilon}_0(\omega)\ge c_0 >0$ {\it a.e.} 
in $\mathbb{R}$ for all $\omega \in \Omega$, 
and $U^{\varepsilon}_0$ satisfies
$$
\begin{aligned}
  &U^{\varepsilon}_0\in \Gamma_{\mathcal{H}^{\varepsilon}} \ \mathbb{P}\text{-{\it a.s.}},\\
  &\E \big[\Big(\int_{\mathbb{R}} \big(\frac12 \rho^{\varepsilon}_0 (u^{\varepsilon}_0)^2 +e^*(\rho^{\varepsilon}_0,\rho_{\infty})\big)\, \d x \Big)^p\big] 
  \le C \E \big[\Big(\int_{\mathbb{R}} \big(\frac12 \rho_{0} u_{0}^2 +e^*(\rho_{0},\rho_{\infty})\big)\,
  \d x\Big)^p\big]\quad\mbox{for $\,1\le p< \infty$}.
  \end{aligned}
$$
Then the parabolic approximation equations \eqref{eq-parabolic-approximation} admit a global unique 
strong solution $U^{\varepsilon}$ that is a predictable process and satisfies that, for any $T>0$,
  \begin{itemize}
  \item[\rm (i)] $(\rho^{\varepsilon}-\rho_{\infty},u^{\varepsilon }) \in C([0,T];H^3(\mathbb{R}))\cap L^2([0,T];H^4(\mathbb{R}))$ $\mathbb{P}$-{\it a.s.}{\rm ;}
  \item[\rm (ii)] $U^{\varepsilon}\in \Gamma_{\mathcal{H}^{\varepsilon}}$ and $\|\frac{1}{\rho^{\varepsilon}}\|_{L^{\infty}([0,T]\times \mathbb{R})}\le e^{C(T, c_0, \varepsilon, \mathcal{H}^{\varepsilon})}$ $\mathbb{P}$-{\it a.s.}{\rm ;}
  \item[\rm (iii)] The parabolic approximation equations \eqref{eq-parabolic-approximation} hold in the classical strong sense{\rm ;}
  \item[\rm (iv)] The following uniform energy estimate holds for $1\le p< \infty${\rm :}
    \begin{align}
    &\mathbb{E} \big[\Big(\sup_{t\in [0,T]}\int_{\mathbb{R}} \big(\frac12 \rho^{\varepsilon} (u^{\varepsilon})^2 +e^*(\rho^{\varepsilon},\rho_{\infty})\big)\, \d x + \varepsilon \int_0^T \int_{\mathbb{R}} \big(\big(\rho^{\varepsilon} e(\rho^{\varepsilon})\big)^{\prime \prime}(\partial_x\rho^{\varepsilon})^2 
    + \rho^{\varepsilon}(\partial_x u^{\varepsilon})^2\big)\, \d x \d t \Big)^p\big]\nonumber\\[1mm]
    &\le C(p,T,E_{0,p}),\label{iq-energy-estimate-on-whole-space}
    \end{align}
    where $E_{0,p}=\mathbb{E}\big[ \big(\int_{\mathbb{R}} \big(\frac12 \rho_{ 0} u_{ 0}^2 +e^*(\rho_{ 0 },\rho_{\infty})\big)\, \d x \big)^p\big]$, 
    and $C(p,T,E_{0,p})$ is independent of $\varepsilon$.
  \end{itemize}
\end{theorem}

The proof of Theorem \ref{thm-wellposedness-parabolic-approximation-R} is postponed 
to \S \ref{sec-proof-of-well-posedness-of-parabolic-approximation}. 
In this section below, we first derive some essential uniform estimates, 
which will be used to establish the compensated compactness framework.

\subsection{Uniform Estimates of the Stochastic Parabolic 
Approximation}\label{sec-uniform-estmates-for-parabolic-approximation}

For simplicity, in this subsection, we omit superscript $\varepsilon$ whenever no confusion arises.
First we rewrite \eqref{eq-parabolic-approximation} as
  \begin{equation}\label{eq-parabolic-approximation-seperate}
  \begin{cases}
  \d \rho + \partial_x m\,\d t =\varepsilon \partial_x^2 \rho\,\d t,\\[0.5mm]
  \d m + \partial_x\big(\frac{m^2}{\rho} +P(\rho)\big)\,\d t=\varepsilon \partial_x^2 m\,\d t 
  +\Phi^{\varepsilon}(\rho,m)\,\d W(t),\\[0.5mm]
  (\rho, m)|_{t=0}=(\rho_0,m_0).
  \end{cases}
  \end{equation}

\subsubsection{Higher integrability of the density}\label{sec-higher-itegrability-for-density}
\begin{proposition}\label{prop-uniform-estimates-rho-gamma+1}
The solution of the Cauchy problem \eqref{eq-parabolic-approximation} obtained 
in {\rm Theorem \ref{thm-wellposedness-parabolic-approximation-R}} satisfies that,
for any $K\Subset \mathbb{R}$,
there exists $C({\gamma_2},K,p,T,E_{0,3p})>0$ independent of $\varepsilon$ such that
  \begin{equation}\label{iq-uniform-estimate-rho-gamma+1}
  \begin{aligned}
  \E \big[\Big( \int_0^t \int_{K} \rho P(\rho) \,\d x\d \tau \Big)^p\big]
  \le  C(\gamma_2,K,p,T,E_{0,3p})\qquad\mbox{for any $t\in[0,T]$}.
  \end{aligned}
  \end{equation}
\end{proposition}

\begin{proof}
Let $w(x)\in C_{\rm c}^{\infty}(\mathbb{R})$ be an arbitrary function 
satisfying $0\le w(x)\le 1$ and $w\equiv 1$ on $K$. Theorem \ref{thm-wellposedness-parabolic-approximation-R}(iii) 
guarantees that, for any fixed $x$, the equations in \eqref{eq-parabolic-approximation} hold 
as a finite-dimensional stochastic differential system. 
This then enables us to calculate similarly to that in the deterministic case \cite[Lemma 3.3]{chenperepelitsa15CMP}, 
especially integrate the equation for $wm$ on $(-\infty,x)$ and finally obtain 
$$
\begin{aligned}
  &\int_0^r \int_{\mathbb{R}} w^2\rho P(\rho)\,\d x\d \tau\\
  &=-\int_0^r \int_{\mathbb{R}} \,\d \Big(w\rho\int_{-\infty}^x wm \,\d y\Big)\,\d x \d \tau+ \varepsilon \int_0^r \int_{\mathbb{R}} w^2 \rho^2 \partial_x u \,\d x\d \tau
  - \varepsilon \int_0^r \int_{\mathbb{R}} \partial_xw \rho w m \,\d x\d \tau\\
  &\quad-\int_0^r \int_{\mathbb{R}} \partial_x w m \Big(\int_{-\infty}^x wm \,\d y\Big)\d x\d \tau 
  +\varepsilon \int_0^r \int_{\mathbb{R}}\partial_x^2 w \rho \Big(\int_{-\infty}^x wm \,\d y\Big)\d x\d \tau\\
  &\quad+ \int_0^r \int_{\mathbb{R}}f_1 \rho w\,\d x\d \tau
  + \int_0^r \int_{\mathbb{R}}w\rho \int_{-\infty}^x w\Phi^{\varepsilon}(\rho,m)\,\d y \d x\d W(\tau)\\
  &:= \sum_{i=1}^7 J_i,
  \end{aligned}
$$
where
  $$
  f_1:=\int_{-\infty}^x \Big(\partial_y w \big(\frac{m^2}{\rho} +P(\rho)\big)
  -\varepsilon \partial_y w \partial_y m\Big) \,\d y.
  $$

For the stochastic integral $J_7$, by the Burkholder-Davis-Gundy inequality, we have
  $$\begin{aligned}
  &\E \big[\sup_{r\in[0,t]}\Big| \int_0^r \int_{\mathbb{R}}w\rho 
  \int_{-\infty}^x \Phi^{\varepsilon}(\rho,m)w\,\d y \,\d x\d W(\tau) \Big|^p\big]\\
  &\le  \E \big[  \Big( \int_0^t \sum_k \Big(\int_{\mathbb{R}}w\rho \int_{-\infty}^x 
  a_k 
  \zeta_k^{\varepsilon}(\rho,m)w \,\d y \d x  \Big)^2 \,\d \tau\Big)^{\frac{p}{2}} \big]\\
  &\le  \E \big[  \Big( \int_0^t \Big( \int_{{\rm supp}(w)} \rho\, \,\d x \Big)^3 
  \int_{{\rm supp}(w)} \frac{1}{\rho}\sum_k  a_k^2 ( \zeta_k^{\varepsilon})^2 \,\d x   \d \tau\Big)^{\frac{p}{2}} \big]\\
  &\le  C(\gamma_2,B_0)\E \big[  \Big( \sup_{\tau\in[0,t]} 
  \int_{{\rm supp}(w)} \big(1+e^*(\rho, \rho_{\infty})+ \rho_{\infty}^{\gamma_2} +\rho u^2 \big) \,\d x \Big)^{2p} \big]r^{\frac{p}{2}},
  \end{aligned}$$
where, in the last inequality, we have 
used \eqref{iq-lower-upper-bound-for-general-internal-energy-2}, 
\eqref{iq-density-control-by-relative-internal-energy-general-pressure-law-2}, 
\eqref{iq-noise-coefficience-growth-conditions-for-parabolic-approximation},
and the fact that $\rho \le 1+\rho^{\gamma_2}$. 

All the other terms can be estimated similarly as $J_7$ and that in the deterministic 
case \cite[Lemma 3.3]{chenperepelitsa15CMP}, so we omit the details. 
Combining the above estimates and applying \eqref{iq-energy-estimate-on-whole-space}, we obtain
  \begin{equation}\label{iq-uniform-estimate-rho-gamma+1-unfinished}
  \begin{aligned}
  \E \big[\Big(\sup_{r\in[0,t]} \int_0^r \int_{\mathbb{R}} w^2\rho P(\rho)\,\d x\d \tau \Big)^p \big]
  \le  \E \big[\Big(\frac{3\varepsilon}{8}\int_0^t \int_{{\rm supp}(w)} \rho^3 w^2\,\d x \d \tau\Big)^p\big]
  + C({\gamma_2},w,p,T,E_{0,2p}).
  \end{aligned}
  \end{equation}

Next, we deal with the first term on the right-hand side to close the estimate. 
We make the following decomposition:
$$
\varepsilon \int_0^t \int_{{\rm supp} (w)} \rho^{3} w^2\,\d x \d \tau
=\varepsilon \int_0^t \int_{{\rm supp}(w)} I_{B(\tau,\omega)^{\rm c}} \rho^{3} w^2\,\d x \d \tau 
+\varepsilon \int_0^t \int_{{\rm supp} (w)} I_{B(\tau,\omega)} \rho^{3} w^2\,\d x \d \tau,
$$
where $B(t,\omega):=\{ x\in \mathbb{R}\,:\, \rho(t,x,\omega)>\rho^{\star} \}$. 
Then, to deal with the second term,
we apply the iteration technique in \cite[Lemma 3.3]{chenperepelitsa15CMP}. 
Using \eqref{iq-density-control-by-relative-internal-energy-general-pressure-law-2} and 
the Sobolev embedding: $W^{1,1}(\mathbb{R})\hookrightarrow L^{\infty}(\mathbb{R})$ (see \cite[Theorem 4.12]{adams03}), by \eqref{eq-general-pressure-law-2}, we calculate similarly
to that in \cite[Lemma 3.3]{chenperepelitsa15CMP} to obtain
  $$
  \begin{aligned}
  &\varepsilon \int_0^t \int_{{\rm supp}(w)} \rho^{3} w^2\,\d x \d \tau\\
  &\le  \varepsilon C(t,w,\rho^{\star})+\varepsilon \int_0^t\int_{B(\tau,\omega)} \big(1+ \rho P(\rho)\big)w^2 \,\d x \d \tau \\
  &\quad+ C(T)\Big(\sup_{\tau\in[0,t]}\int_{{\rm supp} (w)} \big(e^*(\rho, \rho_{\infty})+ \rho_{\infty}^{\gamma_2}\big) \,\d x + \varepsilon\int_0^t\int_{\mathbb{R}} \big(\rho e(\rho)\big)^{\prime \prime}(\rho_x)^2w^2 \,\d x \d \tau\Big)^3.
  \end{aligned}
  $$

Substituting this estimate for $\varepsilon \int_0^t \int_{{\rm supp} (w)} \rho^{3} w^2\,\d x \d \tau$ in \eqref{iq-uniform-estimate-rho-gamma+1-unfinished}, we complete the proof.
\end{proof}

\subsubsection{Higher integrability of the velocity}
\label{sec-Higher-integrability-for-velocity-in-general-pressure-law-case}

For the general pressure law case, there is no explicit formula for the entropy kernel, 
which makes it difficult to obtain the higher integrability of the velocity. 
We adopt the special entropy pairs constructed in \cite[\S 4]{chenhuangLiWangWang24CMP} 
to overcome this difficulty. The special entropy $\breve \eta$ satisfies
  \begin{equation}\label{eq-special-entropy-for-large-u}
  \breve \eta(\rho,u)=\begin{cases}
  \frac12 \rho u^2+ \rho e(\rho)\qquad\quad &\text{for $u\ge \mathcal{K}(\rho)$},\\[0.5mm]
  -\frac12 \rho u^2- \rho e(\rho)\quad &\text{for $u\le -\mathcal{K}(\rho)$},
  \end{cases}
  \end{equation}
and $\breve \eta(\rho,u)$ is the unique solution of the Goursat problem in the region
$\{-\mathcal{K}(\rho)\le u\le \mathcal{K}(\rho)\}$:
  \begin{equation}\label{eq-special-entropy-for-small-u}
  \begin{cases}
  \partial_{\rho}^2 \breve\eta-\mathcal{K}^{\prime}(\rho)^2 \partial_{u}^2 \breve\eta =0,\\[0.5mm]
  \breve \eta(\rho,u)|_{u=\pm \mathcal{K}(\rho)}=\pm \big(\frac12 \rho u^2+ \rho e(\rho)\big).
  \end{cases}
  \end{equation}
We recall some useful properties for this special entropy pair $(\breve \eta, \breve q)$.

\begin{lemma}[\hspace{-1pt}\cite{chenhuangLiWangWang24CMP}, Lemma 4.1]\label{lem-properties-for-special-entropy-pair}
The special entropy pair $(\breve \eta, \breve q)$ satisfies the following{\rm :}
\begin{itemize}
\item[\rm (i)] When $|u|\le \mathcal{K}(\rho)$,
  $$\begin{aligned}
  \breve q(\rho,u) \le
  \begin{cases}
  C\rho^{\gamma(\rho)+\theta(\rho)}\qquad\quad &\text{for $\rho\le \rho_{\star}$ or $\rho\ge \rho^{\star}$},\\
  C(\rho_{\star},\rho^{\star})\quad &\text{for $\rho_{\star}\le \rho\le \rho^{\star}$},
  \end{cases}
  \end{aligned}$$
  $$\begin{aligned}
  \partial_u \breve\eta(\rho,u) \le
  \begin{cases}
  C \rho^{\theta(\rho)+1}\qquad\quad &\text{for $\rho\le \rho_{\star}$ or $\rho\ge \rho^{\star}$},\\
  C\rho^{1+M_1}\quad &\text{for $\rho_{\star}\le \rho\le \rho^{\star}$},
  \end{cases}
  \end{aligned}$$
  $$\begin{aligned}
  \partial_{\rho} \breve\eta(\rho,u) \le
  \begin{cases}
  C \rho^{\gamma(\rho)-1}\qquad\quad &\text{for $\rho\le \rho_{\star}$ or $\rho\ge \rho^{\star}$},\\
  C(\rho_{\star},\rho^{\star})\rho^{M_1}\quad &\text{for $\rho_{\star}\le \rho\le \rho^{\star}$},
  \end{cases}
  \end{aligned}$$
where $\gamma(\rho)=\gamma_1$ when $\rho\le \rho_{\star}$ and $\gamma(\rho)=\gamma_2$ when $\rho\ge \rho^{\star}$, 
and $\theta(\rho)=\frac{\gamma(\rho)-1}{2}$.

\item[\rm (ii)] When $|u|\ge \mathcal{K}(\rho)$,
  $$
  \breve q (\rho,u) \ge \frac12 \rho|u|^3.
  $$

\item[\rm (iii)] If $\breve \eta_m$ is regarded as a function of $(\rho,m)$, then, 
for $(\rho,m)\in \mathbb{R}_+ \times \mathbb{R}$,
  $$
  |\partial_{m\rho}\breve \eta(\rho,m)| \le C\rho^{\theta(\rho)-1},\ 
  \qquad |\partial_{m}^2\breve \eta(\rho,m)| \le C\rho^{-1}.
  $$
\end{itemize}
\end{lemma}
Note that, even though some properties in Lemma \ref{lem-properties-for-special-entropy-pair} are not covered by \cite[Lemma 4.1]{chenhuangLiWangWang24CMP}, they can be proved similarly 
as that in \cite[Lemma 4.1]{chenhuangLiWangWang24CMP}.

We have the following lemma.

\begin{lemma}\label{lem-dervivatives-of-special-entropy-controlled-by-hessian-of-energy}
There exists a constant $C>0$ that may depend on $\rho_{\star}$ and $\rho^{\star}$ that
$$
  \big|\partial_x U\, \nabla^2 \breve\eta \,(\partial_x U)^{\top}\big| 
  \le C \big( (\rho e(\rho))^{\prime \prime}(\partial_x \rho)^2 + \rho(\partial_x u)^2 \big).
  $$

\end{lemma}
\begin{proof}
The key of the proof is to estimate the second-order derivatives of $\breve\eta$, which are similar to that in the proof of \cite[Lemma 4.1]{chenhuangLiWangWang24CMP},  so we omit it.
\end{proof}

The relative-entropy pair that we consider is as follows:
  \begin{equation}\label{eq-relative-entropy-pair-general-pressure-law}
  \begin{aligned}
  &\tilde \eta (\rho,m)=\breve \eta (\rho,m)-\breve \eta (\rho_{\infty},0)-\nabla_{(\rho,m)}\breve \eta (\rho_{\infty},0)\cdot (\rho-\rho_{\infty},\,m),\\
  &\tilde q(\rho,m)=\breve q(\rho,m)-\nabla_{(\rho,m)}\breve \eta(\rho_{\infty},0)\cdot (m,\,\frac{m^2}{\rho}+P(\rho)),
  \end{aligned}
  \end{equation}
which is still an entropy pair of \eqref{eq-stochastic-Euler-system}.

\begin{lemma}\label{lem-properties-of-relative-special-entropy-pair}
There exists a constant $C>0$ such that
  \begin{align}
  \tilde\eta(\rho,m)+ \rho |\partial_m \tilde \eta(\rho,m)|^2
  \le C
  \eta^*_E(\rho,m). \label{iq-special-relative-entropy-control-by-relative-energy}
  \end{align}
When $|u|\ge \mathcal{K}(\rho)$,
  \begin{equation}\label{iq-control-for-relative-special-entropy-flux-general-pressure-law}
  \tilde q(\rho,m) \ge \frac12 \rho|u|^3 -C(\rho_{\star},\rho^{\star},\rho_{\infty})\big(1+ \eta^*_E(\rho,m) \big).
  \end{equation}
\end{lemma}

The proof of \eqref{iq-special-relative-entropy-control-by-relative-energy} is similar to that in \cite[Proof of Proposition 4.1]{chenwangyong22ARMA}, using Lemma \ref{lem-properties-for-special-entropy-pair}(iii). \eqref{iq-control-for-relative-special-entropy-flux-general-pressure-law} follows from direct calculation and Lemma \ref{lem-properties-for-special-entropy-pair}(i)-(ii). We omit the details.

Now we are ready to prove the higher integrability of the velocity:

\begin{proposition}\label{prop-higher-integrability-of-velocity-general-pressure-law}
The solution of \eqref{eq-parabolic-approximation} obtained 
in {\rm Theorem \ref{thm-wellposedness-parabolic-approximation-R}} satisfies 
  \begin{equation}\label{iq-uniform-estimate-rho-u3}
  \begin{aligned}
  \E \big[\Big( \int_0^t \int_{K} \rho |u|^3 \,\d x\d \tau \Big)^p\big]
  \le  C(K,p,T,\gamma_2, E_{0,3p}) \qquad \mbox{for any $t\in[0,T]$},
  \end{aligned}
  \end{equation}
where $K\Subset \mathbb{R}$ is any compact set, and $C(K,p,T,\gamma_2, E_{0,3p})$ is independent of $\varepsilon$.
\end{proposition}

\begin{proof}
By Theorem \ref{thm-wellposedness-parabolic-approximation-R}(iii), we apply the It\^o formula 
to $\tilde \eta$ to obtain
  $$
  \d \tilde \eta (U)=-\partial_x \tilde q(U) \,\d t + \varepsilon \nabla \tilde \eta (U) \partial_x^2 U \,\d t+ \partial_m \tilde \eta (U) \Phi^{\varepsilon}(U) \,\d W(t)+ \frac12 \partial_{m}^2 \tilde \eta(U) \sum_k a_k^2 ( \zeta_k^{\varepsilon})^2 \,\d t.
  $$
Then integrating on $(0,t)\times (-\infty,x)$ gives
  $$\begin{aligned}
  \int_0^t \tilde q(U)(x) \,\d \tau   &=\int_0^t \tilde q(U)(-\infty) \,\d \tau-\int_{-\infty}^x \big(\tilde \eta(U)(t)-\tilde \eta(U)(0)\big) \,\d y
  +\varepsilon \int_0^t \int_{-\infty}^x \nabla \tilde \eta (U) \partial_y^2 U \,\d y \d \tau\\
  &\quad+\int_0^t \int_{-\infty}^x \partial_m \tilde \eta (U) \Phi^{\varepsilon}(U) \,\d y\d W(\tau)+ \frac12 \int_0^t \int_{-\infty}^x \partial_{m}^2 \tilde \eta(U) \sum_k a_k^2 ( \zeta_k^{\varepsilon})^2 \,\d y \d \tau.
  \end{aligned}
  $$
We first estimate the first three terms on the right-hand side.
We use Lemma \ref{lem-properties-for-special-entropy-pair}(i) to estimate the first term, 
use \eqref{iq-special-relative-entropy-control-by-relative-energy} for the second term,
use integration by parts and Lemma \ref{lem-dervivatives-of-special-entropy-controlled-by-hessian-of-energy}
for the third term, and then finally obtain 
 $$
 \begin{aligned}
  \int_0^t \tilde q(U)(x) \,\d \tau
  &\le
  C(\rho_{\star},\rho^{\star},\rho_{\infty})\sup_{\tau\in[0,t]} \int_{-\infty}^x \big(\frac{m^2}{2\rho}+e^*(\rho, \rho_{\infty}) \big)
  (\tau)
  \,\d y
  +\varepsilon \int_0^t \partial_x \tilde \eta (U)(x)  \,\d \tau\\
  &\quad+ \varepsilon C(\rho_{\star},\rho^{\star}) \int_0^t \int_{-\infty}^x  
  \Big( \big(\rho e(\rho)\big)^{\prime \prime}(\partial_y \rho)^2 + \rho(\partial_y u)^2 \Big) \,\d y\d \tau
  + C(\rho_{\infty},\rho_{\star},\rho^{\star}) t\\
  &\quad+\int_0^t \int_{-\infty}^x \partial_m \tilde \eta (U) \Phi^{\varepsilon}(U) \,\d y\d W(\tau)
  + \frac12 \int_0^t \int_{-\infty}^x \partial_{m}^2 \tilde \eta(U) \sum_k a_k^2 ( \zeta_k^{\varepsilon})^2 \,\d y \d \tau.
  \end{aligned}
 $$
Now, let $w(x)\in C_c^{\infty}(\mathbb{R})$ be an arbitrary function satisfying $0\le w(x)\le 1$ and $w\equiv 1 $ on $K$. 
Then multiplying the above inequality by $w(x)$ and integrating on $\mathbb{R}$ give
  \begin{equation}\label{iq--uniform-estimate-rho-u3-unfinished-general-pressure}
  \begin{aligned}
  &\int_{\mathbb{R}} w(x)\int_0^t \tilde q(U)(x) \,\d \tau d x\\
  &\le
  C\int_{\mathbb{R}} 
  w
  \sup_{\tau\in[0,t]} \int_{-\infty}^x \big(\frac{m^2}{2\rho}+e^*(\rho, \rho_{\infty}) \big) \,\d y\d x
  +\varepsilon \int_{\mathbb{R}} w
  \int_0^t 
  \partial_x \tilde \eta (U)
  \,\d \tau \d x\\
  &\quad+ \varepsilon 
  C\int_{\mathbb{R}} 
  w
  \int_0^t \int_{-\infty}^x  \Big( (\rho e(\rho))^{\prime \prime}(\partial_y \rho)^2 
  + \rho(\partial_y u)^2 \Big) \,\d y\d \tau\d x+ C t \int_{\mathbb{R}} w
  \,\d x\\
  &\quad+\int_{\mathbb{R}} 
  w
  \int_0^t \int_{-\infty}^x \partial_m \tilde \eta (U) \Phi^{\varepsilon}(U) \,\d y\d W(\tau)\d x
  + \frac12 \int_{\mathbb{R}} 
  w
  \int_0^t \int_{-\infty}^x \partial_{m}^2 \tilde \eta(U) \sum_k a_k^2 ( \zeta_k^{\varepsilon})^2 \,\d y \d \tau\d x.
  \end{aligned}
  \end{equation}
Applying integration by parts and \eqref{iq-special-relative-entropy-control-by-relative-energy}, we have
  $$
  \varepsilon \int_{\mathbb{R}} 
  w
  \int_0^t 
  \partial_x \tilde \eta (U)
  \,\d \tau \d x
  \le 
  C(w)\varepsilon \int_0^t \int_{{\rm supp}(w)} \big(\frac{m^2}{2\rho}+e^*(\rho, \rho_{\infty})\big) \,\d x \d \tau.
  $$
For the stochastic integral, it follows from the Burkholder-Davis-Gundy inequality that
$$
\begin{aligned}
  &\E \big[\Big|\int_0^t \int_{\mathbb{R}}
  w
  \int_{-\infty}^x \partial_m \tilde \eta (U) \Phi^{\varepsilon}(U) \,\d y\d x\d W(\tau) \Big|^p\big]\\
  &\le  \E \big[ \Big( \int_0^t |{\rm supp}(w)|^2 \int_{\mathbb{R}} \rho (\partial_m \tilde \eta )^2 \,\d y 
  \int_{\mathbb{R}} \frac{1}{\rho} \sum_k a_k^2 ( \zeta_k^{\varepsilon})^2 \,\d y \d \tau \Big)^{\frac{p}{2}} \big]\\
  &\le  C(w)t^{\frac{p}{2}} \Big(\E \big[ \Big( \sup_{\tau\in[0,t]} 
  \int_{\mathbb{R}} \rho (\partial_m \tilde \eta )^2 \,\d y \Big)^{p} \big]+ \E \big[ \Big( \sup_{\tau\in[0,t]} 
  \int_{\mathbb{R}} \frac{1}{\rho} \sum_k a_k^2 ( \zeta_k^{\varepsilon})^2 \,\d y  \Big)^{p} \big]\Big).
  \end{aligned}$$
For the second term on the right-hand side of the above inequality, we only show the estimate for Case 3 (see Hypotheses \ref{hypo-for-far-field-of-initial-data-and-noise-coefficient}--\ref{hypo-for-far-field-of-initial-data-and-noise-coefficient-parabolic-approximation}), 
since the estimates for the other cases are similar and easier. 
For Case 3, by assumption \eqref{iq-noise-coefficience-growth-conditions-for-parabolic-approximation},
we obtain that, for $\mathfrak{G}_1^{\varepsilon}$,
  \begin{equation}\label{iq-estimate-for-noise-term-1/rho--in-energy-estimate}
  \begin{aligned}
  &\E \big[\Big( \sup_{\tau\in[0,t]}
  \sum_k  \int_{{\rm supp}_x (\zeta_k^{\varepsilon})} \frac{1}{\rho}a_k^2 ( \zeta_k^{\varepsilon})^2 \,\d x \d \tau \Big)^p\big]\\
  &\le  \E \big[\Big(B_0^2(1+\varepsilon^2)|{\rm supp}_x (\zeta_k^{\varepsilon})|t
  +C \sup_{\tau\in[0,t]}
  \int_{{\rm supp}_x (\zeta_k^{\varepsilon})} 
  B_0^2 \big(\rho u^2 +\rho^{\gamma_2}+e(\rho) \big)\,\d x \d \tau \Big)^p\big]\\
  &\le  
  C(p)\big((1+\varepsilon^2+\rho_{\infty}^{\gamma_2})|{\rm supp}_x (\zeta_k^{\varepsilon})|t\big)^p + 
  C(p)\E \big[\Big(\sup_{\tau\in[0,t]}
  \int_{{\rm supp}_x (\zeta_k^{\varepsilon})}
  \big(\rho u^2 +e^*(\rho, \rho_{\infty})\big) \,\d x \d \tau \Big)^p\big],
  \end{aligned}
  \end{equation}
where we have used $\rho\le 1+\rho^{\gamma_2}$ in the 
first inequality and 
used \eqref{iq-lower-upper-bound-for-general-internal-energy-2} and 
\eqref{iq-density-control-by-relative-internal-energy-general-pressure-law-2}
in the last inequality.
For Case 3, 
${\rm supp}_x (\zeta_k^{\varepsilon})\subset\mathbb{K}^1$, so that
the first term on the right-hand side of the above estimate is uniformly bounded 
with respect to $\varepsilon$. 
Case 2 can be treated similarly. 
While, for Case 1, 
${\rm supp}_x (\zeta_k^{\varepsilon})=\Lambda^{\varepsilon}$, 
letting $i=1$ when $\rho\le \rho_{\star}$ and $i=2$ when $\rho\ge \rho^{\star}$, 
then the first term on the right-hand side of the above 
estimate \eqref{iq-estimate-for-noise-term-1/rho--in-energy-estimate} becomes
$$
C(p)\big((\varepsilon^2+\rho_{\infty}(\varepsilon)^{\gamma_i})|{\rm supp}_x (\zeta_k^{\varepsilon})|t\big)^p,
$$
which is uniformly bounded with respect to $\varepsilon$, by choosing $\alpha_0$ in $\rho_{\infty}(\varepsilon)=\varepsilon^{\alpha_0}$
such that $\alpha_0\gamma_1>1$ and $\alpha_0\gamma_2>1$, respectively.
When $\rho\in [\rho_{\star},\rho^{\star}]$, we 
use \eqref{iq-density-control-by-relative-internal-energy-general-pressure-law-3}, recall that 
$\rho_{\infty}(\varepsilon)<\rho_{\star}$ (see \eqref{iq-rho-infinity-varepsilon-sufficiently-small}), then utilize 
\eqref{iq-lower-upper-bound-for-general-internal-energy-1} and 
choose $\alpha_0$ such that $\alpha_0\gamma_1>1$ to obtain its uniform bound.
For the first term, we can use \eqref{iq-special-relative-entropy-control-by-relative-energy} to estimate it. 
Thus, we finally obtain
  $$
  \begin{aligned}
  &\E \big[\Big|\int_0^t \int_{\mathbb{R}} 
  w
  \int_{-\infty}^x \partial_m \tilde \eta (U) \Phi^{\varepsilon}(U) \,\d y\d x\d W(\tau) \Big|^p\big]\\
  &\le  
  C(w,p)t^{\frac{p}{2}} \Big( 1+\E \big[\Big( \sup_{\tau\in[0,t]}
  \int_{{\rm supp}_x(\zeta_k^{\varepsilon})} \big( \rho u^2 +e^*(\rho, \rho_{\infty})\big) \,\d x  \Big)^p \big]\Big).
  \end{aligned}
  $$
Noticing that $|\partial_{m}^2 \tilde \eta(U)|\le \frac{1}{\rho}$. Therefore, similarly, we have
$$
\begin{aligned}
  \frac12 \int_{\mathbb{R}} w
  \int_0^t \int_{-\infty}^x \partial_{m}^2 \tilde \eta(U) \sum_k a_k^2 ( \zeta_k^{\varepsilon})^2 \,\d y \d \tau\d x
  \le 
  C(w)t \Big( 1+\sup_{\tau\in[0,t]}\int_{{\rm supp}_x(\zeta_k^{\varepsilon})} 
  \big(\rho u^2 +e^*(\rho, \rho_{\infty})\big) \,\d x  \Big).
\end{aligned}
$$
Substituting the above estimates in \eqref{iq--uniform-estimate-rho-u3-unfinished-general-pressure}, 
taking the $p$-th power and the expectation, and applying 
the energy estimate \eqref{iq-energy-estimate-on-whole-space}, 
we obtain
  \begin{equation}\label{iq--uniform-estimate-rho-u3-unfinished-2}
  \begin{aligned}
  \E \big[\Big(\int_{\mathbb{R}} 
  w
  \int_0^t 
  \tilde q(U)
  \,\d \tau d x\Big)^p\big]
  \le C(w,p,T, E_{0,p}).
  \end{aligned}
  \end{equation}

By \eqref{iq-control-for-relative-special-entropy-flux-general-pressure-law}, we have
  $$
  \begin{aligned}
  &\int_0^t \int_{|u|\ge \mathcal{K}(\rho)} 
  \rho|u|^3 w\, \d x \d \tau \\
  &\le 2\int_0^t \int_{\mathbb{R}} 
  \tilde q(\rho,m) w\, \d x \d \tau + 2\int_0^t \int_{{\rm supp}(w)}
   C(\rho_{\star},\rho^{\star},\rho_{\infty})\big(1+ \eta^*_E(\rho,m) \big) \, \d x \d \tau.
  \end{aligned}$$
When $|u|\le \mathcal{K}(\rho)$, for $\rho\le \rho_{\star}$, 
it follows from \eqref{iq-lower-upper-bound-for-k(rho)-1} 
and \eqref{iq-density-control-by-relative-internal-energy-general-pressure-law-1}
that
  $
  \rho |u|^3\le C \big( e^*(\rho,\rho_{\infty})+\rho_{\infty}^{\gamma_1} \big).
  $
For $\rho\ge \rho_{\star}$, it follows from
\eqref{iq-lower-upper-bound-for-general-pressure-2},
\eqref{iq-lower-upper-bound-for-k(rho)-2}, and $\gamma_2 <3$, that
  $
  \rho |u|^3\le C\rho^{\frac32 \gamma_2 -\frac12}\le C\big(1+\rho P(\rho) \big).
  $
For $\rho_{\star}\le \rho\le \rho^{\star}$, it follows from 
\eqref{iq-density-control-by-relative-internal-energy-general-pressure-law-3} that
  $
  \rho |u|^3\le C(\rho_{\star},\rho^{\star},\rho_{\infty})\rho e(\rho) \le C(\rho_{\star},\rho^{\star},\rho_{\infty})\big( e^*(\rho,\rho_{\infty})+\rho_{\infty}e(\rho_{\infty}) \big).
  $
Therefore, we obtain
  $$
  \int_0^t \int_{|u|\le \mathcal{K}(\rho)} 
  \rho|u|^3 w\, \d x \d \tau \le C\int_0^t \int_{{\rm supp} w} \big( 1+e^*(\rho,\rho_{\infty}) + \rho P(\rho) \big) \, \d x \d \tau.
  $$
Thus, combining the above estimates and using Proposition \ref{prop-uniform-estimates-rho-gamma+1}, we conclude \eqref{iq-uniform-estimate-rho-u3}.
\end{proof}

\section{Tightness of the Parabolic Approximation}\label{sec-tightness}

In this section, we prove the tightness of 
the parabolic approximation.
We first follow the notions introduced in Ball \cite{J.Ball89}. Denote $\mathbb{R}^2_+ :=\mathbb{R}_+ \times \mathbb{R}$.
Let $\mathcal{M}(\mathbb{R}^2_+)$ denote the space of bounded Radon measures 
on $\mathbb{R}^2_+$, 
$\mathcal{M}^+(\mathbb{R}^2_+)$ denote the space of nonnegative bounded Radon measures 
on $\mathbb{R}^2_+$, $C_0(\mathbb{R}^2_+)$ denote the Banach space of continuous real functions $f$ 
satisfying $\lim_{|x|\to \infty}f(x)=0$, and $\mathcal{P}(\mathbb{R}^2_+)$ 
denote the set of all probability measures on $\mathbb{R}^2_+$. $\mathcal{M}(\mathbb{R}^2_+)$ is identified with the dual of $C_0(\mathbb{R}^2_+)$ by the duality pairing:
\begin{equation}\label{eq-duality-pair-for-radon-measure}
\langle\mu, f \rangle:=\int_{\mathbb{R}^2_+} f \,\d \mu 
  \qquad\mbox{for $\mu\in \mathcal{M}(\mathbb{R}^2_+)$ and $f\in C_0(\mathbb{R}^2_+)$}.
\end{equation}
In this paper, $\mathcal{M}(\mathbb{R}^2_+)$ is equipped with the weak-* topology. 
The norm on $\mathcal{M}(\mathbb{R}^2_+)$ is given by $\|\mu\|_{\mathcal{M}}:=\int_{\mathbb{R}^2_+}\,\d |\mu|$. 
The space $L^{\infty}_{\rm w} (\mathbb{R}^2_+; \mathcal{M}(\mathbb{R}^2_+))$ is a Banach space of equivalence 
classes of weak-* measurable mapping $\mu : \mathbb{R}^2_+ \to \mathcal{M}(\mathbb{R}^2_+)$ 
that are essentially bounded, endowed with the norm
  $$
  \|\mu\|_{\infty,\mathcal{M}}:={{\rm ess}\sup}_{\mathbf{x}\in \mathbb{R}^2_+}\|\mu_{\mathbf{x}}\|_{\mathcal{M}}.
  $$
Here we call $\mu$ is weak-* measurable if $\langle \mu_\mathbf{x}, f \rangle $ is measurable 
with respect to $\mathbf{x}$ for every 
$f\in C_0(\mathbb{R}^2_+)$. $L^{\infty}_{\rm w} (\mathbb{R}^2_+; \mathcal{P}(\mathbb{R}^2_+))$ 
is a space of Young measures equipped with the same topology 
as $L^{\infty}_{\rm w} (\mathbb{R}^2_+; \mathcal{M}(\mathbb{R}^2_+))$. 
The Young measure associated to a measurable map $h:\mathbb{R}^2_+ \to \mathbb{R}^2_+$ is the Dirac mass 
given by $\delta_{h(\mathbf{x})}$ for $\mathbf{x}\in \mathbb{R}^2_+$. 
The space $L^{\infty}_{\rm w} (\mathbb{R}^2_+; \mathcal{M}(\mathbb{R}^2_+))$ is isometric isomorphic to the dual of 
$L^1(\mathbb{R}^2_+;C_0(\mathbb{R}^2_+))$ via the duality pairing:
  $$
  \langle \mu, g \rangle_{L^{\infty},L^1}:=\int_{\mathbb{R}^2_+} \langle \mu_\mathbf{x}, g 
  \rangle \,\d \mathbf{x}
\qquad\mbox{for $\mu \in L^{\infty}_{\rm w} (\mathbb{R}^2_+; \mathcal{M}(\mathbb{R}^2_+))$ 
and $g\in L^1(\mathbb{R}^2_+;C_0(\mathbb{R}^2_+))$}. 
$$
For $\mu \in L^{\infty}_{\rm w} (\mathbb{R}^2_+; \mathcal{M}(\mathbb{R}^2_+))$, we also use the notation $\mu_{\mathbf{x}}(\mathbf{f}):=\mu(\mathbf{x},\mathbf{f})$, where $\mathbf{x}, \mathbf{f}\in \mathbb{R}^2_+$ 
such that $\mu_\mathbf{x}\in \mathcal{M}(\mathbb{R}^2_+)$.

From now on,
the space $L^{\infty}_{\rm w} (\mathbb{R}^2_+; \mathcal{M}(\mathbb{R}^2_+))$ and hence $L^{\infty}_{\rm w} (\mathbb{R}^2_+; \mathcal{P}(\mathbb{R}^2_+))$ are equipped with the weak-* topology. 
Since $\|\nu\|_{\infty,\mathcal{M}}=1$ for any Young measure $\nu \in L^{\infty}_{\rm w} (\mathbb{R}^2_+; \mathcal{P}(\mathbb{R}^2_+))$,
by Banach-Alaoglu's theorem, $L^{\infty}_{\rm w} (\mathbb{R}^2_+; \mathcal{P}(\mathbb{R}^2_+))$ 
is relatively compact in $L^{\infty}_{\rm w} (\mathbb{R}^2_+; \mathcal{M}(\mathbb{R}^2_+))$.
However, we should note that the limit of the Young measures is not necessarily a Young measure. 
We now provide a compactness criterion for the Young measures. 
There are various results of this form such as \cite[Proposition 4.1]{BV19} and \cite[Theorem 5]{DebusscheVovelle10},
among others. Here our compactness criterion is similar to the one in \cite[Proposition 4.1]{BV19}.

\begin{proposition}\label{prop-compactness-criteria-for-young-measure}
Let $\varsigma\in C(\mathbb{R}_+^2;\mathbb{R}_+)$ satisfy the growth condition
  $$
  \lim_{\mathbf{f}\in \mathbb{R}_+^2,|\mathbf{f}|\to \infty} \varsigma (\mathbf{f})=\infty.
  $$
Then the set
  $$
  \Lambda_C =\{ \mu \in L^{\infty}_{\rm w} (\mathbb{R}^2_+; \mathcal{P}(\mathbb{R}^2_+))\,:\, \int_{[0,T]\times \hat{K}} \int_{\mathbb{R}^2_+}\varsigma (\mathbf{f}) \,\d \mu_\mathbf{x}(\mathbf{f})\d \mathbf{x} \le C
  \quad\mbox{for any $T \ge 0$ and $\hat K \Subset \mathbb{R}$} \}
  $$
is a compact subset in  $L^{\infty}_{\rm w} (\mathbb{R}^2_+; \mathcal{P}(\mathbb{R}^2_+))$, 
where $C$ is a positive constant, and 
$\hat K \Subset  \mathbb{R}$ means $\hat K$ is a compact subset of $\mathbb{R}$.
\end{proposition}

The proof of Proposition \ref{prop-compactness-criteria-for-young-measure} is in the spirit of Ball \cite{J.Ball89}; 
see Appendix \ref{appendix-some-proofs}.

\medskip
For the initial law $\Im$ given in Theorem \ref{thm-well-posedness-for-euler-on-whole-space-general-pressure-law},
by \cite[Corollary 2.6.4]{BFHbook18}, there exist random variables $(\rho_0,m_0)=(\rho_0, \rho_0 u_0)$ 
with law $\Im$. 
We now construct $(\rho^{\varepsilon}_0,m^{\varepsilon}_0)$ satisfying the conditions 
in Theorem \ref{thm-wellposedness-parabolic-approximation-R}. 
By cut-off technique and mollification, we can directly 
construct $\rho^{\varepsilon}_0$ such that $\rho^{\varepsilon}_0-\rho_{\infty} \in H^5({\mathbb{R}})$ $\mathbb{P}$-{\it a.s.}, 
$\rho^{\varepsilon}_0(\omega)\ge c_0(\varepsilon) >0$ {\it a.e.} in $\mathbb{R}$ for all $\omega \in \Omega$,
and $\rho^{\varepsilon}_0 \to \rho_0$ in $L^p(\Omega;L^{\gamma_2}_{\rm loc}(\mathbb{R}))$ for 
$1\le p<\infty$. 
Using the cut-off technique and mollification again, we can construct $\vartheta_0^{\varepsilon}$ such that $\vartheta_0^{\varepsilon} \to \vartheta_0:=\frac{m_0}{\sqrt{\rho_0}}$ in $L^p(\Omega; L^2(\mathbb{R}))$ for 
$1\le p<\infty$,
$\,m^{\varepsilon}_0:=\sqrt{\rho^{\varepsilon}_0}\,\vartheta_0^{\varepsilon}\in H^5({\mathbb{R}})$ 
$\mathbb{P}$-{\it a.s.}, $m^{\varepsilon}_0 \to m_0$ in $L^p(\Omega;L^1_{\rm loc}(\mathbb{R}))$ for $1\le p<\infty$,
and $U^{\varepsilon}_0=(\rho^{\varepsilon}_0, m^{\varepsilon}_0)^\top$ satisfies
  $$
  \begin{aligned}
  &U^{\varepsilon}_0\in \Gamma_{\mathcal{H}^{\varepsilon}}\quad \mathbb{P}\text{-{\it a.s.}}, \\[1mm]
  &\E \Big[\Big(\int_{\mathbb{R}} \big(\frac12 \frac{(m^{\varepsilon}_0)^2}{\rho^{\varepsilon}_0} 
  +e^*(\rho^{\varepsilon}_0,\rho_{\infty})\big) \,\d x \Big)^p\Big] 
  \le C\,\E \big[\Big(\int_{\mathbb{R}} \big(\frac12 \frac{m_{0}^2}{\rho_{0}} 
  +e^*(\rho_{0},\rho_{\infty})\big)  \,\d x \Big)^p\big]\qquad \text{for}\ p>3.
  \end{aligned}
  $$
See also \cite[Appendix A]{chenwangyong22ARMA} for the construction of the initial data 
in the multi-dimensional case when the initial density $\rho_0 \to \rho_{\infty}$ at the far field.

Therefore, with the initial data $U^{\varepsilon}_0$ constructed above, 
Theorem \ref{thm-wellposedness-parabolic-approximation-R} yields the existence of 
strong solutions $(\rho^{\varepsilon},m^{\varepsilon})$ to the parabolic approximation \eqref{eq-parabolic-approximation}. 
We now establish suitable compactness and perform the vanishing viscosity limit $\varepsilon \to 0$ 
to obtain the solution of the stochastic Euler system \eqref{eq-stochastic-Euler-system}.

Let $\mu^{\varepsilon}:=\delta_{(\rho^{\varepsilon},m^{\varepsilon})}$ 
be the random Young measures (Young measure-valued random variables) 
associated with the parabolic approximation $(\rho^{\varepsilon},m^{\varepsilon})$.
We now define the path space 
$\mathcal{X}=\mathcal{X}_{\mu}\times \mathcal{X}_{W}\times \mathcal{X}_{\rho}\times \mathcal{X}_{m}
\times \mathcal{X}_{\rho_0}\times \mathcal{X}_{m_0}\times \mathcal{X}_{\frac{m_0}{\sqrt{\rho_0}}}$, where
$$
\begin{aligned}
  &\mathcal{X}_{\mu}=L^{\infty}_{\rm w} (\mathbb{R}^2_+; \mathcal{P}(\mathbb{R}^2_+)), \qquad
  \mathcal{X}_W=C([0,T];\mathfrak{A}_0),\\[1mm]
  &\mathcal{X}_{\rho}=C_{\rm w}([0,T];L^{\gamma_2}_{\rm loc}(\mathbb{R}))\cap 
  L^{\gamma_2+1}_{\rm w,loc}(\mathbb{R}^2_+),\\
  &\mathcal{X}_m=C_{\rm w}([0,T];L^{\frac{2\gamma_2}{\gamma_2+1}}_{\rm loc}(\mathbb{R}))\cap 
  L^{\frac{3(\gamma_2+1)}{\gamma_2+3}}_{\rm w,loc}(\mathbb{R}^2_+),\\[1mm]
  &\mathcal{X}_{\rho_0}=L^{\gamma_2}_{\rm loc}(\mathbb{R}),\qquad
  \mathcal{X}_{m_0}=L^1_{\rm loc}(\mathbb{R}),\qquad
  \mathcal{X}_{\frac{m_0}{\sqrt{\rho_0}}}=L^2(\mathbb{R}).
  \end{aligned}
  $$
We use $\mathcal{L}_g:=\mathcal{L}(g)$ to denote the law of random variable $g$. The joint law of all variables $(\mu^{\varepsilon},W,\rho^{\varepsilon},m^{\varepsilon},\rho^{\varepsilon}_0,m^{\varepsilon}_0,
\vartheta_0^{\varepsilon})$ is denoted by $\mathcal{L}^{\varepsilon}$.

\begin{proposition}\label{cor-tightness-of-L-varepsilon}
The joint law $\mathcal{L}^{\varepsilon}$ is tight on $\mathcal{X}$.
\end{proposition}

\begin{proof}
The proof is divided into four steps.

\smallskip
\textbf{1}. {\it $\{\rho^{\varepsilon}\,:\,\varepsilon >0 \}$ is tight on $\mathcal{X}_{\rho}$}:
$\,\,$ On one hand, by \eqref{iq-density-control-by-relative-internal-energy-general-pressure-law-2} 
and the energy estimate \eqref{iq-energy-estimate-on-whole-space}, we immediately see that 
$\{\rho^{\varepsilon}\,:\,\varepsilon >0 \}$ is uniformly bounded 
in $L^1(\Omega;L^{\infty}([0,T];L^{\gamma_2}_{\rm loc}(\mathbb{R})))$.

On the other hand, recall equation $\eqref{eq-parabolic-approximation-seperate}_1$ satisfied by the density:
  $
  \d \rho + \partial_x m \,\d t=\varepsilon \partial_x^2 \rho \,\d t.
  $
Since $|m|^q\le \rho u^2+ \rho^{\gamma_2}$ for $q=\frac{2\gamma_2}{\gamma_2+1}$, 
it follows from the energy estimate \eqref{iq-energy-estimate-on-whole-space} that
$\{m^{\varepsilon}\,:\,\varepsilon >0 \}$ is uniformly bounded in $L^1(\Omega;L^{\infty}([0,T];L^{\frac{2\gamma_2}{\gamma_2+1}}_{\rm loc}(\mathbb{R})))$, which implies $\{\partial_x m^{\varepsilon}\,:\,\varepsilon >0 \}$ 
is uniformly bounded in $L^1(\Omega;L^{\infty}([0,T];W^{-1,\frac{2\gamma_2}{\gamma_2+1}}_{\rm loc}(\mathbb{R})))$. 
Morover, since $\{\rho^{\varepsilon}\,:\,\varepsilon >0 \}$ is uniformly bounded 
in $L^1(\Omega;L^{\infty}([0,T];L^{\gamma_2}_{\rm loc}(\mathbb{R})))$, 
then $\{\partial_x^2 \rho^{\varepsilon}\,:\,\varepsilon >0 \}$ is uniformly bounded 
in $L^1(\Omega;L^{\infty}([0,T];W^{-2,\gamma_2}_{\rm loc}(\mathbb{R})))$. 
As a consequence of the continuity equation $\eqref{eq-parabolic-approximation-seperate}_1$, 
we conclude that $\{\rho^{\varepsilon}\,:\,\varepsilon >0 \}$ is uniformly bounded 
in $L^1(\Omega;C^{\alpha}([0,T];W^{-2,\frac{2\gamma_2}{\gamma_2+1}}_{\rm loc}(\mathbb{R})))$ for any $0<\alpha<1$.

Therefore, the required tightness of $\{\rho^{\varepsilon}\,:\,\varepsilon >0 \}$ 
on $C_{\rm w}([0,T];L^{\gamma_2}_{\rm loc}(\mathbb{R}))$ follows due to the compact embedding:
$$
L^{\infty}([0,T];L^{\gamma_2}_{\rm loc}(\mathbb{R})) \cap C^{\alpha}([0,T];W^{-2,\frac{2\gamma_2}{\gamma_2+1}}_{\rm loc}(\mathbb{R})) 
\hookrightarrow C_{\rm w}([0,T];L^{\gamma_2}_{\rm loc}(\mathbb{R}))
$$
from Lemma \ref{lem-compactness-result-from-ondrejat}.

The tightness in 
$L^{\gamma_2+1}_{\rm w,loc}(\mathbb{R}^2_+)$
follows directly from Proposition \ref{prop-uniform-estimates-rho-gamma+1}.

\medskip
\textbf{2}. {\it $\{m^{\varepsilon}\,:\,\varepsilon >0 \}$ is tight on $\mathcal{X}_{m}$}: $\,\,$
From Step \textbf{1}, we know that $\{m^{\varepsilon}\,:\,\varepsilon >0 \}$ is uniformly bounded in $L^1(\Omega;L^{\infty}([0,T];L^{\frac{2\gamma_2}{\gamma_2+1}}_{\rm loc}(\mathbb{R})))$.

Next, we prove the uniform time-regularity of $m^{\varepsilon}$. 
We rewrite equation $\eqref{eq-parabolic-approximation-seperate}_2$ into the integral form:
  $$\begin{aligned}
  m^{\varepsilon}(t)&=m^{\varepsilon}(0)- \int_0^t \partial_x\big(\frac{(m^{\varepsilon})^2}{\rho^{\varepsilon}}
  +P(\rho^{\varepsilon})\big)\,\d \tau + \varepsilon \int_0^t \partial_x^2 m^{\varepsilon} \,\d \tau 
  + \int_0^t \Phi^{\varepsilon}(\rho^{\varepsilon},m^{\varepsilon})\,\d W(\tau)\\
  &:= D_{\varepsilon}(t)+Z_{\varepsilon}(t),
  \end{aligned}$$
where $Z_{\varepsilon}(t):=\int_0^t \Phi^{\varepsilon}(\rho^{\varepsilon},m^{\varepsilon})\,\d W(\tau)$, 
and $D_{\varepsilon}(t)$ denotes the deterministic part.

We first show the uniform time-regularity of the deterministic part $D_{\varepsilon}(t)$. 
The energy estimate \eqref{iq-energy-estimate-on-whole-space} implies that
$\{\frac{(m^{\varepsilon})^2}{\rho^{\varepsilon}}\}_{\varepsilon >0}$ is uniformly bounded 
in $L^1(\Omega;L^{\infty}([0,T];L^1_{\rm loc}(\mathbb{R})))$, and hence uniformly bounded 
in $L^1(\Omega;L^{\infty}([0,T];W^{-1,q}_{\rm loc}(\mathbb{R})))$ for $q>1$ by the Sobolev 
embedding: $L^1 \hookrightarrow W^{-1,q}$ for $q>1$. 
Then $\{\partial_x(\frac{(m^{\varepsilon})^2}{\rho^{\varepsilon}})\}_{\varepsilon >0}$ 
is uniformly bounded in $L^1(\Omega;L^{\infty}([0,T];W^{-2,q}_{\rm loc}(\mathbb{R})))$ for $q>1$. 
The same argument yields that, for $q>1$,  $\{\partial_x P(\rho^{\varepsilon})\}_{\varepsilon >0}$ is uniformly 
bounded in $L^1(\Omega;L^{\infty}([0,T];W^{-2,q}_{\rm loc}(\mathbb{R})))$.
Since $\{m^{\varepsilon}\}_{\varepsilon >0}$ is uniformly bounded 
in $L^1(\Omega;L^{\infty}([0,T];L^{\frac{2\gamma_2}{\gamma_2+1}}_{\rm loc}(\mathbb{R})))$, 
we obtain that $\{\partial_x^2 m^{\varepsilon}\}_{\varepsilon >0}$ is uniformly bounded 
in $L^1(\Omega;L^{\infty}([0,T];W^{-2,\frac{2\gamma_2}{\gamma_2+1}}_{\rm loc}(\mathbb{R})))$.  
Therefore, $\{D_{\varepsilon}(t)\}_{\varepsilon >0}$ is uniformly bounded 
in $L^1(\Omega;C^{\alpha_1}([0,T];W^{-2,k_1}_{\rm loc}(\mathbb{R})))$ 
with $k_1 \le \min\{q,\frac{2\gamma_2}{\gamma_2+1}\}$ for any $0<\alpha_1<1$.

We next show the uniform time-regularity for the stochastic part $Z_{\varepsilon}(t):=\int_0^t \Phi^{\varepsilon}(\rho^{\varepsilon},m^{\varepsilon})\,\d W(\tau)$. By the Burkholder-Davis-Gundy inequality, 
for any $r_1\ge 1$,
  $$
  \begin{aligned}
   \E\big[\Big\|\int_{t_1}^{t_2} \Phi^{\varepsilon}(\rho^{\varepsilon},m^{\varepsilon})\,\d W(\tau)\Big\|_{H^{-1}(\hat K)}^{r_1}\big]
  &\le  \E \big[\Big(\int_{t_1}^{t_2}\sum_k \|a_k \zeta_k^{\varepsilon}\|_{H^{-1}(\hat K)}^2 \,\d \tau \Big)^{\frac{r_1}{2}}\big]\\
  &= \E \big[\Big(\int_{t_1}^{t_2}\sum_k \sup_{\varphi \in H^1_0(\hat K),\|\varphi\|_{H^1_0(\hat K)}\le 1}|\langle a_k \zeta_k^{\varepsilon}, \varphi \rangle|^2 \,\d \tau \Big)^{\frac{r_1}{2}}\big]\\
  &\le  \E \big[\Big(\int_{t_1}^{t_2}\sum_k \Big\|\frac{a_k\zeta_k^{\varepsilon}}{\sqrt{\rho^{\varepsilon}}}
  \Big\|_{L^2(\hat K)}^2 \|\sqrt{\rho^{\varepsilon}}\|_{L^2(\hat K)}^2 \,\d \tau \Big)^{\frac{r_1}{2}}\big],
  \end{aligned}
  $$
where, in the last inequality, we have used the Sobolev 
embedding: $H^1_0 \hookrightarrow L^{\infty}$. 
Using $\rho\le 1+ \rho^{\gamma_2}$
and \eqref{iq-density-control-by-relative-internal-energy-general-pressure-law-2}, we obtain
  $$
  \|\sqrt{\rho^{\varepsilon}}\|_{L^2(\hat K)}^2 \le (1+\bar C_{\gamma_2}\rho_{\infty}^{\gamma_2})|\hat K| 
  + \bar C_{\gamma_2} \int_{\hat K} e^*(\rho^{\varepsilon}, \rho_{\infty}) \,\d x.
  $$
By the same argument and using \eqref{iq-noise-coefficience-growth-conditions-for-parabolic-approximation}, we have
  \begin{equation}\label{iq-noise-estimate-control-by-energy}
  \sum_k \Big\|\frac{a_k\zeta_k^{\varepsilon}}{\sqrt{\rho^{\varepsilon}}}\Big\|_{L^2(\hat K)}^2 
  \le 
  C(|\hat K|)+ C\int_{\hat K} \big(\frac{(m^{\varepsilon})^2}{\rho^{\varepsilon}}
   + e^*(\rho^{\varepsilon},\rho_{\infty})\big)  \,\d x.
  \end{equation}
Therefore, combining the above estimates and applying the energy 
estimate \eqref{iq-energy-estimate-on-whole-space}, we have
  $$
  \E \big[\Big\|\int_{t_1}^{t_2} \Phi^{\varepsilon}(\rho^{\varepsilon},m^{\varepsilon})\,\d W(\tau)\Big\|_{H^{-1}(\hat K)}^{r_1}\big] \le 
  C(|\hat K|, E_{0,r_1})|t_2-t_1|^{\frac{r_1}{2}}.
  $$
Choosing $\frac{r_1}{2}>1$ and applying the Kolmogorov continuity criterion, we obtain
  $$
  \E \big[ \Big\| \int_0^{\cdot} \Phi^{\varepsilon}(\rho^{\varepsilon},m^{\varepsilon})\,\d W 
  \Big\|^{r_1}_{C^{\alpha}([0,T];H^{-1}(\hat K))} \big] 
  \le C(r_1,|\hat K|, E_{0,r_1}),
  $$
where $\alpha \in (0,\frac{\frac{r_1}{2}-1}{r_1})$.

Hence, combining the uniform time-regularity of $D_{\varepsilon}(t)$ and $Z_{\varepsilon}(t)$, 
we finally obtain that $\{m^{\varepsilon}\}_{\varepsilon >0}$ is uniformly bounded in 
$L^{r_1}(\Omega;C^{\alpha}([0,T];W^{-2,k_1}_{\rm loc}(\mathbb{R})))$ with $k_1 \le \min\{q,\frac{2\gamma_2}{\gamma_2+1}\}$ and some $r_1>2$.

By Lemma \ref{lem-compactness-result-from-ondrejat} again, 
for $k_1 \le \min\{q,\frac{2\gamma_2}{\gamma_2+1}\}, q>1, 
\alpha \in (0,\frac{\frac{r_1}{2}-1}{r_1})$, and $r_1>2$, the following embedding:
  $$
  L^{\infty}([0,T];L^{\frac{2\gamma_2}{\gamma_2+1}}_{\rm loc}(\mathbb{R})) \cap C^{\alpha}([0,T];W^{-2,k_1}_{\rm loc}(\mathbb{R})) \hookrightarrow C_{\rm w} ([0,T];L^{\frac{2\gamma_2}{\gamma_2+1}}_{\rm loc}(\mathbb{R}))
  $$
is compact. Thus, the required tightness of $\{m^{\varepsilon}\}_{\varepsilon >0}$ is proved.

Since $|m|^{\frac{3(\gamma_2+1)}{\gamma_2+3}}\le \frac{|m|^3}{\rho^2}+\rho^{\gamma_2+1}$, 
the tightness in $L^{\frac{3(\gamma_2+1)}{\gamma_2+3}}_{\rm w,loc}(\mathbb{R}^2_+)$
follows 
from Propositions \ref{prop-uniform-estimates-rho-gamma+1}--\ref{prop-higher-integrability-of-velocity-general-pressure-law}.

\medskip
\textbf{3}. {\it $\{\mu^{\varepsilon}\,:\,\varepsilon >0 \}$ is tight on $\mathcal{X}_{\mu}$}: $\,\,$ Denote
  $$
  \Lambda_R=\left\{\mu \in L^{\infty}_{\rm w} (\mathbb{R}^2_+; \mathcal{P}(\mathbb{R}^2_+))\,:\,
  {\begin{array}{lll}
  &\int_{0}^T\int_{\hat K} \int_{\mathbb{R}^2_+} \big(\rho^{\gamma_2}+|m|^{\frac{2\gamma_2}{\gamma_2+1}} \big) 
  \,\d \mu_{(t,x)}(\rho,m) \d x \d t \le R  \\[0.5mm]
  &\,\,\mbox{for any $T\ge 0$ and $\hat K \Subset \mathbb{R}$}
  \end{array}}
  \right\}.
  $$
By Proposition \ref{prop-compactness-criteria-for-young-measure}, 
$\Lambda_R$ is a compact subset in $L^{\infty}_{\rm w} (\mathbb{R}^2_+; \mathcal{P}(\mathbb{R}^2_+))$.

By the Chebyshev inequality,
  $$
  \begin{aligned}
  \mathbb{P}(\mu^{\varepsilon} \notin \Lambda_R) 
  \le \frac{1}{R}\E \big[ \int_{[0,T]\times \hat K} \big((\rho^{\varepsilon})^{\gamma_2}+|m^{\varepsilon}|^{\frac{2\gamma_2}{\gamma_2+1}} \big) \,\d x \d t\big] 
  \le \frac{1}{R}C(T, |\hat K|, E_{0,1}),
  \end{aligned}
  $$
where, in the last inequality, we have used the fact that
$|m^{\varepsilon}|^{\frac{2\gamma_2}{\gamma_2+1}} \le \rho^{\varepsilon} (u^{\varepsilon})^2+ (\rho^{\varepsilon})^{\gamma_2}$, \eqref{iq-density-control-by-relative-internal-energy-general-pressure-law-2}, 
and the energy estimate \eqref{iq-energy-estimate-on-whole-space}. 
Thus, the required tightness of $\{\mu^{\varepsilon}\,:\,\varepsilon >0 \}$ follows.

\medskip
\textbf{4}. {\it $\{\rho^{\varepsilon}_0\}_{\varepsilon >0 }$ is tight on $\mathcal{X}_{\rho_0}$, 
$\,\{m^{\varepsilon}_0\}_{\varepsilon >0}$ is tight on $\mathcal{X}_{m_0}$, 
and $\{\frac{m^{\varepsilon}_0}{\sqrt{\rho^{\varepsilon}_0}}\}_{\varepsilon >0}$ 
is tight on $\mathcal{X}_{\frac{m_0}{\sqrt{\rho_0}}}$}:
The conclusion follows from Prokhorov's theorem \cite[Theorem 5.2]{billingsley99}, 
since the constructed initial data $(\rho^{\varepsilon}_0, m^{\varepsilon}_0)$ satisfies $\rho^{\varepsilon}_0 \to \rho_0$ in $L^{\gamma_2}_{\rm loc}(\mathbb{R})$, $m^{\varepsilon}_0 \to m_0$ in $L^1_{\rm loc}(\mathbb{R})$, 
and $\frac{m^{\varepsilon}_0}{\sqrt{\rho^{\varepsilon}_0}}\to\frac{m_0}{\sqrt{\rho_0}}$ in $L^2(\mathbb{R})$.

Finally, since the singleton $\mathcal{L}_W$ is a Radon measure on the Polish space $\mathcal{X}_W$, it is, of course, tight. 
Therefore, combining the above results, we obtain our final conclusion.
\end{proof}

\smallskip
Note that the path space $\mathcal{X}$ is not a Polish space, since the weak topology is generally not metrizable. 
However, $\mathcal{X}$ belongs to the class of sub-Polish spaces, which enables us to use the Jakubowski-Skorokhod 
representation theorem \cite{jakubowski97}. 
More specifically, recall that 
$L^{\infty}_{\rm w} (\mathbb{R}^2_+; \mathcal{M}(\mathbb{R}^2_+))=\big(L^1(\mathbb{R}^2_+;C_0(\mathbb{R}^2_+))\big)^*$ 
and $L^1(\mathbb{R}^2_+;C_0(\mathbb{R}^2_+))$ is a separable Banach space ({\it cf.} \cite[Section 8.18]{Edwards95}). 
Thus, $\mathcal{X}_{\mu}=L^{\infty}_{\rm w} (\mathbb{R}^2_+; \mathcal{M}(\mathbb{R}^2_+))$ equipped with the weak-star 
topology is a sub-Polish space. 
Moreover, $C_{\rm w}([0,T];W^{k,p}_{\rm loc}(\mathbb{R}))$ with $k\ge 0$ and $1<p<\infty$ 
also belongs to the class of sub-Polish spaces ({\it cf}. \cite[Proposition B.3]{ondrejat10}). 
Hence, $\mathcal{X}_{\rho}=C_{\rm w}([0,T];L^{\gamma_2}_{\rm loc}(\mathbb{R}))$ and $\mathcal{X}_m=C_{\rm w}([0,T];L^{\frac{2\gamma_2}{\gamma_2+1}}_{\rm loc}(\mathbb{R}))$ are also sub-Polish spaces. Therefore, by Proposition \ref{cor-tightness-of-L-varepsilon}, we apply the Jakubowski-Skorokhod representation theorem to obtain the following result:

\begin{proposition}\label{prop-apply-jakubowski-skorokhod-representation-to-take-limit}
There exists a complete probability space $(\tilde{\Omega}, \tilde{\mathcal{F}}, \tilde{\mathbb{P}})$ 
with $\mathcal{X}$-valued Borel measurable random variables 
$\tilde X^{\varepsilon}:=(\tilde \mu^{\varepsilon},\tilde W^{\varepsilon},\tilde \rho^{\varepsilon},\tilde m^{\varepsilon},\tilde \rho^{\varepsilon}_0,\tilde m^{\varepsilon}_0, \tilde \vartheta^{\varepsilon}_0 )$ for $\varepsilon >0$ 
and $\tilde X:=(\tilde \mu,\tilde W, \tilde \rho, \tilde m, \tilde \rho_0, \tilde m_0, \tilde \vartheta_0)$ 
such that the following properties hold $($up to a subsequence$)${\rm :}
  \begin{itemize}
  \item[\rm (i)] The law of $\tilde X^{\varepsilon}$ is equal to $\mathcal{L}^{\varepsilon}${\rm ;}
  \item[\rm (ii)] $\tilde X^{\varepsilon}$ converges $\tilde{\mathbb{P}}$-{\it a.s.} to $\tilde X$ in $\mathcal{X}$.
  \end{itemize}
\end{proposition}

\smallskip
We identify the initial data as follows:
\begin{proposition}\label{prop-identify-initial-data}
For {\it a.e.} $x\in \mathbb{R}$, $\,\,\tilde \rho^{\varepsilon}_0=\tilde \rho^{\varepsilon}|_{t=0}, 
\,\,\tilde m^{\varepsilon}_0=\tilde m^{\varepsilon}|_{t=0}$,  
and $\,\tilde \vartheta^{\varepsilon}_0=\frac{\tilde m^{\varepsilon}_0}{\sqrt{\tilde\rho^{\varepsilon}_0}}$ 
$\tilde{\mathbb{P}}$-{\it a.s.}.
\end{proposition}

\begin{proof}
$\tilde \rho^{\varepsilon}_0=\tilde \rho^{\varepsilon}|_{t=0}$ and 
$\tilde m^{\varepsilon}_0=\tilde m^{\varepsilon}|_{t=0}$ directly follow 
from the equality of laws in Proposition \ref{prop-apply-jakubowski-skorokhod-representation-to-take-limit}.

\smallskip
We now prove that $\tilde \vartheta^{\varepsilon}_0=\frac{\tilde m^{\varepsilon}_0}{\sqrt{\tilde\rho^{\varepsilon}_0}}$. 
Denote $\mathcal{L}^{\varepsilon}_{m,\rho,\frac{m}{\sqrt{\rho}}}:=\mathcal{L}(m^{\varepsilon}_0, \rho^{\varepsilon}_0, \frac{m^{\varepsilon}_0}{\sqrt{\rho^{\varepsilon}_0}})$ 
and $\tilde{\mathcal{L}}^{\varepsilon}_{m,\rho,\frac{m}{\sqrt{\rho}}}:=\mathcal{L}(\tilde m^{\varepsilon}_0, \tilde \rho^{\varepsilon}_0, \frac{\tilde m^{\varepsilon}_0}{\sqrt{\tilde \rho^{\varepsilon}_0}})$. 
Then, by Proposition \ref{prop-apply-jakubowski-skorokhod-representation-to-take-limit}(i), 
we see that $\mathcal{L}^{\varepsilon}_{m,\rho,\frac{m}{\sqrt{\rho}}}
=\tilde{\mathcal{L}}^{\varepsilon}_{m,\rho,\frac{m}{\sqrt{\rho}}}$. Thus, for any $\varphi \in C_{\rm c}^{\infty}(\mathbb{R})$,
  \begin{equation}\label{eq-identify-initial-data}
  \begin{aligned}
  0&= \E \big[\Big|\int_{\mathbb{R}} \big(m^{\varepsilon}_0-\sqrt{\rho^{\varepsilon}_0}\frac{m^{\varepsilon}_0}{\sqrt{\rho^{\varepsilon}_0}}\big)\varphi\, \,\d x \Big|\big]\\
  &= \int_{\mathcal{X}_m \times \mathcal{X}_{\rho} \times \mathcal{X}_{\frac{m}{\sqrt{\rho}}}} \Big| \int_{\mathbb{R}} (z-\sqrt{y}v)\varphi\, \,\d x \Big| \,\d \mathcal{L}^{\varepsilon}_{m,\rho,\frac{m}{\sqrt{\rho}}}(z,y,v)\\
  &= \int_{\mathcal{X}_m \times \mathcal{X}_{\rho} \times \mathcal{X}_{\frac{m}{\sqrt{\rho}}}} \Big| \int_{\mathbb{R}} (z-\sqrt{y}v)\varphi\, \,\d x \Big| \,\d \tilde{\mathcal{L}}^{\varepsilon}_{m,\rho,\frac{m}{\sqrt{\rho}}}(z,y,v)\\
  &=\tilde \E \big[\Big|\int_{\mathbb{R}} \big(\tilde m^{\varepsilon}_0-\sqrt{\tilde \rho^{\varepsilon}_0}\tilde \vartheta^{\varepsilon}_0\big)\varphi\, \,\d x \Big|\big],
  \end{aligned}
  \end{equation}
where we have used the fact that the map:
  $$
  \begin{aligned}
  \mathcal{X}_m \times \mathcal{X}_{\rho} \times \mathcal{X}_{\frac{m}{\sqrt{\rho}}} &\longmapsto \mathbb{R},\\
  (z,y,v) &\longmapsto \int_{\mathbb{R}} (z-\sqrt{y}v)\varphi\, \,\d x,
  \end{aligned}
  $$
is continuous and hence measurable. Since $\varphi \in C_c^{\infty}(\mathbb{R})$ is arbitrary and 
$\tilde m^{\varepsilon}_0-\sqrt{\tilde \rho^{\varepsilon}_0}\tilde \vartheta^{\varepsilon}_0 \in L^1_{\rm loc}(\mathbb{R})$, 
we obtain from \eqref{eq-identify-initial-data} that, 
for {\it a.e.} $x\in \mathbb{R}$, 
$\tilde \vartheta^{\varepsilon}_0=\frac{\tilde m^{\varepsilon}_0}{\sqrt{\tilde\rho^{\varepsilon}_0}}$ $\tilde{\mathbb{P}}$-{\it a.s.}. 
\end{proof}

\section{Stochastic Compensated Compactness Framework}\label{sec-compensated-compactness-and-reduction-of-young-measure}
This section is devoted to establishing the following stochastic compensated compactness framework in $L^p$:

\begin{theorem}\label{thm-stochastic-Lp-compensated-compactness-framework}
Let $(\rho^{\varepsilon}, m^{\varepsilon})$ be a sequence of random variables satisfying that,
for any $K\Subset \mathbb{R}$, there exists $C(K,p,T)>0$ independent of $\varepsilon$ such that
  $$
  \begin{aligned}
  \E \big[\Big( \int_0^t \int_{K} \big(\rho^{\varepsilon} P(\rho^{\varepsilon})+\frac{|m^{\varepsilon}|^3}{(\rho^{\varepsilon})^2}\big) \,\d x\d \tau \Big)^p\big]
  \le  C(K,p,T) \qquad \mbox{for any $t\in[0,T]$},
  \end{aligned}
  $$
where $P(\rho)$ satisfies \eqref{eq-general-pressure-law-1}--\eqref{eq-general-pressure-law-2}, 
so that $(\rho^{\varepsilon}, m^{\varepsilon})$ are tight in
$L^{\gamma_2+1}_{\rm w,loc}(\mathbb{R}^2_+)\times L^{\frac{3(\gamma_2+1)}{\gamma_2+3}}_{\rm w,loc}(\mathbb{R}^2_+)$
and possess a Skorokhod representation $(\tilde \rho^{\varepsilon}, \tilde m^{\varepsilon})$. 
Assume that
  $$
  \big\{\eta^{\psi}(\rho^{\varepsilon},m^{\varepsilon})_t +q^{\psi}(\rho^{\varepsilon},m^{\varepsilon})_x\big\}_{\varepsilon>0}
  \qquad \text{are tight in $W^{-1,1}_{\rm loc}(\mathbb{R}^2_+)$}
  $$
for all weak entropy pairs $(\eta^{\psi},q^{\psi})$ of system \eqref{eq-stochastic-Euler-system} defined in \eqref{eq-entropy-flux-representation} with generating function $\psi\in C_{\rm c}^2(\mathbb{R})$. 
Then there exist random variables $(\tilde\rho,\tilde m)$ such that 
almost surely $(\tilde \rho^{\varepsilon},\,\tilde m^{\varepsilon}) \to (\tilde \rho,\, \tilde m)$  
almost everywhere and $($up to a subsequence$)$ almost surely,
  $$
  (\tilde \rho^{\varepsilon},\,\tilde m^{\varepsilon}) \to (\tilde \rho,\, \tilde m)\qquad 
  \text{in $L^{\hat p}_{\rm loc}(\mathbb{R}^2_+)\times L^{\hat q}_{\rm loc}(\mathbb{R}^2_+)$}
  $$
for $\hat p\in [1, \gamma_2+1 )$ and $\hat q\in [ 1, \frac{3(\gamma_2+1)}{\gamma_2+3} )$.
\end{theorem}

We first provide the following result regarding the admissible test functions and 
the convergence of the random Young measure $\tilde \mu$ obtained 
in Proposition \ref{prop-apply-jakubowski-skorokhod-representation-to-take-limit}. 
The proposition below is a version of Proposition 5.1 in \cite{chenperepelitsa10CPAM} 
for the random Young measure (however, we don't use the compactification technique here). 
From now on, we denote $\mathfrak{T}:=\{(\rho,m)\in \mathbb{R}^2_+\,:\, \rho >0\}$.

\begin{proposition}\label{prop-test-func-and-convergence-for-random-young-measure}
For the limit random Young measure $\tilde \mu$ obtained in {\rm Proposition \ref{prop-apply-jakubowski-skorokhod-representation-to-take-limit}}, the following properties hold{\rm :}
  \begin{itemize}
  \item[\rm (i)] $\tilde \mu$ satisfies the following additional integrability property{\rm :} 
  For any $\hat K\Subset \mathbb{R}$ and any $T>0$, there exists 
  $\hat C>0$
  such that
     \begin{equation}\label{iq-higher-integrability-for-young-measure}
     \tilde{\E} \big[\int_0^T \int_{\hat K} \int_{\mathfrak{T}} \big(\rho P(\rho)+\frac{|m|^3}{\rho^2}\big) \,\d \tilde \mu_{(t,x)}(\rho,m) \d x \d t\big] \le 
     \hat C.
     \end{equation}
  \item[\rm (ii)] The space of admissible test functions for $\tilde \mu$ can be extended as follows{\rm :} 
  Let $\phi$ satisfy
    \begin{itemize}
    \item[\rm (a)]$\phi\in C(\mathbb{R}^2_+)$ with $\phi|_{\partial{\mathfrak{T}}}=0$, where $\partial{\mathfrak{T}}$ denotes the boundary of $\mathbb{R}^2_+${\rm ;}
    \item[\rm (b)]${\rm supp} (\phi) \subset \{(\rho,m)\,:\, w_1(\rho,m) \le a, w_2(\rho,m) \ge -a\}$ 
    for some constant $a>0$, where $(w_1, w_2)$ are the Riemann invariants 
    defined in \eqref{eq-riemann-invariants-representation-formula}{\rm ;}
    \item[\rm (c)]$|\phi(\rho,m)|\le C\rho^{\beta(\gamma_2+1)}$ uniformly in $m$ for large $\rho >1$ and some $\beta\in (0,1)$, where $C$ is a constant independent of $(\rho, m)$.
    \end{itemize}
\noindent
    Then, for any $\hat K\Subset \mathbb{R}$ and any $T>0$,
      \begin{align}
      &\tilde{\E}\big[\int_0^T \int_{\hat K} \int_{\mathfrak{T}} \phi(\rho,m) \,\d \tilde \mu_{(t,x)}( \rho,m) \d x \d t\big] 
      < \infty,\label{iq-admissible-test-func}\\
      &\lim_{\varepsilon \to 0}\tilde{\E} \big[\Big| \int_0^T \int_{\hat K} \Big(\int_{\mathfrak{T}}  \phi(\rho,m) \,\d (\tilde \mu^{\varepsilon}_{(t,x)}
      -\tilde \mu_{(t,x)})( \rho,m) \Big)\varphi(t,x) \,\d x \d t  \Big|\big]= 0
      \label{eq-convergence-of-young-measure-for-admissible-func}
      \end{align}
    for any $\varphi \in L^q_{\rm loc}(\mathbb{R}^2_+)$ with $q=\frac{1}{1-\beta}$.
  \end{itemize}
\end{proposition}

\begin{proof}
We first prove (i). Let $w_k(\rho,m) \ge 0$ be a continuous cut-off function such that $w_k\equiv 1$ 
on the set $\{(\rho,m)\in \mathbb{R}^2_+\,:\, \frac{1}{k}\le \mathcal{K}(\rho)\le k, |m|\le k\}$ 
and $w_k \equiv 0$ outside the set $\{(\rho,m)\in \mathbb{R}^2_+\,:\, \frac{1}{2k}\le \mathcal{K}(\rho)\le 2k, |m|\le 2k\}$. 
Then $\big(\rho P(\rho)+\frac{|m|^3}{\rho^2}\big)w_k(\rho,m) \in C_0(\mathbb{R}^2_+)$.
We obtain that, for any $\hat K\Subset\mathbb{R}$,  $\tilde{\mathbb{P}}$-{\it a.s.},
  $$
  \begin{aligned}
  &\int_0^T \int_{\hat K} \int_{\mathfrak{T}} \big(\rho P(\rho)+\frac{|m|^3}{\rho^2}\big)w_k(\rho,m) \,\d \tilde \mu^{\varepsilon}_{(t,x)}(\rho,m) \d x \d t\\
   &\to \int_0^T \int_{\hat K} \int_{\mathfrak{T}} 
   \big(\rho P(\rho)+\frac{|m|^3}{\rho^2}\big)w_k(\rho,m) \,\d \tilde \mu_{(t,x)}(\rho,m) \d x \d t  
   \qquad\mbox{as $\varepsilon \to 0$}.
  \end{aligned}
$$
Additionally, by the equality of laws from Proposition \ref{prop-apply-jakubowski-skorokhod-representation-to-take-limit}, 
it follows from 
Propositions \ref{prop-uniform-estimates-rho-gamma+1}--\ref{prop-higher-integrability-of-velocity-general-pressure-law}
that
  $$
  \begin{aligned}
  &\tilde{\E} \big[\int_0^T \int_{\hat K} \int_{\mathfrak{T}} \big(\rho P(\rho)+\frac{|m|^3}{\rho^2}\big)w_k(\rho,m) \,\d \tilde \mu^{\varepsilon}_{(t,x)}(\rho,m) \d x \d t\big]\\
  &\le \E \big[\int_0^T \int_{\hat K} \big(\rho^{\varepsilon} P(\rho^{\varepsilon})+\frac{|m^{\varepsilon}|^3}{(\rho^{\varepsilon})^2}\big) w_k(\rho^{\varepsilon},m^{\varepsilon}) \,\d x \d t \big]\le {\hat C},
  \end{aligned}$$
where $C$
is independent of $(\varepsilon,k)$. Thus, by Fatou's lemma, we have
  $$
  \begin{aligned}
  &\tilde{\E} \big[\int_0^T \int_{\hat K} \int_{\mathfrak{T}} \big(\rho P(\rho)+\frac{|m|^3}{\rho^2}\big)w_k(\rho,m) \,\d \tilde \mu_{(t,x)}(\rho,m) \d x \d t\big]\\
  & \le \liminf_{\varepsilon \to 0}\,\tilde{\E} \big[\int_0^T \int_{\hat K} \int_{\mathfrak{T}} \big(\rho P(\rho)+\frac{|m|^3}{\rho^2}\big)w_k(\rho,m) \,\d \tilde \mu^{\varepsilon}_{(t,x)}(\rho,m) \d x \d t \big]\le C,
  \end{aligned}
  $$
which then implies from the monotone convergence theorem that
  \begin{equation}\label{iq-uniform-higher-integrability-bound-for-tilde-mu}
  \begin{aligned}
  &\tilde{\E}\big[ \int_0^T \int_{\hat K} \int_{\mathfrak{T}} \big(\rho P(\rho)+\frac{|m|^3}{\rho^2}\big)
   \,\d \tilde \mu_{(t,x)}(\rho,m) \d x \d t\big]\\
  & = \lim_{k \to \infty}\tilde{\E} \big[\int_0^T \int_{\hat K} \int_{\mathfrak{T}} \big(\rho P(\rho)
  +\frac{|m|^3}{\rho^2}\big)w_k(\rho,m) \,\d \tilde \mu_{(t,x)}(\rho,m) \d x \d t\big] \le C.
  \end{aligned}
  \end{equation}
Then the conclusion in (i) follows.

\medskip
Next, we prove (ii) by two steps.

\smallskip
\textbf{1}. Without loss of generality, we may suppose that $\phi \ge 0$; otherwise, 
we can use the decomposition $\phi=\phi^+ -\phi^-$ with $\phi^+=\max\{\phi,0\}$ 
and $\phi^- =\max\{-\phi,0\}$ satisfying the same conditions as $\phi$. 
Let $w_k$ be as in the proof of (i). Then $\phi w_k \in C_0(\mathbb{R}^2_+)$ 
and $\langle \tilde \mu_{(t,x)} , \phi w_k \rangle $ is well-defined. Since
  $$
  \phi w_k \le C\rho^{\beta(\gamma_2+1)}\le C\big(1+\rho^{\gamma_2+1}\big),
  $$
by the conclusion of (i) and \eqref{iq-lower-upper-bound-for-general-pressure-2}, 
we apply the dominated convergence theorem to yield 
$$
\begin{aligned}
  \tilde{\E} \big[\int_0^T \int_{\hat K} \int_{\mathfrak{T}} \phi \,\d \tilde \mu_{(t,x)}(\rho,m) \d x \d t\big]
  &=\lim_{k\to \infty} 
  \tilde{\E}\big[ \int_0^T \int_{\hat K} \int_{\mathfrak{T}} \phi w_k\, \,\d \tilde \mu_{(t,x)}(\rho,m) \d x \d t\big]\\
  &\le  \tilde{\E}\big[ \int_0^T \int_{\hat K} \int_{\mathfrak{T}} C(1+\rho^{\gamma_2+1})\, \,\d \tilde \mu_{(t,x)}(\rho,m) 
  \d x \d t\big]< \infty.
  \end{aligned}$$
Thus, the first conclusion \eqref{iq-admissible-test-func} in (ii) is proved.

\smallskip
\textbf{2}. We now prove the second conclusion \eqref{eq-convergence-of-young-measure-for-admissible-func} in (ii). 
We first assume $\varphi \in L^{\infty}_{\rm loc}(\mathbb{R}^2_+)$. 
Using Proposition \ref{prop-apply-jakubowski-skorokhod-representation-to-take-limit} and similar arguments as in Step \textbf{1},
it is direct to adapt those estimates in the deterministic case (see \cite[Proposition 5.1]{chenperepelitsa10CPAM}) 
to the stochastic case here (hence we omit the details) to derive \eqref{eq-convergence-of-young-measure-for-admissible-func} 
for $\varphi \in L^{\infty}_{\rm loc}(\mathbb{R}^2_+)$.

We now prove \eqref{eq-convergence-of-young-measure-for-admissible-func} for more general $\varphi \in L^q_{\rm loc}(\mathbb{R}^2_+)$
with $q=\frac{1}{1-\beta}$. 
Let $\nu_{(t,x)} \in \{\tilde \mu^{\varepsilon}_{(t,x)}\}_{\varepsilon >0} \cup \{\tilde \mu_{(t,x)}\}$. 
Propositions \ref{prop-uniform-estimates-rho-gamma+1}--\ref{prop-higher-integrability-of-velocity-general-pressure-law}
enable us to utilize similar arguments as those in the proof of (i) above to obtain
  \begin{equation}\label{iq-uniform-higher-integrability-bound-for-tilde-mu-varepsilon}
  \tilde{\E}\big[ \int_0^T \int_{\hat K} \int_{\mathfrak{T}} \big(\rho P(\rho)+\frac{|m|^3}{\rho^2}\big) \,\d \tilde \mu^{\varepsilon}_{(t,x)}(\rho,m) \d x \d t\big] \le {\hat C}.
  \end{equation}
Thus, the H\"older inequality gives
  $$
  \begin{aligned}
  &\tilde{\E}\big[\Big|\int_0^T \int_{\hat K} \varphi \int_{\mathfrak{T}} \phi(\rho,m) 
   \,\d \nu_{(t,x)}(\rho,m) \d x \d t\Big|\big]\\[0.5mm]
  & \le\Big( \tilde{\E}\big[ \int_0^T \int_{\hat K} \int_{\mathfrak{T}} |\phi(\rho,m)|^{\frac{1}{\beta}}
   \, \d \nu_{(t,x)}(\rho,m) \d x \d t\big] \Big)^{\beta}
  \Big( \tilde{\E}\big[ \int_0^T \int_{\hat K} \int_{\mathfrak{T}} |\varphi|^q \, \d \nu_{(t,x)}(\rho,m) \d x \d t \big]\Big)^{\frac{1}{q}}\\[0.5mm]
  &\le {\hat C}
  \Big( \int_0^T \int_{\hat K} |\varphi|^q \, \d x \d t \Big)^{\frac{1}{q}}.
  \end{aligned}
  $$
Since
  $$
  \varphi \mapsto \tilde{\E} \big[ \Big|\int_0^T \int_{\hat K} \varphi \int_{\mathfrak{T}} 
  \phi(\rho,m) \, \d \nu_{(t,x)}(\rho,m) \d x \d t \Big|\big]
  $$
is linear, by density argument, then 
the second conclusion \eqref{eq-convergence-of-young-measure-for-admissible-func} 
for more general $\varphi \in L^q_{\rm loc}(\mathbb{R}^2_+)$ with $q=\frac{1}{1-\beta}$ follows.
\end{proof}

From now on, for an admissible function $f(\rho,m)$, we denote
  $$
  \bar f :=\int_{\mathbb{R}^2_+} f(\rho,m) \,\d \tilde \mu_{(t,x)}(\rho,m)
  =\langle \tilde \mu_{(t,x)}, f \rangle.
  $$

We now collect some properties for the entropy pairs generated by compactly supported functions.
\begin{lemma}\label{lem-properties-for-entropy-generate-by-compact-support-func}
Let $(\eta^{\psi},q^{\psi})$ be a weak entropy pair generated by 
$\psi \in C_{\rm c}^2(\mathbb{R})$ as defined 
in \eqref{eq-entropy-flux-representation}. 
Then the following statements hold{\rm :}
\begin{itemize}
\item[\rm (i)] For the general pressure law, if $\eta^{\psi}$ is regarded as a function of $(\rho,u)$, then
  \begin{align}
  &|\eta^{\psi}_{\rho}(\rho,u)|\le C_{\psi}(1+ \rho^{\theta_1}) \qquad \text{when $\rho\le \rho^{\star}$},\label{iq-control-for-partial-rho-of-entropy-generate-by-compact-support-func-1}\\
  &|\eta^{\psi}_{\rho}(\rho,u)|\le C_{\psi} \rho^{\theta_2} \qquad\qquad\,\,\,
  \text{when $\rho\ge \rho^{\star}$}.\label{iq-control-for-partial-rho-of-entropy-generate-by-compact-support-func-2}
  \end{align}
If $\eta^{\psi}$ is regarded as a function of $(\rho,m)$, then
  \begin{equation}\label{iq-control-for-partial-m-of-entropy-generate-by-compact-support-func}
  |\eta_{m}^{\psi}(\rho,m)|+|\rho \partial_m^2 \eta^{\psi}(\rho,m) | \le C_{\psi}.
  \end{equation}
If $\eta^{\psi}_m$ is regarded as a function of $(\rho,u)$, then
  \begin{equation}\label{iq-control-for-partial-m-of-entropy-in-rho-u-coordinate-generate-by-compact-support-func}
  \begin{aligned}
  &|\rho^{1-\theta_1}\partial_{m \rho}\eta^{\psi}(\rho,u)|\le C_{\psi} \qquad \text{when $\rho\le \rho^{\star}$},\\
  &|\rho^{1-\theta_2}\partial_{m \rho}\eta^{\psi}(\rho,u)|\le C_{\psi} \qquad \text{when $\rho\ge \rho^{\star}$}.
  \end{aligned}
  \end{equation}
\item[\rm (ii)] If $\eta^{\psi}$ is regarded as a function of $(\rho,m)$, then
    \begin{equation}\label{iq-second-order-derivative-of-entropy-control-by-energy-general-pressure-law}
    |(\partial_x U)^{\top}\,\nabla^2 \eta^{\psi}(U)\,\partial_x U|
    \le C_{\psi} (\partial_x U)^{\top}\, \nabla^2 \eta_E(U)\,\partial_x U.
    \end{equation}
\end{itemize}
\end{lemma}

\begin{proof}
For (i),
all the estimates come from \cite[Lemmas 4.5 and 4.7]{chenhuangLiWangWang24CMP}.

\smallskip
(ii) follows from \eqref{iq-control-for-partial-m-of-entropy-generate-by-compact-support-func}--\eqref{iq-control-for-partial-m-of-entropy-in-rho-u-coordinate-generate-by-compact-support-func} once we notice that 
$$
\begin{aligned}
  &(\partial_x U)^{\top}\,\nabla^2 \eta^{\psi}(U)\, \partial_x U\\
  &=(\partial_x \rho)^2 \big(\mathcal{K}^{\prime}(\rho)\big)^2 \partial_u^2 \eta^{\psi}(\rho,u) + 2\partial_x \rho \partial_x u \big( \partial_{\rho u}\eta^{\psi}(\rho,u)-\frac{1}{\rho}\partial_u \eta^{\psi}(\rho,u) \big) 
  + (\partial_x u)^2 \partial_u^2 \eta^{\psi}(\rho,u)\\
  &=(\partial_x \rho)^2 \big(\mathcal{K}^{\prime}(\rho)\big)^2 \rho^2 \partial_m^2 \eta^{\psi} + 2\partial_x \rho \partial_x u \rho \partial_{ m \rho}\eta^{\psi}(\rho,u) 
  + (\partial_x u)^2 \rho^2 \partial_m^2 \eta^{\psi}.
\end{aligned}
$$
\end{proof}

\begin{lemma}\label{prop-tightness-of-partial-x-eta}
Let $(\eta^{\psi},q^{\psi})$ be a weak entropy pair generated by 
$\psi \in C_{\rm c}^2(\mathbb{R})$ as defined 
in \eqref{eq-entropy-flux-representation}. 
Suppose that $(\rho^{\varepsilon},m^{\varepsilon})$ with $m^{\varepsilon}=\rho^{\varepsilon}u^{\varepsilon}$ 
is a sequence of solutions to the parabolic approximation equations \eqref{eq-parabolic-approximation}. 
Then $\{\varepsilon \partial_x^2 \eta^{\psi}(U^{\varepsilon})\}_{\varepsilon >0}$ 
is tight in $H^{-1}_{\rm loc}(\mathbb{R}^2_+)$.
\end{lemma}

\begin{proof}
We prove the conclusion by verifying that $\varepsilon \partial_x \eta^{\psi}(U^{\varepsilon})\to 0$ 
in probability in $L^2_{\rm loc}(\mathbb{R}^2_+)$.
Note that
  $$
  |\partial_x \eta^{\psi}(\rho^{\varepsilon},m^{\varepsilon})|=|(\eta^{\psi}_{\rho^{\varepsilon}}+ u^{\varepsilon}\eta^{\psi}_{m^{\varepsilon}})\partial_x \rho^{\varepsilon}+ \eta_{m^{\varepsilon}}^{\psi}\rho^{\varepsilon}\partial_x u^{\varepsilon}|.
  $$

Notice that, for the general pressure law case, $\eta^{\psi}_{\rho^{\varepsilon}}(\rho^{\varepsilon},m^{\varepsilon})+ u^{\varepsilon}\eta^{\psi}_{m^{\varepsilon}}(\rho^{\varepsilon},m^{\varepsilon})= \eta^{\psi}_{\rho^{\varepsilon}}(\rho^{\varepsilon},u^{\varepsilon})$. Thus,
for any $\hat K \Subset \mathbb{R}$, 
$$
\begin{aligned}
&\int_0^T \int_{\hat K} \varepsilon^2 |\partial_x \eta^{\psi}(\rho^{\varepsilon},m^{\varepsilon})|^2 \,\d x \d t\\
 &\le C_{\psi}\int_0^T \int_{\hat K} \varepsilon^2 |\eta^{\psi}_{\rho^{\varepsilon}}(\rho^{\varepsilon},u^{\varepsilon})|^2|\partial_x\rho^{\varepsilon}|^2 \,\d x\d t 
 + C_{\psi}\int_0^T \int_{\hat K} \varepsilon^2 (\rho^{\varepsilon})^2|\partial_x u^{\varepsilon}|^2 \,\d x \d t.
 \end{aligned}
$$
For the first term on the right-hand side, by \eqref{iq-control-for-partial-rho-of-entropy-generate-by-compact-support-func-1}--\eqref{iq-control-for-partial-rho-of-entropy-generate-by-compact-support-func-2}, we have
  $$
  |\eta^{\psi}_{\rho^{\varepsilon}}(\rho^{\varepsilon},u^{\varepsilon})|^2|\partial_x\rho^{\varepsilon}|^2 \le \begin{cases}
  C_{\psi} \big(1+ (\rho^{\varepsilon})^{\gamma_1-1}\big)|\partial_x\rho^{\varepsilon}|^2\qquad \text{when $\rho^{\varepsilon}\le \rho^{\star}$},\\[0.5mm]
  C_{\psi} (\rho^{\varepsilon})^{\gamma_2-1}|\partial_x\rho^{\varepsilon}|^2\quad\qquad\quad\,\,\,\, \text{when $\rho^{\varepsilon}\ge \rho^{\star}$}.
  \end{cases}
  $$
We first treat $(\rho^{\varepsilon})^{\gamma_1-1} |\partial_x\rho^{\varepsilon}|^2$ and $(\rho^{\varepsilon})^{\gamma_2-1}|\partial_x\rho^{\varepsilon}|^2$. 
From \eqref{iq-lower-upper-bound-for-general-pressure-1}, \eqref{iq-lower-upper-bound-for-k(rho)-1}, and 
Theorem \ref{thm-wellposedness-parabolic-approximation-R}(ii), we see that,
when $\rho^{\varepsilon}\le \rho_{\star}$, $(\rho^{\varepsilon})^{\gamma_1-2}\le C \frac{P^{\prime}(\rho^{\varepsilon})}{\rho^{\varepsilon}}=C\big(\rho^{\varepsilon} e(\rho^{\varepsilon})\big)^{\prime\prime}$ and 
$(\rho^{\varepsilon})^{\theta_1}\le C\mathcal{H}^{\varepsilon}$. When $\rho_{\star}\le \rho^{\varepsilon} \le \rho^{\star}$, by direct calculation and again Theorem \ref{thm-wellposedness-parabolic-approximation-R}(ii), 
we obtain that $(\rho^{\varepsilon})^{\gamma_1-2}\le C(\rho_{\star}, \rho^{\star}) \big(\rho^{\varepsilon} e(\rho^{\varepsilon})\big)^{\prime\prime}$ 
and $C(\rho_{\star},\rho^{\star})\big(\rho^{\varepsilon}-\rho_{\star}\big)
\le \mathcal{K}(\rho^{\varepsilon})\le \mathcal{H}^{\varepsilon}$. 
Therefore, we have
  $$
  (\rho^{\varepsilon})^{\gamma_1-1}|\partial_x\rho^{\varepsilon}|^2 \le \begin{cases}
  C(\mathcal{H}^{\varepsilon})^{\frac{1}{\theta_1}} \big(\rho^{\varepsilon} e(\rho^{\varepsilon})\big)^{\prime\prime} |\partial_x\rho^{\varepsilon}|^2
     \qquad\text{when $\rho^{\varepsilon}\le \rho_{\star}$},\\[0.5mm]
  C\mathcal{H}^{\varepsilon} \big(\rho^{\varepsilon} e(\rho^{\varepsilon})\big)^{\prime\prime} |\partial_x\rho^{\varepsilon}|^2
  \qquad\qquad\,\text{when $\rho_{\star}\le \rho^{\varepsilon}\le \rho^{\star}$}.
  \end{cases}
  $$
Similarly, from \eqref{iq-lower-upper-bound-for-general-pressure-2}, \eqref{iq-lower-upper-bound-for-k(rho)-2}, and 
Theorem \ref{thm-wellposedness-parabolic-approximation-R}(ii), 
we obtain that, when $\rho^{\varepsilon}\ge \rho^{\star}$,
  $$
  (\rho^{\varepsilon})^{\gamma_2-1}|\partial_x\rho^{\varepsilon}|^2 \le C(\mathcal{H}^{\varepsilon})^{\frac{1}{\theta_2}} \big(\rho^{\varepsilon} e(\rho^{\varepsilon})\big)^{\prime\prime} |\partial_x\rho^{\varepsilon}|^2.
  $$
By the energy estimate \eqref{iq-energy-estimate-on-whole-space} and the Chebyshev inequality, we have
  $$
  \mathbb{P} \Big\{ \varepsilon \int_0^T \int_{\mathbb{R}} 
  \big((\rho^{\varepsilon} e(\rho^{\varepsilon}))^{\prime \prime}(\partial_x\rho^{\varepsilon})^2 
  + \rho^{\varepsilon} (\partial_x u^{\varepsilon})^2\big) \,\d x \d t \le R \Big\} \to 1
  \qquad\,\, \text{as}\ R\to \infty.
  $$
Denote
  $$
  A_R:=\Big\{ \varepsilon \int_0^T \int_{\mathbb{R}} \big((\rho^{\varepsilon} e(\rho^{\varepsilon}))^{\prime \prime}(\partial_x\rho^{\varepsilon})^2 + \rho^{\varepsilon} (\partial_x u^{\varepsilon})^2\big) \,\d x \d t \le R \Big\}.
  $$
Then, restricted to the event $A_R$, letting $i=1$ when $\rho_{\varepsilon}\le \rho^{\star}$ 
and $i=2$ when $\rho_{\varepsilon}\ge \rho^{\star}$, we have
  $$
  \int_0^T \int_{\hat K} \varepsilon^2 (\rho^{\varepsilon})^{2\theta_i}|\partial_x\rho^{\varepsilon}|^2 \,\d x\d t \to 0
  \qquad \text{as}\ \varepsilon\to 0
  $$
when $1-\frac{\alpha_1}{\theta(\rho^{\varepsilon})}>0$ for $\rho^{\varepsilon}\le \rho_{\star}$ 
or $\rho^{\varepsilon} \ge \rho^{\star}$ and $1-\alpha_1>0$ for $\rho_{\star}\le \rho^{\varepsilon}\le \rho^{\star}$.

We now deal with $|\partial_x\rho^{\varepsilon}|^2$ for $\rho^{\varepsilon}\le \rho^{\star}$. When $\gamma_1 \le 2$, using similar arguments as above, we have
  $$
  |\partial_x\rho^{\varepsilon}|^2=(\rho^{\varepsilon})^{2-\gamma_1} (\rho^{\varepsilon})^{\gamma_1-2}|\partial_x\rho^{\varepsilon}|^2 \le \begin{cases}
  C(\mathcal{H}^{\varepsilon})^{\frac{2-\gamma_1}{\theta_1}} \big(\rho^{\varepsilon} e(\rho^{\varepsilon})\big)^{\prime\prime} |\partial_x\rho^{\varepsilon}|^2\qquad \text{when $\rho^{\varepsilon}\le \rho_{\star}$},\\[0.5mm]
  C(\mathcal{H}^{\varepsilon})^{2-\gamma_1} \big(\rho^{\varepsilon} e(\rho^{\varepsilon})\big)^{\prime\prime} |\partial_x\rho^{\varepsilon}|^2\qquad \text{when $\rho_{\star}\le \rho^{\varepsilon}\le \rho^{\star}$}.
  \end{cases}
  $$
Thus, restricted to $A_R$, when $1-\alpha_1 \frac{2-\gamma_1}{\theta_1}>0$ for $\rho^{\varepsilon}\le \rho_{\star}$ and $1-\alpha_1(2-\gamma_1)>0$ for $\rho_{\star}\le \rho^{\varepsilon}\le \rho^{\star}$, we have
  $$
  \int_0^T \int_{\hat K} \varepsilon^2 |\partial_x\rho^{\varepsilon}|^2 \,\d x \d t \to 0
  \qquad\, \text{as}\ \varepsilon\to 0.
  $$
When $\gamma_1 > 2$, we multiply $\eqref{eq-parabolic-approximation-seperate}_1$ by $\big(\delta-(\rho^{\varepsilon} \land \delta)\big)\phi^2$, where $\delta>0$ is a constant and $\phi \in C_c^2(\mathbb{R})$ with ${\rm supp} (\phi) =K$
and $\phi|_{\hat K}= 1$, 
and then integrate by parts to obtain
  $$\begin{aligned}
  &\varepsilon \int_0^T \int_{K} (\partial_x \rho^{\varepsilon})^2 I_{\rho^{\varepsilon}\le \delta} \phi^2 \,\d x \d t\\
  & =2\varepsilon \int_0^T \int_{K} \partial_x \rho^{\varepsilon}\big(\delta-(\rho^{\varepsilon} \land \delta)\big)\phi \phi_x \,\d x \d t + \int_0^T \int_{K} \partial_t \rho^{\varepsilon}\big(\delta-(\rho^{\varepsilon} \land \delta)\big)\phi^2 \,\d x \d t\\
  &\quad+ \int_0^T \int_{K} \rho^{\varepsilon} \partial_x u^{\varepsilon} \big(\delta-(\rho^{\varepsilon} \land \delta)\big)\phi^2 \,\d x \d t +\int_0^T \int_{K}  \partial_x \rho^{\varepsilon} u^{\varepsilon} \big(\delta-(\rho^{\varepsilon} \land \delta)\big)\phi^2 \,\d x \d t\\
  &:= \sum_{k=1}^4 J_k.
  \end{aligned}$$
By the H\"older inequality,
  $$
  \begin{aligned}
  J_1
  &\le  2\varepsilon \delta \Big(\int_0^T \int_{K} (\partial_x \rho^{\varepsilon})^2 \phi^2 I_{\rho^{\varepsilon}\le\delta}\,\d x \d t\Big)^{\frac12}\Big( \int_0^T \int_{K} \phi_x^2 \,\d x \d t \Big)^{\frac12}\\
  &\le  \varepsilon \delta \int_0^T \int_{K} (\partial_x \rho^{\varepsilon})^2 \phi^2 I_{\rho^{\varepsilon}\le\delta}\,\d x \d t+ \varepsilon \delta C(|K|,T).
  \end{aligned}$$
Integrating by parts again gives
  $$
  J_2
  = \int_0^T \int_{K} \partial_t \Big( \big(\rho^{\varepsilon}(\delta-\rho^{\varepsilon})+\frac{(\rho^{\varepsilon})^2}{2}\big) I_{\rho^{\varepsilon}\le\delta} \phi^2 \Big) \,\d x \d t \le \delta^2 C(|K|).
  $$
By the H\"older inequality again, we have
  $$\begin{aligned}
  J_3
  &\le  \delta \Big(  \int_0^T \int_{K} \rho^{\varepsilon} (\partial_x u^{\varepsilon})^2 \,\d x \d t \Big)^{\frac12}\Big( \int_0^T \int_{K} \rho^{\varepsilon} I_{\rho^{\varepsilon}\le\delta}\phi^4 \,\d x \d t \Big)^{\frac12}\\
  &\le  \frac{1}{\sqrt{\varepsilon}}\delta^{\frac32}C(T,|K|) \Big( \varepsilon \int_0^T \int_{K} \rho^{\varepsilon} (\partial_x u^{\varepsilon})^2 \,\d x \d t \Big)^{\frac12}.
  \end{aligned}$$
By Theorem \ref{thm-wellposedness-parabolic-approximation-R}(ii), 
we know that $u^{\varepsilon}\le \mathcal{H}^{\varepsilon}=c_1 \varepsilon^{-\alpha_1}$.
Thus, the H\"older inequality again gives
  $$\begin{aligned}
  J_4
  &\le  \delta \Big( \int_0^T \int_{K} (\partial_x \rho^{\varepsilon})^2 \phi^2 I_{\rho^{\varepsilon}\le\delta}\,\d x \d t  \Big)^{\frac12} \Big( \int_0^T \int_{K} (u^{\varepsilon})^2 \phi^2 \,\d x \d t \Big)^{\frac12}\\
  &\le  \delta \int_0^T \int_{K} (\partial_x \rho^{\varepsilon})^2 \phi^2 I_{\rho^{\varepsilon}\le\delta}\,\d x \d t
  + \delta c_1^2 \varepsilon^{-2\alpha_1} C(T,|K|).
  \end{aligned}$$
Combining the above estimates for $J_1$--$J_4$, we obtain that, by taking $\delta$ sufficiently small,
  $$\begin{aligned}
  &\varepsilon^2 \int_0^T \int_{K} (\partial_x \rho^{\varepsilon})^2 I_{\rho^{\varepsilon}\le \delta} \phi^2 \,\d x \d t\\
  &\le \frac{\varepsilon^2}{\varepsilon -\delta\varepsilon -\delta} \Big(\varepsilon \delta C(T,|K|)+\delta^2 C(|K|)
  + \frac{\delta^{\frac32}}{\sqrt{\varepsilon}}C(T,|K|) \Big( \varepsilon \int_0^T \int_{K} \rho^{\varepsilon} (\partial_x u^{\varepsilon})^2 \,\d x \d t \Big)^{\frac12}\\
  & \quad \quad \qquad\qquad\,\,\, + \delta c_1^2 \varepsilon^{-2\alpha_1} C(T,|K|)\Big).
  \end{aligned}$$
Therefore, restricted to the event $A_R$, when $1-2\alpha_1>0$, we have
  $$\begin{aligned}
  &\varepsilon^2 \int_0^T \int_{K} (\partial_x \rho^{\varepsilon})^2 I_{\rho^{\varepsilon}\le \delta} \phi^2 \,\d x \d t\\[1mm]
  &\le \frac{\varepsilon}{\varepsilon -\delta\varepsilon -\delta} \big(\varepsilon^2 \delta C(T,|K|)+\varepsilon \delta^2 C(|K|)+ \sqrt{\varepsilon}\delta^{\frac32}C(T,|K|) R^{\frac12}
  + \delta c_1^2 \varepsilon^{1-2\alpha_1} C(T,|K|)\big)\\[1mm]
  & \to 0\qquad\,\, \text{as $\varepsilon\to 0$},
  \end{aligned}$$
which implies
 $$
 \varepsilon^2 \int_0^T \int_{\hat K} (\partial_x \rho^{\varepsilon})^2 I_{\rho^{\varepsilon}\le \delta}  \,\d x \d t \to 0
 \qquad\,\, \text{as $\varepsilon\to 0$}.
 $$
On the other hand, since $(\rho^{\varepsilon})^{\gamma_1-2}\le C\big(\rho^{\varepsilon} e(\rho^{\varepsilon})\big)^{\prime\prime}$ by \eqref{iq-lower-upper-bound-for-general-pressure-1},
  $$
  \varepsilon^2 \int_0^T \int_{\hat K} (\partial_x \rho^{\varepsilon})^2 I_{\rho^{\varepsilon}> \delta}  \,\d x \d t\le \frac{\varepsilon}{\delta^{\gamma_1-2}}\int_0^T \int_{\hat K} \varepsilon (\rho^{\varepsilon})^{\gamma_1-2}(\partial_x \rho^{\varepsilon})^2  \,\d x \d t  \to 0\qquad \text{as}\ \varepsilon\to 0,
  $$
when restricted to the event $A_R$. Hence, 
restricted to the event $A_R$, when $1-2\alpha_1>0$, we arrive at
  $$
  \varepsilon^2 \int_0^T \int_{\hat K} (\partial_x \rho^{\varepsilon})^2 \,\d x \d t  \to 0
  \qquad \,\text{as $\varepsilon\to 0$}.
  $$

For the last term $C_{\psi}\int_0^T \int_{\hat K} \varepsilon^2 (\rho^{\varepsilon})^2|\partial_x u^{\varepsilon}|^2 \,\d x \d t$, similarly, restricted to the event $A_R$, we have
  $$
  \int_0^T \int_{\hat K} \varepsilon^2 (\rho^{\varepsilon})^2|\partial_x u^{\varepsilon}|^2 \,\d x \d t=\varepsilon \int_0^T \int_{\hat K} \rho^{\varepsilon} \varepsilon \rho^{\varepsilon}|\partial_x u^{\varepsilon}|^2 \,\d x \d t \to 0
  \qquad \text{as}\ \varepsilon\to 0,
  $$
when $1-\frac{\alpha_1}{\theta(\rho^{\varepsilon})}>0$ for $\rho^{\varepsilon}\le \rho_{\star}$ 
or $\rho^{\varepsilon} \ge \rho^{\star}$ and $1-\alpha_1>0$ for $\rho_{\star}\le \rho^{\varepsilon}\le \rho^{\star}$.

\smallskip
Combining the above results and choosing $\alpha_1$ satisfying
  $$
  \begin{cases}
  \alpha_1 < \theta_2 
  \qquad\qquad\quad \text{for $\gamma_1\le 2$},\\
  \alpha_1 < \min\{\frac12, \theta_2\}\quad\,\, \text{for $\gamma_1> 2$},
  \end{cases}
  $$
we obtain that, conditionally on the event $A_R$,
  $$
  \int_0^T \int_{\hat K} \varepsilon^2 |\partial_x \eta^{\psi}(\rho^{\varepsilon},m^{\varepsilon})|^2 \,\d x \d t   \to 0
  \qquad \text{as}\ \varepsilon\to 0.
  $$
Thus, $\{\varepsilon \partial_x^2 \eta^{\psi}(U^{\varepsilon})\}_{\varepsilon >0}$ is tight in $H^{-1}_{\rm loc}(\mathbb{R}^2_+)$.
\end{proof}

\begin{lemma}\label{prop-tightness-of-nabla-2-eta}
Let $(\eta^{\psi},q^{\psi})$ be a weak entropy pair generated by 
$\psi \in C_{\rm c}^2(\mathbb{R})$ as defined 
in \eqref{eq-entropy-flux-representation}. 
Suppose that $(\rho^{\varepsilon},m^{\varepsilon})$ with $m^{\varepsilon}=\rho^{\varepsilon}u^{\varepsilon}$ is a sequence of solutions to parabolic approximation equations \eqref{eq-parabolic-approximation}. 
Then $\{\varepsilon (\partial_x U^{\varepsilon})^{\top} \, \nabla^2\eta^{\psi}(U^{\varepsilon})\, 
\partial_x U^{\varepsilon}\}_{\varepsilon >0}$ is tight in $W^{-1,q_2}_{\rm loc}(\mathbb{R}^2_+)$
for any $q_2\in (1,2)$.
\end{lemma}

\begin{proof}
By
\eqref{iq-second-order-derivative-of-entropy-control-by-energy-general-pressure-law}, we have
  $$
  \begin{aligned}
  |(\partial_x U^{\varepsilon})^{\top} \,\nabla^2\eta^{\psi}(U^{\varepsilon})\,\partial_x U^{\varepsilon}| 
  \le  C_{\psi} (\partial_x U^{\varepsilon})^{\top} \,\nabla^2\eta_E(U^{\varepsilon})\,\partial_x U^{\varepsilon}
  = C_{\psi}\big( (\rho^{\varepsilon} e(\rho^{\varepsilon}))^{\prime\prime}(\partial_x \rho^{\varepsilon})^2 
  + \rho^{\varepsilon}(\partial_x u^{\varepsilon})^2 \big),
  \end{aligned}
  $$
which implies, by the energy estimate \eqref{iq-energy-estimate-on-whole-space}, 
that $\varepsilon (\partial_x U^{\varepsilon})^{\top} \,\nabla^2\eta^{\psi}(U^{\varepsilon})\,\partial_x U^{\varepsilon}$ 
is uniformly bounded in $L^1(\Omega;L^1([0,T]\times \mathbb{R}))$.
Thus, by the Sobolev compact embedding: $L^1 \hookrightarrow W^{-1,q_2}$ for any $q_2\in (1,2)$, 
$\{\varepsilon (\partial_x U^{\varepsilon})^{\top} \,\nabla^2\eta^{\psi}(U^{\varepsilon})\,\partial_x U^{\varepsilon}\}_{\varepsilon >0}$ 
is tight in $W^{-1,q_2}_{\rm loc}(\mathbb{R}^2_+)$ for any $q_2\in (1,2)$.
\end{proof}

\begin{lemma}\label{prop-tightness-of-stochastic-integral}
Let $(\eta^{\psi},q^{\psi})$ be a weak entropy pair generated by 
$\psi \in C_{\rm c}^2(\mathbb{R})$ as defined 
in \eqref{eq-entropy-flux-representation}. 
Suppose that $(\rho^{\varepsilon},m^{\varepsilon})$ 
with $m^{\varepsilon}=\rho^{\varepsilon}u^{\varepsilon}$ is a sequence of solutions to the parabolic 
approximation equations \eqref{eq-parabolic-approximation}. 
Then $\{\partial_t M^{\varepsilon}(t)\}_{\varepsilon >0}$ is tight in $H^{-1}_{\rm loc}(\mathbb{R}^2_+)$, where $M^{\varepsilon}(t):=\int_0^t \partial_{m^{\varepsilon}}\eta^{\psi}(U^{\varepsilon})\Phi^{\varepsilon}(U^{\varepsilon}) \,\d W(\tau)$.
\end{lemma}

\begin{proof}
We are going to use the Aubin-Simon lemma \cite[Corollary 9]{simon87} to prove the conclusion. 
To do this, we first estimate
$$
\E\big[\int_0^T \int_0^T \frac{\|M^{\varepsilon}(t)-M^{\varepsilon}(\tau)\|_{H^{-\ell}(\hat K)}^{r_1}}{|t-\tau|^{s_1r_1+1}}
\,\d t \d \tau\big],
$$
where $\hat K$ is any compact subset of $\mathbb{R}$, $\ell >\frac12$, and $1\le r_1\le \infty$ 
and $s_1\in \mathbb{R}$ will be determined later.

Similar to the arguments in Step \textbf{2} in the proof of Proposition \ref{cor-tightness-of-L-varepsilon}, 
by the Burkholder-Davis-Gundy inequality, we have
  $$
  \begin{aligned}
  \E\big[ \|M^{\varepsilon}(t)-M^{\varepsilon}(\tau)\|_{H^{-\ell}(\hat K)}^{r_1}\big]
  \le  \E \big[\Big( \int_{\tau}^t  \sum_k 
  \Big\|\frac{a_k\zeta_k^{\varepsilon}}{\sqrt{\rho^{\varepsilon}}}
  \Big\|_{L^2(\hat K)}^2 \|\sqrt{\rho^{\varepsilon}}\partial_{m^{\varepsilon}}\eta^{\psi}
  \|_{L^2(\hat K)}^2 \,\d r \Big)^{\frac{r_1}{2}}\big],
  \end{aligned}
  $$
where 
we have used the 
Sobolev embedding: $H^{\ell}_0 \hookrightarrow L^{\infty}$ since $\ell >\frac12$. Using $\rho\le 1+ \rho^{\gamma_2}$ and \eqref{iq-density-control-by-relative-internal-energy-general-pressure-law-2}, we obtain
  $$
  \|\sqrt{\rho^{\varepsilon}}\partial_{m^{\varepsilon}}\eta^{\psi}
  \|_{L^2(\hat K)}^2 
  \le C_{\psi}\big(1+\bar C_{\gamma_2}\rho_{\infty}^{\gamma_2}\big)|\hat K| 
  + C_{\psi}\bar C_{\gamma_2} \int_{\hat K} e^*(\rho^{\varepsilon}, \rho_{\infty}) \,\d x,
  $$
where we have used the important fact 
that $|\partial_{m^{\varepsilon}}\eta^{\psi}|\le C_{\psi}$ since $\psi \in C_c^2(\mathbb{R})$; see
\eqref{iq-control-for-partial-m-of-entropy-generate-by-compact-support-func}. 
The estimate of 
$\sum_k \Big\|\frac{a_k \zeta_k^{\varepsilon}}{\sqrt{\rho^{\varepsilon}}}\Big\|_{L^2(\hat K)}^2$ 
is the same as \eqref{iq-noise-estimate-control-by-energy}.
Combining the above estimates and applying the energy estimate \eqref{iq-energy-estimate-on-whole-space}, we obtain
  \begin{equation}\label{iq-main-estimate-in-tightness-of-stochastic-integral}
  \E \big[\|M^{\varepsilon}(t)-M^{\varepsilon}(\tau)\|_{H^{-\ell}(\hat K)}^{r_1}\big] 
  \le C(\psi, |\hat K|, E_{0,r_1})|t-\tau|^{\frac{r_1}{2}},
  \end{equation}
which implies that, when $\ell=2$ and $s_1 < \frac12$,
$$
\begin{aligned}
  \E \big[\int_0^T \int_0^T \frac{\|M^{\varepsilon}(t)-M^{\varepsilon}(\tau)\|_{H^{-2}(\hat K)}^{r_1}}{|t-\tau|^{s_1r_1+1}}
  \,\d t \d \tau\big]
  \le  
  C(\psi, |\hat K|, E_{0,r_1})\int_0^T \int_0^T \frac{|t-\tau|^{\frac{r_1}{2}}}{|t-\tau|^{s_1r_1+1}}\,\d t \d \tau < \infty,
\end{aligned}
$$
and, by letting $\ell=2$ and $\tau=0$,
  $$
  \E \big[\int_0^T \|M^{\varepsilon}(t)\|_{H^{-2}(\hat K)}^{r_1} \,\d t\big] 
  \le C(\psi,|\hat K|, E_{0,r_1},r_1, T).
  $$
Thus, $M^{\varepsilon}(t)$ is uniformly bounded in $L^{r_1}(\Omega; W^{s_1,r_1}(0,T; H^{-2}_{\rm loc}(\mathbb{R})))$ 
when $s_1 < \frac12$.

On the other hand, letting $\tau=0$ and choosing $\frac12<\ell<1$ in \eqref{iq-main-estimate-in-tightness-of-stochastic-integral},
we have
$$
\E \big[\int_0^T \|M^{\varepsilon}(t)\|_{H^{-\ell}(\hat K)}^{p} \,\d t\big] 
\le C(\psi,|\hat K|, E_{0,p},p, T),
$$
which shows that $M^{\varepsilon}(t)$ is uniformly bounded in $L^p(\Omega; L^p(0,T;H^{-\ell}_{\rm loc}(\mathbb{R})))$
with $\frac12<\ell<1$ and $1\le p< \infty$ whenever $E_{0,p}<\infty$.

By the Sobolev imbedding theorem, for $\ell<1$,
  $
  H^{-\ell}(\hat K) \hookrightarrow H^{-1}(\hat K) \hookrightarrow H^{-2}(\hat K),
  $
where $\hat K$ is any compact subset of $\mathbb{R}$. Moreover, the imbedding $H^{-\ell}(\hat K) \hookrightarrow H^{-2}(\hat K)$ is compact. Therefore, by the Aubin-Simon lemma \cite[Corollary 9]{simon87}, the embedding:
  $$
  L^2(0,T;H_{\rm loc}^{-\ell}(\mathbb{R}))\cap W^{\frac14,2}(0,T;H_{\rm loc}^{-2}(\mathbb{R})) \hookrightarrow L^2(0,T;H^{-1}_{\rm loc}(\mathbb{R}))
  $$
is compact. The tightness of $M^{\varepsilon}(t)$ in $L^2(0,T;H^{-1}_{\rm loc}(\mathbb{R}))$ then follows, which implies that $\{\partial_t M^{\varepsilon}(t)\}_{\varepsilon >0}$ is tight in $H^{-1}_{\rm loc}(\mathbb{R}^2_+)$.
\end{proof}

\begin{remark}
To obtain the tightness of $M^{\varepsilon}(t)$, it is essential to note that $\psi \in C_{\rm c}^2(\mathbb{R})$, 
which is the key point of our estimates here, compared to the proof in {\rm \cite[Proposition 5.3]{BV19}}.
\end{remark}

\begin{lemma}\label{prop-tightness-of-quadratic-variation}
Let $(\eta^{\psi},q^{\psi})$ be a weak entropy pair generated by 
$\psi \in C_{\rm c}^2(\mathbb{R})$ as defined 
in \eqref{eq-entropy-flux-representation}. 
Suppose that $(\rho^{\varepsilon},m^{\varepsilon})$ with $m^{\varepsilon}=\rho^{\varepsilon}u^{\varepsilon}$ 
is a sequence of solutions to the parabolic approximation equations \eqref{eq-parabolic-approximation}. 
Then $\{\frac12 \partial_{m^{\varepsilon}}^2\eta^{\psi}(U^{\varepsilon})\sum_k a_k^2 ( \zeta_k^{\varepsilon})^2\}_{\varepsilon >0}$ 
is tight in $W^{-1,q_3}_{\rm loc}(\mathbb{R}^2_+)$ for $1<q_3 <2$.
\end{lemma}

\begin{proof}
By
\eqref{iq-control-for-partial-m-of-entropy-generate-by-compact-support-func},
$|\rho^{\varepsilon} \partial_{m^{\varepsilon}}^2\eta^{\psi}(U^{\varepsilon})| \le C_{\psi}$. 
Thus, for any $\hat K \Subset \mathbb{R}$,
  $$
  \begin{aligned}
  \int_0^T \int_{\hat K} \frac12 \partial_{m^{\varepsilon}}^2\eta^{\psi}(U^{\varepsilon})\sum_k a_k^2 ( \zeta_k^{\varepsilon})^2 \,\d x \d t
  &\le C_{\psi} \int_0^T \int_{\hat K} \frac{1}{\rho^{\varepsilon} }\sum_k a_k^2 ( \zeta_k^{\varepsilon})^2 \,\d x \d t\\
  &\le 
  C(\psi,T,|\hat K|)+ C(\psi)\int_0^T \int_{\hat K} \big(\frac{(m^{\varepsilon})^2}{\rho^{\varepsilon}} + e^*(\rho^{\varepsilon},\rho_{\infty})\big)\, \d x \d t,
  \end{aligned}
  $$
where, in the last inequality, we have used \eqref{iq-noise-estimate-control-by-energy}. 
Then the energy estimate \eqref{iq-energy-estimate-on-whole-space} 
again implies that $\frac12 \partial_{m^{\varepsilon}}^2\eta^{\psi}(U^{\varepsilon})\sum_k ( a_k \zeta_k^{\varepsilon})^2$ 
is uniformly bounded in $L^1(\Omega;L^1([0,T]\times \hat K))$. 
Therefore, by the Sobolev compact embedding: $L^1 \hookrightarrow W^{-1,q_3}$ for any $q_3\in (1,2)$, 
we obtain that $\{\frac12 \partial_{m^{\varepsilon}}^2\eta^{\psi}(U^{\varepsilon})
\sum_k a_k^2 ( \zeta_k^{\varepsilon})^2\}_{\varepsilon >0}$ is tight in $W^{-1,q_3}_{\rm loc}(\mathbb{R}^2_+)$ for $1<q_3 <2$.
\end{proof}

\begin{proposition}\label{prop-H^-1-compactness}
Let $(\eta^{\psi},q^{\psi})$ be a weak entropy pair generated by 
$\psi \in C_{\rm c}^2(\mathbb{R})$ as defined 
in \eqref{eq-entropy-flux-representation}. 
Suppose that $(\rho^{\varepsilon},m^{\varepsilon})$ with $m^{\varepsilon}=\rho^{\varepsilon}u^{\varepsilon}$ 
is a sequence of solutions to the parabolic approximation equations \eqref{eq-parabolic-approximation}. Then,
for the general pressure law case, $\{\eta^{\psi}(\rho^{\varepsilon},m^{\varepsilon})_t +q^{\psi}(\rho^{\varepsilon},m^{\varepsilon})_x\}_{\varepsilon >0}$ is tight in 
$W^{-1,p_0}_{\rm loc}(\mathbb{R}^2_+)$ with $1\le p_0 <2$.
\end{proposition}

\begin{proof}
For the general pressure law case, by \cite[Lemmas 4.5, 4.7, and 4.10]{chenhuangLiWangWang24CMP}, we have
  $$
  |\eta^{\psi}(\rho,m)|\le C_{\psi}\rho,\qquad |q^{\psi}(\rho,m)|\le C_{\psi}(\rho+ \rho^{1+\theta_2}).
  $$
Thus, by Proposition \ref{prop-uniform-estimates-rho-gamma+1} and \eqref{iq-lower-upper-bound-for-general-pressure-2}, 
we obtain that $\eta^{\psi}(\rho^{\varepsilon},m^{\varepsilon})$ and $q^{\psi}(\rho^{\varepsilon},m^{\varepsilon})$ are uniformly bounded in $L^1(\Omega;L^{2}_{\rm loc}(\mathbb{R}^2_+))$. 
Thus, $\{\eta^{\psi}(\rho^{\varepsilon},m^{\varepsilon})_t +q^{\psi}(\rho^{\varepsilon},m^{\varepsilon})_x\}_{\varepsilon >0}$ 
is stochastically bounded in $H^{-1}_{\rm loc}(\mathbb{R}^2_+)$ (see Lemma \ref{lem-stochastic-version-murat-lemma}).

On the other hand, by Theorem \ref{thm-wellposedness-parabolic-approximation-R}(iii), 
we apply the It\^o formula to $\eta^{\psi}$ to obtain
   $$\begin{aligned}
   \d \eta^{\psi}(U^{\varepsilon})+\partial_x q^{\psi}(U^{\varepsilon}) \,\d t
   &= \varepsilon \partial_x^2 \eta^{\psi}(U^{\varepsilon})\,\d t 
   -\varepsilon (\partial_x U^{\varepsilon})^{\top} \, \nabla^2\eta^{\psi}(U^{\varepsilon})\,\partial_x U^{\varepsilon} \,\d t \\
   &\quad\, + \partial_{m^{\varepsilon}}\eta^{\psi}(U^{\varepsilon})\Phi^{\varepsilon}(U^{\varepsilon}) \,\d W(t) 
   + \frac12 \partial_{m^{\varepsilon}}^2\eta^{\psi}(U^{\varepsilon})\sum_k a_k^2 ( \zeta_k^{\varepsilon})^2 \,\d t.
   \end{aligned}$$
It follows from Lemmas \ref{prop-tightness-of-partial-x-eta} and \ref{prop-tightness-of-stochastic-integral}
that $\{\varepsilon \partial_x^2 \eta^{\psi}(U^{\varepsilon})\}_{\varepsilon >0}$ and 
$\{\partial_t M^{\varepsilon}(t)\}_{\varepsilon >0}$ are tight in $H^{-1}_{\rm loc}(\mathbb{R}^2_+)$. 
Lemmas \ref{prop-tightness-of-nabla-2-eta} and \ref{prop-tightness-of-quadratic-variation} imply that 
$$
\{\varepsilon (\partial_x U^{\varepsilon})^{\top} \,\nabla^2\eta^{\psi}(U^{\varepsilon})\,
\partial_x U^{\varepsilon}\}_{\varepsilon >0}
\quad\mbox{and}\quad \{\frac12 \partial_{m^{\varepsilon}}^2\eta^{\psi}(U^{\varepsilon})\sum_k a_k^2 ( \zeta_k^{\varepsilon})^2\}_{\varepsilon >0}
$$ 
are tight in $W^{-1,q}_{\rm loc}(\mathbb{R}^2_+)$ for all $q\in (1,2)$. 
These also imply
  \begin{equation}\label{eq-entropy-dissipation-measure-compactness-in-W-1q0}
  \{\eta^{\psi}(\rho^{\varepsilon},m^{\varepsilon})_t +q^{\psi}(\rho^{\varepsilon},m^{\varepsilon})_x\}_{\varepsilon >0} 
  \qquad \text{are tight in $W^{-1,q_0}_{\rm loc}(\mathbb{R}^2_+)$ with $1\le q_0\le q$}.
  \end{equation}
Therefore, our conclusion follows from Lemma \ref{lem-stochastic-version-murat-lemma} and \eqref{eq-entropy-dissipation-measure-compactness-in-W-1q0}.
\end{proof}

\begin{lemma}\label{prop-tartar-commutation}
Let $\tilde \mu$ be the limit random Young measure obtained in 
{\rm Proposition \ref{prop-apply-jakubowski-skorokhod-representation-to-take-limit}}. 
Then, $\tilde{\mathbb{P}}$-{\it a.s.} for {\it a.e.} $(t,x)\in \mathbb{R}^2_+$,
  \begin{equation}\label{eq-tartar-commutation}
  \overline{\chi(s_1)\sigma(s_2)-\chi(s_2)\sigma(s_1)}=\overline{\chi(s_1)}\,\overline{\sigma(s_2)}- \overline{\chi(s_2)}\,\overline{\sigma(s_1)}\qquad \mbox{for any $\,s_1, s_2 \in \mathbb{R}$}.
  \end{equation}
\end{lemma}

\begin{proof}
For the general pressure law case, from \cite[Lemmas 4.5, 4.7, and 4.10]{chenhuangLiWangWang24CMP}, we have
  \begin{equation}\label{iq-control-of-eta-q-with-compact-supported-psi-general-pressure-law}
  |\eta^{\psi}(\rho,m)|\le C_{\psi}\rho,\qquad |q^{\psi}(\rho,m)|\le C_{\psi}\big(\rho+ \rho^{1+\theta_2}\big).
  \end{equation}
It is clear that ${\rm supp}(\eta^{\psi}),\,{\rm supp} (q^{\psi})\subset \{(\rho,m)\,:\,\frac{m}{\rho}+\mathcal{K}(\rho) \ge a_*,\ \frac{m}{\rho}-\mathcal{K}(\rho) \le a^*\}$, 
where $a_*$ and $a^*$ depend only on the support of $\psi$.
Moreover, by the representation of the entropy pairs \eqref{eq-entropy-flux-representation}, 
$\eta^{\psi}(0,m)=q^{\psi}(0,m)=0$ since $\psi \in C_{\rm c}^2(\mathbb{R})$. 
Therefore, in the following proof, if necessary, 
we always use the fact that, with $\phi=\eta^{\psi}$ or $q^{\psi}$, and $\nu\in \{\tilde\mu^{\varepsilon}, \tilde\mu, \mu^{\varepsilon}\}$,
  $$
  \int_{\mathbb{R}^2_+} \phi(\rho,m) \,\d \nu_{(t,x)}(\rho,m)=\int_{\mathfrak{T}} \phi(\rho,m) \,\d \nu_{(t,x)}(\rho,m).
  $$

For any $\psi_1, \psi_2 \in C_{\rm c}^2(\mathbb{R})$,
  $$
  \begin{aligned}
  &\tilde X_1^{\varepsilon}:=(\int_{\mathbb{R}^2_+} \eta^{\psi_1}(\rho,m) \,\d \tilde\mu^{\varepsilon}_{(t,x)}(\rho,m),\ \int_{\mathbb{R}^2_+} q^{\psi_1}(\rho,m) \,\d \tilde\mu^{\varepsilon}_{(t,x)}(\rho,m)),\\
  &\tilde X_2^{\varepsilon}:=(\int_{\mathbb{R}^2_+} q^{\psi_2}(\rho,m) \,\d \tilde\mu^{\varepsilon}_{(t,x)}(\rho,m),\ -\int_{\mathbb{R}^2_+} \eta^{\psi_2}(\rho,m) \,\d \tilde\mu^{\varepsilon}_{(t,x)}(\rho,m))
  \end{aligned}
  $$
are well-defined on $[0,T]\times \hat K$ for any $T>0$ 
and any $\hat K \Subset \mathbb{R}$ 
by \eqref{iq-uniform-higher-integrability-bound-for-tilde-mu-varepsilon}. 
Applying Proposition \ref{prop-test-func-and-convergence-for-random-young-measure}(ii) with $\beta=\frac12$,
we obtain that, for any $\varphi \in L^2_{\rm loc}(\mathbb{R}^2_+;\mathbb{R}^2)$,
  $$
  \lim_{\varepsilon \to 0}\tilde{\E}\big[ \Big| \int_0^T \int_{\hat K} (\tilde X_i^{\varepsilon}-\tilde X_i )\varphi(t,x) \,\d x \d t  \Big|\big]= 0 \qquad\mbox{for $i=1,2$},
  $$
where
  $$
  \begin{aligned}
  &\tilde X_1:=(\int_{\mathbb{R}^2_+} \eta^{\psi_1}(\rho,m) \,\d \tilde\mu_{(t,x)}(\rho,m),\ \int_{\mathbb{R}^2_+} q^{\psi_1}(\rho,m) \,\d \tilde\mu_{(t,x)}(\rho,m))=(\overline{\eta^{\psi_1}},\ \overline{q^{\psi_1}})\\
  &\tilde X_2:=(\int_{\mathbb{R}^2_+} q^{\psi_2}(\rho,m) \,\d \tilde\mu_{(t,x)}(\rho,m),\ -\int_{\mathbb{R}^2_+} \eta^{\psi_2}(\rho,m) \,\d \tilde\mu_{(t,x)}(\rho,m))=(\overline{ q^{\psi_2}},\ -\overline{\eta^{\psi_2}}).
  \end{aligned}
  $$
Therefore, for each $\varphi \in L^2_{\rm loc}(\mathbb{R}^2_+;\mathbb{R}^2)$, up to a subsequence, we have
  $$
  \lim_{\varepsilon \to 0} \int_0^T \int_{\hat K} (\tilde X_i^{\varepsilon}-\tilde X_i )\varphi(t,x) \,\d x \d t=0
  \qquad\mbox{for $i=1,2,$} \quad \tilde{\mathbb{P}}\text{-{\it a.s.}}
  $$
Since $L^2$ is separable, there exists a further subsequence (without changing the notation) such that
  $$
  \lim_{\varepsilon \to 0} \int_0^T \int_{\hat K} (\tilde X_i^{\varepsilon}-\tilde X_i )\varphi(t,x) \,\d x \d t=0
  \qquad \mbox{for any $\varphi \in L^2_{\rm loc}(\mathbb{R}^2_+;\mathbb{R}^2)$}\quad \mbox{for $i=1,2$,} \,\,\, \tilde{\mathbb{P}}\text{-{\it a.s.}},
  $$
which implies that
  \begin{equation}\label{eq-convergence-by-separability}
  \tilde X_i^{\varepsilon} \rightharpoonup \tilde X_i\  \qquad \text{in}\ L^2_{\rm loc}(\mathbb{R}^2_+;\mathbb{R}^2)\quad \tilde{\mathbb{P}}\text{-{\it a.s.}} \,\,\,\mbox{as $\varepsilon\to 0$}.
  \end{equation}

Notice that, for any fixed $\varepsilon>0$, 
$\mu^{\varepsilon}_{(t,x)}(\rho,m)=\delta_{(\rho^{\varepsilon},m^{\varepsilon})}$ and $(\rho^{\varepsilon},m^{\varepsilon})\in L^{\frac{2\gamma_2}{\gamma_2+1}}(\Omega; L^{\frac{2\gamma_2}{\gamma_2+1}}_{\rm loc}(\mathbb{R}^2_+))$ 
by Proposition \ref{cor-tightness-of-L-varepsilon}, 
which implies that $\mu^{\varepsilon}_{(t,x)}(\rho,m)$ 
is an $L^{\frac{2\gamma_2}{\gamma_2+1}}$-random Dirac mass (see \cite[Definition 4.2]{BV19}). 
Thus, utilizing \cite[Proposition 4.5]{BV19} and 
the equality of laws from Proposition \ref{prop-apply-jakubowski-skorokhod-representation-to-take-limit}, 
$\tilde \mu^{\varepsilon}_{(t,x)}(\rho,m)$ is also an $L^{\frac{2\gamma_2}{\gamma_2+1}}$-random Dirac mass, 
which enables us to calculate
  \begin{align}
  \tilde X_1^{\varepsilon} \cdot \tilde X_2^{\varepsilon}
  &= \int_{\mathbb{R}^2_+} \eta^{\psi_1}(\rho,m) \,\d \tilde\mu^{\varepsilon}_{(t,x)}(\rho,m)\int_{\mathbb{R}^2_+} q^{\psi_2}(\rho,m) \,\d \tilde\mu^{\varepsilon}_{(t,x)}(\rho,m)\nonumber\\
  &\quad-\int_{\mathbb{R}^2_+} \eta^{\psi_2}(\rho,m) \,\d \tilde\mu^{\varepsilon}_{(t,x)}(\rho,m)\int_{\mathbb{R}^2_+} q^{\psi_1}(\rho,m) \,\d \tilde\mu^{\varepsilon}_{(t,x)}(\rho,m)\nonumber\\
  &= \int_{\mathbb{R}^2_+} \big(\eta^{\psi_1}(\rho,m) q^{\psi_2}(\rho,m)- \eta^{\psi_2}(\rho,m)q^{\psi_1}(\rho,m)\big)\, \,\d \tilde\mu^{\varepsilon}_{(t,x)}(\rho,m).\label{eq-application-of-random-dirac-mass}
  \end{align}
By 
\eqref{iq-control-of-eta-q-with-compact-supported-psi-general-pressure-law}, we have
  \begin{equation}\label{iq-control-of-eta-q-with-compact-supported-psi-multiply}
  |\eta^{\psi_1}(\rho,m) q^{\psi_2}(\rho,m)- \eta^{\psi_2}(\rho,m)q^{\psi_1}(\rho,m)|\le C_{\psi}\rho^{\beta_0(\gamma_2+1)}
  \end{equation}
with $\beta_0=\frac{2+\bar \theta}{\gamma_2+1}<1$. 
Thus, by similar arguments as in the justification of \eqref{eq-convergence-by-separability} above, we have
  \begin{equation}\label{eq-convergence-for-X1X2}
  \tilde X_1^{\varepsilon}\cdot  \tilde X_2^{\varepsilon}\,
  \rightharpoonup \,\overline{\eta^{\psi_1}q^{\psi_2}- \eta^{\psi_2}q^{\psi_1}} \qquad \text{in}\ L^{\frac{1}{\beta_0}}_{\rm loc}(\mathbb{R}^2_+) \quad \tilde{\mathbb{P}}\text{-{\it a.s.}} 
  \,\,\,\mbox{as $\varepsilon\to 0$}.
  \end{equation}

On the other hand, we claim that $\div_{(t,x)}\, \tilde X_1^{\varepsilon}$ and $\text{curl}_{(t,x)}\, \tilde X_2^{\varepsilon}$ are tight in $W^{-1,1}_{\rm loc}(\mathbb{R}^2_+)$ and $W^{-1,1}_{\rm loc}(\mathbb{R}^2_+;\mathbb{R}^{2\times 2})$, respectively.
Denote
  $$
  \begin{aligned}
  &X_1^{\varepsilon}:=(\int_{\mathbb{R}^2_+} \eta^{\psi_1}(\rho,m) \,\d \mu^{\varepsilon}_{(t,x)}(\rho,m),\ \int_{\mathbb{R}^2_+} q^{\psi_1}(\rho,m) \,\d \mu^{\varepsilon}_{(t,x)}(\rho,m)),
  \\
  &X_2^{\varepsilon}:=(\int_{\mathbb{R}^2_+} q^{\psi_2}(\rho,m) \,\d \mu^{\varepsilon}_{(t,x)}(\rho,m),\ -\int_{\mathbb{R}^2_+} \eta^{\psi_2}(\rho,m) \,\d \mu^{\varepsilon}_{(t,x)}(\rho,m)).
  \end{aligned}
  $$
Proposition \ref{prop-H^-1-compactness} implies that 
$\div_{(t,x)}\, X_1^{\varepsilon}$ and 
$\text{curl}_{(t,x)}\, X_2^{\varepsilon}$ are tight in $W^{-1,1}_{\rm loc}(\mathbb{R}^2_+)$ 
and $W^{-1,1}_{\rm loc}(\mathbb{R}^2_+;\mathbb{R}^{2\times 2})$ respectively. 
Therefore, to prove the claim, it suffices to prove that $\mathcal{L}(\div_{(t,x)}\, \tilde X_1^{\varepsilon})=\mathcal{L}(\div_{(t,x)}\, X_1^{\varepsilon})$ in $W^{-1,1}_{\rm loc}(\mathbb{R}^2_+)$ and $\mathcal{L}(\text{curl}_{(t,x)}\, \tilde X_2^{\varepsilon})=\mathcal{L}(\text{curl}_{(t,x)}\, X_2^{\varepsilon})$ in $W^{-1,1}_{\rm loc}(\mathbb{R}^2_+;\mathbb{R}^{2\times 2})$. By Lemma 2.1.4 in \cite{BFHbook18}, it suffices to verify
  $$
  \begin{aligned}
  &\tilde \E\big[\int_0^T \int_{\hat K} \varphi\, \div_{(t,x)}\, \tilde X_1^{\varepsilon} \,\d x \d t\big]
    = \E\big[\int_0^T \int_{\hat K} \varphi\, \div_{(t,x)}\, X_1^{\varepsilon} \,\d x \d t\big],\\
 &\tilde \E\big[\int_0^T \int_{\hat K} \hat\varphi \,\text{curl}_{(t,x)}\, \tilde X_2^{\varepsilon} \,\d x \d t\big]
   = \E\big[\int_0^T \int_{\hat K} \hat\varphi\, \text{curl}_{(t,x)}\, X_2^{\varepsilon} \,\d x \d t\big],
  \end{aligned}
  $$
for any $T>0$, any $\hat K \Subset \mathbb{R}$, 
and any $\varphi\in W^{1,\infty}_0([0,T]\times \hat K)$ and 
$\hat \varphi \in W^{1,\infty}_0([0,T]\times \hat K;\mathbb{R}^{2\times 2})$.

To do this, let $w_k(\rho,m) \ge 0$ be a continuous cut-off function such that $w_k\equiv 1$
on the set $\{(\rho,m)\in \mathbb{R}^2_+\,:\, \frac{1}{k}\le \mathcal{K}(\rho)\le k,\ |m|\le k\}$ 
and $w_k \equiv 0$ outside the set $\{(\rho,m)\in \mathbb{R}^2_+\,:\, \frac{1}{2k}\le \mathcal{K}(\rho)\le 2k,\ |m|\le 2k\}$. 
Then $\eta^{\psi_i}(\rho,m)w_k(\rho,m),\,q^{\psi_i}(\rho,m)w_k(\rho,m) \in C_0(\mathbb{R}^2_+)$ for $i=1,2$. 
Thus, it follows from the equality of laws 
in Proposition \ref{prop-apply-jakubowski-skorokhod-representation-to-take-limit}
that
  $$
  \begin{aligned}
  &\tilde \E \big[\int_0^T \int_{\hat K} \int_{\mathbb{R}^2_+} \eta^{\psi_1}(\rho,m)w_k(\rho,m) 
  \,\d \tilde \mu^{\varepsilon}_{(t,x)}(\rho,m) \partial_t \varphi \,\d x \d t\big]\\
  &\quad\, +\tilde \E\big[ \int_0^T \int_{\hat K} \int_{\mathbb{R}^2_+} q^{\psi_1}(\rho,m)w_k(\rho,m) \,\d \tilde \mu^{\varepsilon}_{(t,x)}(\rho,m) \partial_x \varphi \,\d x \d t\big]\\
  &=\E\big[ \int_0^T \int_{\hat K} \int_{\mathbb{R}^2_+} \eta^{\psi_1}(\rho,m)w_k(\rho,m) 
  \,\d \mu^{\varepsilon}_{(t,x)}(\rho,m) \partial_t \varphi \,\d x \d t\big]\\
  &\quad\,+ \E \big[\int_0^T \int_{\hat K} \int_{\mathbb{R}^2_+} q^{\psi_1}(\rho,m)w_k(\rho,m) \,\d \mu^{\varepsilon}_{(t,x)}(\rho,m) \partial_x \varphi \,\d x \d t\big].
  \end{aligned}
  $$
Utilizing \eqref{iq-uniform-higher-integrability-bound-for-tilde-mu-varepsilon}, 
Proposition \ref{prop-uniform-estimates-rho-gamma+1}, and the dominated convergence theorem, we have
  $$
  \begin{aligned}
  &\tilde \E \big[\int_0^T \int_{\hat K} \int_{\mathbb{R}^2_+} \eta^{\psi_1}(\rho,m) \,\d \tilde \mu^{\varepsilon}_{(t,x)}(\rho,m) \partial_t \varphi \,\d x \d t\big]\\
  &=\lim_{k\to \infty} \tilde \E \big[\int_0^T \int_{\hat K} \int_{\mathbb{R}^2_+} \eta^{\psi_1}(\rho,m)w_k(\rho,m) \,\d \tilde \mu^{\varepsilon}_{(t,x)}(\rho,m) \partial_t \varphi \,\d x \d t\big]\\
  &=\lim_{k\to \infty} \E \big[\int_0^T \int_{\hat K} \int_{\mathbb{R}^2_+} \eta^{\psi_1}(\rho,m)w_k(\rho,m) \,\d \mu^{\varepsilon}_{(t,x)}(\rho,m) \partial_t \varphi \,\d x \d t\big]\\
  &= \E \big[\int_0^T \int_{\hat K} \int_{\mathbb{R}^2_+} \eta^{\psi_1}(\rho,m) \,\d \mu^{\varepsilon}_{(t,x)}(\rho,m) \partial_t \varphi \,\d x \d t\big].
  \end{aligned}
  $$
Similarly, we have
  $$
  \begin{aligned}
  \tilde \E \big[\int_0^T \int_{\hat K} \int_{\mathbb{R}^2_+} q^{\psi_1}(\rho,m) 
  \,\d \tilde \mu^{\varepsilon}_{(t,x)}(\rho,m) \partial_x \varphi \,\d x \d t\big]
  =\E \big[\int_0^T \int_{\hat K} \int_{\mathbb{R}^2_+} q^{\psi_1}(\rho,m) 
  \,\d \mu^{\varepsilon}_{(t,x)}(\rho,m) \partial_x \varphi \,\d x \d t\big].
  \end{aligned}
  $$
Therefore, integrating by parts implies
  $$
  \tilde \E\big[\int_0^T \int_{\hat K} \varphi\,\div_{(t,x)}\, \tilde X_1^{\varepsilon} \,\d x \d t\big]
  = \E\big[\int_0^T \int_{\hat K} \varphi\,\div_{(t,x)}\, X_1^{\varepsilon} \,\d x \d t\big].
  $$
Similarly, we have
  $$
  \tilde \E\big[\int_0^T \int_{\hat K} \hat\varphi \,\text{curl}_{(t,x)}\, \tilde X_2^{\varepsilon} \,\d x \d t\big]
  = \E\big[\int_0^T \int_{\hat K} \hat\varphi\, \text{curl}_{(t,x)}\, X_2^{\varepsilon} \,\d x \d t\big].
  $$
These prove our claim that
$\div_{(t,x)}\, \tilde X_1^{\varepsilon}$ and $\text{curl}_{(t,x)}\, \tilde X_2^{\varepsilon}$ 
are tight in $W^{-1,1}_{\rm loc}(\mathbb{R}^2_+)$ and $W^{-1,1}_{\rm loc}(\mathbb{R}^2_+;\mathbb{R}^{2\times 2})$, 
respectively.

Since $\div_{(t,x)}\, \tilde X_1^{\varepsilon}$ and $\text{curl}_{(t,x)}\, \tilde X_2^{\varepsilon}$ are tight 
in $W^{-1,1}_{\rm loc}(\mathbb{R}^2_+)$ and $W^{-1,1}_{\rm loc}(\mathbb{R}^2_+;\mathbb{R}^{2\times 2})$ respectively, 
we see that, for any $\beta\in (0, \frac{1}{2})$, 
there exist compact subsets $K_{\beta}\Subset W^{-1,1}_{\rm loc}(\mathbb{R}^2_+)$
and $\bar{K}_{\beta}\subset W^{-1,1}_{\rm loc}(\mathbb{R}^2_+;\mathbb{R}^{2\times 2})$ such that
  $$
  \tilde{\mathbb{P}}\big\{\omega\,:\, \div_{(t,x)}\, \tilde X_1^{\varepsilon} \in K_{\beta}\,\,\text{and}\,\,
  \text{curl}_{(t,x)}\, \tilde X_2^{\varepsilon} \in \bar{K}_{\beta} \big\} > 1-\beta.
  $$
Denote $A_{\beta}:=\{\omega\,:\, \div_{(t,x)}\, \tilde X_1^{\varepsilon} \in K_{\beta}\,\,\text{and}\,\,
\text{curl}_{(t,x)}\, \tilde X_2^{\varepsilon} \in \bar{K}_{\beta} \}$.

The tightness of $\div_{(t,x)}\, \tilde X_1^{\varepsilon}$ and $\text{curl}_{(t,x)}\, \tilde X_2^{\varepsilon}$ is 
weaker than that required in \cite{murat78}, 
so we need an improvement of the classical div-curl lemma (see Lemma \ref{lem-generalized-div-curl-lemma} below). 
Therefore, we need to prove additionally that $\tilde X_1^{\varepsilon} \cdot \tilde X_2^{\varepsilon}$ is equi-integrable. 
By \eqref{iq-uniform-higher-integrability-bound-for-tilde-mu-varepsilon} and 
\eqref{eq-application-of-random-dirac-mass}--\eqref{iq-control-of-eta-q-with-compact-supported-psi-multiply}, 
we obtain
$$
\tilde{\E} \big[\int_0^T \int_{\hat K} \big|\tilde X_1^{\varepsilon} \cdot \tilde X_2^{\varepsilon}\big|^{\frac{1}{\beta_0}} \d x \d t\big] 
\le 
{\hat C}
$$
for any $T>0$ and any 
$\hat K \Subset \mathbb{R}$. 
Thus, by the Chebyshev inequality, 
for any $0<\beta<\frac12$, there exists a constant $R_{\beta}>0$ such that
  $$
  \tilde{\P}\big\{\omega\,:\, \int_0^T \int_{\hat K} \big|\tilde X_1^{\varepsilon} \cdot \tilde X_2^{\varepsilon}\big|^{\frac{1}{\beta_0}} \d x \d t \le R_{\beta} \big\} > 1-\beta.
  $$
Then, restricted to the event $B_{\beta}:=\big\{\omega\,:\, \int_0^T \int_{\hat K} \big|\tilde X_1^{\varepsilon} \cdot \tilde X_2^{\varepsilon}\big|^{\frac{1}{\beta_0}} \d x \d t \le R_{\beta} \big\}$, 
since $\frac{1}{\beta_0}>1$, we see that 
$\tilde X_1^{\varepsilon} \cdot \tilde X_2^{\varepsilon}$ is equi-integrable uniformly in $\varepsilon$. Therefore,
restricted to the event $A_{\beta}\cap B_{\beta}$ (note that $\tilde{\P}\{A_{\beta}\cap B_{\beta}\}>0$), by Lemma \ref{lem-generalized-div-curl-lemma}, we obtain
  \begin{equation}\label{eq-convergence-by-div-curl-lemma}
  \tilde X_1^{\varepsilon} \cdot \tilde X_2^{\varepsilon} \rightharpoonup \tilde X_1 \cdot \tilde X_2=\overline{\eta^{\psi_1}}\,\overline{q^{\psi_2}} -\overline{q^{\psi_1}}\,\overline{\eta^{\psi_2}}\qquad \text{in}\ \mathcal{D}^{\prime}.
  \end{equation}

Combining \eqref{eq-convergence-for-X1X2} with \eqref{eq-convergence-by-div-curl-lemma} implies that
  \begin{equation}\label{eq-relative-to-tartar-commutation}
  \overline{\eta^{\psi_1}}\,\overline{q^{\psi_2}} -\overline{q^{\psi_1}}\,\overline{\eta^{\psi_2}}
  = \overline{\eta^{\psi_1}q^{\psi_2}- \eta^{\psi_2}q^{\psi_1}} 
  \qquad \text{in}\ \mathcal{D}^{\prime} \cap L^{\frac{1}{\beta_0}}_{\rm loc}(\mathbb{R}^2_+),
  \end{equation}
and particularly, in $\mathcal{D}^{\prime} \cap L^1_{\rm loc}(\mathbb{R}^2_+)$. Therefore,
equality \eqref{eq-relative-to-tartar-commutation} holds pointwise almost everywhere in $\left[0,\infty\right)\times \mathbb{R}$.

Choose a decreasing sequence $\{\beta_n\}$ satisfying $\beta_n <\frac12$ with $\beta_n \to 0$ and 
an increasing sequence $\{\mathcal{B}_{\beta_n}\}$ with $\mathcal{B}_{\beta_n}:=A_{\beta_n}\cap B_{\beta_n}$. 
Denoting $\mathcal{B}:=\bigcup_{n\in \mathbb{N}} \mathcal{B}_{\beta_n}$, then $\tilde{\mathbb{P}}(\mathcal{B})=1$. 
Equality \eqref{eq-relative-to-tartar-commutation} holds on $\mathcal{B}$, hence $\tilde{\mathbb{P}}$-almost surely and pointwise almost everywhere in $\left[0,\infty\right)\times \mathbb{R}$. Recall the representation of 
weak entropy pairs in \eqref{eq-entropy-flux-representation}, by arbitrariness of test functions $\psi_1$ and $\psi_2$, 
we obtain that, for {\it a.e.} $(t,x)\in \mathbb{R}^2_+$, the Tartar commutation \eqref{eq-tartar-commutation} 
holds $\tilde{\mathbb{P}}$-{\it a.s.}.
\end{proof}

\begin{lemma}\label{prop-identify-tilde-mu-varepsilon-delta-mass}
If $\tilde \mu^{\varepsilon}$ is the random Young measure 
obtained in {\rm Proposition \ref{prop-apply-jakubowski-skorokhod-representation-to-take-limit}},
then 
$\tilde \mu^{\varepsilon}=\delta_{(\tilde\rho^{\varepsilon},\tilde m^{\varepsilon})}$.
\end{lemma}

\begin{proof}
From the proof of Lemma \ref{prop-tartar-commutation}, 
we already know that $\tilde \mu^{\varepsilon}$ is an $L^{\frac{2\gamma_2}{\gamma_2+1}}$-random Dirac mass. 
We first assume $\tilde \mu^{\varepsilon}=\delta_{(\hat\rho^{\varepsilon},\hat m^{\varepsilon})}$.
Then $(\hat\rho^{\varepsilon}, \hat m^{\varepsilon})$ are 
$L^{\frac{2\gamma_2}{\gamma_2+1}}_{\rm loc}(\mathbb{R}^2_+)$-valued random variables 
and hence also $L^1_{\rm loc}(\mathbb{R}^2_+)$-valued random variables.

Since $\mu^{\varepsilon}=\delta_{(\rho^{\varepsilon},m^{\varepsilon})}$, we see that, for $\phi(\rho,m)= \rho$ or $m$,
  $$
  \mathbb{P}\Big\{ \int_{\mathbb{R}^2_+} \Big(\int_{\mathbb{R}^2_+} \phi w_k \,\d \mu^{\varepsilon}_{(t,x)}\Big)\varphi\,\d t\d x = \int_{\mathbb{R}^2_+} \phi(\rho^{\varepsilon},m^{\varepsilon}) w_k(\rho^{\varepsilon},m^{\varepsilon})\,\varphi \,\d t\d x
  \,\,\,\,\mbox{for any $\varphi \in C_{\rm c}^{\infty}(\mathbb{R}^2_+)$} \Big\}=1,
  $$
where $w_k(\rho,m) \ge 0$ is a continuous cut-off function such that $w_k\equiv 1$ 
on the set $\{(\rho,m)\in \mathbb{R}^2_+\,:\, \frac{1}{k}\le \mathcal{K}(\rho)\le k, |m|\le k\}$ 
and $w_k \equiv 0$ outside the 
set $\{(\rho,m)\in \mathbb{R}^2_+\,:\, \frac{1}{2k}\le \mathcal{K}(\rho)\le 2k, |m|\le 2k\}$. 
Then, by the equality of laws from Proposition \ref{prop-apply-jakubowski-skorokhod-representation-to-take-limit}, we have
  $$
  \begin{aligned}
  1&=\mathbb{P}\Big\{ \int_{\mathbb{R}^2_+} \Big(\int_{\mathbb{R}^2_+} \phi w_k \,\d \mu^{\varepsilon}_{(t,x)}\Big)\varphi\,\d t\d x = \int_{\mathbb{R}^2_+}\phi(\rho^{\varepsilon},m^{\varepsilon}) w_k(\rho^{\varepsilon},m^{\varepsilon})\,\varphi  \,\d t\d x\,\,\,\, 
  \mbox{for any $\varphi \in C_{\rm c}^{\infty}(\mathbb{R}^2_+)$}\Big\}\\
  &=\tilde{\mathbb{P}}\Big\{ \int_{\mathbb{R}^2_+} \varphi \int_{\mathbb{R}^2_+} \phi w_k \,\d \tilde \mu^{\varepsilon}_{(t,x)}\,\d t\d x = \int_{\mathbb{R}^2_+} \phi(\tilde \rho^{\varepsilon},\tilde m^{\varepsilon}) w_k(\tilde \rho^{\varepsilon},\tilde m^{\varepsilon})\,\varphi \,\d t\d x\,\,\,\,
  \mbox{for any $\varphi \in C_{\rm c}^{\infty}(\mathbb{R}^2_+)$}\Big\},
  \end{aligned}
  $$
which implies that
  $$
  \int_{\mathbb{R}^2_+} \big(\phi(\hat\rho^{\varepsilon},\hat m^{\varepsilon}) w_k(\hat\rho^{\varepsilon},\hat m^{\varepsilon}) -\phi(\tilde \rho^{\varepsilon},\tilde m^{\varepsilon})w_k(\tilde \rho^{\varepsilon},\tilde m^{\varepsilon}) \big)\, \varphi
  \,\d t\d x = 0\qquad \mbox{for any $\varphi \in C_{\rm c}^{\infty}(\mathbb{R}^2_+)$} \,\,\,\,\tilde{\mathbb{P}}\text{-{\it a.s.}}.
  $$
Notice that 
$\phi(\hat\rho^{\varepsilon},\hat m^{\varepsilon}) w_k(\hat\rho^{\varepsilon},\hat m^{\varepsilon}) 
-\phi(\tilde \rho^{\varepsilon},\tilde m^{\varepsilon})w_k(\tilde \rho^{\varepsilon},\tilde m^{\varepsilon}) 
\in L^1_{\rm loc}(\mathbb{R}^2_+)$, 
so that
$$
\phi(\hat\rho^{\varepsilon},\hat m^{\varepsilon}) w_k(\hat\rho^{\varepsilon},\hat m^{\varepsilon}) -\phi(\tilde \rho^{\varepsilon},\tilde m^{\varepsilon})w_k(\tilde \rho^{\varepsilon},\tilde m^{\varepsilon})=0
$$
holds $\tilde{\mathbb{P}}$-{\it a.s.} pointwise almost everywhere 
in $\left[0,\infty\right)\times \mathbb{R}$. 
Letting $k\to \infty$ gives 
$(\hat \rho^{\varepsilon},\hat m^{\varepsilon})=(\tilde \rho^{\varepsilon}, \tilde m^{\varepsilon})$ 
$\tilde{\mathbb{P}}$-{\it a.s.} almost everywhere in $\left[0,\infty\right)\times \mathbb{R}$. Thus, the conclusion follows.
\end{proof}

\begin{proposition}\label{prop-reduction-of-young-measure}
Let $\tilde \mu$ be the limit random Young measure obtained 
in {\rm Proposition \ref{prop-apply-jakubowski-skorokhod-representation-to-take-limit}}. 
Then it holds $\tilde{\mathbb{P}}$-{\it a.s.} that, for {\it a.e.} $(t,x)\in \mathbb{R}^2_+$, 
either $\tilde \mu_{(t,x)}$ is concentrated on the vacuum region $\mathcal{V}:=\{(
\rho,m)\in \mathbb{R}^2_+\,:\, \rho=0\}$ or $\tilde \mu_{(t,x)}$ is reduced to a Dirac mass $\delta_{(\rho^*(t,x,\omega),m^*(t,x,\omega))}$ on the region $\mathfrak{T}=\{(\rho,m)\in \mathbb{R}^2_+\,:\, \rho>0\}$.
\end{proposition}

\begin{proof}
If $\tilde \mu_{(t,x)}$ is not concentrated on the vacuum region 
$\{(\rho,m)\in \mathbb{R}^2_+\,:\, \rho=0\}$, 
we fix $(t,x,\omega)\in \mathbb{R}_+\times\mathbb{R}\times\Omega$ such that 
the Tartar commutation \eqref{eq-tartar-commutation} is satisfied. Then
we use \cite[Theorem 8.5]{chenhuangLiWangWang24CMP} to obtain that $\tilde \mu_{(t,x)}$ is reduced to a Dirac mass $\delta_{(\rho^*(t,x,\omega),m^*(t,x,\omega))}$ on the region $\{(\rho,m)\in \mathbb{R}^2_+\,:\, \rho>0\}$.
\end{proof}

\begin{remark}
Notice that all the entropy pairs under our consideration are weak entropy pairs, {\it i.e.}, 
they vanish at the vacuum. Moreover, in the vacuum, the entropy kernel $\chi(\rho=0)=0$.
\end{remark}

We now give a generalized version of Proposition \ref{prop-test-func-and-convergence-for-random-young-measure} 
without proof, since it can be proved similarly as Proposition \ref{prop-test-func-and-convergence-for-random-young-measure}.

\begin{proposition}\label{prop-generalized-convergence-of-young-measure}
Let $\tilde \mu^{\varepsilon}_{(t,x)}$ and $\tilde \mu_{(t,x)}$ 
be the random Young measures obtained in {\rm Proposition \ref{prop-apply-jakubowski-skorokhod-representation-to-take-limit}}. 
Then 
\begin{itemize}
  \item[\rm (i)] For $\nu_{(t,x)} \in \{\tilde \mu^{\varepsilon}_{(t,x)}\}_{\varepsilon >0} \cup \{\tilde \mu_{(t,x)}\}$,
  there exists some $p_0>1$ such that
    \begin{equation}\label{iq-higher-integrability-on-new-probability-space}
    \begin{aligned}
    \tilde{\E} \big[\Big(\int_0^T \int_{\hat K} \int_{\mathfrak{T}} \big(\rho P(\rho)+\frac{|m|^3}{\rho^2} \big) \,\d \nu_{(t,x)}(\rho,m) \d x \d t\Big)^{p_0}\big] \le {\hat C}
    (p_0).
    \end{aligned}
    \end{equation}
  \item[\rm (ii)] Let $p_0>1$ be as in {\rm (i)} above. 
  Let $\phi(t,x,\rho,m; \omega) \in L^{p_0}(\Omega, L^1_{\rm loc}(\mathbb{R}^2_+;C(\mathbb{R}_+^2)))$ 
  satisfy that $\phi(t,x,0,m;\omega)=0$ and, for any $\delta>0$, there exists $C_{\delta}>0$ such that, 
  for {\it a.e.} $(t,x,\omega)\in \mathbb{R}^2_+\times \tilde{\Omega}$,
  $$
  |\phi(t,x,\rho,m;\omega)|\le C_{\delta}+ h(t,x,\omega) +\delta \big(\rho^{\gamma_2+1}+\frac{|m|^3}{\rho^2}\big)
  \qquad\mbox{for large $\rho$ and $|m|$},
  $$
 where $h(t,x;\omega)\in L^{p_0}(\Omega;L^{q_0}_{\rm loc}(\mathbb{R}^2_+))$ with $q_0>1$.
 Then, for any compact subset $\hat K \Subset \mathbb{R}$ and any $T>0$, 
  \begin{align}
  &\tilde{\E}\big[ \Big| \int_0^T \int_{\hat K} \Big(\int_{\mathbb{R}^2_+}\phi(t,x,\rho,m;\omega) 
   \,\d (\tilde \mu^{\varepsilon}_{(t,x)}
   -\tilde \mu_{(t,x)})( \rho,m) \Big)\varphi(t,x) \,\d x \d t  \Big|\big]\nonumber\\[1mm]
  &\longrightarrow 0  \qquad\,\,\mbox{
  as $\varepsilon \to 0\quad$ for any $\varphi \in L^{\infty}_{\rm loc}(\mathbb{R}^2_+)$}.
    \label{eq-generalized-convergence-of-young-measure-for-admissible-func}
   \end{align}
\end{itemize}
\end{proposition}

\medskip
\begin{corollary}\label{coro-almost-everywhere-convergence-and-delta-mass-coincide-with-limit}
Let $(\tilde \rho^{\varepsilon},\tilde m^{\varepsilon})$ and $(\tilde \rho, \tilde m)$ 
be the random variables, and let $\tilde \mu_{(t,x)}$ be the limit random Young measure obtained 
in {\rm Proposition \ref{prop-apply-jakubowski-skorokhod-representation-to-take-limit}}. Then 
\begin{equation}
(\int_{\mathbb{R}^2_+} \rho\, \d \tilde \mu_{(t,x)},\, \int_{\mathbb{R}^2_+} m\, \d \tilde \mu_{(t,x)})
=(\tilde \rho(t,x),\, \tilde m(t,x)) \qquad\,\mbox{almost everywhere}. 
\end{equation}
In particular, 
\begin{equation}\label{eq-vanish-on-vacuum-for-tilde-m}
(\tilde \rho, \tilde m)=(0,0)
\end{equation}
on the set $\{(t,x,\omega)\in \mathbb{R}^2_+\times \tilde{\Omega}\,:\,{\rm supp} \tilde \mu_{(t,x)}=\mathcal{V}\}$,
while  
\begin{equation}\label{eq-coincide-outside-vacuum-for-tilde-rho-m}
(\rho^*(t,x,\omega),\, m^*(t,x,\omega) =(\tilde \rho(t,x,\omega),\, \tilde m(t,x,\omega))
\end{equation}
on the set $\{(t,x,\omega)\in \mathbb{R}^2_+\times\tilde{\Omega}\,:\, \tilde \mu_{(t,x)}=\delta_{(\rho^*(t,x,\omega),m^*(t,x,\omega))},\ \rho^*(t,x,\omega)>0 \}$.
Furthermore, there exists a function $\tilde u$ such that $\tilde m=\tilde \rho \tilde u$ and $\tilde u=0$ on the set $\{(t,x,\omega)\in \mathbb{R}^2_+\times \tilde{\Omega}\,:\,\tilde \rho=0\}$. 
In addition,
$$
(\tilde \rho^{\varepsilon},\,\tilde m^{\varepsilon}) \to (\tilde \rho,\, \tilde m)
\qquad\mbox{$\tilde{\mathbb{P}}$-{\it a.s.} almost everywhere}.
$$
\end{corollary}

\begin{proof}  We divide the proof into two steps.

\smallskip
\textbf{1}. Let
$\phi(\rho,m)=\rho$ or $m$. Note that $\phi(\rho,m)$ is a natural weak entropy, 
which vanishes at the vacuum, so $\phi$ satisfies the conditions in 
Proposition \ref{prop-generalized-convergence-of-young-measure}(ii). Therefore, 
we apply Proposition \ref{prop-generalized-convergence-of-young-measure}(ii) to obtain that,
for any $\varphi \in L^{\infty}_{\rm loc}(\mathbb{R}^2_+)$,
  \begin{equation}\label{eq-convergence-of-young-measure-for-rho-m-L-infinity}
  \lim_{\varepsilon \to 0}\tilde{\E}\big[ \Big| \int_0^T \int_{\hat K} \Big(\int_{\mathbb{R}^2_+}  \phi(\rho,m)
  \,\d (\tilde \mu^{\varepsilon}_{(t,x)}
  -\tilde
  \mu_{(t,x)})( \rho,m) \Big)\varphi(t,x) \,\d x \d t  \Big|\big]= 0.
  \end{equation}
As in the proof of Proposition \ref{prop-test-func-and-convergence-for-random-young-measure}(ii), 
we can also extend the above convergence. 
More specifically, since $|m|^{\frac{3(\gamma_2+1)}{\gamma_2+3}}\le \frac{|m|^3}{\rho^2}+\rho^{\gamma_2+1}$, 
by Proposition \ref{prop-generalized-convergence-of-young-measure}(i) and the same arguments as 
the justification of \eqref{eq-convergence-of-young-measure-for-admissible-func}, we see that 
the convergence \eqref{eq-convergence-of-young-measure-for-rho-m-L-infinity} also holds 
for any $\varphi\in L^{q_0}_{\rm loc}(\mathbb{R}^2_+)$ 
with $q_0=\max \{\frac{3(\gamma_2+1)}{2\gamma_2}, \frac{\gamma_2+1}{\gamma_2}\} >1$. 
Hence, by Lemma \ref{prop-identify-tilde-mu-varepsilon-delta-mass}, 
for each $\varphi\in L^{q_0}_{\rm loc}(\mathbb{R}^2_+)$, up to a subsequence,
  $$
  \lim_{\varepsilon \to 0} \int_0^T \int_{\hat K} \Big(\phi(\tilde \rho^{\varepsilon},\tilde m^{\varepsilon}) -\int_{\mathbb{R}^2_+}  \phi(\rho,m) \,\d \tilde \mu_{(t,x)}( \rho,m) \Big)\varphi(t,x) \,\d x \d t =0\qquad \tilde{\mathbb{P}}\text{-{\it a.s.}}.
  $$
Since $L^{q_0}_{\rm loc}(\mathbb{R}^2_+)$ is separable, there exists a further subsequence (without changing the notation) 
such that, for any $\varphi\in L^{q_0}_{\rm loc}(\mathbb{R}^2_+)$,
  $$
  \lim_{\varepsilon \to 0} \int_0^T \int_{\hat K} \Big(\phi(\tilde \rho^{\varepsilon},\tilde m^{\varepsilon}) -\int_{\mathbb{R}^2_+}  \phi(\rho,m) \,\d \tilde \mu_{(t,x)}( \rho,m) \Big)\varphi(t,x) \,\d x \d t =0
  \qquad \,\,\tilde{\mathbb{P}}\text{-{\it a.s.}}.
  $$

On the other hand, by Proposition \ref{prop-apply-jakubowski-skorokhod-representation-to-take-limit}, 
we know that $(\tilde \rho^{\varepsilon},\tilde m^{\varepsilon})\to (\tilde \rho,\tilde m)$ $\tilde{\mathbb{P}}$-{\it a.s.} 
in $\mathcal{X}_{\rho}\times \mathcal{X}_{m}$.  
Then we obtain that, for any $\varphi \in C_{\rm c}^{\infty}(\mathbb{R}^2_+)$,
$$
\lim_{\varepsilon \to 0} \int_0^T \int_{\hat K} \big(\phi(\tilde \rho^{\varepsilon},\tilde m^{\varepsilon}) - \phi(\tilde\rho,\tilde m)  \big)\varphi(t,x) \,\d x \d t =0 \qquad\,\, \tilde{\mathbb{P}}\text{-{\it a.s.}}.
$$
Therefore, it follows from the uniqueness of the limit that, for any $\varphi \in C_{\rm c}^{\infty}(\mathbb{R}^2_+)$,
  $$
  \int_0^T \int_{\hat K} \Big(\int_{\mathbb{R}^2_+}  \phi(\rho,m) \,\d \tilde \mu_{(t,x)}( \rho,m) -\phi(\tilde\rho,\tilde m) \Big) \varphi(t,x) \,\d x \d t=0 \qquad\,\, \tilde{\mathbb{P}}\text{-{\it a.s.}}.
  $$
Since $\int_{\mathbb{R}^2_+}  \phi(\rho,m) \,\d \tilde \mu_{(t,x)}( \rho,m)
-\phi(\tilde\rho,\tilde m) \in L^1_{\rm loc}(\mathbb{R}^2_+)$, we obtain
  $$
  (\int_{\mathbb{R}^2_+} \rho\, \d \tilde \mu_{(t,x)},\, \int_{\mathbb{R}^2_+} m\, \d \tilde \mu_{(t,x)})
  =(\tilde \rho(t,x),\, \tilde m(t,x)) \qquad\,\,\mbox{$\tilde{\mathbb{P}}$-{\it a.s.} almost everywhere}.
  $$
 Note that, on the set $\{(t,x,\omega)\in \mathbb{R}^2_+\times \tilde{\Omega}\,:\,{\rm supp} \tilde \mu_{(t,x)}=\mathcal{V}\}$, 
 the weak entropies vanish on the vacuum so that
  \begin{equation*}
  (\int_{\mathbb{R}^2_+} \rho\, \d \tilde \mu_{(t,x)}, \int_{\mathbb{R}^2_+} m\, \d \tilde \mu_{(t,x)})
  =(0,0)=(\tilde \rho(t,x,\omega), \tilde m(t,x,\omega)).
  \end{equation*}
On the set $\{(t,x,\omega)\in \mathbb{R}^2_+\times \tilde{\Omega}\,:\, \tilde \mu_{(t,x)}=\delta_{(\rho^*(t,x,\omega),m^*(t,x,\omega))},\ \rho^*(t,x,\omega)>0 \}$, we obtain \eqref{eq-coincide-outside-vacuum-for-tilde-rho-m}.
Thus, if we define $\tilde u:=0$ on the set 
$\{(t,x,\omega)\in \mathbb{R}^2_+\times \tilde{\Omega}\,:\,\tilde \rho=0\}$ 
and $\tilde u:=\frac{\tilde m}{\tilde \rho}$ 
on the set $\{(t,x,\omega)\in \mathbb{R}^2_+\times \Omega\,:\,\tilde \rho>0\}$, 
then $\tilde m=\tilde \rho \tilde u\,\,$ $\tilde{\mathbb{P}}$-{\it a.s.} almost everywhere.

\smallskip
\textbf{2}. Now we apply Proposition \ref{prop-generalized-convergence-of-young-measure}(ii) 
to prove that $(\tilde \rho^{\varepsilon},\tilde m^{\varepsilon}) \to (\tilde \rho, \tilde m)$ $\tilde{\mathbb{P}}$-{\it a.s.} almost everywhere. To do this, let $\phi(t,x,\rho,m;\omega)=|\rho-\tilde \rho(t,x,\omega)|$ or $|m-\tilde m(t,x,\omega)|$, then $\phi(t,x,\rho,m;\omega)$ satisfies the conditions in Proposition \ref{prop-generalized-convergence-of-young-measure}(ii)
by using \eqref{eq-vanish-on-vacuum-for-tilde-m}, again the fact $|m|^{\frac{3(\gamma_2+1)}{\gamma_2+3}}\le \frac{|m|^3}{\rho^2}+\rho^{\gamma_2+1}$, and \eqref{iq-higher-integrability-on-new-probability-space}. 
Thus, \eqref{eq-generalized-convergence-of-young-measure-for-admissible-func} yields that,
  for any 
  $\hat K \Subset \mathbb{R}$ and for any $T>0$,
  \begin{equation}
  \tilde{\E} \big[\Big| \int_0^T \int_{\hat K} \Big(\int_{\mathbb{R}^2_+}  \phi(t,x,\rho,m;\omega)
  \,\d \big(\tilde \mu^{\varepsilon}_{(t,x)}-\tilde \mu_{(t,x)}\big)( \rho,m) \Big)\varphi(t,x) \,\d x \d t \Big|
  \big] \to 0
  \qquad\mbox{as $\varepsilon \to 0$},
\end{equation}
for any $\varphi \in L^{\infty}_{\rm loc}(\mathbb{R}^2_+)$. 
Moreover, as the arguments in Step \textbf{1} above, 
we can also extend this convergence to be satisfied by a more general 
$\varphi\in L^{q_0}_{\rm loc}(\mathbb{R}^2_+)$ with $q_0=\max \{\frac{3(\gamma_2+1)}{2\gamma_2}, \frac{\gamma_2+1}{\gamma_2}\} >1$. 
Similarly, the separability of $L^{q_0}_{\rm loc}(\mathbb{R}^2_+)$ enables us to obtain that, for any $\varphi\in L^{q_0}_{\rm loc}(\mathbb{R}^2_+)$,
 $\tilde{\mathbb{P}}$-{\it a.s.}  (up to a subsequence),
  $$
  \int_0^T \int_{\hat K} \Big(\int_{\mathbb{R}^2_+}  \phi(t,x,\rho,m;\omega) \,
  \d \big(\tilde \mu^{\varepsilon}_{(t,x)}-\tilde \mu_{(t,x)}\big)( \rho,m) \Big)\varphi(t,x) \,\d x \d t \to 0
  \qquad\mbox{as $\varepsilon\to 0$}.
  $$
Choosing $\varphi \equiv 1$ and noting by
\eqref{eq-vanish-on-vacuum-for-tilde-m}--\eqref{eq-coincide-outside-vacuum-for-tilde-rho-m} that
  $$
  \int_{\mathbb{R}^2_+}  \phi(t,x,\rho,m;\omega) \,\d \tilde \mu_{(t,x)}( \rho,m)=0,
  $$
we obtain that, as $\varepsilon \to 0$,
  $$
  \int_0^T \int_{\hat K} |\tilde \rho^{\varepsilon}-\tilde \rho|  \,\d x \d t 
  +\int_0^T \int_{\hat K} |\tilde m^{\varepsilon}-\tilde m|  \,\d x \d t \to 0\qquad \tilde{\mathbb{P}}\text{-{\it a.s.}},
  $$
which implies that there exists a further subsequence (without changing the notation) such that $(\tilde \rho^{\varepsilon},\tilde m^{\varepsilon}) \to (\tilde \rho, \tilde m)\,\,$ $\tilde{\mathbb{P}}$-{\it a.s.} almost everywhere.
\end{proof}

\smallskip
\section{Martingale Solutions: 
Proof of Theorems \ref{thm-well-posedness-for-euler-on-whole-space-general-pressure-law}--
\ref{thm-better-well-posedness-for-euler-on-whole-space-general-pressure-law}}\label{sec-martingale-solution}

We are now ready to prove our main results: Theorem \ref{thm-well-posedness-for-euler-on-whole-space-general-pressure-law} and Theorem \ref{thm-better-well-posedness-for-euler-on-whole-space-general-pressure-law}.

\subsection{The Entropy Inequality for Global Martingale Solutions}

\begin{proposition}\label{prop-take-limit-obtain-relative-energy-estiate-for-Euler}
Let $(\tilde \rho^{\varepsilon},\tilde m^{\varepsilon})$ and $(\tilde \rho, \tilde m)$ be the random variables, 
and let $\tilde \mu^{\varepsilon}_{(t,x)}$ and $\tilde \mu_{(t,x)}$ be 
the random Young measures obtained
in {\rm Proposition \ref{prop-apply-jakubowski-skorokhod-representation-to-take-limit}}. Then,
for $p\ge 1$, there exists a constant $C(T, E_{0,p})>0$ such that
  $$
  \tilde \E \big[\Big\|\int_{\mathbb{R}} \big(\frac12 \frac{\tilde m^2}{\tilde \rho} +e(\tilde \rho,\rho_{\infty})\big) \,\d x\Big\|_{L^{\infty}([0,T])}^p \big] \le C(T, E_{0,p}).
  $$
\end{proposition}

\begin{proof}
From \eqref{iq-energy-estimate-on-whole-space}, we obtain that, for any $t_1<t_2 \le T$,
    $$
    \mathbb{E} \big[\Big(\int_{t_1}^{t_2}\int_{\mathbb{R}} \big(\frac12 \frac{(m^{\varepsilon})^2}{\rho^{\varepsilon}} 
    +e^*(\rho^{\varepsilon },\rho_{\infty})\big) \,\d x \d \tau \Big)^p\big]
    \le (t_2-t_1)^p\, C(T,E_{0,p}).
    $$
Let $\tilde w_k(\rho,m) \ge 0$ be a continuous cut-off function such that $\tilde w_k\equiv 1$ 
on the set $\{(\rho,m)\in \mathbb{R}^2_+\,:\, \mathcal{K}(\rho)\le k,\ |m|\le k\}$ 
and $\tilde w_k \equiv 0$ outside the set $\{(\rho,m)\in \mathbb{R}^2_+\,:\, \mathcal{K}(\rho)\le 2k,\ |m|\le 2k\}$. 
Let $B_n\subset \mathbb{R}$ be an open ball centered at the origin with radius $n$. 
Then, by the equality of laws from 
Proposition \ref{prop-apply-jakubowski-skorokhod-representation-to-take-limit}, we have
  $$
  \begin{aligned}
  &\tilde \E  \big[\Big(\int_{t_1}^{t_2}\int_{B_n} \int_{\mathbb{R}^2_+} \big( \frac12 \frac{m^2}{\rho} +e^*(\rho, \rho_{\infty})\big)\tilde w_k \,\d \tilde \mu^{\varepsilon}_{(t,x)} \d x\d \tau \Big)^p \big]\\
  &= \E  \big[\Big(\int_{t_1}^{t_2}\int_{B_n} \big( \frac12 \frac{(m^{\varepsilon})^2}{\rho^{\varepsilon}} +e^*(\rho^{\varepsilon},\rho_{\infty})\big)\tilde w_k(\rho^{\varepsilon},m^{\varepsilon}) \,\d x \d \tau \Big)^p \big]\\
  &\le C(T,E_{0,p})(t_2-t_1)^p.
 \end{aligned}
 $$
Thus, by the monotone convergence theorem 
and the fact that the weak entropies vanish at the vacuum, we obtain that, 
for any $t_1<t_2 \le T$,
  \begin{equation}\label{iq-energy-bound-for-tilde-rho-m-varepsilon}
  \tilde \E  \big[\Big(\int_{t_1}^{t_2} \int_{B_n}  \big( \frac12 \frac{(\tilde m^{\varepsilon})^2}{\tilde \rho^{\varepsilon}} +e^*(\tilde \rho^{\varepsilon},\rho_{\infty})\big) \,\d x \d \tau\Big)^p \big] \le C(T,E_{0,p})(t_2-t_1)^p.
  \end{equation}
By Corollary \ref{coro-almost-everywhere-convergence-and-delta-mass-coincide-with-limit}, 
we know that $\tilde m=0$ on the set $\{(t,x,\omega)\in \mathbb{R}^2_+\times \tilde{\Omega}\,:\,\tilde \rho=0\}$ and $(\tilde \rho^{\varepsilon},\tilde m^{\varepsilon}) \to (\tilde \rho, \tilde m)$ $\tilde{\mathbb{P}}$-{\it a.s.} almost everywhere, which implies that $\frac{(\tilde m^{\varepsilon})^2}{\tilde \rho^{\varepsilon}} \to \frac{\tilde m^2}{\tilde \rho}$ almost everywhere on the set $\{(t,x,\omega)\in \mathbb{R}^2_+\times \tilde{\Omega}\,:\,\tilde \rho>0\}$ 
and $e^*(\tilde \rho^{\varepsilon},\rho_{\infty}) \to e^*(\tilde \rho,\rho_{\infty})$ almost everywhere. Therefore, it follows from Fatou's lemma that, for any $t_1<t_2 \le T$,
  $$
  \tilde \E  \big[\Big(\int_{t_1}^{t_2}\int_{B_n}  \big( \frac12 \frac{\tilde m^2}{\tilde \rho} +e^*(\tilde \rho,\rho_{\infty})\big) \,\d x \d \tau\Big)^p \big] 
  \le C(T,E_{0,p})(t_2-t_1)^p,
  $$
which implies, by Fatou's lemma again, that, for almost every $t\in [0,T]$,
  $$
  \tilde \E  \big[\Big(\int_{B_n}  \big( \frac12 \frac{\tilde m^2}{\tilde \rho} +e^*(\tilde \rho,\rho_{\infty})\big)(t) \,\d x  \Big)^p \big] \le C(T,E_{0,p}).
  $$
Then, by the monotone convergence theorem, we have
  $$
  \tilde \E  \big[ \Big\|\int_{\mathbb{R}}  \big( \frac12 \frac{\tilde m^2}{\tilde \rho} +e^*(\tilde \rho,\rho_{\infty})\big) \,\d x  \Big\|_{L^{\infty}([0,T])}^p \big] \le C(T,E_{0,p}).
  $$
\end{proof}

In the general pressure law case, the explicit formula for the entropies is unavailable, 
which makes it difficult to calculate the entropies and their derivatives. 
We first recall the asymptotic expansions for the entropy kernel $\chi(\rho,u)$ and the entropy flux 
kernel $\sigma(\rho,u)$ when $\rho$ is bounded.

\begin{lemma}[\hspace{-1pt}\cite{chenhuangLiWangWang24CMP}, Lemmas 4.2--4.3]\label{lem-asymptotic-expansion-for-entropy-and-flux-kernels}
The entropy kernel $\chi(\rho,u)$ and entropy flux kernel $\sigma(\rho,u)$ admit the expansion{\rm :} for $\rho\in \left[0,\infty \right)$,
  $$
  \begin{aligned}
  &\chi(\rho,u)=\hat f_1(\rho)\hat F_{\lambda_1}(\rho,u)+ \hat f_2(\rho)\hat F_{\lambda_1+1}(\rho,u)+a_1(\rho,u),\\
  &(\sigma-u\chi)(\rho,u)=-u\big( \hat g_1(\rho)\hat F_{\lambda_1}(\rho,u)+\hat g_2(\rho) \hat F_{\lambda_1+1}(\rho,u)  \big) + a_2(\rho,u),
  \end{aligned}
  $$
where
  $$
  \begin{aligned}
 \lambda_1=\frac{3-\gamma_1}{2(\gamma_1-1)}>0,
 \qquad\,\hat F_{\lambda}(\rho,u) =[\mathcal{K}(\rho)^2-u^2]^{\lambda}_+ \,\,\,\,\,\mbox{for any $\lambda>0$}.
  \end{aligned}
  $$
The support of entropy kernel ${\rm supp}(\chi(\rho,u))\subset \{(\rho,u):|u|\le \mathcal{K}(\rho)\}$, 
and $\chi(\rho,u)>0$ in $\{(\rho,u):|u|<\mathcal{K}(\rho)\}$. 
The remainder terms $a_1(\rho,u)$ and $a_2(\rho,u)$ are H\"older continuous. 
Moreover, for any fixed $\rho_{\rm max}>0$, there exist $\delta_0 \in (0 ,1)$ 
and $C(\rho_{\rm max})>0$ depending only on $\rho_{\rm max}$ such that,
for any $\rho\in [0,\rho_{\rm \max}]$,
  $$
  |a_1(\rho,u-s)|+|a_2(\rho,u-s)|\le C(\rho_{\rm max})[\mathcal{K}(\rho)^2 -(u-s)^2]^{\lambda_1+\delta_0+1}_+.
  $$
Furthermore, for any $\rho\in [0,\rho_{\rm \max}]$,
  $$
  |\hat f_1(\rho)|+|\hat f_2(\rho)|+|\hat g_1(\rho)|+|\hat g_2(\rho)|\le C(\rho_{\rm max}).
  $$
\end{lemma}

The next lemma shows what the growth of the generating function in the entropy is allowed
in the general pressure law case.

\begin{lemma}\label{lem-growth-of-entropy-and-derivatives-general-pressure-law}
Let $\psi\in C^2(\mathbb{R})$ be convex and satisfy
  $$
  |\psi(s)|\le C|s|^{2-\delta},\qquad |\psi^{\prime}(s)|\le C|s|^{1-\delta}.
  $$
Then, when $\rho\ge \rho^{\star}$ and $\delta\ge \frac{3}{2}\big(1- \frac{1}{\gamma_2}\big)$,
  $$\begin{aligned}
  &|\eta^{\psi}(\rho,u)|+|q^{\psi}(\rho,u)| +\frac{1}{\rho}|\partial_u\eta^{\psi}(\rho,u)|^2+\frac{1}{\rho^2}|\partial_u^2 \eta^{\psi}(\rho,u)|\sum_k\zeta_k^2 \le C\big(1+\rho|u|^3+ \rho P(\rho) \big),
  \end{aligned}
  $$
and, when $\rho\le \rho^{\star}$ and $\delta\ge 0$,
  $$
  |\eta^{\psi}(\rho,u)|+|q^{\psi}(\rho,u)| +\frac{1}{\rho}|\partial_u\eta^{\psi}(\rho,u)|^2+\frac{1}{\rho^2}|\partial_u^2 \eta^{\psi}(\rho,u)|\sum_k\zeta_k^2 \le C(\rho^{\star})\big(1+\rho|u|^3+ \rho P(\rho)\big).
  $$
\end{lemma}

\begin{proof}
The proof can be done by direct
calculation. Indeed, we have
  $$\begin{aligned}
  \eta^{\psi}(\rho,u)&=\int_{-\mathcal{K}(\rho)}^{\mathcal{K}(\rho)} \chi(\rho,s)\psi(u-s)\, \d s,\\
  q^{\psi}(\rho,u)&=\int_{-\mathcal{K}(\rho)}^{\mathcal{K}(\rho)} \sigma(\rho,s)\psi(u-s)\, \d s\\
  &=\int_{-\mathcal{K}(\rho)}^{\mathcal{K}(\rho)} \big(\sigma-u\chi \big)(\rho,s)\psi(u-s)\, \d s + u\int_{-\mathcal{K}(\rho)}^{\mathcal{K}(\rho)} \chi(\rho,s)\psi(u-s)\, \d s\\
  &:=g^{\psi}(\rho,u)+u\eta^{\psi},
  \end{aligned}$$
where we have used the fact that 
${\rm supp}(\chi(\rho,u))\cup {\rm supp}(\sigma(\rho,u))
\subset \{(\rho,u):|u|\le \mathcal{K}(\rho)\}$
(see \cite[Theorem 2.1]{chenlefloch00ARMA}).

Then, using the following bounds (see \cite[Lemma 4.6 and 4.9]{chenhuangLiWangWang24CMP}):
  $$
  \|\chi(\rho,\cdot)\|_{L^{\infty}_u} \le C\rho,\quad \|\big(\sigma-u\chi \big)(\rho,\cdot)\|_{L^{\infty}_u} \le C\rho^{1+\theta_2},
  $$
we obtain that, when $\rho\ge \rho^{\star}$,
  $$\begin{aligned}
  &|\eta^{\psi}(\rho,u)|\le C\big(\rho^{1+\theta_2}|u|^{2-\delta}+\rho^{1+(3-\delta)\theta_2}+\rho^{\gamma_2}|u|^{1-\delta} \big),\\
  &|q^{\psi}(\rho,u)|\le C\big( \rho^{\gamma_2}|u|^{2-\delta}+\rho^{(4-\delta)\theta_2+1}+\rho^{3\theta_2 +1}|u|^{1-\delta} + \rho^{\theta_2+1}|u|^{3-\delta}+ \rho^{\gamma_2}|u|^{2-\delta} \big),\\
  &|\partial_u\eta^{\psi}(\rho,u)|\le C\big(\rho^{\theta_2+1}|u|^{1-\delta}+ \rho^{(2-\delta)\theta_2+1}\big),\\
  &|\partial_u^2 \eta^{\psi}(\rho,u)|\le C \big( \rho|u|^{1-\delta}+\rho^{(1-\delta)\theta_2+1} \big).
  \end{aligned}$$
We now take $\partial_u^2 \eta^{\psi}$ for an example, since the other estimates are similar. 
Since $\psi$ is convex, we have
$$\begin{aligned}
  |\partial_u^2 \eta^{\psi}(\rho,u)|&=\Big|\int_{-\mathcal{K}(\rho)}^{\mathcal{K}(\rho)} \chi(\rho,s)\psi^{\prime \prime}(u-s)\, \d s\Big|\\
  &\le \|\chi(\rho,\cdot)\|_{L^{\infty}_u}\big| \psi^{\prime}(u-\mathcal{K}(\rho))-\psi^{\prime}(u+\mathcal{K}(\rho)) \big|\\
  &\le C\rho \big(|u|+\mathcal{K}(\rho)\big)^{1-\delta}\\
  &\le C \big( \rho|u|^{1-\delta}+\rho^{(1-\delta)\theta_2+1} \big),
  \end{aligned}$$
where, in the last inequality, 
we have used \eqref{iq-lower-upper-bound-for-k(rho)-2}. Then the conclusion for $\rho\ge \rho^{\star}$ follows easily.

When $\rho$ is bounded, let $z=\frac{s}{\mathcal{K}(\rho)}$.
By Lemma \ref{lem-asymptotic-expansion-for-entropy-and-flux-kernels}, we have
  $$\begin{aligned}
  \eta^{\psi}(\rho,u)
  &\le C(\rho_{\rm max})\mathcal{K}(\rho)^{2\lambda_1+1}\int_{-1}^{1} |\psi(u-\mathcal{K}(\rho)z)|[1-z^2]^{\lambda_1}_+ ( \cdots ) \, \d z\\
  &\le C(\rho_{\rm max})\big(\rho|u|^{2-\delta}+\rho^{\gamma_1-\frac{\delta}{2}(\gamma_1-1)}+\rho^{\theta_1}|u|^{1-\delta}  \big),
  \end{aligned}$$
where $(\cdots)=1+\mathcal{K}(\rho)^2[1-z^2]_+ + \mathcal{K}(\rho)^{2\alpha_0+2}[1-z^2]^{\alpha_0+1}_+\le C(\rho_{\rm max})$ 
and, in the second inequality, we have used \eqref{iq-lower-upper-bound-for-k(rho)-1}. 
Similarly, using the convexity of $\psi$, we obtain
  $$\begin{aligned}
  &|q^{\psi}(\rho,u)|\le C(\rho_{\rm max})\big( \rho |u|^{2-\delta}+\rho^{\gamma_1-\frac{\delta}{2}(\gamma_1-1)}+\rho^{\theta_1 }|u|^{1-\delta} + \rho|u|^{3-\delta}+ \rho^{\gamma_1-\frac{\delta}{2}(\gamma_1-1)}|u|+\rho^{\theta_1}|u|^{2-\delta} \big),\\
  &|\partial_u\eta^{\psi}(\rho,u)|\le C(\rho_{\rm max})\big(\rho|u|^{1-\delta}+ \rho^{(1-\delta)\theta_1+1}\big),\\
  &|\partial_u^2 \eta^{\psi}(\rho,u)|\le C(\rho_{\rm max}) \big( \rho|u|^{1-\delta}+\rho^{(1-\delta)\theta_1+1} \big).
  \end{aligned}$$
Then the conclusion for $\rho\le \rho^{\star}$ follows after a direct calculation.
\end{proof}

\begin{proposition}\label{prop-entropy-inequality-general-pressure-law}
Let $(\eta,q)=(\eta^\psi, q^\psi)$ be the entropy pair defined 
in \eqref{eq-entropy-flux-representation} with convex generating 
function $\psi\in C^2(\mathbb{R})$ satisfying \eqref{eq-definition-of-growth-for-weight-function-general-pressure}.
Then, for any nonnegative $\varphi \in C_{\rm c}^{\infty}((0,T)\times\mathbb{R})$, the following entropy inequality holds $\tilde{\mathbb{P}}$-{\it a.s.}{\rm :}
  $$
  \begin{aligned}
  &\int_0^T \int_{\mathbb{R}} \big( \eta(\tilde  U)\varphi_t+ q(\tilde U) \varphi_x \big) \,\d x \d t+ \int_0^T \int_{\mathbb{R}} \varphi \partial_{\tilde m} \eta (\tilde U) \Phi( \tilde U) \,\d x \d \tilde W(t) \\
  &+\frac12 \int_0^T \int_{\mathbb{R}} \varphi \partial_{\tilde m}^2 \eta(\tilde U) \sum_k a_k^2  \zeta_k^2(\tilde U) \,\d x \d t
  \ge  0,
  \end{aligned}
  $$
where $\tilde U=(\tilde \rho, \tilde m)^\top$.
\end{proposition}

\begin{proof}
By the It\^o formula, the solution $U^{\varepsilon}$ of the parabolic approximation \eqref{eq-parabolic-approximation} satisfy $\tilde{\mathbb{P}}$-{\it a.s.} the following entropy balance equation:
  $$
  \d \eta(U^{\varepsilon}) + \partial_x q(U^{\varepsilon})\,\d t =\varepsilon \nabla \eta(U^{\varepsilon})\partial_x^2 U^{\varepsilon} \,\d t + \partial_m \eta(U^{\varepsilon}) \Phi^{\varepsilon}(U^{\varepsilon}) \,\d W(t)+\frac12 \partial_m^2 \eta(U^{\varepsilon})\sum_k a_k^2 ( \zeta_k^{\varepsilon})^2 \,\d t,
  $$
which implies that, for any $\varphi \in C_{\rm c}^{\infty}((0,T)\times\mathbb{R})$,
  $$
  \begin{aligned}
  &\int_0^T \int_{\mathbb{R}} \big(\eta( U^{\varepsilon})\varphi_t+q( U^{\varepsilon}) \varphi_x ) \,\d x \d t +\varepsilon \int_0^T \int_{\mathbb{R}} \eta(U^{\varepsilon})\partial_x^2 \varphi \,\d x \d t \\
  &+ \int_0^T \int_{\mathbb{R}} \partial_{m^{\varepsilon}} \eta ( U^{\varepsilon}) \Phi^{\varepsilon}( U^{\varepsilon})\, \varphi \,\d x \d W(t)
  +\frac12 \int_0^T \int_{\mathbb{R}} \partial_{m^{\varepsilon}}^2 \eta( U^{\varepsilon}) \sum_k a_k^2 ( \zeta_k^{\varepsilon})^2(U^{\varepsilon}) \,\varphi \,\d x \d t \\
  &= \varepsilon \int_0^T \int_{\mathbb{R}} (\partial_x U^{\varepsilon})^{\top} \,\nabla^2 \eta(U^{\varepsilon}) \, \partial_x U^{\varepsilon}\,\varphi\,\d x \d t\quad \tilde{\mathbb{P}}\text{-{\it a.s.}}.
  \end{aligned}
  $$
Since the entropy $\eta$ is convex, then, for any $\varphi\ge 0$, we have
  $$
  \varepsilon \int_0^T \int_{\mathbb{R}}(\partial_x U^{\varepsilon})^{\top} \, \nabla^2 \eta(U^{\varepsilon}) 
  \, \partial_x U^{\varepsilon} \, \varphi\,\d x \d t \ge 0.
  $$
Therefore, by the equality of laws from Proposition \ref{prop-apply-jakubowski-skorokhod-representation-to-take-limit} 
and applying \cite[Theorem 2.9.1]{BFHbook18} for the stochastic integral, we obtain that, $\tilde{\mathbb{P}}$-{\it a.s.} for any nonnegative 
$\varphi \in C_{\rm c}^{\infty}((0,T)\times\mathbb{R})$,
  $$
  \begin{aligned}
  &\int_0^T \int_{\mathbb{R}} \big(\eta( \tilde U^{\varepsilon})\varphi_t+q(\tilde U^{\varepsilon}) \varphi_x ) \,\d x \d t +\varepsilon \int_0^T \int_{\mathbb{R}} \eta(\tilde U^{\varepsilon})\,\partial_x^2 \varphi \,\d x \d t \\
  &+ \int_0^T \int_{\mathbb{R}} \partial_{\tilde m^{\varepsilon}} \eta (\tilde U^{\varepsilon}) \Phi^{\varepsilon}(\tilde U^{\varepsilon}) \, \varphi\,\d x \d W(t)
  +\frac12 \int_0^T \int_{\mathbb{R}}  \partial_{\tilde m^{\varepsilon}}^2 \eta(\tilde U^{\varepsilon}) \sum_k a_k^2 ( \zeta_k^{\varepsilon})^2(\tilde U^{\varepsilon}) \,\varphi\,\d x \d t
  \ge  0,
  \end{aligned}
  $$
where $\tilde U^{\varepsilon}=(\tilde \rho^{\varepsilon}, \tilde m^{\varepsilon})^\top$.

We now perform the limit $\varepsilon \to 0$. 
Recall that the entropy pair $(\eta,q)$ vanishes at the vacuum. 
On the set $\{(t,x,\omega)\in \mathbb{R}^2_+\times \tilde{\Omega}\,:\,\tilde \rho>0\}$, 
by Lemma \ref{lem-growth-of-entropy-and-derivatives-general-pressure-law}, 
\eqref{iq-higher-integrability-on-new-probability-space} provides the equi-integrability 
of $(\eta,q)(\tilde U^{\varepsilon})$ for the entropy 
pair $(\eta,q)(\tilde U^{\varepsilon})$ with generating function $\psi$ 
satisfying \eqref{eq-definition-of-growth-for-weight-function-general-pressure}.
This equi-integrability, the $\tilde{\mathbb{P}}$-{\it a.s.} almost everywhere convergence in Corollary \ref{coro-almost-everywhere-convergence-and-delta-mass-coincide-with-limit},
and the Vitali convergence theorem are enough 
to deduce that, up to a subsequence,
$$
\begin{aligned}
&\int_0^T\int_{\mathbb{R}} \eta(\tilde U^{\varepsilon})
\varphi_t \,\d x \d t\, \to\, \int_0^T \int_{\mathbb{R}} \eta(\tilde U)\varphi_t \,\d x \d t
\quad\,\, \tilde{\mathbb{P}}\text{-{\it a.s.}}.
\\
  &\int_0^T \int_{\mathbb{R}} q(\tilde U^{\varepsilon})  \varphi_x \,\d x \d t
  \,\to \,\int_0^T \int_{\mathbb{R}} q(\tilde U)  \varphi_x \,\d x \d t\quad\,\, \tilde{\mathbb{P}}\text{-{\it a.s.}}.
  \end{aligned}
$$

Notice that
$$
\begin{aligned}
\tilde \E \big[\int_0^T \int_{\mathbb{R}}  \partial_{\tilde m}^2 \eta(\tilde U) \sum_k a_k^2 \zeta_k^2(\tilde U) \,\varphi\,\d x \d t\big]
\le \sum_k |a_k|^2 \tilde \E \big[\int_0^T \int_{\mathbb{R}}  \big|\partial_{\tilde m}^2 \eta(\tilde U) \big| 
 \zeta_k^2(\tilde U) \,\varphi\,\d x \d t\big].
\end{aligned}
$$ 
Using Lemma \ref{lem-growth-of-entropy-and-derivatives-general-pressure-law} and 
$\partial_{\tilde m}^2 \eta(\tilde U)=\frac{1}{\tilde\rho^2}\partial_{\tilde u}^2 \eta(\tilde \rho, \tilde u)$,
we have
\begin{equation}\label{iq-uniform-bdd-for-quadratic-variation-in-entropy-inequality-tilde-U}  
\sum_k \tilde \E \big[\int_0^T \int_{\mathbb{R}}  \big|\partial_{\tilde m}^2 \eta(\tilde U) \big| 
 (\zeta_k)^2(\tilde U) \,\varphi\,\d x \d t\big]
 \le C\tilde \E \big[\int_0^T \int_{{\rm supp}_x (\varphi)} \big(1+\frac{|\tilde m|^3}{\tilde\rho^2}+\tilde \rho P(\tilde \rho)
   \big) \,\d x \d t\big]\le {\hat C},
  \end{equation}
where
we have used \eqref{iq-higher-integrability-on-new-probability-space} in the last inequality.
Similarly, 
$$
\begin{aligned}
  \tilde \E \big[\int_0^T \int_{\mathbb{R}} \partial_{\tilde m^{\varepsilon}}^2 \eta(\tilde U^{\varepsilon}) \sum_k a_k^2 (\zeta_k^{\varepsilon})^2(\tilde U^{\varepsilon}) \, \varphi\,\d x \d t\big]
  \le \sum_k |a_k|^2 \tilde \E \big[\int_0^T \int_{\mathbb{R}} \big|\partial_{\tilde m^{\varepsilon}}^2 \eta(\tilde U^{\varepsilon}) \big| (\zeta_k^{\varepsilon})^2(\tilde U^{\varepsilon}) \, \varphi\,\d x \d t\big]
\end{aligned}
$$
and, using $\eqref{iq-noise-coefficience-growth-conditions-for-parabolic-approximation}$ 
and Lemma \ref{lem-growth-of-entropy-and-derivatives-general-pressure-law},
we have
\begin{equation}\label{iq-uniform-bdd-for-quadratic-variation-in-entropy-inequality-tilde-U-varepsilon}
  \sum_k \tilde \E \big[\int_0^T \int_{\mathbb{R}} \big|\partial_{\tilde m^{\varepsilon}}^2 \eta(\tilde U^{\varepsilon}) \big| (\zeta_k^{\varepsilon})^2(\tilde U^{\varepsilon}) \, \varphi\,\d x \d t\big]
  \le {\hat C}.
\end{equation}

Applying the Dirichlet criterion, these two bounds 
\eqref{iq-uniform-bdd-for-quadratic-variation-in-entropy-inequality-tilde-U} 
and \eqref{iq-uniform-bdd-for-quadratic-variation-in-entropy-inequality-tilde-U-varepsilon}, 
combined with the fact that $|a_k|$ is monotone and $|a_k|\to 0$ as $k\to \infty$, 
imply the uniform convergence of the sums
  $
  \sum_k |a_k|^2 \tilde \E \big[\int_0^T \int_{\mathbb{R}}  \big|\partial_{\tilde m}^2 \eta(\tilde U) \big| 
 \zeta_k^2(\tilde U) \,\varphi\,\d x \d t\big]
  $
and $\sum_k |a_k|^2 \tilde \E \big[\int_0^T \int_{\mathbb{R}} \big|\partial_{\tilde m^{\varepsilon}}^2
 \eta(\tilde U^{\varepsilon}) \big| (\zeta_k^{\varepsilon})^2(\tilde U^{\varepsilon}) \, \varphi\,\d x \d t\big]$. 
 Therefore, for any $\delta >0$, there exists $N_{\delta}>0$ 
such that, for any $N\ge N_{\delta}$,
\begin{equation}\label{iq-quadratic-variation-remaining-term-small}
\begin{aligned}
&\sum_{k=N}^{\infty} |a_k|^2 \tilde \E \big[\int_0^T \int_{\mathbb{R}}  \big|\partial_{\tilde m}^2 \eta(\tilde U) \big|
 \zeta_k^2(\tilde U) \,\varphi\,\d x \d t\big]\le \frac{\delta}{2}\\
&\sum_{k=N}^{\infty} |a_k|^2 \tilde \E \big[\int_0^T 
\int_{\mathbb{R}} \big|\partial_{\tilde m^{\varepsilon}}^2 \eta(\tilde U^{\varepsilon}) \big| 
(\zeta_k^{\varepsilon})^2(\tilde U^{\varepsilon}) \, \varphi\,\d x \d t\big]\,  \le \frac{\delta}{2}.
\end{aligned}
\end{equation}
Note that, in the definition of $\zeta_k^{\varepsilon}$, 
we use the cut-off such that $\zeta_k^{\varepsilon}=0$ when $k> \lfloor \varepsilon^{-1}\rfloor$; 
in fact, without this cut-off, the above bounds are also valid. 
For simplicity,
we do not distinguish between them. 
For $N_0=N_{\delta}+1$, we choose $\varepsilon_0$ such that, for any 
$\varepsilon \le \varepsilon_0$, 
$\lfloor \varepsilon^{-1} \rfloor \ge N_0$. 
Then 
$$
\begin{aligned}
&\tilde \E \big[\Big|\int_0^T \int_{\mathbb{R}} \Big(\partial_{\tilde m^{\varepsilon}}^2 \eta(\tilde U^{\varepsilon}) \sum_{k=1}^{\lfloor \varepsilon^{-1} \rfloor} a_k^2(\zeta_k^{\varepsilon})^2(\tilde U^{\varepsilon})
\,-\,\partial_{\tilde m}^2 \eta(\tilde U) \sum_k a_k^2\,\zeta_k^2(\tilde U)\Big)
\varphi \,\d x \d t\Big| \big]\\
&\le \tilde \E \big[\Big|\int_0^T \int_{\mathbb{R}} \Big(\partial_{\tilde m^{\varepsilon}}^2 \eta(\tilde U^{\varepsilon}) \sum_{k=1}^{N_0} a_k^2 (\zeta_k^{\varepsilon})^2(\tilde U^{\varepsilon})\,  -\, \partial_{\tilde m}^2 \eta(\tilde U) \sum_{k=1}^{N_0} a_k^2 \,\zeta_k^2(\tilde U)\Big)\varphi \,\d x \d t\Big| \big]\\
  &\quad\,+\,\tilde \E \big[\Big|\int_0^T \int_{\mathbb{R}} \Big(\partial_{\tilde m^{\varepsilon}}^2 \eta(\tilde U^{\varepsilon}) \sum_{k=N_0+1}^{\lfloor \varepsilon^{-1}\rfloor} a_k^2 (\zeta_k^{\varepsilon})^2(\tilde U^{\varepsilon})\, 
  -\,\partial_{\tilde m}^2 \eta(\tilde U) \sum_{k=N_0+1}^{\infty} a_k^2\,\zeta_k^2(\tilde U)\,\Big)\varphi \,\d x \d t\Big| \big].
\end{aligned}
$$
For the second term on the right-hand side, \eqref{iq-quadratic-variation-remaining-term-small} implies that it is less than $\delta$. 
For the first term on the right-hand side, 
using the $\tilde{\mathbb{P}}$-{\it a.s.} almost everywhere 
convergence $\tilde U^{\varepsilon} \to \tilde U$, it can be directly verified 
that
$$
\Big(\partial_{\tilde m^{\varepsilon}}^2 \eta(\tilde U^{\varepsilon}) 
\sum_{k=1}^{N_0} a_k^2 (\zeta_k^{\varepsilon})^2(\tilde U^{\varepsilon})\, 
-\,\partial_{\tilde m}^2 \eta(\tilde U) \sum_{k=1}^{N_0} a_k^2 \,\zeta_k^2(\tilde U)\Big)\varphi\, \to 0.
$$
Then, again by Lemma \ref{lem-growth-of-entropy-and-derivatives-general-pressure-law}, this convergence, combined with the equi-integrability provided by \eqref{iq-higher-integrability-on-new-probability-space}, enables us to 
apply the Vitali convergence theorem to deduce that
 $$
  \tilde \E \big[\Big|\int_0^T \int_{\mathbb{R}} \Big(\partial_{\tilde m^{\varepsilon}}^2 \eta(\tilde U^{\varepsilon}) \sum_{k=1}^{N_0} a_k^2 (\zeta_k^{\varepsilon})^2(\tilde U^{\varepsilon})\,  -\, \partial_{\tilde m}^2 \eta(\tilde U) \sum_{k=1}^{N_0} a_k^2 (\zeta_k)^2(\tilde U)\Big)\varphi \,\d x \d t\Big| \big] \to 0.
  $$
Thus combing together leads to 
   \begin{equation}\label{iq-convergence-of-quadratic-variation-term}
  \tilde \E \big[\Big|\int_0^T \int_{\mathbb{R}} \Big(\partial_{\tilde m^{\varepsilon}}^2 \eta(\tilde U^{\varepsilon}) \sum_{k=1}^{\lfloor \varepsilon^{-1} \rfloor} a_k^2 (\zeta_k^{\varepsilon})^2(\tilde U^{\varepsilon})\, -
\,\partial_{\tilde m}^2 \eta(\tilde U) \sum_k a_k^2\,\zeta_k^2(\tilde U)\Big)\varphi 
\,\d x \d t\Big| \big] \to 0,
  \end{equation}
which implies that, up to a subsequence,
  $$
  \int_0^T \int_{\mathbb{R}} \Big(\partial_{\tilde m^{\varepsilon}}^2 \eta(\tilde U^{\varepsilon}) \sum_{k=1}^{\lfloor \varepsilon^{-1} \rfloor} a_k^2 (\zeta_k^{\varepsilon})^2(\tilde U^{\varepsilon})\,
  -
\,\partial_{\tilde m}^2 \eta(\tilde U) \sum_k a_k^2\,\zeta_k^2(\tilde U)\Big)\varphi
\,\d x \d t \to 0 \qquad \tilde{\mathbb{P}}\text{-{\it a.s.}}. 
  $$

For the term: $\varepsilon \int_0^T \int_{\mathbb{R}} \eta(\tilde U^{\varepsilon})\partial_x^2 \varphi \,\d x \d t$,  
by Lemma \ref{lem-growth-of-entropy-and-derivatives-general-pressure-law} 
and \eqref{iq-higher-integrability-on-new-probability-space}, we have
$$
\varepsilon \,\tilde \E \big[\Big|\int_0^T \int_{\mathbb{R}} \eta(\tilde U^{\varepsilon})\partial_x^2 \varphi \,\d x \d t \Big| \big] \le \varepsilon {\hat C}
\to 0,
$$
which implies that
$$
\varepsilon \int_0^T \int_{\mathbb{R}} \eta(\tilde U^{\varepsilon})\partial_x^2 \varphi \,\d x \d t \to 0 \qquad \tilde{\mathbb{P}}\text{-{\it a.s.}}.
$$

We now look at the stochastic integral. 
By Proposition \ref{prop-apply-jakubowski-skorokhod-representation-to-take-limit}, 
we know that $\tilde W^{\varepsilon} \to \tilde W$ in $C([0,T];\mathfrak{A}_0)$. 
Therefore, by \cite[Lemma 2.6.6]{BFHbook18}, to prove the convergence of the stochastic integral, it remains to show that
  \begin{equation}\label{iq-convergence-of-coefficience-of-stochastic-integral}
  \int_{\mathbb{R}} \partial_{\tilde m^{\varepsilon}} \eta (\tilde U^{\varepsilon}) \Phi^{\varepsilon}( \tilde U^{\varepsilon}) \,\varphi\,\d x 
  \to \int_{\mathbb{R}} \partial_{\tilde m} \eta (\tilde U) \Phi (\tilde U) \,\varphi\,\d x\qquad \ 
  \text{in $L^2(0,T;L_2(\mathfrak{A};\mathbb{R})) \,\,\,\,\tilde{\mathbb{P}}\text{-{\it a.s.}}$}.
  \end{equation}
In fact, we have 
$$
\begin{aligned}
  \tilde \E \big[\int_0^T \sum_k \Big|\int_{\mathbb{R}}\partial_{\tilde m} \eta(\tilde U) a_k\zeta_k(\tilde U)
    \,\varphi\,\d x\Big|^2 \,\d t\big]
  \le  \sum_k |a_k|^2 \tilde \E \big[\int_0^T  \Big|\int_{\mathbb{R}}\partial_{\tilde m} \eta(\tilde U) \zeta_k(\tilde U)
    \,\varphi\,\d x\Big|^2 \,\d t\big].
\end{aligned}
$$
Using 
conditions \eqref{iq-noise-coefficience-growth-condition-on-R-compact-supp}--\eqref{iq-noise-coefficience-growth-condition-on-R}, 
Lemma \ref{lem-growth-of-entropy-and-derivatives-general-pressure-law},
and $\partial_{\tilde m} \eta(\tilde U)=\frac{1}{\tilde \rho}\partial_{\tilde u} \eta(\tilde \rho, \tilde u)$, we have
\begin{equation}\label{iq-estimate-for-partial-m-eta-noise-in-entropy-inequality-general-pressure}
\begin{aligned}
  &\sum_k \tilde \E \big[\int_0^T  \Big|\int_{\mathbb{R}}\partial_{\tilde m} \eta(\tilde U) \zeta_k(\tilde U)
    \,\varphi\,\d x\Big|^2 \,\d t\big]\\
  &\le  \tilde \E \big[\int_0^T \int_{{\rm supp}_x(\varphi)}
  \big(\sqrt{\tilde \rho}\partial_{\tilde m} \eta(\tilde U)\varphi\big)^2 \,\d x
  \int_{{\rm supp}_x(\varphi)}\sum_k \big(\frac{1}{\sqrt{\tilde \rho}}\zeta_k(\tilde U)\big)^2 \,\d x\d t\big]\\
  &\le C(\varphi,B_0,\gamma,\rho_{\infty})\tilde \E \big[\sup_{t\in[0,T]} \Big(\int_{{\rm supp}_x (\varphi)} 
  \big(1+\frac{\tilde m^2}{\tilde \rho}+ e^*(\tilde \rho,\rho_{\infty})\big) \,\d x \Big)^2 \big]\\
  &\quad+ \tilde \E\big[ \Big(\int_0^T \int_{{\rm supp}_x(\varphi)} 
  \big(1+ \frac{\tilde m^3}{\tilde \rho^2}+ \tilde \rho P(\tilde \rho) \big) \d x\, \d t \Big)^2 \big]\\
  &\le  C(T,E_{0,2})+ {\hat C}(2),
\end{aligned}
\end{equation}
where we have used similar arguments to those 
for \eqref{iq-estimate-for-noise-term-1/rho--in-energy-estimate}
in the second inequality
and used 
Proposition \ref{prop-take-limit-obtain-relative-energy-estiate-for-Euler} and \eqref{iq-higher-integrability-on-new-probability-space}
in the last inequality. 
Similarly, 
$$
\begin{aligned}
  \tilde \E \big[\int_0^T \sum_k \Big|\int_{\mathbb{R}} 
  \partial_{\tilde m^{\varepsilon}} \eta(\tilde U^{\varepsilon}) a_k \zeta_k^{\varepsilon}(\tilde U^{\varepsilon}) \,\varphi\,\d x\Big|^2 \,\d t\big]\le \sum_k |a_k|^2 \tilde \E \big[\int_0^T \Big|\int_{\mathbb{R}} 
  \partial_{\tilde m^{\varepsilon}} \eta(\tilde U^{\varepsilon}) \zeta_k^{\varepsilon}(\tilde U^{\varepsilon}) \,\varphi\,\d x\Big|^2 \,\d t\big].
\end{aligned}
$$
Using
\eqref{iq-noise-coefficience-growth-conditions-for-parabolic-approximation}, \eqref{iq-higher-integrability-on-new-probability-space}, 
and \eqref{iq-energy-bound-for-tilde-rho-m-varepsilon}, we have
\begin{equation}\label{iq-estimate-for-partial-m-eta-noise-in-entropy-inequality-tilde-U-varepsilon-general-pressure}
\begin{aligned}
  \sum_k \tilde \E \big[\int_0^T  \Big|\int_{\mathbb{R}} 
  \partial_{\tilde m^{\varepsilon}} \eta(\tilde U^{\varepsilon}) \zeta_k^{\varepsilon}(\tilde U^{\varepsilon}) \,\varphi\,\d x\Big|^2 \,\d t\big]
  \le  C(T,E_{0,2})+ {\hat C}(2).
\end{aligned}
\end{equation}
By the same argument as that for 
\eqref{iq-quadratic-variation-remaining-term-small}, 
these two bounds 
\eqref{iq-estimate-for-partial-m-eta-noise-in-entropy-inequality-general-pressure}--\eqref{iq-estimate-for-partial-m-eta-noise-in-entropy-inequality-tilde-U-varepsilon-general-pressure} imply that, for any $\delta >0$, there exists $N_{\delta}>0$ such that,
for any $N\ge N_{\delta}$,
$$
\sum_{k=N}^{\infty} |a_k|^2 \tilde \E \big[\int_0^T \Big|\int_{\mathbb{R}} \partial_{\tilde m} \eta(\tilde U) \zeta_k(\tilde U) \,\varphi\,\d x\Big|^2\,\d t\big]
+\sum_{k=N}^{\infty} |a_k|^2 \tilde \E \big[\int_0^T \Big|\int_{\mathbb{R}} 
\partial_{\tilde m^{\varepsilon}} \eta(\tilde U^{\varepsilon}) 
\zeta_k^{\varepsilon}(\tilde U^{\varepsilon}) \,\varphi\,\d x\Big|^2\,\d t\big]  \le \delta.
$$
As above, for $N_0=N_{\delta}+1$, we choose $\varepsilon_0$ 
such that $\lfloor \varepsilon^{-1} \rfloor \ge N_0$ 
for any $\varepsilon \le \varepsilon_0$. 
Note that
$$\begin{aligned}
  &\tilde \E \big[\int_0^T  \Big\|\int_{\mathbb{R}} \partial_{\tilde m^{\varepsilon}} \eta(\tilde U^{\varepsilon}) \Phi^{\varepsilon}(\tilde U^{\varepsilon}) \,\varphi\,\d x- \int_{\mathbb{R}} 
  \partial_{\tilde m} \eta(\tilde U) \Phi(\tilde U)\,\varphi\,\d x \Big\|_{L_2(\mathfrak{A};\mathbb{R})}^2 \,\d t \big] \\
  &= \tilde \E \big[\int_0^T \sum_k \Big|\int_{\mathbb{R}} \partial_{\tilde m^{\varepsilon}} \eta(\tilde U^{\varepsilon}) a_k \zeta_k^{\varepsilon}(\tilde U^{\varepsilon}) \,\varphi\,\d x
  - \int_{\mathbb{R}} \partial_{\tilde m} \eta(\tilde U) a_k \zeta_k(\tilde U)\,\varphi \,\d x \Big|^2\,\d t\big]  \\
  &\le \tilde \E \big[\int_0^T \sum_{k=1}^{N_0} \Big|\int_{\mathbb{R}} 
   \partial_{\tilde m^{\varepsilon}} \eta(\tilde U^{\varepsilon}) a_k \zeta_k^{\varepsilon}(\tilde U^{\varepsilon})\,\varphi\,\d x
   - \int_{\mathbb{R}} \partial_{\tilde m} \eta(\tilde U) a_k \zeta_k(\tilde U) \,\varphi\,\d x \Big|^2 \,\d t\big]\\
  &\quad\,+ 2\tilde \E\big[ \int_0^T \Big(\sum_{k=N_0+1}^{\lfloor \varepsilon^{-1}\rfloor} \Big|\int_{\mathbb{R}} 
  \partial_{\tilde m^{\varepsilon}} \eta(\tilde U^{\varepsilon}) a_k \zeta_k^{\varepsilon}(\tilde U^{\varepsilon}) 
  \,\varphi\,\d x\Big|^2\,
 +\sum_{k=N_0+1}^{\infty} \Big| \int_{\mathbb{R}} 
  \partial_{\tilde m} \eta(\tilde U) a_k\, \zeta_k(\tilde U) \,\varphi\,\d x \Big|^2\Big)\,\d t\big].
  \end{aligned}
  $$
Therefore, similar arguments to the justification of \eqref{iq-convergence-of-quadratic-variation-term} yield that
  \begin{equation}\label{eq-convergence-of-stochastic-integral-hilbert-schmit-norm}
  \tilde \E\big[ \Big|\int_0^T  \Big\|\int_{\mathbb{R}} \partial_{\tilde m^{\varepsilon}} \eta(\tilde U^{\varepsilon}) \Phi^{\varepsilon}(\tilde U^{\varepsilon})\,\varphi \,\d x
  - \int_{\mathbb{R}} \partial_{\tilde m} \eta(\tilde U) \Phi(\tilde U)\,\varphi \,\d x \Big\|_{L_2(\mathfrak{A};\mathbb{R})}^2\, 
  \d t \Big|\big] \to 0,
  \end{equation}
which implies \eqref{iq-convergence-of-coefficience-of-stochastic-integral}. 
Then, by \cite[Lemma 2.6.6]{BFHbook18}, we have
  $$
  \int_0^{\cdot} \int_{\mathbb{R}} 
  \partial_{\tilde m^{\varepsilon}} \eta (\tilde U^{\varepsilon}) \Phi^{\varepsilon}( \tilde U^{\varepsilon})\, 
  \varphi\,\d x \d \tilde W^{\varepsilon}(\tau) 
  \to \int_0^{\cdot} \int_{\mathbb{R}} \partial_{\tilde m} \eta (\tilde U) \Phi( \tilde U) \,\varphi\,\d x \d \tilde W(\tau)
  $$
in $L^2(0,T)$ in probability. 
Thus, up to a subsequence, for {\it a.e.} $t\in [0,T]$,
  \begin{equation}\label{iq-convergence-of-stochastic-integral-in-entropy-inequality-general-pressure}
  \int_0^t \int_{\mathbb{R}} \partial_{\tilde m^{\varepsilon}} \eta (\tilde U^{\varepsilon}) \Phi^{\varepsilon}( \tilde U^{\varepsilon}) \,\varphi\,\d x \d \tilde W^{\varepsilon}(\tau) 
  \to \int_0^t \int_{\mathbb{R}} \partial_{\tilde m} \eta (\tilde U) \Phi( \tilde U) \,\varphi\,\d x \d \tilde W(\tau)\qquad \tilde{\mathbb{P}}\text{-{\it a.s.}},
  \end{equation}
which implies 
that
  $$
  \int_0^T \int_{\mathbb{R}} \partial_{\tilde m^{\varepsilon}} \eta (\tilde U^{\varepsilon}) \Phi^{\varepsilon}( \tilde U^{\varepsilon})\, \varphi\,\d x \d \tilde W^{\varepsilon}(\tau) 
  \to \int_0^T \int_{\mathbb{R}} \partial_{\tilde m} \eta (\tilde U) \Phi( \tilde U) \,\varphi\,\d x \d \tilde W(\tau)
  \qquad \tilde{\mathbb{P}}\text{-{\it a.s.}}.
  $$

Combining the above results yields that 
  $$
  \begin{aligned}
  &\int_0^T \int_{\mathbb{R}} \big(\eta(\tilde  U)\varphi_t +q(\tilde U) \varphi_x \big) \,\d x\d t
  + \int_0^T \int_{\mathbb{R}} \partial_{\tilde m} \eta (\tilde U) \Phi( \tilde U) \,\varphi\,\d x \d \tilde W(\tau)\\
  &+\frac12 \int_0^T \int_{\mathbb{R}} 
  \partial_{\tilde m}^2 \eta(\tilde U) \sum_k a_k^2\,\zeta_k^2(\tilde U) \,\varphi\,\d x \d t
  \ge  0\qquad \text{$\tilde{\mathbb{P}}$-{\it a.s.}}.
  \end{aligned}
  $$
\end{proof}

For the general pressure law case, we now show that $(\tilde \rho, \tilde m)$ 
is a weak solution to \eqref{eq-stochastic-Euler-system}.
\begin{proposition}\label{prop-weak-solution-for-general-pressure-law}
For any $\varphi \in C_{\rm c}^{\infty}((0,T)\times\mathbb{R})$, 
the mass and momentum equations hold $\tilde{\mathbb{P}}$-{\it a.s.}{\rm :}
$$
\begin{aligned}
&\int_0^T \int_{\mathbb{R}} \tilde \rho
\partial_t\varphi \, \d x \d t
+\int_0^T \int_{\mathbb{R}} \tilde m \partial_x \varphi \, \d x \d t=0,\\
&\int_0^T \int_{\mathbb{R}}  \tilde m
\partial_t \varphi\, \d x \d t
+ \int_0^T \int_{\mathbb{R}} \big( \frac{\tilde m^2}{\tilde \rho}+P(\tilde \rho) \big) \partial_x \varphi
\,\d x \d t
+\int_0^T \int_{\mathbb{R}} \Phi(\tilde U)\,\varphi \, \d x \d \tilde W(t)= 0.
\end{aligned}
$$
\end{proposition}
The proof is similar to that of 
Proposition \ref{prop-entropy-inequality-general-pressure-law} and even simpler, so we omit the details.

\smallskip
\subsection{Higher-Order Energy Estimates and More Convex Entropy Pairs for the Entropy Inequality}

We now establish a higher-order energy estimate, 
under which the entropy inequality holds for more entropy pairs.

\begin{proposition}\label{prop-rho|u|4-general-pressure-law}
Assume that $U^{\varepsilon}_0$ satisfy
  $$\begin{aligned}
  \E \big[\Big(\int_{\mathbb{R}} \eta_{\diamond}^*(\rho^{\varepsilon}_0,m^{\varepsilon}_0)\, \d x \Big)^p\big] \le C\, \E \big[\Big(\int_{\mathbb{R}} \eta_{\diamond}^*(\rho_0,m_0)\,  \d x \Big)^p\big]<\infty \qquad \mbox{for $1\le p< \infty$},
  \end{aligned}$$
and let $(\rho^{\varepsilon},u^{\varepsilon})$ be the strong solution of the parabolic approximation \eqref{eq-parabolic-approximation}, which is predictable and satisfies $(\rho^{\varepsilon}-\rho_{\infty},\,u)\in C([0,T];H^3(\mathbb{R}))$ $\P$-{\it a.s.}. 
Then the following higher-order energy estimates hold{\rm :} 
For $1\le p < \infty$,
    \begin{equation}\label{iq-energy-estimate-stopping-time}
    \begin{aligned}
    \mathbb{E} \big[\Big(\sup_{r\in [0,t]}\int_{\mathbb{R}} \eta_{\diamond}^*(\rho^{\varepsilon},m^{\varepsilon})(r) \,\d x  \Big)^p\big]
    \le C(p,T,E_{0,p})
    +C\, \E \big[\Big(\int_{\mathbb{R}} \eta_{\diamond}^*(\rho_0,m_0)\,  \d x \Big)^p\big] \quad \mbox{for any $t\in[0,T]$},
    \end{aligned}
    \end{equation}
where $C(p,T,E_{0,p})$ is independent of $\varepsilon$.
\end{proposition}

\begin{proof}
Theorem \ref{thm-wellposedness-parabolic-approximation-R}(iii) 
guarantees that, for any fixed $x$, the equations in \eqref{eq-parabolic-approximation} hold 
as a finite-dimensional stochastic differential system. Thus,
applying the It\^o formula to the relative higher-order energy $\eta^*_{\diamond}$ and 
integrating by parts yield that, for $i=1$ or $2$,
$$
\begin{aligned}
&\int_{\mathbb{R}} \eta_{\diamond}^* (U^{\varepsilon})(\tau)\,\d x +\varepsilon \int_0^{\tau} \int_{\mathbb{R}} (\partial_x U^{\varepsilon})^{\top}\nabla^2\eta_{\diamond}^{*}(U^{\varepsilon}) \partial_x U^{\varepsilon} \,\d x\d r \\
&= \int_{\mathbb{R}} \eta_{\diamond}^* (U^{\varepsilon}_0) \,\d x+ \int_0^{\tau} \int_{\mathbb{R}} \nabla \eta_{\diamond}^*(U^{\varepsilon})\Psi^{\varepsilon}(U^{\varepsilon}) \,\d x\d W(r) + \int_0^{\tau} \int_{\mathbb{R}} \frac12 \partial_{m^{\varepsilon}}^2 \eta_{\diamond}^*(U^{\varepsilon}) \big(\mathfrak{G}_i^{\varepsilon} \big)^2 \,\d x \d r,
\end{aligned}
$$
where $\mathfrak{G}_i^{\varepsilon}, i=1,2$, are defined in the beginning of \S \ref{sec-parabolic-approximation}.

Since it is not clear 
whether $\eta_{\diamond}^{*}$ is convex with respect to all $U$ owing to the generality of the pressure, 
we need to estimate  $\varepsilon \int_0^{\tau} \int_{\mathbb{R}} (\partial_x U^{\varepsilon})^{\top}\nabla^2\eta_{\diamond}^{*}(U^{\varepsilon}) 
\partial_x U^{\varepsilon} \,\d x\d r$ carefully. 
Note that
$$
\nabla^2\eta_{\diamond}^{*}(U^{\varepsilon})=\begin{pmatrix}
  \frac{(m^{\varepsilon})^4}{(\rho^{\varepsilon})^5}+\big(\frac{e(\rho^{\varepsilon})}{\rho^{\varepsilon}}\big)^{\prime \prime}(m^{\varepsilon})^2 + g^{\prime \prime}(\rho^{\varepsilon}) 
  &\quad -\frac{(m^{\varepsilon})^3}{(\rho^{\varepsilon})^4}+2m^{\varepsilon} \big( \frac{e(\rho^{\varepsilon})}{\rho^{\varepsilon}}\big)^{\prime}\\[2mm]
  -\frac{(m^{\varepsilon})^3}{(\rho^{\varepsilon})^4}+2m^{\varepsilon} \big( \frac{e(\rho^{\varepsilon})}{\rho^{\varepsilon}}\big)^{\prime} 
  &\quad\frac{(m^{\varepsilon})^2}{(\rho^{\varepsilon})^3}+2\frac{e(\rho^{\varepsilon})}{\rho^{\varepsilon}}
  \end{pmatrix}=:\begin{pmatrix}
  A_{11} &\ A_{12}\\
  A_{21} &\ A_{22}
  \end{pmatrix}.
$$
We are going to prove first that $\eta_{\diamond}^{*}$ is convex, when $\rho^{\varepsilon}<\rho_{\star}$ 
or $\rho^{\varepsilon}>\rho^{\star}$.
When $\rho^{\varepsilon}<\rho_{\star}$, without loss of generality, 
we assume that $\kappa_1=1$ in \eqref{eq-general-pressure-law-1}. 
On one hand, a direct calculation yields that
  $$
  \begin{aligned}
  A_{11}&= \frac{(m^{\varepsilon})^4}{(\rho^{\varepsilon})^5}+\Big(\frac{e(\rho^{\varepsilon})}{\rho^{\varepsilon}}\Big)^{\prime \prime}(m^{\varepsilon})^2 + g^{\prime \prime}(\rho^{\varepsilon})\\
  &=\frac{(m^{\varepsilon})^4}{(\rho^{\varepsilon})^5}+\Big(\frac{(\gamma_1-2)(\gamma_1-3)}{\gamma_1-1}+f_1 \Big)(\rho^{\varepsilon})^{\gamma_1-4}(m^{\varepsilon})^2+\Big(\frac{2\gamma_1}{\gamma_1-1}+f_2\Big)(\rho^{\varepsilon})^{2\gamma_1-3}\\
  &=a_1 \frac{(m^{\varepsilon})^4}{(\rho^{\varepsilon})^5} + b_1 (\rho^{\varepsilon})^{\gamma_1-4}(m^{\varepsilon})^2 + c_1 (\rho^{\varepsilon})^{2\gamma_1-3},
  \end{aligned}
  $$
where $f_1$ and $f_2$ are sufficiently small when $\rho_{\star}$ is sufficiently small. Let
  $$
   \Delta:=b_1^2 -4 a_1 c_1= \Big(\frac{(\gamma_1-2)(\gamma_1-3)}{\gamma_1-1}+f_1 \Big)^2- 4\Big(\frac{2\gamma_1}{\gamma_1-1}+f_2\Big).
  $$
When $1<\gamma_1<\frac32$, we see that $b_1>0$ since $f_1$ is sufficiently small. 
Therefore, the roots of $A_{11}=0$ are all less than zero, 
so that $A_{11}>0$. 
When $\frac32 \le \gamma_1 <3$, we find that $\Delta <0$ since $f_1$ and $f_2$ are sufficiently small, so that  $A_{11}>0$. In conclusion, $A_{11}>0$.
On the other hand,
$$
\begin{aligned}
|\nabla^2\eta_{\diamond}^{*}(U^{\varepsilon})|
  &=A_{11}A_{22}-A_{12}A_{21}\\
  &=(m^{\varepsilon})^4(\rho^{\varepsilon})^{\gamma_1-7}(\gamma_1 + \tilde f_1)+ (m^{\varepsilon})^2 (\rho^{\varepsilon})^{2\gamma_1-6}\Big( \frac{4}{\gamma_1-1} + \tilde f_2 \Big) + (\rho^{\varepsilon})^{3\gamma_1 -5} \Big( \frac{4\gamma_1}{(\gamma_1-1)^2} + \tilde f_3 \Big),
  \end{aligned}$$
where $\tilde f_i,\ i=1,2,3$, are all sufficiently small 
when $\rho_{\star}$ is sufficiently small. 
Then it is direct to obtain that $|\nabla^2\eta_{\diamond}^{*}(U^{\varepsilon})| >0$. 
Thus,  $\eta_{\diamond}^{*}(U^{\varepsilon})$ is convex, when $\rho^{\varepsilon}<\rho_{\star}$.
Similarly, we obtain that 
$\eta_{\diamond}^{*}(U^{\varepsilon})$ is convex, when $\rho^{\varepsilon}>\rho^{\star}$. 
Therefore, we conclude that
$$
 \varepsilon \iint_{([0,\tau]\times \mathbb{R})\cap\{\rho^{\varepsilon}<\rho_{\star}\}\cap \{\rho^{\varepsilon}>\rho^{\star}\}}
 (\partial_x U^{\varepsilon})^{\top}\nabla^2\eta_{\diamond}^{*}(U^{\varepsilon}) \partial_x U^{\varepsilon} \,\d x\d r \ge 0.
$$

When $\rho_{\star}\le \rho^{\varepsilon} \le \rho^{\star}$,
$$
\begin{aligned}
&(\partial_x U^{\varepsilon})^{\top}\nabla^2\eta_{\diamond}^{*}(U^{\varepsilon})
\partial_x U^{\varepsilon}\\
&=\frac{(m^{\varepsilon})^2}{(\rho^{\varepsilon})^3}
\big( \frac{m^{\varepsilon}}{\rho^{\varepsilon}}\partial_x \rho^{\varepsilon} -\partial_x m^{\varepsilon} \big)^2
+ 2\frac{e(\rho^{\varepsilon})}{\rho^{\varepsilon}}
\big( \frac{m^{\varepsilon}}{\rho^{\varepsilon}}\partial_x \rho^{\varepsilon}-\partial_x m^{\varepsilon} \big)^2\\ 
&\quad\, + \big(\frac{e^{\prime \prime}(\rho^{\varepsilon})}{\rho^{\varepsilon}}- 2\frac{e^{\prime}(\rho^{\varepsilon})}{(\rho^{\varepsilon})^2}\big) (\partial_x \rho^{\varepsilon})^2 (m^{\varepsilon})^2
  + 4 m^{\varepsilon} \frac{e^{\prime}(\rho^{\varepsilon})}{\rho^{\varepsilon}} \partial_x \rho^{\varepsilon} \partial_x m^{\varepsilon}\\
  &=\rho^{\varepsilon}(u^{\varepsilon})^2 \big(\partial_x u^{\varepsilon}\big)^2 + 2\rho^{\varepsilon} e(\rho^{\varepsilon}) \big(\partial_x u^{\varepsilon}\big)^2 +  \big(\rho^{\varepsilon} e(\rho^{\varepsilon})\big)^{\prime \prime} (\partial_x \rho^{\varepsilon})^2 (u^{\varepsilon})^2
  + 4 \rho^{\varepsilon}u^{\varepsilon} e^{\prime}(\rho^{\varepsilon}) \partial_x \rho^{\varepsilon} \partial_x u^{\varepsilon},
  \end{aligned}$$
where
  $$\begin{aligned}
  |4 \rho^{\varepsilon}u^{\varepsilon} e^{\prime}(\rho^{\varepsilon}) \partial_x \rho^{\varepsilon} \partial_x u^{\varepsilon}|
  &\le \beta \rho^{\varepsilon} (u^{\varepsilon})^2 (\partial_x u^{\varepsilon})^2 + C_{\beta} \rho^{\varepsilon} \big(e^{\prime}(\rho^{\varepsilon})\big)^2 (\partial_x \rho^{\varepsilon})^2\\
  &\le \beta \rho^{\varepsilon} (u^{\varepsilon})^2 (\partial_x u^{\varepsilon})^2 + C_{\beta}C(\rho_{\star},\rho^{\star}) \big(\rho^{\varepsilon}e(\rho^{\varepsilon})\big)^{\prime \prime} (\partial_x \rho^{\varepsilon})^2.
  \end{aligned}$$
Thus, when $\rho_{\star} \le \rho^{\varepsilon} \le \rho^{\star}$, we have
$$
\begin{aligned}
(\partial_x U^{\varepsilon})^{\top}
 \nabla^2\eta_{\diamond}^{*}(U^{\varepsilon})\partial_x U^{\varepsilon}
&\ge (1-\delta)\rho^{\varepsilon}(u^{\varepsilon})^2 \big(\partial_x u^{\varepsilon}\big)^2 + 2\rho^{\varepsilon} e(\rho^{\varepsilon}) \big(\partial_x u^{\varepsilon}\big)^2 +  \big(\rho^{\varepsilon} e(\rho^{\varepsilon})\big)^{\prime \prime} (\partial_x \rho^{\varepsilon})^2 (u^{\varepsilon})^2\\
&\quad - C_{\delta}C(\rho_{\star},\rho^{\star}) \big(\rho^{\varepsilon}e(\rho^{\varepsilon})\big)^{\prime \prime} (\partial_x \rho^{\varepsilon})^2.
\end{aligned}
$$
Therefore, we have
\begin{align*}
&\varepsilon \iint_{([0,\tau]\times \mathbb{R})\cap\{\rho_{\star}\le \rho^{\varepsilon} \le \rho^{\star}\}} 
 (\partial_x U^{\varepsilon})^{\top}\nabla^2\eta_{\diamond}^{*}(U^{\varepsilon}) \partial_x U^{\varepsilon} \,\d x\d r\\
&\ge \varepsilon\iint_{\{\rho_{\star}\le \rho^{\varepsilon} \le \rho^{\star}\}}
\big((1-\delta)\rho^{\varepsilon}(u^{\varepsilon})^2 \big(\partial_x u^{\varepsilon}\big)^2 + 2\rho^{\varepsilon} e(\rho^{\varepsilon}) \big(\partial_x u^{\varepsilon}\big)^2 + \big(\rho^{\varepsilon} e(\rho^{\varepsilon})\big)^{\prime \prime} (\partial_x \rho^{\varepsilon})^2 (u^{\varepsilon})^2\big)\d x\d r\\
&\quad\,\, - C_{\delta}C(\rho_{\star},\rho^{\star})\iint_{\{\rho_{\star}\le \rho^{\varepsilon} \le \rho^{\star}\}}\varepsilon\big(\rho^{\varepsilon}e(\rho^{\varepsilon})\big)^{\prime \prime} (\partial_x \rho^{\varepsilon})^2\d x\d r.
\end{align*}
It 
can directly be proved that the stochastic integral
  $
  G_{\diamond}^{\varepsilon}(\tau):= \int_0^{\tau} \int_{\mathbb{R}} \nabla \eta_{\diamond}^*(U^{\varepsilon})\Psi^{\varepsilon}(U^{\varepsilon}) \,\d x\d W(r)
  $
is a local martingale. Then, by the Burkholder-Davis-Gundy inequality, we obtain
$$
\begin{aligned}
  \E\big[ \sup_{\tau\in [0,t]}|G_{\diamond}^{\varepsilon}(\tau)|^p\big]
  \le C\, \E \big[\Big( \int_0^t \sum_k \Big( \int_{{\rm supp}_x (\zeta_k^{\varepsilon})} 
  \big(\frac13 (u^{\varepsilon})^3 + 2 e(\rho^{\varepsilon})u^{\varepsilon} \big) \zeta_k^{\varepsilon} \,\d x \Big)^2 \d r \Big)^{\frac{p}{2}}\big].
\end{aligned}
$$

Denote $K_1^\varepsilon=\{0\le \rho^\varepsilon<\rho_{\star}\}$ and $K_2^\varepsilon=\{\rho^\varepsilon>\rho^{\star}\}$. 
Notice that,
for $i=1,2$,
$$
\begin{aligned}
  &\E \big[\Big(  \sum_k \Big( \iint_{([0,t]\times {\rm supp}_x  (\zeta_k^{\varepsilon}) )\cap\, K_i^{\varepsilon}}
  \big(\frac13 (u^{\varepsilon})^3 + 2 e(\rho^{\varepsilon})u^{\varepsilon} \Big) \zeta_k^{\varepsilon} \,\d x \big)^2 \d r \Big)^{\frac{p}{2}}\big]\\
  &\le \delta\, \E\big[ \Big(\sup_{r\in[0,t]} \int_{\mathbb{R}}
    \big(\rho^{\varepsilon}(u^{\varepsilon})^4 + \frac{(e(\rho^{\varepsilon}))^2}{(\rho^{\varepsilon})^{\gamma_i-2}}(u^{\varepsilon})^2 \big)  \,\d x\Big)^p \big]\\
  &\quad + C_{\delta}\E\big[ \Big( \sum_k \iint_{([0,t]\times {\rm supp}_x  (\zeta_k^{\varepsilon}) )\cap\,K_i^{\varepsilon}}
  \big(\frac{(u^{\varepsilon})^2}{\rho^{\varepsilon}}+ (\rho^{\varepsilon})^{\gamma_i-2} \big)(\zeta_k^{\varepsilon})^2 \,\d x \d r \Big)^p \big].
  \end{aligned}
$$

For the second term on the right-hand side of the above inequality, we only show the estimate for Case 3 (see Hypotheses \ref{hypo-for-far-field-of-initial-data-and-noise-coefficient}--\ref{hypo-for-far-field-of-initial-data-and-noise-coefficient-parabolic-approximation}). The estimates for the other cases are similar and easier. 
For Case 3, by assumption \eqref{iq-noise-coefficience-growth-conditions-for-parabolic-approximation},
we obtain that, for $\mathfrak{G}_1^{\varepsilon}$,
  $$
  \begin{aligned}
  &\E \big[\Big( \sum_k \iint_{([0,t]\times {\rm supp}_x  (\zeta_k^{\varepsilon}) )\cap\, K_i^{\varepsilon}}
    \Big(\frac{(u^{\varepsilon})^2}{\rho^{\varepsilon}}+ (\rho^{\varepsilon})^{\gamma_i-2} \Big)(\zeta_k^{\varepsilon})^2 \,\d x \d r \Big)^p\big]\\
  &\le  B_0^{2p}(1+\varepsilon^2)^p t^p E_{0, p} +C(p,t)\\
  &\quad+C \E \big[\Big( \iint_{([0,t]\times {\rm supp}_x  (\zeta_k^{\varepsilon}) )\cap\, K_i^{\varepsilon}}
  B_0^2 \big(\rho^{\varepsilon} (u^{\varepsilon})^2 + (\rho^{\varepsilon})^{\gamma_i}  \big)\big( (u^{\varepsilon})^2 +e(\rho^{\varepsilon}) \big) \,\d x \d r \Big)^p\big],
  \end{aligned}
  $$
where we have used \eqref{iq-lower-upper-bound-for-general-internal-energy-1}--\eqref{iq-lower-upper-bound-for-general-internal-energy-2} in the 
last inequality. 
Since, by similar arguments 
to \eqref{iq-relative-internal-energy-control-general-pressure-law-1}--\eqref{iq-density-control-by-relative-internal-energy-general-pressure-law-2},
$$
\begin{aligned}
&\rho^{2\gamma_1-1}\le \bar C_{\gamma_1} \big(e_{\diamond}^*(\rho, \rho_{\infty}) + \rho_{\infty}^{2\gamma_1-1}\big)\qquad \text{ when $\rho\in \left(0,\rho_{\star}\right]$},\\
&\rho^{2\gamma_2-1}\le \bar C_{\gamma_2} \big(e_{\diamond}^*(\rho, \rho_{\infty}) + \rho_{\infty}^{2\gamma_2-1}\big)\qquad \text{ when $\rho\in \left[\rho^{\star},\infty \right)$},
\end{aligned}
$$
it follows from \eqref{iq-lower-upper-bound-for-general-internal-energy-1}--\eqref{iq-lower-upper-bound-for-general-internal-energy-2} that
  $$
  \begin{aligned}
  &\E \big[\Big(\iint_{([0,t]\times {\rm supp}_x  (\zeta_k^{\varepsilon}) )\cap\, K_i^{\varepsilon}} 
    B_0^2 \big(\rho^{\varepsilon} (u^{\varepsilon})^2 + (\rho^{\varepsilon})^{\gamma_i}  \big)\big( (u^{\varepsilon})^2 +e(\rho^{\varepsilon}) \big) \,\d x \d r \Big)^p\big]\\
  &\le C \E \big[\Big( 
  \iint_{([0,t]\times {\rm supp}_x  (\zeta_k^{\varepsilon}) )\cap\, K_i^{\varepsilon}}
  \big(\rho^{\varepsilon} (u^{\varepsilon})^4+ \rho^{\varepsilon} e(\rho^{\varepsilon}) (u^{\varepsilon})^2 + e_{\diamond}^*(\rho, \rho_{\infty})  \big) \,\d x \d r \Big)^p\big]+ \big(t|{\rm supp}_x(\zeta_k^{\varepsilon})|\rho_{\infty}^{2\gamma_i-1} \big)^p.
  \end{aligned}
  $$
For the term $\big(t|{\rm supp}_x(\zeta_k^{\varepsilon})|\rho_{\infty}^{2\gamma_i-1} \big)^p$, in Case 3, ${\rm supp}_x (\zeta_k^{\varepsilon})\subset \mathbb{K}$,
so this term is uniformly bounded with respect to $\varepsilon$. 
Case 2 can be treated similarly. For Case 1, ${\rm supp}_x (\zeta_k^{\varepsilon})=\Lambda^{\varepsilon}$, 
thus $\big(t|{\rm supp}_x (\zeta_k^{\varepsilon})|\rho_{\infty}^{2\gamma_i-1} \big)^p$ is uniformly bounded with respect to $\varepsilon$ by choosing $\alpha_0$ in $\rho_{\infty}(\varepsilon)=\varepsilon^{\alpha_0}$ such that $\alpha_0 (2\gamma_i -1)>1,\ i=1,2$. When $\rho^{\varepsilon}\in [\rho_{\star},\rho^{\star}]$,
$$
\begin{aligned}
&\E \big[\Big(  \sum_k \big( \int_0^t\int_{({\rm supp}_x(\zeta_k^{\varepsilon}))\cap\{\rho_\star\le \rho^{\varepsilon} \le \rho^\star\}} 
\big(\frac13 (u^{\varepsilon})^3 + 2 e(\rho^{\varepsilon})u^{\varepsilon} \big) \zeta_k^{\varepsilon} \,\d x  \d r\big)^2 \Big)^{\frac{p}{2}}\big]\\
&\le \delta \E \big[\Big(\sup_{r\in[0,t]} \int_{\mathbb{R}} \big(\rho^{\varepsilon}(u^{\varepsilon})^4 + \rho^{\varepsilon}(u^{\varepsilon})^2 \big)  \,\d x\Big)^p\big] \\
  &\quad+ C_{\delta}C(\rho^{\star})\E \big[\Big( 
  \sum_k  \iint_{([0,t]\times {\rm supp}_x  (\zeta_k^{\varepsilon} ))\cap \{\rho_\star\le \rho^{\varepsilon} \le \rho^\star\}}
  \Big(\frac{(u^{\varepsilon})^2}{\rho^{\varepsilon}}+ \frac{1}{\rho^{\varepsilon}} \Big)(\zeta_k^{\varepsilon})^2 \,\d x \d r \Big)^p \big].
\end{aligned}
$$
We deal with the second term on the right-hand side of the above inequality similarly as above
to obtain 
  $$
  \begin{aligned}
  &\E \big[\Big( \sum_k \iint_{([0,t]\times {\rm supp}_x  (\zeta_k^{\varepsilon}) )\cap \{\rho_\star\le \rho^{\varepsilon} \le \rho^\star\}}
  \Big(\frac{(u^{\varepsilon})^2}{\rho^{\varepsilon}}+ \frac{1}{\rho^{\varepsilon}} \Big)(\zeta_k^{\varepsilon})^2 \,\d x \d r \Big)^p \big]\\
  &\le B_0^{2p}(1+\varepsilon^2)^p t^p E_{0, p} + C \E \big[\Big( 
  \iint_{([0,t]\times {\rm supp}_x  (\zeta_k^{\varepsilon}) )\cap \{\rho_\star\le \rho^{\varepsilon} \le \rho^\star\}}
  \big(\rho^{\varepsilon} (u^{\varepsilon})^4+ \rho^{\varepsilon} e(\rho^{\varepsilon}) (u^{\varepsilon})^2   \big) \,\d x \d r \Big)^p\big]\\
  &\quad+C(p,t) +
  C(p)\E \big[\Big(\int_0^t  \int_{{\rm supp}_x (\zeta_k^{\varepsilon})} \big(\rho^{\varepsilon} (u^{\varepsilon})^2 +e^*(\rho^{\varepsilon}, \rho_{\infty})\big) \,\d x \d r \Big)^p\big].
  \end{aligned}
  $$
Therefore, using \eqref{iq-lower-upper-bound-for-general-internal-energy-1}--\eqref{iq-lower-upper-bound-for-general-internal-energy-2}, we have
$$
\begin{aligned}
\E\big[ \sup_{\tau\in [0,t]}|G_{\diamond}^{\varepsilon}(\tau)|^p\big]
  &\le \delta\, \E \big[ \Big(\sup_{r\in[0,t]} \int_{\mathbb{R}} \big(\rho^{\varepsilon}(u^{\varepsilon})^4 + \rho^{\varepsilon} e(\rho^{\varepsilon}) (u^{\varepsilon})^2 \big)  \,\d x\Big)^p \big] + C_{\delta} C(\rho^{\star},p,t)(1+E_{0,p}) \\
  &\quad + C_{\delta}C(\rho^{\star}) \E \big[\Big( \int_0^t  \int_{{\rm supp}_x (\zeta_k^{\varepsilon})}  \big(\rho^{\varepsilon} (u^{\varepsilon})^4+ \rho^{\varepsilon} e(\rho^{\varepsilon}) (u^{\varepsilon})^2 + e_{\diamond}^*(\rho^{\varepsilon}, \rho_{\infty})  \big) \,\d x \d r \Big)^p\big].
\end{aligned}
$$

Since $\partial_{m^{\varepsilon}}^2 \eta_{\diamond}^*(U^{\varepsilon})=\frac{(u^{\varepsilon})^2}{\rho^{\varepsilon}}+\frac{2e(\rho^{\varepsilon})}{\rho^{\varepsilon}}$, by similar arguments as above, we obtain that, for $j=1,2$, when $\alpha_0 \gamma_1>1$,
$$
\begin{aligned}
  &\E\big[ \sup_{\tau\in [0,t]}\Big|\int_0^{\tau} \int_{\mathbb{R}} \frac12 \partial_{m^{\varepsilon}}^2 \eta_{\diamond}^*(U^{\varepsilon}) \big(\mathfrak{G}_j^{\varepsilon} \big)^2 \,\d x \d r \Big|^p\big]\\
  &\le C(p,t)(1+E_{0,p})
  + C(p)\E \big[\Big(\int_0^t  \int_{{\rm supp}_x(\zeta_k^{\varepsilon})} \big(\rho^{\varepsilon} (u^{\varepsilon})^4 +  \rho^{\varepsilon} e(\rho^{\varepsilon})(u^{\varepsilon})^2 +e_{\diamond}^*(\rho^{\varepsilon}, \rho_{\infty})\big) \,\d x \d r \Big)^p \big],
\end{aligned}
$$
where we have used the estimate that, for $\rho^{\varepsilon} \in [\rho_{\star},\rho^{\star}]$, 
$\rho^{\varepsilon} (e(\rho^{\varepsilon}))^2 \le C \big(e_{\diamond}^*(\rho^{\varepsilon},\rho_{\infty})+ \rho_{\infty} (e(\rho_{\infty}))^2 \big)$.

Combining the above estimates, we conclude that, for $i=1,2$,
$$
\begin{aligned}
  &\mathbb{E} \big[\Big(\sup_{r\in [0,t]}\int_{\mathbb{R}} \big(\frac{1}{12}\frac{(m^{\varepsilon})^4}{(\rho^{\varepsilon})^3}+\frac{e(\rho^{\varepsilon})}{\rho^{\varepsilon}}m^2+e_{\diamond}^*(\rho^{\varepsilon},\rho_{\infty}) \big)(r,x) \,\d x
  \Big)^p\big]\\
&\quad+ \varepsilon \iint_{\{\rho_{\star}\le \rho^{\varepsilon} \le\rho^{\star}\}}
  \big((1-\delta)\rho^{\varepsilon}(u^{\varepsilon})^2 \big(\partial_x u^{\varepsilon}\big)^2 
  + 2\rho^{\varepsilon} e(\rho^{\varepsilon}) \big(\partial_x u^{\varepsilon}\big)^2 
  + \big(\rho^{\varepsilon} e(\rho^{\varepsilon})\big)^{\prime \prime} 
  (\partial_x \rho^{\varepsilon})^2 (u^{\varepsilon})^2\big) \,\d x\d r
  \Big)^p\big]\\
   &\le  \delta\,\E \big[\Big(\sup_{r\in[0,t]} \int_{\mathbb{R}}
  \big(\rho^{\varepsilon}(u^{\varepsilon})^4 + \rho^{\varepsilon} e(\rho^{\varepsilon}) 
  (u^{\varepsilon})^2 \big)(r,x)  \,\d x\Big)^p \big] \\
  &\quad + C_{\delta}C(p, \rho^{\star}) \E \big[\Big(\iint_{([0,t]\times {\rm supp}_x(\zeta_k^{\varepsilon}))} 
   \big(\rho^{\varepsilon} (u^{\varepsilon})^4+ \rho^{\varepsilon} e(\rho^{\varepsilon}) (u^{\varepsilon})^2 + e_{\diamond}^*(\rho^{\varepsilon}, \rho_{\infty})  \big) \,\d x \d r \Big)^p\big]\\
  &\quad + C_{\delta} C(p,t,\rho^\star)(1+E_{0,p})+C(p)\E \big[\Big(\int_{\mathbb{R}} \eta_{\diamond}^*(\rho^{\varepsilon}_0,m^{\varepsilon}_0)\, \d x \Big)^p\big].
\end{aligned}
$$
Then the proof is completed by taking $\delta$ sufficiently small and using the Gr\"onwall inequality.
\end{proof}

Having this higher-order energy estimate at hand,
the higher-order growth of the generating function in the entropy is allowed in the general pressure law case.
\begin{lemma}\label{lem-higher-order-growth-of-entropy-and-derivatives-general-pressure-law}
Let $\psi\in C^2(\mathbb{R})$ be convex and satisfy
  $$
  |\psi(s)|\le C|s|^{3-\delta},\qquad |\psi^{\prime}(s)|\le C|s|^{2-\delta},
  $$
then, when $\rho\ge \rho^{\star}$ and $\delta\ge 2\big(1- \frac{1}{\gamma_2}\big)$ for $\gamma_2 \le 2$ and $\delta \ge 1$ for $\gamma_2>2$,
  $$\begin{aligned}
  &|\eta^{\psi}(\rho,u)|+|q^{\psi}(\rho,u)| +\frac{1}{\rho}|\partial_u\eta^{\psi}(\rho,u)|^2+\frac{1}{\rho^2}|\partial_u^2 \eta^{\psi}(\rho,u)|\sum_k\zeta_k^2 \le C\big(1+\rho|u|^4+g(\rho)+ e(\rho) 
  \big).
  \end{aligned}$$
When $\rho\le \rho^{\star}$ and $\delta\ge 0$,
  $$
  |\eta^{\psi}(\rho,u)|+|q^{\psi}(\rho,u)| +\frac{1}{\rho}|\partial_u\eta^{\psi}(\rho,u)|^2+\frac{1}{\rho^2}|\partial_u^2 \eta^{\psi}(\rho,u)|\sum_k\zeta_k^2 \le C(\rho_{\star},\rho^{\star})\big(1+\rho|u|^4 + g(\rho)\big).
  $$
\end{lemma}

Since the proof is similar to that of Lemma \ref{lem-growth-of-entropy-and-derivatives-general-pressure-law}, so we omit the details.

Using Proposition \ref{prop-rho|u|4-general-pressure-law} and 
Lemma \ref{lem-higher-order-growth-of-entropy-and-derivatives-general-pressure-law}, 
it follows from a similar proof to 
Proposition \ref{prop-entropy-inequality-general-pressure-law} that the following entropy inequality satisfied by more entropy pairs.

\begin{proposition}\label{prop-entropy-inequality-general-pressure-law-under-higher-energy}
Let $(\eta,q)$ be the entropy pair defined in \eqref{eq-entropy-flux-representation} with convex generating function $\psi\in C^2(\mathbb{R})$ satisfies \eqref{eq-definition-of-higher-growth-for-weight-function-general-pressure}.
Assume that $U_0$ satisfies $E_{0,2}<\infty$ and
  $$\begin{aligned}
  \E \big[\Big(\int_{\mathbb{R}} \eta_{\diamond}^*(\rho_0,m_0)\,  \d x \Big)^2\big]<\infty.
  \end{aligned}$$
Then, for any nonnegative $\varphi \in C_{\rm c}^{\infty}((0,T)\times\mathbb{R})$, the following entropy inequality holds $\tilde{\mathbb{P}}$-{\it a.s.}{\rm :}
  $$
  \begin{aligned}
  &\int_0^T \int_{\mathbb{R}} \big( \eta(\tilde  U)\varphi_t+ q(\tilde U) \varphi_x \big) \,\d x \d t+ \int_0^T \int_{\mathbb{R}}  \partial_{\tilde m} \eta (\tilde U) \Phi( \tilde U) \varphi\,\d x \d W(t) \\
  &+\frac12 \int_0^T \int_{\mathbb{R}} \partial_{\tilde m}^2 \eta(\tilde U) \sum_k a_k^2 \,\zeta_k^2(\tilde U) \varphi\,\d x \d t
  \ge  0,
  \end{aligned}
  $$
where $\tilde U=(\tilde \rho, \tilde m)^\top$.
\end{proposition}

\subsection{Compactness of the Martingale Entropy Solution Sequence to \eqref{eq-stochastic-Euler-system}}

We now prove (ii) in Theorem \ref{thm-better-well-posedness-for-euler-on-whole-space-general-pressure-law},
which can be stated in the following: 

\begin{theorem}\label{thm-compactness-of-solution-sequence}
Let $U^{\delta}:=(\rho^{\delta},m^{\delta})$ be a sequence of entropy solutions 
to \eqref{eq-stochastic-Euler-system} obtained 
in {\rm Theorem \ref{thm-better-well-posedness-for-euler-on-whole-space-general-pressure-law}} 
corresponding to a sequence of initial law $\Im^{\delta}=\mathbb{P}\circ (\rho_0^\delta,m_0^\delta)^{-1}$ 
satisfying 
$$
\begin{aligned}
  \int_{L^{\gamma_2}_{\rm loc}(\mathbb{R})\times L^1_{\rm loc}(\mathbb{R})} \big| \mathcal{E}(\rho,m)\big|^{2}\, \d \Im^{\delta}(\rho,m)+\int_{L^{\gamma_2}_{\rm loc}(\mathbb{R})\times L^1_{\rm loc}(\mathbb{R})} \big| \mathcal{E}_{\diamond}(\rho,m)\big|^{2}\, \d \Im^{\delta}(\rho,m) \le C,
  \end{aligned}
  $$
where constant $C$ is independent of $\delta$. 
Then, for any $K\Subset \mathbb{R}$, there exists $C(K,p,T)>0$ independent of $\delta$
such that
\begin{equation}\label{iq-higher-integrability-in-compactness-of-solu}
  \begin{aligned}
  \E \big[\Big( \int_0^t \int_{K} \big(\rho^{\delta} P(\rho^{\delta})+ \frac{|m^{\delta}|^3}{(\rho^{\delta})^2} + \eta_{\diamond}(\rho^{\delta},m^{\delta})
  \big) \,\d x\d \tau \Big)^p\big]
  \le  C(K,p,T)  \quad \mbox{for any $t\in[0,T]$},
  \end{aligned}
  \end{equation}
so that $U^{\delta}:=(\rho^{\delta},m^{\delta})$ are tight in 
$L^{\gamma_2+1}_{\rm w,loc}(\mathbb{R}^2_+) \times L^{\frac{3(\gamma_2+1)}{\gamma_2+3}}_{\rm w,loc}(\mathbb{R}^2_+)$ 
and admit the Skorokhod representation $(\tilde \rho^{\delta}, \tilde m^{\delta})$. 
Moreover, when $\gamma_2 <2$, there exist random variables $(\tilde\rho,\tilde m)$ such that 
almost surely $(\tilde \rho^{\delta},\,\tilde m^{\delta}) \to (\tilde \rho,\, \tilde m)$ 
almost everywhere and $($up to a subsequence$)$ almost surely,
$$
(\tilde \rho^{\delta},\,\tilde m^{\delta}) \to (\tilde \rho,\, \tilde m)\qquad 
\text{in $L^{\bar p}_{\rm loc}(\mathbb{R}^2_+)\times L^{\bar q}_{\rm loc}(\mathbb{R}^2_+)$}
$$
for $\bar p\in [1, \gamma_2+1 )$ and $\bar q\in [ 1, \frac{3(\gamma_2+1)}{\gamma_2+3} )$. In addition, $(\tilde\rho,\tilde m)$ is also an entropy solution to \eqref{eq-stochastic-Euler-system}, {\it i.e.}, almost surely satisfies the entropy inequality \eqref{iq-entropy-inequality-for-Euler-general-pressure}.
\end{theorem}

\begin{proof}
We divide the proof into four steps.

\smallskip
\textbf{1}. The uniform bound \eqref{iq-higher-integrability-in-compactness-of-solu} can be proved 
by the arguments similar to those for \eqref{iq-higher-integrability-for-young-measure} 
and Proposition \ref{prop-take-limit-obtain-relative-energy-estiate-for-Euler}.

\textbf{2}. 
Denote 
  $$
  \mu_{\delta}:= \d \eta^{\psi}(U^{\delta}) + \partial_x q^{\psi}(U^{\delta})\d t - \partial_m \eta^{\psi}(U^{\delta})\Phi(U^{\delta}) \d W(t)-\frac12 \partial_m^2 \eta^{\psi}(U^{\delta}) \sum_k a_k^2(\zeta_k)^2 (U^{\delta})\d t,
  $$
where $(\eta^{\psi},q^{\psi})$ is any entropy pair defined in \eqref{eq-entropy-flux-representation} with generating function $\psi\in C_{\rm c}^2(\mathbb{R})$.
When the generating function $\psi\in C_c^2(\mathbb{R})$ is convex and 
satisfies \eqref{eq-definition-of-growth-for-weight-function-general-pressure}, 
by the proof of Lemmas \ref{prop-tightness-of-stochastic-integral}--\ref{prop-tightness-of-quadratic-variation} 
and Proposition \ref{prop-H^-1-compactness}, we obtain 
that $\mu_{\delta}$ is stochastically bounded in $H^{-1}_{\rm loc}(\mathbb{R}^2_+)$. 

On the other hand, the entropy inequality \eqref{iq-entropy-inequality-for-Euler-general-pressure} gives that $\mu_{\delta}\ge 0$ in the sense of distribution (see Lemma \ref{lem-stochastic-version-murat-lemma-81}). 
Thus, combining them together and applying Lemma \ref{lem-stochastic-version-murat-lemma-81} yield
that $\mu_{\delta}$ is tight in $W^{-1,p_1}$ for any $1<p_1 < 2$. 

By Proposition \ref{prop-entropy-inequality-general-pressure-law-under-higher-energy}, when $\gamma_2 <2$, 
$$
\mu^{\delta}_{E}:=\d \eta_E (U^{\delta}) + \partial_x q_E(U^{\delta})\d t 
- \partial_m \eta_E (U^{\delta})\Phi(U^{\delta}) \d W(t)
-\frac12 \partial_m^2 \eta_E (U^{\delta}) \sum_k a_k^2 (\zeta_k)^2 (U^{\delta})\d t\ge 0
$$
in the sense of distributions. Additionally, using \eqref{iq-higher-integrability-in-compactness-of-solu}, we obtain that $\{\partial_t \eta_E (U^{\delta}) +\partial_x q_E (U^{\delta}) \}_{\delta>0}$ 
is stochastically bounded in $W^{-1,\frac{4\gamma_2+4}{3\gamma_2+4}}_{\rm loc}(\mathbb{R}^2_+)$. 
Then, by similar arguments to those in the justification of 
Lemmas \ref{prop-tightness-of-stochastic-integral}--\ref{prop-tightness-of-quadratic-variation}, we see 
that $\mu^{\delta}_{E}$ is stochastically bounded 
in $W^{-1,\frac{4\gamma_2+4}{3\gamma_2+4}}_{\rm loc}(\mathbb{R}^2_+)$.

\smallskip
\textbf{3}. Notice that, for any not necessarily convex entropy pair $(\eta^{\psi}, q^{\psi})$ with generating function $\psi \in C_{\rm c}^2(\mathbb{R})$, there exists a constant 
$\bar{M}:= M(\psi, \rho_{\star},\rho^{\star})$ such that $|\nabla^2 \eta^{\psi}|\le \bar{M} \nabla^2 \eta_E$. 
In fact, $|\nabla^2 \eta^{\psi}|\le \bar{M}\, \nabla^2 \eta_E$ is equivalent 
to $|\xxi \, \nabla^2 \eta^{\psi} \,\xxi^\top| \le \bar{M}\, \xxi \, \nabla^2 \eta_E \,\xxi^\top$ 
for any vector $\xxi=(\xi_1,\xi_2)\in \mathbb{R}^2$.
By direct calculation, we see that $\xxi\,\nabla^2 \eta_E \,\xxi^\top
=\big(\rho e(\rho)\big)^{\prime \prime}\xi_1^2 + \frac{1}{\rho} \big( \frac{m}{\rho}\xi_1 - \xi_2\big)^2$. 
Regarding $\eta^{\psi}$ as a function of $(\rho,u)$, we have 
  $$
  \begin{aligned}
  &\xxi\,\nabla^2 \eta^{\psi} \xxi^\top\\
  &=\xi_1^2 \kappa^{\prime}(\rho)^2 \partial_u^2 \eta^{\psi}(\rho,u) + \partial_u^2 \eta^{\psi}(\rho,u) \frac{1}{\rho^2} \big( \frac{m}{\rho}\xi_1 - \xi_2\big)^2 
  - 2\xi_1\big( \frac{m}{\rho}\xi_1 - \xi_2\big) \big(\partial_{\rho u}\eta^{\psi}(\rho,u) 
  - \frac{1}{\rho}\partial_u \eta^{\psi}(\rho,u) \big).
  \end{aligned}
  $$
Then, using the fact that $\partial_u^2 \eta^{\psi}(\rho,u)=\rho^2 \partial_m^2 \eta^{\psi}(\rho,m)$ 
and $\partial_{\rho u}\eta^{\psi}(\rho,u) - \frac{1}{\rho}\partial_u \eta^{\psi}(\rho,u)=\rho \partial_{m\rho}\eta^{\psi}(\rho,u)$, 
it follows from 
\eqref{iq-control-for-partial-m-of-entropy-generate-by-compact-support-func}--\eqref{iq-control-for-partial-m-of-entropy-in-rho-u-coordinate-generate-by-compact-support-func} that 
there exists a constant $M(\psi, \rho_{\star},\rho^{\star})$ such that 
$|\xxi\, \nabla^2 \eta^{\psi} \,\xxi^\top| 
\le \bar{M}\, \xxi\, \nabla^2 \eta_E \,\xxi^\top$.

Utilizing the idea from \cite{chen1991hyperbolic-symmetric}, 
$\,(\bar \eta, \bar q):=(\bar{M}\eta_E -\eta^{\psi}, \bar{M} q_E -q^{\psi})$ is a convex entropy pair. 
Define 
  $$
  \bar \mu_{\delta}:= \d \bar \eta(U^{\delta}) + \partial_x \bar q(U^{\delta})\d t - \partial_m \bar \eta(U^{\delta})\Phi(U^{\delta}) \d W(t)-\frac12 \partial_m^2 \bar \eta(U^{\delta}) \sum_k a_k^2(\zeta_k)^2 (U^{\delta})\d t.
  $$
Note that the stochastic boundedness of $\mu_{\delta}$ is nothing to do with the convexity. 
Thus, by the stochastic boundedness of $\mu_{\delta}$ and $\mu^{\delta}_E$, 
we obtain that $\bar \mu_{\delta}$ is stochastically bounded 
in $W^{-1,\frac{4\gamma_2+4}{3\gamma_2+4}}_{\rm loc}(\mathbb{R}^2_+)$. 
On the other hand, by the entropy inequality \eqref{iq-entropy-inequality-for-Euler-general-pressure}, 
$\bar \mu_{\delta}\ge 0$ in the sense of distribution. Thus, 
by Lemma \ref{lem-stochastic-version-murat-lemma-81} again, 
$\bar \mu_{\delta}$ is tight in $W^{-1,p_1}$ for any $1<p_1 < \frac{4\gamma_2+4}{3\gamma_2+4}$. 
By linearity, we obtain that $\mu_{\delta}$ is tight in 
$W^{-1,p_1}_{\rm loc}(\mathbb{R}^2_+)$ for any $1<p_1 < \frac{4\gamma_2+4}{3\gamma_2+4}$.

\smallskip
\textbf{4}. By Lemmas \ref{prop-tightness-of-stochastic-integral}--\ref{prop-tightness-of-quadratic-variation} 
and Step \textbf{3}, we obtain 
that $\{ \partial_t \eta^{\psi}(U^{\delta}) + \partial_x q^{\psi}(U^{\delta}) \}_{\delta}$ 
are tight in $W^{-1,p_1}_{\rm loc}(\mathbb{R}^2_+)$ with $1<p_1 < \frac{4\gamma_2+4}{3\gamma_2+4}$. 
Thus, all the conditions in the stochastic compensated compactness 
framework (Theorem \ref{thm-stochastic-Lp-compensated-compactness-framework}) are satisfied. 
Then Theorem \ref{thm-stochastic-Lp-compensated-compactness-framework} implies that 
there exist random variables $(\tilde\rho,\tilde m)$ such that 
almost surely $(\tilde \rho^{\delta},\,\tilde m^{\delta}) \to (\tilde \rho,\, \tilde m)$ 
almost everywhere and $($up to a sequence$)$ almost surely,
$$
(\tilde \rho^{\delta},\,\tilde m^{\delta}) \to (\tilde \rho,\, \tilde m)\qquad 
\text{in $L^{\bar p}_{\rm loc}(\mathbb{R}^2_+)\times L^{\bar q}_{\rm loc}(\mathbb{R}^2_+)$}
$$
for $\bar p\in [1, \gamma_2+1 )$ and $\bar q\in [ 1, \frac{3(\gamma_2+1)}{\gamma_2+3} )$. 

Additionally,  by arguments similar to those for the proof 
of Proposition \ref{prop-entropy-inequality-general-pressure-law}, 
we obtain that $(\tilde\rho,\tilde m)$ almost surely satisfies the entropy 
inequality \eqref{iq-entropy-inequality-for-Euler-general-pressure}. This indicates that 
$(\tilde\rho,\tilde m)$ is an entropy solution to \eqref{eq-stochastic-Euler-system}. 
\end{proof}

Having all these at hand, we now prove the existence of global martingale entropy solutions 
in the sense that they satisfy conditions (i)--(vi) in 
Definition \ref{def-martingale-entropy-solutions-with-relative-finite-energy-general-pressure} 
and the entropy inequality \eqref{iq-entropy-inequality-for-Euler-general-pressure} 
with generating function $\psi$ satisfying \eqref{eq-definition-of-growth-for-weight-function-general-pressure} 
or \eqref{eq-definition-of-higher-growth-for-weight-function-general-pressure}. 
More specifically, conditions (i)--(v) in 
Definition \ref{def-martingale-entropy-solutions-with-relative-finite-energy-general-pressure} are 
the consequence of Propositions \ref{prop-apply-jakubowski-skorokhod-representation-to-take-limit} 
and \ref{prop-take-limit-obtain-relative-energy-estiate-for-Euler}.
Condition (vi) in Definition \ref{def-martingale-entropy-solutions-with-relative-finite-energy-general-pressure} 
is due to Proposition \ref{prop-weak-solution-for-general-pressure-law}. 
Inequality \eqref{iq-entropy-inequality-for-Euler-general-pressure} follows from Propositions \ref{prop-entropy-inequality-general-pressure-law} and \ref{prop-entropy-inequality-general-pressure-law-under-higher-energy}. 
Condition (ii) in Theorem \ref{thm-better-well-posedness-for-euler-on-whole-space-general-pressure-law} 
has been proved in Theorem \ref{thm-compactness-of-solution-sequence} above. Therefore, by \cite[Corollary 2.6.4]{BFHbook18}, we have verified 
Theorems \ref{thm-well-posedness-for-euler-on-whole-space-general-pressure-law}--\ref{thm-better-well-posedness-for-euler-on-whole-space-general-pressure-law}.

\section{The Polytropic Pressure Case}\label{sec-polytropic-gas-case}

In the polytropic pressure case, we obtain more information about the martingale entropy solutions, based on the explicit formulas of entropy pairs and the particular features of the system. 

First, we have the following simpler properties for the relative internal energy than those
in Lemma \ref{lem-properties-for-relative-internal-energy}:

\begin{lemma}\label{lem-properties-for-relative-internal-energy-polytropic-gas}
For the polytropic pressure case \eqref{eq-gamma-law}, the following bounds hold{\rm :}
  \begin{align}
  &e^*(\rho, \rho_{\infty}) \ge C_{\gamma}\rho(\rho^{\theta}-\rho_{\infty}^{\theta})^2,\label{iq-relative-internal-energy-control}\\[0.5mm]
  &\rho^{\gamma}\le \bar C_{\gamma} \big(e^*(\rho, \rho_{\infty}) + \rho_{\infty}^{\gamma}\big),\label{iq-density-control-by-relative-internal-energy}
  \end{align}
  where $C_{\gamma} >0$ is a constant, and $\bar C_{\gamma}$ is another constant depending on $C_{\gamma}$.
\end{lemma}

These properties follow from Lemma \ref{lem-properties-for-relative-internal-energy} by taking $\gamma_1=\gamma_2=\gamma$ and $\rho_{\star}=\rho^{\star}$ since the proof does not restrict that $\gamma\le 3$. For a more detailed proof, see \cite[(3.1)]{chenperepelitsa10CPAM} for \eqref{iq-relative-internal-energy-control}.

Lemma \ref{lem-properties-for-relative-internal-energy-polytropic-gas} above, replacing Lemma \ref{lem-properties-for-relative-internal-energy}, will be used frequently in the proofs for the polytropic pressure case, as in the general pressure law case.

\subsection{Definition of Martingale Entropy Solutions and Main Theorems: \\
Theorem \ref{thm-well-posedness-for-euler-on-whole-space-deterministic-initial-data}
and Theorem \ref{thm-well-posedness-for-euler-on-whole-space}}

When the initial density has a positive far-field $\rho_{\infty}>0$, 
for the polytropic pressure case \eqref{eq-gamma-law}, we use the following more precise 
definition of solutions:

\begin{definition}\label{def-martingale-weak-entropy-solution-with-relative-finite-energy-polytropic-gas}
We call
  $$
  (\Omega,\mathcal{F},(\mathcal{F}_t),\mathbb{P},\rho,m,W)
  $$
a \textit{martingale entropy solution with finite relative-energy} of \eqref{eq-stochastic-Euler-system}
if it satisfies conditions {\rm (i)--(v)} 
in {\rm Definition \ref{def-martingale-entropy-solutions-with-relative-finite-energy-general-pressure}} 
with $\gamma_2$ replaced by $\gamma$ and, 
for any entropy pair $(\eta^{\psi},q^{\psi})$ defined in \eqref{eq-entropy-flux-representation} with $\psi\in C^2(\mathbb{R})$ that is convex and satisfies \eqref{eq-definition-of-subquadratic-growth} {\rm ({\it i.e.}, has subquadratic growth at infinity)} and additionally allows the case $\psi(s)=\frac12 s^2$ {\rm (}corresponding to the mechanical energy{\rm )}, 
the following entropy inequality is almost surely satisfied 
for any nonnegative $\varphi \in C_{\rm c}^{\infty}((0,T)\times\mathbb{R})${\rm :}
    \begin{equation}\label{iq-entropy-inequality-for-Euler}
    \begin{aligned}
    &\int_0^T \int_{\mathbb{R}}  \eta( U)\, \varphi_t\, \d x \d t
    + \int_0^T \int_{\mathbb{R}} q( U)  \varphi_x \,\d x \d t\\
    &\,\,\,+\int_0^T \int_{\mathbb{R}}  \partial_{m} \eta ( U) \Phi( U)\,\varphi\,\d x \d W(t)
    +\frac12 \int_0^T \int_{\mathbb{R}}\partial_{ m}^2 \eta( U) \sum_k a_k^2\,\zeta_k^2( U)\, \varphi \, \d x \d t
    \ge  0.
    \end{aligned}
    \end{equation}
\end{definition}

We say a function $\psi\in C^2(\mathbb{R})$ has subquadratic growth at infinity if
  \begin{equation}\label{eq-definition-of-subquadratic-growth}
  \begin{aligned}
  &
  \lim_{s\to \pm\infty } 
  \frac{\psi(s)}{s^2}=0,\quad 
  \lim_{s\to \pm\infty } 
  \frac{\psi^{\prime}(s)}{s}=0 &\qquad \text{for $\gamma\le 3$},\\
  &
  \lim_{s\to \pm\infty } 
  \frac{\psi(s)}{s^2}=0,\quad 
  \lim_{s\to \pm\infty } \frac{\psi^{\prime}(s)}{s}=0,\quad 
  \lim_{s\to \pm\infty } 
  \psi^{\prime \prime}(s)=0 
  &\qquad \text{for $\gamma > 3$}.
  \end{aligned}
  \end{equation}

Our main result is the following:

\begin{theorem}\label{thm-well-posedness-for-euler-on-whole-space-deterministic-initial-data}
Assume that the initial data $(\rho_0, m_0)\in L^{\gamma}_{\rm loc}(\mathbb{R})\times L^1_{\rm loc}(\mathbb{R})$ have finite 
relative-energy{\rm :}
  $$\begin{aligned}
  &\mathcal{E}(\rho_0,m_0)=\int_{\mathbb{R}} \big(\frac12 \frac{m_0^2}{\rho_0} +e^*(\rho_0,\rho_{\infty})\big)\, \d x <\infty.
  \end{aligned}$$
Then, for all $\gamma>1$, there exists a martingale entropy solution with finite relative-energy to \eqref{eq-stochastic-Euler-system} in the sense of {\rm Definition \ref{def-martingale-weak-entropy-solution-with-relative-finite-energy-polytropic-gas}}.
\end{theorem}

As in the general pressure law case, the above result can be more general, allowing random initial data. 

\begin{theorem}\label{thm-well-posedness-for-euler-on-whole-space}
Assume that $\Im$ is a Borel probability measure on $L^{\gamma}_{\rm loc}(\mathbb{R})\times L^1_{\rm loc}(\mathbb{R})$ satisfying
$$
\begin{aligned}
  &{\rm supp}\, \Im=\{(\rho,m)\in \mathbb{R}^2_+\,:\, \rho \ge 0\},\\
  &
  E_{p_0}(\Im):=\int_{L^{\gamma}_{\rm loc}(\mathbb{R})\times L^1_{\rm loc}(\mathbb{R})} \big| \mathcal{E}(\rho,m)\big|^{p_0}\, \d \Im(\rho,m) <\infty
  \quad\mbox{for some $p_0>3$}. 
  \end{aligned}
$$
Then, for all $\gamma>1$, there exists a martingale entropy solution with 
finite relative-energy to \eqref{eq-stochastic-Euler-system} 
in the sense of {\rm Definition \ref{def-martingale-weak-entropy-solution-with-relative-finite-energy-polytropic-gas}} with initial law $\Im$, {\it i.e.}, there exists $\mathcal{F}_0$-measurable random variables $(\rho_0,m_0)$ such that $\Im=\mathbb{P}\circ (\rho_0,m_0)^{-1}$ 
for which we denote $E_{0,p_0}:=E_{p_0}(\Im)$.
\end{theorem}

\begin{remark}
For the polytropic pressure case, we have used the same notation $E_{0,p_0}$ as the general pressure law case for simplicity of notation.
\end{remark}

\begin{remark}
{\rm Theorem \ref{thm-well-posedness-for-euler-on-whole-space-deterministic-initial-data}} 
is the direct corollary of {\rm Theorem \ref{thm-well-posedness-for-euler-on-whole-space}},
so it suffices to 
prove {\rm Theorem \ref{thm-well-posedness-for-euler-on-whole-space}}.
\end{remark}

\begin{remark}
In {\rm Theorem \ref{thm-well-posedness-for-euler-on-whole-space}}, 
we require $p_0>3$ $($compared to the general pressure law case$)$ 
in order to pass to the limit in the entropy flux and derive the entropy inequality{\rm ;} 
see \eqref{iq-lower-growth-for-entropy-in-entropy-inequality},
{\rm Propositions \ref{prop-uniform-estimates-rho-gamma+1}}
{\rm and}  {\rm \ref{prop-entropy-inequality-general-pressure-law}}--{\rm \ref{prop-weak-solution-for-general-pressure-law}}.
\end{remark}

\medskip
When the initial data satisfy
$
(\rho_0,m_0) \to (0,0)\ \text{ as $|x|\to \infty$},
$
we can also formulate the theorems as above with an obvious change.

\subsection{Proof of Theorem \ref{thm-well-posedness-for-euler-on-whole-space}}
We use similar strategies to those for the general pressure law case 
to prove Theorem \ref{thm-well-posedness-for-euler-on-whole-space}.

\subsubsection{Stochastic parabolic approximation}
For the polytropic pressure case, we have the same result, Theorem \ref{thm-wellposedness-parabolic-approximation-R}, regarding the parabolic approximation. 
In fact, the proof is the same as that in the general pressure law case since it does not restrict that $\gamma \le 3$.

\subsubsection{Higher integrability of the density and the velocity}\label{sec-Higher-integrability-in-polytropic-gas-case}
In the polytropic pressure case, 
we have the same uniform energy estimates 
as \eqref{iq-energy-estimate-on-whole-space} 
since the proof does not restrict that $\gamma \le 3$. 
Regarding the higher integrability of the density and the velocity, 
we have similar results to those for the general pressure law case. 
Firstly, for the higher integrability of the density, 
applying the iteration technique in \cite[Lemma 3.3]{chenperepelitsa15CMP} 
and using \eqref{iq-density-control-by-relative-internal-energy}, 
the proof is similar to that of the general pressure law case (even easier). More specifically, in the polytropic pressure case, we obtain the higher integrability of the density for all $\gamma>1$.

Secondly, for the higher integrability of the velocity in the polytropic 
gas case, we use a special entropy generated by a non-convex function.
We need the following two lemmas in the proof.

\begin{lemma}\label{lem-hessian-of-mechanical-energy-control-other-weak-entropy}
Let $(\eta^{\psi},q^{\psi})$ be an entropy pair in \eqref{eq-entropy-flux-representation} with $\psi(s)$ satisfying
  $
  \sup |\psi^{\prime \prime}(s)| < \infty.
  $
Then, for any $(\rho,m)\in \mathbb{R}^2$ and $\xxi=(\xi_1,\xi_2)\in \mathbb{R}^2$,
  $$
  |\xxi\, \nabla^2 \eta^{\psi}(\rho,m)\, \xxi^{\top} |
  \le M_{\psi}\, \xxi\, \nabla^2 \eta_E(\rho,m)\, \xxi^{\top},
  $$
where $M_{\psi}$ is a constant depending only on $\psi$.
\end{lemma}

The proof is direct. See \cite{Diperna83} or \cite{chenperepelitsa10CPAM}.

\smallskip
Let $(\breve \eta, \breve q)$ be an entropy pair corresponding to $\psi(s)=\frac12 s|s|$ in \eqref{eq-entropy-flux-representation}. We define the relative-entropy pair as
  \begin{equation}\label{eq-relative-entropy-pair}
  \begin{aligned}
  &\tilde \eta (\rho,m)=\breve \eta (\rho,m)-\nabla_{(\rho,m)}\breve \eta (\rho_{\infty},0)\cdot (\rho-\rho_{\infty},m),\\
  &\tilde q(\rho,m)=\breve q(\rho,m)-\nabla_{(\rho,m)}\breve \eta(\rho_{\infty},0)\cdot (m,\frac{m^2}{\rho}+P(\rho)).
  \end{aligned}
  \end{equation}
Note that $(\tilde \eta, \tilde q)$ is still an entropy pair of \eqref{eq-stochastic-Euler-system}. We collect here some properties related to $(\breve \eta, \breve q)$ and $(\tilde \eta, \tilde q)$.
\begin{lemma}\label{lem-properties-for-relative-entropy-pair-s|s|}
The entropy pairs $(\breve\eta, \breve q)$ and $(\tilde\eta, \tilde q)$ satisfy
the following properties{\rm :}
  \begin{align}
  &\partial_{\rho} \breve \eta (\rho,0)=0,\quad \partial_{m} \breve \eta (\rho,0)=\rho^{\theta}\int_{-1}^1 |s|[1-s^2]^{\lambda}_{+} \,\d s,\label{eq-derivative-of-relative-entropy-s|s|}\\
  &|\tilde \eta (\rho,m)|\le C(\gamma)\eta_E^*(\rho,m),\label{iq-relative-entropy-s|s|-control-by-relative-energy}\\
  &|\partial_m \tilde \eta(\rho,m)|
  \le C(\gamma)\big(|\rho^{\theta}-\rho_{\infty}^{\theta}|+|u|\big),\label{iq-relative-entropy-s|s|-derivative-partial-m}\\
  &\breve q(\rho,m) \ge C(\gamma)^{-1}\big(\rho|u|^3 + \rho^{\gamma+\theta}\big),\label{iq-entropy-flux-s|s|-control-rhou3}
  \end{align}
 where $C(\gamma)>0$ is a constant depending only on $\gamma>1$. 
\end{lemma}

\begin{proof}
$\eqref{eq-derivative-of-relative-entropy-s|s|}$ follows from direct calculation. For \eqref{iq-relative-entropy-s|s|-control-by-relative-energy}--\eqref{iq-relative-entropy-s|s|-derivative-partial-m}, see \cite[Proof of Proposition 4.1]{chenwangyong22ARMA}. For \eqref{iq-entropy-flux-s|s|-control-rhou3}, see \cite{LPT94} for the details.
\end{proof}

Now we are ready to prove the higher integrability of the velocity.

\begin{proposition}\label{prop-uniform-estimates-rho-u3}
For any $K\Subset \mathbb{R}$, there exists $C(K,p,T, E_{0,p})>0$ independent of $\varepsilon$ such that
the solution of \eqref{eq-parabolic-approximation} obtained 
in {\rm Theorem \ref{thm-wellposedness-parabolic-approximation-R}} satisfies 
  \begin{equation}\label{iq-uniform-estimate-rho-u3-a}
  \begin{aligned}
  \E \big[\Big( \int_0^t \int_{K}\big( \rho |u|^3 +\rho^{\gamma+\theta}\big)\,\d x\d \tau \Big)^p\big]
  \le  C(K,p,T, E_{0,p}) \qquad\mbox{for any $t\in[0,T]$}.
  \end{aligned}
  \end{equation}

\end{proposition}
Using Lemmas \ref{lem-hessian-of-mechanical-energy-control-other-weak-entropy}--\ref{lem-properties-for-relative-entropy-pair-s|s|}, we can prove Proposition \ref{prop-uniform-estimates-rho-u3} similarly as Proposition \ref{prop-higher-integrability-of-velocity-general-pressure-law}, so we omit the details.

Having these higher integrability results at hand, we have the same results as those in \S \ref{sec-tightness} 
for the general pressure law case by taking $\gamma_1=\gamma_2=\gamma$ in those proofs, 
which do not restrict $\gamma\le 3$. 
More specifically, we have the same tightness results as those in \S \ref{sec-tightness} 
for all $\gamma>1$. 
Now we are going to establish the {\it Stochastic $L^p$ Compensated Compactness Framework} 
for the polytropic pressure case.

\subsubsection{Stochastic compensated compactness framework in $L^p$}
For the polytropic pressure case, our Stochastic $L^p$ Compensated Compactness Framework can be formulated
as follows: 

\begin{theorem}\label{thm-stochastic-Lp-compensated-compactness-framework-polytropic-gas}
Let $\gamma>1$.
Let $(\rho^{\varepsilon}, m^{\varepsilon})$ be a sequence of random variables satisfying that,
for any $K\Subset \mathbb{R}$, there exists a constant $C(K,p,T)>0$ independent of $\varepsilon$ such that
  \begin{equation}\label{iq-higher-integrability-in-Lp-compensated-compactness-framework-polytropic-gas}
  \begin{aligned}
  \E \big[\Big( \int_0^t \int_{K} \big(\rho P(\rho)+\frac{|m|^3}{\rho^2}\big) \,\d x\d \tau \Big)^p\big]
  \le  C(K,p,T) \qquad\mbox{for any $t\in[0,T]$},
  \end{aligned}
  \end{equation}
where $P(\rho)$ satisfies \eqref{eq-gamma-law},
so that $(\rho^{\varepsilon}, m^{\varepsilon})$ are tight in 
$L^{\gamma+1}_{\rm w,loc}(\mathbb{R}^2_+)\times L^{\frac{3(\gamma+1)}{\gamma+3}}_{\rm w,loc}(\mathbb{R}^2_+)$
and admit the Skorokhod representation $(\tilde \rho^{\varepsilon}, \tilde m^{\varepsilon})$. 
Assume that
\begin{equation}\label{con-tightness-of-dissipation-measure-in-Lp-compensated-compactness-framework-polytropic-gas}
\big\{\eta^{\psi}(\rho^{\varepsilon},m^{\varepsilon})_t +q^{\psi}(\rho^{\varepsilon},m^{\varepsilon})_x\big\}_{\varepsilon >0}
\qquad \text{are tight in $H^{-1}_{\rm loc}(\mathbb{R}^2_+)$}
\end{equation}
for any weak entropy pair $(\eta^{\psi},q^{\psi})$ of system \eqref{eq-stochastic-Euler-system} 
defined in \eqref{eq-entropy-flux-representation} with generating function $\psi\in C_{\rm c}^2(\mathbb{R})$. 
Then there exist random variables $(\tilde\rho,\tilde m)$ such that 
almost surely $(\tilde \rho^{\varepsilon},\,\tilde m^{\varepsilon}) \to (\tilde \rho,\, \tilde m)$ 
almost everywhere and $($up to a sequence$)$ almost surely,
$$
(\tilde \rho^{\varepsilon},\,\tilde m^{\varepsilon}) \to (\tilde \rho,\, \tilde m)\qquad 
\text{in $L^{\bar p}_{\rm loc}(\mathbb{R}^2_+)\times L^{\bar q}_{\rm loc}(\mathbb{R}^2_+)$}
$$
for $\bar p\in [1, \gamma+1 )$ and $\bar q\in [ 1, \frac{3(\gamma+1)}{\gamma+3} )$.
\end{theorem}

To establish this framework, the same procedure as that for the general pressure law case 
in \S\ref{sec-compensated-compactness-and-reduction-of-young-measure} can be applied to 
the polytropic pressure case with minor changes in the proof. 
We first need to verify the tightness of the random entropy dissipation measures \eqref{con-tightness-of-dissipation-measure-in-Lp-compensated-compactness-framework-polytropic-gas}.

\subsubsection{Tightness of the random entropy dissipation measures}\label{section-tightness-random-dissipation-measure-polytropic-gas}

For the polytropic pressure case, we need the following simpler lemma,
Lemma \ref{lem-properties-for-entropy-generate-by-compact-support-func-polytropic-gas}, 
regarding the properties of the entropy pairs with generating functions of compact support.

\begin{lemma}\label{lem-properties-for-entropy-generate-by-compact-support-func-polytropic-gas}
Let $\psi \in C_{\rm c}^2(\mathbb{R})$, and let $(\eta^{\psi},q^{\psi})$ be a weak entropy pair 
defined in \eqref{eq-entropy-flux-representation}. 
Then, for the polytropic pressure case,
$$
|\eta^{\psi}_{\rho}(\rho,m)+ u\eta^{\psi}_{m}(\rho,m)| \le C_{\psi}(1+\rho^{\theta}),\quad\ |\eta_{m}^{\psi}(\rho,m)|\le C_{\psi}.
$$
\end{lemma}

The proof can be achieved by a direct calculation and \cite[Lemma 2.1]{chenperepelitsa10CPAM}.

We can obtain a better tightness result for the random entropy dissipation measures since the entropy pairs have 
explicit formulas and better estimates.

\begin{proposition}\label{prop-H^-1-compactness-polytropic-gas}
Let $\psi \in C_{\rm c}^2(\mathbb{R})$, and let $(\eta^{\psi},q^{\psi})$ be a weak entropy pair defined in \eqref{eq-entropy-flux-representation}. Suppose that 
$(\rho^{\varepsilon},m^{\varepsilon})$ with $m^{\varepsilon}=\rho^{\varepsilon}u^{\varepsilon}$ 
is a sequence of solutions of 
the parabolic approximation equations \eqref{eq-parabolic-approximation}. 
Then, for the polytropic pressure case, the random entropy dissipation measures 
$$
\big\{\eta^{\psi}(\rho^{\varepsilon},m^{\varepsilon})_t +q^{\psi}(\rho^{\varepsilon},m^{\varepsilon})_x\big\}_{\varepsilon >0} 
\qquad\mbox{are tight in $H^{-1}_{\rm loc}(\mathbb{R}^2_+)$}.
$$
\end{proposition}

\begin{proof}
For the polytropic pressure case, by \cite[Lemma 2.1]{chenperepelitsa10CPAM}, we have
  \begin{equation}\label{iq-control-of-eta-q-with-compact-supported-psi}
  \begin{aligned}
  &|\eta^{\psi}(\rho,m)|+|q^{\psi}(\rho,m)| \le C_{\psi}\rho &&\qquad \text{when $\gamma\in \left(1,3\right]$},\\
  &|\eta^{\psi}(\rho,m)|\le C_{\psi}\rho,\ |q^{\psi}(\rho,m)|\le C_{\psi}\rho \max\{1,\rho^{\theta}\} 
  &&\qquad \text{when $\gamma>3$}.
  \end{aligned}
  \end{equation}
Then, by the higher integrability of the density and the velocity 
in \S\ref{sec-Higher-integrability-in-polytropic-gas-case}, 
we conclude that $\eta^{\psi}(\rho^{\varepsilon},m^{\varepsilon})$ and $q^{\psi}(\rho^{\varepsilon},m^{\varepsilon})$ are uniformly bounded in $L^1(\Omega;L^{q_1}_{\rm loc}(\mathbb{R}^2_+))$ with 
$q_1=\gamma+1$ for $\gamma \in \left(1,3\right]$, and $q_1=2+\frac{\gamma-3}{\gamma+1}$ for $\gamma >3$. 
Therefore, 
$\{\eta^{\psi}(\rho^{\varepsilon},m^{\varepsilon})_t +q^{\psi}(\rho^{\varepsilon},m^{\varepsilon})_x\}_{\varepsilon >0}$ 
is stochastically bounded  in $W^{-1,q_1}_{\rm loc}(\mathbb{R}^2_+)$ (see Lemma \ref{lem-stochastic-version-murat-lemma}). 
That is, for each $\beta >0$, there exists $C_{\beta}>0$ such that
  $$
  \sup_{\varepsilon >0}\mathbb{P}\big\{ \|\eta^{\psi}(\rho^{\varepsilon},m^{\varepsilon})_t +q^{\psi}(\rho^{\varepsilon},m^{\varepsilon})_x\|_{W^{-1,q_1}_{\rm loc}(\mathbb{R}^2_+)}> C_{\beta} \big\} < \beta.
  $$

On the other hand, as in the general pressure law case, 
we apply the It\^o formula to $\eta^{\psi}$ to obtain
   $$
   \begin{aligned}
  \d \eta^{\psi}(U^{\varepsilon})+\partial_x q^{\psi}(U^{\varepsilon}) \,\d t   
  &= \varepsilon \partial_x^2 \eta^{\psi}(U^{\varepsilon})\,\d t 
   -\varepsilon (\partial_x U^{\varepsilon})^{\top}\, \nabla^2\eta^{\psi}(U^{\varepsilon})\,\partial_x U^{\varepsilon} \,\d t\\
   &\quad\,+\partial_{m^{\varepsilon}}\eta^{\psi}(U^{\varepsilon})\Phi^{\varepsilon}(U^{\varepsilon}) \,\d W(t) 
   + \frac12 \partial_{m^{\varepsilon}}^2\eta^{\psi}(U^{\varepsilon})\sum_k a_k^2 (\zeta_k^{\varepsilon})^2 \,\d t.
   \end{aligned}
   $$
In the polytropic pressure case, we have the same results as 
Lemmas \ref{prop-tightness-of-partial-x-eta}--\ref{prop-tightness-of-quadratic-variation}, 
since the proofs can be directly applied to the polytropic pressure case with (i) 
in Lemma \ref{lem-properties-for-entropy-generate-by-compact-support-func} replaced 
by Lemma \ref{lem-properties-for-entropy-generate-by-compact-support-func-polytropic-gas}. 
Therefore, the same as that for the general pressure law case, 
$\{\varepsilon \partial_x^2 \eta^{\psi}(U^{\varepsilon})\}_{\varepsilon >0}$ 
and $\{\partial_t M^{\varepsilon}(t)\}_{\varepsilon >0}$ 
are tight in $H^{-1}_{\rm loc}(\mathbb{R}^2_+)$,
and $\{\varepsilon (\partial_x U^{\varepsilon})^{\top}\nabla^2\eta^{\psi}(U^{\varepsilon})
\partial_x U^{\varepsilon}\}_{\varepsilon >0}$ and 
$\{\frac12 \partial_{m^{\varepsilon}}^2\eta^{\psi}(U^{\varepsilon})\sum_k a_k^2 (\zeta_k^{\varepsilon})^2\}_{\varepsilon >0}$ 
are tight 
in $W^{-1,q}_{\rm loc}(\mathbb{R}^2_+)$ for any $q\in (1,2)$. These imply that
  \begin{equation}\label{eq-entropy-dissipation-measure-compactness-in-W-1q0-polytropic-gas}
  \{\eta^{\psi}(\rho^{\varepsilon},m^{\varepsilon})_t +q^{\psi}(\rho^{\varepsilon},m^{\varepsilon})_x\}_{\varepsilon >0}
  \qquad \text{are tight in $W^{-1,q_0}_{\rm loc}(\mathbb{R}^2_+)$ with $1\le q_0\le q$}.
  \end{equation}

Therefore, our conclusion follows from Lemma \ref{lem-stochastic-version-murat-lemma} and \eqref{eq-entropy-dissipation-measure-compactness-in-W-1q0-polytropic-gas}.
\end{proof}

Then we can also derive the same Tartar commutation relation for the polytropic pressure case, 
the proof of which is similar and easier 
with \eqref{iq-control-of-eta-q-with-compact-supported-psi-general-pressure-law} 
replaced by \eqref{iq-control-of-eta-q-with-compact-supported-psi}. 
Moreover, we can prove that 
$\div_{(t,x)}\, \tilde X_1^{\varepsilon}$ and $\text{curl}_{(t,x)}\, \tilde X_2^{\varepsilon}$ are tight 
in $H^{-1}_{\rm loc}(\mathbb{R}^2_+)$ and $H^{-1}_{\rm loc}(\mathbb{R}^2_+;\mathbb{R}^{2\times 2})$ respectively. 
Thus, the classical div-curl lemma \cite{murat78} is enough for the polytropic pressure case.

The remaining results in \S \ref{sec-compensated-compactness-and-reduction-of-young-measure} 
are all valid for the polytropic pressure case 
by taking $\gamma_1=\gamma_2=\gamma$ in those proofs, which do not restrict $\gamma\le 3$, {\it i.e.}, 
for all $\gamma>1$. 
In fact, for the reduction of the random Young measure 
(Proposition \ref{prop-reduction-of-young-measure}) in the polytropic pressure case, 
all the computations in \cite[\S 6--7]{chenperepelitsa10CPAM} with $u$ there 
replaced by $\frac{m}{\rho}$ can be applied. 
Therefore, we have established the stochastic $L^p$ compensated compactness 
framework for the polytropic pressure case.

\begin{proposition}\label{prop-reduction-of-young-measure-a}
Let $\tilde \mu$ be the limit random Young measure obtained in {\rm Proposition \ref{prop-apply-jakubowski-skorokhod-representation-to-take-limit}}. Then $\tilde{\mathbb{P}}$-{\it a.s.}, for {\it a.e.} $(t,x)\in \mathbb{R}^2_+$, 
either $\tilde \mu_{(t,x)}$ is concentrated on the vacuum region 
$\mathcal{V}:=\{(\rho,m)\in \mathbb{R}^2_+\,:\, \rho=0\}$ or $\tilde \mu_{(t,x)}$ is reduced to 
a Dirac mass $\delta_{(\rho^*(t,x,\omega),m^*(t,x,\omega))}$ 
on $\mathfrak{T}=\{(\rho,m)\in \mathbb{R}^2_+\,:\, \rho>0\}$.
\end{proposition}

\begin{proof}
If $\tilde \mu_{(t,x)}$ is not concentrated on the vacuum 
region $\{(\rho,m)\in \mathbb{R}^2_+\,:\, \rho=0\}$, we fix $(t,x,\omega)\in \Omega\times\mathbb{R}\times\mathbb{R}_+$ 
such that the Tartar commutation \eqref{eq-tartar-commutation} is satisfied. Then
we use \cite[Theorem 8.5]{chenhuangLiWangWang24CMP} to obtain 
that $\tilde \mu_{(t,x)}$ is reduced to a Dirac mass $\delta_{(\rho^*(t,x,\omega),m^*(t,x,\omega))}$ 
on the region $\{(\rho,m)\in \mathbb{R}^2_+\,:\, \rho>0\}$.
\end{proof}

However, to complete the proof of Theorem \ref{thm-well-posedness-for-euler-on-whole-space}, 
we need much higher integrability for the density. 
Thanks to the explicit formula of the entropy, 
we have obtained much higher integrability of the density in Proposition \ref{prop-uniform-estimates-rho-u3}. 
Note that $\gamma+\theta>\gamma+1$ when $\gamma>3$. 
Then, by the same arguments as those in the proofs of 
Propositions \ref{prop-test-func-and-convergence-for-random-young-measure}--\ref{prop-generalized-convergence-of-young-measure}, 
we can obtain better results with \eqref{iq-higher-integrability-for-young-measure} in Proposition \ref{prop-test-func-and-convergence-for-random-young-measure} replaced by
  \begin{equation}\label{iq-higher-integrability-for-young-measure-polytropic-gas}
  \tilde{\E} \big[\int_0^T \int_{\hat K} \int_{\mathfrak{T}} 
  \big(\rho P(\rho)+\rho^{\gamma+\theta}+\frac{|m|^3}{\rho^2}\big) \,\d \tilde \mu_{(t,x)}(\rho,m) \d x \d t\big] \le {\hat C},
  \end{equation}
and with \eqref{iq-higher-integrability-on-new-probability-space} in Proposition \ref{prop-generalized-convergence-of-young-measure} replaced by
  \begin{equation}\label{iq-higher-integrability-on-new-probability-space-polytropic-gas}
  \begin{aligned}
  \tilde{\E} \big[\Big(\int_0^T \int_{\hat K} \int_{\mathfrak{T}} \big(\rho P(\rho)+\rho^{\gamma+\theta}+\frac{|m|^3}{\rho^2} \big) \,\d \nu_{(t,x)}(\rho,m) \d x \d t\Big)^{p_0}\big] \le {\hat C}(p_0).
  \end{aligned}
  \end{equation}

\subsubsection{Completion of the proof of Theorem \ref{thm-well-posedness-for-euler-on-whole-space}}
We now prove Theorem \ref{thm-well-posedness-for-euler-on-whole-space}. 

For the energy estimate, Proposition \ref{prop-take-limit-obtain-relative-energy-estiate-for-Euler} can be applied directly 
to the polytropic pressure case for all $\gamma>1$ so that the same result can be obtained.
Thus, in this section, when Proposition \ref{prop-take-limit-obtain-relative-energy-estiate-for-Euler} is used, 
it means that we use the version for the polytropic pressure case. 

For the entropy inequality, the following lemma 
shows the growth of entropy pairs.

\begin{lemma}\label{lem-growth-of-entropy-and-derivatives-polytropic-gas}
Let $\psi\in C^2(\mathbb{R})$ be convex and satisfy
\begin{align*}
  &
  |\psi(s)| 
  \le C|s|^{2},\quad 
  |\psi^{\prime}(s)| \le C|s|, \qquad\qquad\qquad\quad\,\,\,
    \text{for $\gamma\le 3$},\\[0.5mm]
  &
  |\psi(s)| \le C|s|^{2},\quad 
  |\psi^{\prime}(s)| \le C|s|,\quad \psi^{\prime \prime}(s)\le C 
  \qquad \text{for $\gamma > 3$}.
\end{align*}
Then
 \begin{align}
  &|q^{\psi}(\rho,u)| +|\partial_m^2 \eta^{\psi}(\rho,u)|\sum_k\zeta_k^2 
  \le C\big(1+\rho|u|^3+ \rho^{\gamma+\theta} 
  \big),\\
  &|\eta^{\psi}(\rho,u)| +\rho|\partial_m\eta^{\psi}(\rho,u)|^2 \le C\big(1+\rho|u|^2+ \rho^{\gamma} 
  \big).\label{iq-lower-growth-for-entropy-in-entropy-inequality}
  \end{align}
\end{lemma}

\begin{proof}
The proof is similar to and easier than that of 
Lemma \ref{lem-growth-of-entropy-and-derivatives-general-pressure-law}, 
since the entropy has an explicit formula. 
Here we only show the proof for $|\partial_m^2 \eta^{\psi}(\rho,u)|\sum_k\zeta_k^2$ when $\gamma > 3$, 
since the other case is simpler. 
Recall the explicit formula for the entropy in the polytropic pressure case:
  $$
  \begin{aligned}
  \eta^{\psi}(\rho, m)=\int_{\mathbb{R}} [\rho^{2\theta}-(u-s)^2]^{\lambda}_{+}\psi(s) \d s
  =\rho \int_{-1}^1 \psi(u+\rho^{\theta}s)[1-s^2]^{\lambda}_{+} \d s,
  \end{aligned}
  $$
where $\lambda=\frac{3-\gamma}{2(\gamma-1)}>-\frac12$. 
Note that 
$\lambda=\frac{3-\gamma}{2(\gamma-1)}<0$ when $\gamma>3$. 
Then we have to use the additional 
condition $\psi^{\prime \prime}(s)\le C$ to obtain
  $$\begin{aligned}
  \partial_m^2 \eta^{\psi} (\rho,m)
  =\frac{1}{\rho} \int_{-1}^1 \psi^{\prime \prime}(\frac{m}{\rho}+\rho^{\theta}s)[1-s^2]^{\lambda}_{+} \d s
  \le \frac{C}{\rho}.
  \end{aligned}$$
Then the conclusion follows directly.
\end{proof}

By Lemma \ref{lem-growth-of-entropy-and-derivatives-polytropic-gas}, when $\psi$ is of subquadratic growth at infinity, 
the same arguments as those in the proof of Proposition \ref{prop-entropy-inequality-general-pressure-law} 
can be applied to the polytropic pressure case here, so we can obtain the desired entropy inequality 
as in Proposition \ref{prop-entropy-inequality-general-pressure-law}. 
Furthermore, the arguments in Proposition \ref{prop-weak-solution-for-general-pressure-law} 
can be applied directly to the polytropic pressure case for all $\gamma>1$, 
so that we can also prove that $(\tilde \rho, \tilde m)$ is a weak 
solution to \eqref{eq-stochastic-Euler-system}. 

To complete the proof of Theorem \ref{thm-well-posedness-for-euler-on-whole-space}, it remains to prove that, 
when $\psi(s)=\psi_E(s)=\frac12 s^2$, the entropy inequality is still satisfied. 
Inspired by the idea in \cite{chen-kang-vasseur2024arxiv}, 
to approximate $\psi_E(s)=\frac12 s^2$, we construct a new cut-off function $\psi_R$ as
  $$
  \begin{aligned}
  \psi_R(s)=\int_0^s \int_0^y \psi_R^{\prime \prime}(w) \,\d w \d y,\qquad \psi_R^{\prime \prime}(s)=\begin{cases}
  1\ &\,\,\,\text{if $|s|\le R$},\\
  \frac{2R-|s|}{R}\ &\,\,\,\text{if $R\le |s|\le 2R$},\\
  0\ &\,\,\,\text{if $|s|\ge 2R$}.
  \end{cases}
  \end{aligned}
  $$
Then we have
  \begin{equation}\label{eq-cut-off-approximation-of-energy-generating-function}
  \begin{aligned}
  \psi_R(s)=\begin{cases}
  \frac{|s|^2}{2}\ &\quad\text{if $|s|\le R$},\\
  \frac{R^2}{6}-\frac{R|s|}{2}+|s|^2 -\frac{|s|^3}{6R}\,\quad  &\quad\text{if $R\le |s|\le 2R$},\\
  -\frac{7R^2}{6}+\frac{3R}{2}|s|\ &\quad\text{if $|s|\ge 2R$}.
  \end{cases}
  \end{aligned}
  \end{equation}
The corresponding entropy and entropy flux are
  $$
  \begin{aligned}
  \eta^R(\rho, \rho u)
  &=\rho \int_{-1}^1 \psi_R(u+\rho^{\theta}s)[1-s^2]^{\lambda}_{+}\, \d s,\\
  q^R(\rho, \rho u)
  &=\rho \int_{-1}^1 (u+\theta\rho^{\theta}s)\psi_R(u+\rho^{\theta}s)[1-s^2]^{\lambda}_{+} \,\d s.
  \end{aligned}
  $$
To estimate $q^R$, we decompose the interval $[-1,1]$ into three parts: 
$I_1:=\{s\in [-1,1]\,:\,|u+\rho^{\theta}s|\le R\}$, 
$I_2:=\{s\in [-1,1]\,:\,R<|u+\rho^{\theta}s|\le 2R\}$, 
and $I_3:=\{s\in [-1,1]\,:\,|u+\rho^{\theta}s|> 2R\}$, and obtain
  \begin{equation}\label{iq-estimate-for-cut-off-entropy-flux-fix-R}
  \begin{aligned}
  q^R (\rho,\rho u)
  &=\rho \int_{I_1} (u+\theta \rho^{\theta}s)\frac12 (u+\rho^{\theta}s)^2 [1-s^2]^{\lambda}_+ \,{\rm d} s\\
  &\quad+\rho \int_{I_2} (u+\theta \rho^{\theta}s)\Big(\frac{R^2}{6} -\frac{R}{2}|u+\rho^{\theta}s|+(u+\rho^{\theta}s)^2 -\frac{1}{6R}|u+\rho^{\theta}s|^3 \Big) [1-s^2]^{\lambda}_+ \,{\rm d} s\\
  &\quad+\rho \int_{I_3} (u+\theta \rho^{\theta}s)\Big( -\frac{7R^2}{6}+ \frac{3R}{2} |u+\rho^{\theta}s|\Big) [1-s^2]^{\lambda}_+ \,{\rm d} s\\
  &\le CR \rho \big(|u|+ \rho^{\theta} \big)^2.
  \end{aligned}
  \end{equation}
Similarly, we have 
  \begin{equation}\label{iq-estimate-for-cut-off-entropy-fix-R}
  \eta^R (\rho,\rho u) \le CR \rho \big(|u|+ \rho^{\theta} \big).
  \end{equation}
Therefore, for fixed $R$, by the uniform bound \eqref{iq-higher-integrability-on-new-probability-space-polytropic-gas} and Proposition \ref{prop-take-limit-obtain-relative-energy-estiate-for-Euler} (recall that we have the same result for the polytropic pressure case), estimates \eqref{iq-estimate-for-cut-off-entropy-flux-fix-R}--\eqref{iq-estimate-for-cut-off-entropy-fix-R} 
imply the uniform integrability with respect to $\varepsilon$. 
Thus, similar to the proof in Proposition \ref{prop-entropy-inequality-general-pressure-law}, 
the Vitali convergence theorem implies that, for $\tilde{\mathbb{P}}\text{-{\it a.s.}}$,
  $$
  \begin{aligned}
  &\int_0^t\int_{\mathbb{R}}  \eta^R(\tilde U^{\varepsilon}) \, \varphi_{\tau}\,\d x \d \tau
  \to \int_0^t\int_{\mathbb{R}}  \eta^R(\tilde U) \,\varphi_{\tau}\,\d x \d \tau\qquad
    \text{for any $t\in [0,T]$},\\
  &\int_0^t \int_{\mathbb{R}} q^R(\tilde U^{\varepsilon})  \varphi_x \,\d x \d \tau\to \int_0^t \int_{\mathbb{R}} q^R(\tilde U)  \varphi_x \,\d x \d \tau\qquad \text{for any $t\in [0,T]$}.
  \end{aligned}
  $$
 
Again, similar to the proof in Proposition \ref{prop-entropy-inequality-general-pressure-law}, 
\eqref{iq-higher-integrability-on-new-probability-space-polytropic-gas} also implies that, for any $t\in [0,T]$, 
  $$
  \varepsilon \int_0^t \int_{\mathbb{R}} \eta^R(\tilde U^{\varepsilon})\partial_x^2 \varphi \,\d x \d \tau \to 0 
  \qquad \tilde{\mathbb{P}}\text{-{\it a.s.}}.
  $$

Since
  $$
  \partial_m^2 \eta^R(\rho, \rho u)
  =\frac{1}{\rho} \int_{-1}^1 \psi_R^{\prime\prime}(u+\rho^{\theta}s)[1-s^2]^{\lambda}_{+}\, \d s\le \frac{C}{\rho},
  $$
we have
  $$
  \partial_m^2 \eta^R(\rho, \rho u)\sum_k a_k^2 \zeta_k^2 \le C\rho \big(1+\frac{m^2}{\rho^2} + \rho^{\gamma-1} \big).
  $$
Since the constructed cut-off function $\psi_R\in C^2$ and convex, so is the entropy $\eta^R$. 
Therefore, we can apply the same arguments as in the proof of \eqref{iq-convergence-of-quadratic-variation-term} 
to obtain that, for all $t\in[0,T]$,
  $$
  \int_0^t \int_{\mathbb{R}} \varphi \partial_{\tilde m^{\varepsilon}}^2 \eta^R(\tilde U^{\varepsilon}) \sum_k a_k^2\,(\zeta_k^{\varepsilon})^2(\tilde U^{\varepsilon}) \,\d x \d \tau \to \int_0^t \int_{\mathbb{R}} \varphi \partial_{\tilde m}^2 \eta^R(\tilde U) \sum_k a_k^2 \,\zeta_k^2(\tilde U) \,\d x \d \tau\qquad \tilde{\mathbb{P}}\text{-{\it a.s.}}.
  $$

A direct calculation shows that
  $$
  \partial_m \eta^R(\rho, \rho u)
  =\int_{-1}^1 \psi_R^{\prime}(u+\rho^{\theta}s)[1-s^2]^{\lambda}_{+}\, \d s\le CR.
  $$
Thus, similar to \eqref{iq-estimate-for-partial-m-eta-noise-in-entropy-inequality-general-pressure}, we have
  $$\begin{aligned}
  &\tilde \E \big[\int_0^T \sum_k \Big|\int_{\mathbb{R}} \varphi \partial_{\tilde m} \eta^R(\tilde U) \zeta_k(\tilde U) \,\d x\Big|^2 \,\d t\big]\\
  &\le C(\varphi,B_0,\gamma,\rho_{\infty})\tilde \E \big[\sup_{t\in[0,T]} \Big(\int_{{\rm supp}_x \varphi} \big(1+\frac{\tilde m^2}{\tilde \rho}+ e^*(\tilde \rho,\rho_{\infty})\big) \,\d x \Big)^2 \big]\\
  &\quad+ C(\varphi,R,\gamma,\rho_{\infty})\tilde \E\big[ \Big(\int_0^T \int_{{\rm supp}_x \varphi} \big(1+ e^*(\tilde \rho,\rho_{\infty}) \big) \d x\, \d t \Big)^2 \big]\\
  &\le  C(T,E_{0,2},R),
  \end{aligned}$$
where, in the 
first inequality, we have used similar arguments to those 
for \eqref{iq-estimate-for-noise-term-1/rho--in-energy-estimate} and, in the last inequality, 
we have used Proposition \ref{prop-take-limit-obtain-relative-energy-estiate-for-Euler}.
Similarly,
using \eqref{iq-noise-coefficience-growth-conditions-for-parabolic-approximation} and \eqref{iq-energy-bound-for-tilde-rho-m-varepsilon},
we have
  $$
  \begin{aligned}
  \tilde \E \big[\int_0^T \sum_k \Big|\int_{\mathbb{R}} \varphi \partial_{\tilde m^{\varepsilon}} \eta^R(\tilde U^{\varepsilon}) \zeta_k^{\varepsilon}(\tilde U^{\varepsilon}) \,\d x\Big|^2 \,\d t\big]
  \le  C(T,E_{0,2}).
  \end{aligned}
  $$
Then, by the same argument as the justification 
of \eqref{iq-convergence-of-stochastic-integral-in-entropy-inequality-general-pressure}, we obtain that, for {\it a.e.} $t\in[0,T]$,
  $$
  \int_0^t \int_{\mathbb{R}} \varphi \partial_{\tilde m^{\varepsilon}} \eta^R (\tilde U^{\varepsilon}) \Phi^{\varepsilon}( \tilde U^{\varepsilon}) \,\d x \d \tilde W^{\varepsilon}(\tau) \to \int_0^t \int_{\mathbb{R}} \varphi \partial_{\tilde m} \eta^R (\tilde U) \Phi( \tilde U) \,\d x \d \tilde W(\tau)\qquad \tilde{\mathbb{P}}\text{-{\it a.s.}}.
  $$

Combining the above results gives that, $\tilde{\mathbb{P}}$-{\it a.s.} for {\it a.e.} $t\in [0,T]$,
$$
    \begin{aligned}
    &\int_0^T \int_{\mathbb{R}} \eta^R(\tilde  U)\,\varphi_t\, \d x \d t
    + \int_0^T \int_{\mathbb{R}} q^R(\tilde U)   \varphi_x \,\d x \d t\\
    &+ \int_0^T \int_{\mathbb{R}} \varphi \partial_{\tilde m} \eta^R (\tilde U) \Phi( \tilde U) \, \d x \d \tilde W(t)
    +\frac12 \int_0^T \int_{\mathbb{R}} \varphi \partial_{\tilde m}^2 \eta^R(\tilde U) \sum_k a_k^2 \,\zeta_k^2(\tilde U)\, \d x \d t
    \ge  0.
    \end{aligned}
$$

We now pass to the limit: $R\to \infty$.
From the first equality in \eqref{iq-estimate-for-cut-off-entropy-flux-fix-R}, we have
  $$
  q^R(\rho,\rho u)\le C \rho \big(|u|+\rho^{\theta}\big)^3,
  $$
where $C$ is independent of $R$. 
Then, by \eqref{iq-higher-integrability-on-new-probability-space-polytropic-gas}, 
this bound and \eqref{iq-estimate-for-cut-off-entropy-fix-R} imply the uniform integrability with respect to $R$. 
On the other hand, notice that, in the vacuum region, $q^R=q_E=\eta^R=\eta_E=0$ (in this proof, 
we always use this fact and will not repeat it again).
In the non-vacuum region, we have
  \begin{equation}\label{iq-difference-between-q-R-and-q-E}
  \begin{aligned}
  &q^R (\rho,\rho u)-q_E(\rho,\rho u)\\
  &=\rho \int_{I_2} (u+\theta\rho^{\theta}s) \Big(\frac{R^2}{6}+ \frac12 (u+\rho^{\theta}s)^2 - \frac{1}{6R}|u+\rho^{\theta}s|^3 -\frac{R}{2}|u+\rho^{\theta}s| \Big)[1-s^2]_+^{\lambda} \, d s\\
  &\quad+\rho \int_{I_3} (u+\theta\rho^{\theta}s) \Big( -\frac12 (u+\rho^{\theta}s)^2 + \frac{3R}{2}|u+\rho^{\theta}s|-\frac{7R^2}{6}  \Big)[1-s^2]_+^{\lambda} \, d s\\
  &\le C\rho(|u|+\rho^{\theta})^3 \int_{\{s\,:\,|u+\rho^{\theta}s|\ge R\}} [1-s^2]_+^{\lambda} \, \d s,
  \end{aligned}
  \end{equation}
and similarly,
  \begin{equation}\label{iq-difference-between-eta-R-and-eta-E}
  \eta^R (\rho,\rho u)-\eta_E (\rho,\rho u)
  \le C\rho(|u|+\rho^{\theta})^2 \int_{\{\,s:\,|u+\rho^{\theta}s|\ge R\}} [1-s^2]_+^{\lambda} \, \d s,
  \end{equation}
where the measure of set $\{s\,:\,|u+\rho^{\theta}s|\ge R\}$ tends to zero when $R\to \infty$. 
Thus, it follows from \eqref{iq-higher-integrability-on-new-probability-space-polytropic-gas} 
and \eqref{iq-difference-between-q-R-and-q-E} that, as $R\to \infty$, 
  $$
  \begin{aligned}
  q^R (\tilde\rho,\tilde\rho \tilde u) \to q_E(\tilde\rho,\tilde\rho \tilde u)
  \qquad\text{$\tilde{\mathbb{P}}$-{\it a.s.} almost everywhere}.
  \end{aligned}
  $$
By Proposition \ref{prop-take-limit-obtain-relative-energy-estiate-for-Euler} 
and \eqref{iq-difference-between-eta-R-and-eta-E}, we also obtain 
that $\tilde{\mathbb{P}}$-{\it a.s.} for almost every $t\in[0,T]$,
$$
  \eta^R (\tilde\rho,\tilde\rho \tilde u) \to \eta_E (\tilde\rho,\tilde\rho \tilde u)\qquad \text{$\tilde{\mathbb{P}}$-{\it a.s.} {\it a.e.}}.
$$
Therefore, using these convergence results, combined with the integrability ensured 
by \eqref{iq-higher-integrability-on-new-probability-space-polytropic-gas} and Propositions \ref{prop-take-limit-obtain-relative-energy-estiate-for-Euler}, 
we apply the dominated convergence theorem to obtain that, 
for $\tilde{\mathbb{P}}\text{-{\it a.s.}}$,  
  $$
  \begin{aligned}
  &\int_0^t \int_{\mathbb{R}} \eta^R(\tilde  U)\,\varphi_\tau\, \d x \d \tau
  \to \int_0^t \int_{\mathbb{R}} \eta_E(\tilde  U)\,\varphi_\tau\, \d x \d \tau
  \qquad \text{for all $t\in[0,T]$},\\
  &\int_0^t \int_{\mathbb{R}} q^R(\tilde U)   \varphi_x \,\d x \d \tau 
  \to \int_0^t \int_{\mathbb{R}} q_E(\tilde U)   \varphi_x \,\d x \d \tau \qquad \text{for all $t\in[0,T]$}.
  \end{aligned}
  $$

We observe that  
  \begin{align}
  &\big(\partial_m^2 \eta^R(\rho, \rho u)-\partial_m^2 \eta_E(\rho, \rho u) \big)\sum_k a_k^2 \zeta_k^2\nonumber\\
  &=\frac{1}{\rho} \int_{-1}^1 \big(\psi_R^{\prime\prime}(u+\rho^{\theta}s)-\psi_E^{\prime\prime}(u+\rho^{\theta}s) \big)[1-s^2]^{\lambda}_{+}\, \d s \sum_k a_k^2 \zeta_k^2\nonumber\\
  &\le C\rho \Big(1+\frac{m^2}{\rho^2}+e(\rho)\Big) \int_{-1}^1 \big|\psi_R^{\prime\prime}(u+\rho^{\theta}s)-\psi_E^{\prime\prime}(u+\rho^{\theta}s) \big|[1-s^2]^{\lambda}_{+}\, \d s,\label{eq-partial-m2-eta-R-minus-eta-E}
  \end{align}
where, in the last inequality, we have used the growth 
condition \eqref{iq-noise-coefficience-growth-condition-on-R-compact-supp} for the noise coefficient. 
Since, as $R \to \infty$,
  $$
  \psi_R^{\prime\prime}(u+\rho^{\theta}s)-\psi_E^{\prime\prime}(u+\rho^{\theta}s)
  =\big(1-\frac{|u+\rho^{\theta}s|}{R}\big)\mathbf{1}_{\{s\,:\,R\le|u+\rho^{\theta}s|< 2R\}}
  +\mathbf{1}_{\{s\,:\,|u+\rho^{\theta}s|\ge 2R\}} \to 0,
  $$
by \eqref{iq-higher-integrability-on-new-probability-space-polytropic-gas} 
and the bounded convergence theorem, we have
  $$
  \big( \partial_{\tilde m}^2 \eta^R(\tilde\rho, \tilde\rho \tilde u)-\partial_{\tilde m}^2 \eta_E(\tilde\rho, \tilde\rho \tilde u) \big)\sum_k a_k^2 \zeta_k^2 \to 0  \qquad \text{$\tilde{\mathbb{P}}$-{\it a.s.} almost everywhere \quad as $R\to \infty$}. 
  $$ 
Additionally, \eqref{eq-partial-m2-eta-R-minus-eta-E} also implies that 
  $$\begin{aligned}
  |\partial_m^2 \eta^R(\rho, \rho u)-\partial_m^2 \eta_E(\rho, \rho u)|\sum_k a_k^2 \zeta_k^2 
  \le C\rho \Big(1+\frac{m^2}{\rho^2}+e(\rho)\Big),
  \end{aligned}$$
which, by \eqref{iq-higher-integrability-on-new-probability-space-polytropic-gas} again, 
enables us to apply the dominated convergence theorem to obtain that, for almost every $t\in[0,T]$,
  $$
  \int_0^t \int_{\mathbb{R}} \varphi \partial_{\tilde m}^2 \eta^R(\tilde U) \sum_k a_k^2 \,\zeta_k^2(\tilde U) \,\d x \d \tau \to \int_0^t \int_{\mathbb{R}} \varphi \partial_{\tilde m}^2 \eta_E(\tilde U) \sum_k a_k^2 \,\zeta_k^2(\tilde U) \,\d x \d \tau\qquad \tilde{\mathbb{P}}\text{-{\it a.s.}}.
  $$

Notice that 
  $$
  \begin{aligned}
  \psi_R^{\prime}(y)=\begin{cases}
  y\ &\quad\text{if $|y|\le R$},\\
  -\frac{R}{2}+2y-\frac{1}{2R}y^2 \ &\quad\text{if $R <y\le 2R$},\\
  \frac{R}{2}+2y+\frac{1}{2R}y^2 \ &\quad\text{if $-2R \le y< -R$},\\
  \frac{3R}{2} \ &\quad\text{if $ y> 2R$},\\
  -\frac{3R}{2}\ &\quad\text{if $ y< -2R$}.
  \end{cases}
  \end{aligned}
  $$
Then we have
  \begin{align}
  &\big(\partial_m \eta^R(\rho, \rho u)-\partial_m \eta_E(\rho, \rho u)\big) \zeta_k\nonumber\\
  &\le  \Big(\int_{I_2} \big( |u+\rho^{\theta}s| +\frac{R}{2} +\frac{1}{2R} (u+\rho^{\theta}s)^2 \big)[1-s^2]^{\lambda}_{+}\, \d s
  +  \int_{I_3} \big( |u+\rho^{\theta}s| +\frac{3R}{2} \big)[1-s^2]^{\lambda}_{+}\, \d s \Big)\zeta_k
    \nonumber\\
  &\le C \big(|u|+\rho^{\theta}\big) \zeta_k \int_{\{s:|u+\rho^{\theta}s|\ge R\}} [1-s^2]_+^{\lambda} \, \d s
    \nonumber\\
  &\le C\big( \rho (|u|+\rho^{\theta})^2 + \frac{1}{\rho} \zeta_k^2 \big)\int_{\{s:|u+\rho^{\theta}s|\ge R\}} [1-s^2]_+^{\lambda} \, \d s,\label{iq-estimate-for-difference-of-partial-m-eta-R-and-eta-E}
  \end{align}
where $C$ is independent of $R$. 
Thus, by the growth condition of the noise coefficient \eqref{iq-noise-coefficience-growth-condition-on-R-compact-supp} and \eqref{iq-higher-integrability-on-new-probability-space-polytropic-gas} again, for each $k\in\mathbb{N}$, as $R\to \infty$, 
  $$
  \big(\partial_m \eta^R(\rho, \rho u)-\partial_m \eta_E(\rho, \rho u)\big) \zeta_k \to 0
  \qquad \mbox{$\tilde{\mathbb{P}}$-{\it a.s.} almost everywhere}.
  $$
On the other hand, \eqref{iq-estimate-for-difference-of-partial-m-eta-R-and-eta-E} also implies that 
  $$
  \begin{aligned}
  &\tilde \E \big[\int_0^T \sum_k \Big|\int_{\mathbb{R}}  \big(\partial_{\tilde m} \eta^R(\tilde U) -\partial_{\tilde m} \eta_E(\tilde U)\big)\zeta_k(\tilde U) \,\varphi\,\d x\Big|^2 \,\d t\big]\\
  &\le  C(\varphi,\gamma,\rho_{\infty})\tilde \E\big[ \Big(\int_0^T \int_{{\rm supp}_x (\varphi)} 
  \big(1+\frac{\tilde m^2}{\tilde \rho} + e^*(\tilde \rho,\rho_{\infty}) \big) \d x\, \d t \Big)^2 \big]\\
  &\quad +C(\varphi,B_0,\gamma,\rho_{\infty})\tilde \E \big[\sup_{t\in[0,T]} \Big(\int_{{\rm supp}_x (\varphi)} 
  \big(1+\frac{\tilde m^2}{\tilde \rho}+ e^*(\tilde \rho,\rho_{\infty})\big) \,\d x \Big)^2 \big]\\
  &\le  C(T,E_{0,2}).
  \end{aligned}
  $$
Therefore, using a similar argument to the justification 
of \eqref{eq-convergence-of-stochastic-integral-hilbert-schmit-norm} 
but applying the dominated convergence theorem 
(ensured by \eqref{iq-estimate-for-difference-of-partial-m-eta-R-and-eta-E}) twice to pass to the limit that 
  $$
  \tilde \E \big[\int_0^T \sum_{k=1}^{N_0} \Big|\int_{\mathbb{R}} \varphi \big(\partial_{\tilde m} \eta^R(\tilde U) -\partial_{\tilde m} \eta_E(\tilde U)\big) a_k\,\zeta_k(\tilde U) \,\d x\Big|^2 \,\d t\big]\to 0\qquad \text{as $R\to \infty$,}
  $$
where $N_0$ is large enough as in the justification of \eqref{eq-convergence-of-stochastic-integral-hilbert-schmit-norm}, 
we obtain by the same arguments as the justification of 
\eqref{iq-convergence-of-stochastic-integral-in-entropy-inequality-general-pressure} that, for {\it a.e.} $t\in[0,T]$,
  $$
  \int_0^t \int_{\mathbb{R}} \varphi \partial_{\tilde m} \eta^R (\tilde U) \Phi( \tilde U) \,\d x \d \tilde W(\tau) \to \int_0^t \int_{\mathbb{R}} \varphi \partial_{\tilde m} \eta_E (\tilde U) \Phi( \tilde U) \,\d x \d \tilde W(\tau)
  \qquad \tilde{\mathbb{P}}\text{-{\it a.s.}}.
  $$
Finally, we combine the above results to obtain that, $\tilde{\mathbb{P}}$-{\it a.s.},
  $$
  \begin{aligned}
  &\int_0^T \int_{\mathbb{R}}  \eta_E(\tilde  U)\,\varphi_t\, \d x\d t 
  + \int_0^T \int_{\mathbb{R}} q_E(\tilde U)  \varphi_x \,\d x \d t\\
  &+ \int_0^T \int_{\mathbb{R}} \partial_{\tilde m} \eta_E (\tilde U) \Phi( \tilde U) \, \varphi\,\d x \d \tilde W(t)
  +\frac12 \int_0^T \int_{\mathbb{R}}  \partial_{\tilde m}^2 \eta_E(\tilde U) \sum_k a_k^2 \,\zeta_k^2(\tilde U) \,\varphi\,\d x \d t
  \ge  0.  
  \end{aligned}
  $$
This completes the proof of Theorem \ref{thm-well-posedness-for-euler-on-whole-space}.

\subsection{Higher-Order Energy Estimate and More Convex Entropy Pairs for the Entropy Inequality}
For the polytropic pressure case, we also have a higher-order energy estimate as 
in Proposition \ref{prop-rho|u|4-general-pressure-law}. 
The proof is similar to and much simpler than that for 
Proposition \ref{prop-rho|u|4-general-pressure-law}, once we note that the higher-order 
energy $\eta_{\diamond}$ is exactly the entropy $\eta^{\psi_{\diamond}}$ with generating 
function $\psi_{\diamond}=\frac{1}{12}s^4$ so that this entropy $\eta_{\diamond}$ is convex.

In this subsection, we aim at establishing the following theorems:

\begin{theorem}\label{thm-better-well-posedness-for-euler-on-whole-space-deterministic-initial-data}
Assume that the initial data $(\rho_0, m_0)\in L^{\gamma}_{\rm loc}(\mathbb{R})\times L^1_{\rm loc}(\mathbb{R})$ 
satisfy $\mathcal{E}(\rho_0,m_0)<\infty$ and additionally have a finite relative higher-order energy{\rm :}
  $$\begin{aligned}
  &\mathcal{E}_{\diamond}(\rho_0,m_0):=\int_{\mathbb{R}} \big(\frac{1}{12}\frac{m_0^4}{\rho_0^3}+\frac{e(\rho_0)}{\rho_0}m_0^2+e_{\diamond}^*(\rho_0,\rho_{\infty}) \big)\, \d x <\infty.
  \end{aligned}$$
Then 
\begin{itemize}
\item[\rm (i)] There exists a martingale entropy solution with finite relative-energy to \eqref{eq-stochastic-Euler-system} in the sense of {\rm Definition \ref{def-martingale-weak-entropy-solution-with-relative-finite-energy-polytropic-gas}} with \eqref{eq-definition-of-subquadratic-growth} replaced by 
\begin{equation}\label{eq-definition-of-sub-cubic-for-weight-function-polytropic-gas}
  \begin{aligned}
  &\lim_{s\to \pm\infty}\frac{\psi(s)}{s^3}=0,\quad \lim_{s\to \pm\infty}\frac{\psi^{\prime}(s)}{s^2}=0
     \qquad\qquad\qquad\qquad\qquad\quad \text{for $\gamma\le 3$},\\
  &\lim_{s\to \pm\infty}\frac{\psi(s)}{s^3}=0,\quad \lim_{s\to \pm\infty}\frac{\psi^{\prime}(s)}{s^2}=0,\quad \lim_{s\to \pm\infty}\frac{\psi^{\prime \prime}(s)}{s}=0 \qquad\,\, \,\text{for $\gamma > 3$}.
  \end{aligned}
  \end{equation}
\item[\rm (ii)] Let $U^{\delta}:=(\rho^{\delta},m^{\delta})$ be a sequence of entropy solution to \eqref{eq-stochastic-Euler-system} obtained in {\rm (i)} corresponding to a sequence of initial data $(\rho_0^{\delta},m_0^{\delta})$ satisfying 
  $
  \mathcal{E}_{\diamond}(\rho_0^{\delta},m_0^{\delta}) \le C,
  $
where constant $C$ is independent of $\delta$. 
Then $U^{\delta}$ satisfies that, for any $K\Subset \mathbb{R}$, there exists $C(K,p,T)>0$  independent of $\delta$ such that, for any $t\in[0,T]$,
\begin{equation}\label{iq-higher-integrability-in-compactness-of-solu-gamma-law-deterministic-initial}
  \begin{aligned}
  \E \big[\Big( \int_0^t \int_{K} \big(\rho^{\delta} P(\rho^{\delta})+ \frac{|m^{\delta}|^3}{(\rho^{\delta})^2} +\frac{(m^{\delta})^4}{(\rho^{\delta})^3}+(\rho^{\delta})^{\gamma-2} (m^{\delta})^2+ (\rho^{\delta})^{2\gamma-1} \big) \,\d x\d \tau \Big)^p\big]
  \le  C(K,p,T),
  \end{aligned}
  \end{equation}
where $P(\rho)$ satisfies \eqref{eq-gamma-law}, so that $U^{\delta}:=(\rho^{\delta},m^{\delta})$ are tight in 
$L^{\gamma+1}_{\rm w,loc}(\mathbb{R}^2_+) \times L^{\frac{3(\gamma+1)}{\gamma+3}}_{\rm w,loc}(\mathbb{R}^2_+)$ 
and admit the Skorokhod representation $(\tilde \rho^{\delta}, \tilde m^{\delta})$. Moreover, there exist random variables $(\tilde\rho,\tilde m)$ such that 
almost surely $(\tilde \rho^{\delta},\,\tilde m^{\delta}) \to (\tilde \rho,\, \tilde m)$ 
almost everywhere and $($up to a subsequence$)$ almost surely,
$$
(\tilde \rho^{\delta},\,\tilde m^{\delta}) \to (\tilde \rho,\, \tilde m)\qquad 
\text{in $L^{\bar p}_{\rm loc}(\mathbb{R}^2_+)\times L^{\bar q}_{\rm loc}(\mathbb{R}^2_+)$}
$$
for $\bar p\in [1, \gamma+1 )$ and $\bar q\in [ 1, \frac{3(\gamma+1)}{\gamma+3} )$. In addition, $(\tilde\rho,\tilde m)$ is also an entropy solution to \eqref{eq-stochastic-Euler-system}, {\it i.e.}, almost surely satisfies the entropy inequality \eqref{iq-entropy-inequality-for-Euler-general-pressure}.
\end{itemize}
\end{theorem}

Similarly, the above result can be more general, allowing random initial data.
\begin{theorem}\label{thm-better-well-posedness-for-euler-on-whole-space}
Assume that $\Im$ is a Borel probability measure on $L^{\gamma}_{\rm loc}(\mathbb{R})\times L^1_{\rm loc}(\mathbb{R})$ satisfying
  $$\begin{aligned}
  &{\rm supp}\, \Im=\{(\rho,m)\in \mathbb{R}^2_+\,:\,\rho \ge 0\},\\[1mm]
  &\int_{L^{\gamma}_{\rm loc}(\mathbb{R})\times L^1_{\rm loc}(\mathbb{R})} \big| \mathcal{E}(\rho,m)\big|^{2}\, \d \Im(\rho,m)+\int_{L^{\gamma}_{\rm loc}(\mathbb{R})\times L^1_{\rm loc}(\mathbb{R})} \big| \mathcal{E}_{\diamond}(\rho,m)\big|^{2}\, \d \Im(\rho,m) <\infty.
  \end{aligned}$$
Then 
\begin{itemize}
\item[\rm (i)] There exists a martingale entropy solution with finite relative-energy 
to \eqref{eq-stochastic-Euler-system}, having initial 
law $\Im=\mathbb{P}\circ (\rho_0,m_0)^{-1}$,
in the sense of {\rm Definition 
\ref{def-martingale-weak-entropy-solution-with-relative-finite-energy-polytropic-gas}} 
with \eqref{eq-definition-of-subquadratic-growth} replaced by 
\eqref{eq-definition-of-sub-cubic-for-weight-function-polytropic-gas}{\rm ;}
\item[\rm (ii)] Let $U^{\delta}:=(\rho^{\delta},m^{\delta})$ be a sequence of entropy solution to \eqref{eq-stochastic-Euler-system} obtained in {\rm (i)} corresponding to a sequence 
of initial law $\Im^{\delta}=\mathbb{P}\circ (\rho_0^\delta,m_0^\delta)^{-1}$
satisfying 
    $$\begin{aligned}
  \int_{L^{\gamma}_{\rm loc}(\mathbb{R})\times L^1_{\rm loc}(\mathbb{R})} \big| \mathcal{E}(\rho,m)\big|^{2}\, \d \Im^{\delta}(\rho,m)+\int_{L^{\gamma}_{\rm loc}(\mathbb{R})\times L^1_{\rm loc}(\mathbb{R})} \big| \mathcal{E}_{\diamond}(\rho,m)\big|^{2}\, \d \Im^{\delta}(\rho,m) \le C,
  \end{aligned}$$
where constant $C>0$ is independent of $\delta$. 
Then $U^{\delta}$ satisfies that, for any $K\Subset \mathbb{R}$, there exists $C(K,p,T)>0$  independent of $\delta$ such that, for any $t\in[0,T]$,
\begin{equation}\label{iq-higher-integrability-in-compactness-of-solu-gamma-law}
  \begin{aligned}
  \E \big[\Big( \int_0^t \int_{K} \big(\rho^{\delta} P(\rho^{\delta})+ \frac{|m^{\delta}|^3}{(\rho^{\delta})^2} +\frac{(m^{\delta})^4}{(\rho^{\delta})^3}+(\rho^{\delta})^{\gamma-2} (m^{\delta})^2+ (\rho^{\delta})^{2\gamma-1} \big) \,\d x\d \tau \Big)^p\big]
  \le  C(K,p,T),
  \end{aligned}
  \end{equation}
where $P(\rho)$ satisfies \eqref{eq-gamma-law}, so that $U^{\delta}:=(\rho^{\delta},m^{\delta})$ are tight in 
$L^{\gamma+1}_{\rm w,loc}(\mathbb{R}^2_+) \times L^{\frac{3(\gamma+1)}{\gamma+3}}_{\rm w,loc}(\mathbb{R}^2_+)$ 
and admit the Skorokhod representation $(\tilde \rho^{\delta}, \tilde m^{\delta})$. Moreover, there exist random variables $(\tilde\rho,\tilde m)$ such that 
almost surely $(\tilde \rho^{\delta},\,\tilde m^{\delta}) \to (\tilde \rho,\, \tilde m)$ 
almost everywhere and $($up to a subsequence$)$ almost surely,
$$
(\tilde \rho^{\delta},\,\tilde m^{\delta}) \to (\tilde \rho,\, \tilde m)\qquad 
\text{in $L^{\bar p}_{\rm loc}(\mathbb{R}^2_+)\times L^{\bar q}_{\rm loc}(\mathbb{R}^2_+)$}
$$
for $\bar p\in [1, \gamma+1 )$ and $\bar q\in [ 1, \frac{3(\gamma+1)}{\gamma+3} )$. In addition, $(\tilde\rho,\tilde m)$ is also an entropy solution to \eqref{eq-stochastic-Euler-system}, {\it i.e.}, almost surely satisfies the entropy inequality \eqref{iq-entropy-inequality-for-Euler-general-pressure}.
\end{itemize}
\end{theorem}

As before, it remains to prove Theorem \ref{thm-better-well-posedness-for-euler-on-whole-space}.
Thanks to the explicit formula for the entropy pair in the polytropic pressure case, 
we have better control of the growth of the entropy pair, 
which is stated in 
Lemma \ref{lem-higher-order-growth-of-entropy-and-derivatives-polytropic-gas} below, 
and thus we can obtain 
the entropy inequality for more entropy pairs,
compared to the general pressure law case.

\begin{lemma}\label{lem-higher-order-growth-of-entropy-and-derivatives-polytropic-gas}
Let $\psi\in C^2(\mathbb{R})$ be convex and satisfy that there exists $\delta\ge 0$ such that
\begin{align*}
&|\psi(s)| \le C|s|^{3-\delta},\quad |\psi^{\prime}(s)| \le C|s|^{2-\delta}
 \qquad\qquad\qquad \qquad\qquad\text{for $\gamma\le 3$},\\
&|\psi(s)| \le C|s|^{3-\delta},\quad |\psi^{\prime}(s)|\le C|s|^{2-\delta},
\quad \psi^{\prime \prime}(s) \le C|s|^{1-\delta} 
\qquad \text{for $\gamma > 3$}.
\end{align*}
Then 
$$
  |\eta^{\psi}(\rho,u)|+|q^{\psi}(\rho,u)| +\frac{1}{\rho}|\partial_u\eta^{\psi}(\rho,u)|^2+\frac{1}{\rho^2}|\partial_u^2 \eta^{\psi}(\rho,u)|\sum_k\zeta_k^2 \le C\big(1+\rho|u|^4+g(\rho)+ \rho P(\rho) 
  \big).
$$
\end{lemma}

The proof is similar to that of 
Lemmas \ref{lem-growth-of-entropy-and-derivatives-general-pressure-law} and \ref{lem-growth-of-entropy-and-derivatives-polytropic-gas}. Therefore, we omit the details.

By Lemma \ref{lem-higher-order-growth-of-entropy-and-derivatives-polytropic-gas} 
and a similar proof 
to that of 
Proposition \ref{prop-entropy-inequality-general-pressure-law-under-higher-energy},
we can establish the entropy inequality for entropy pairs with convex 
generating functions 
of sub-cubic 
growth (see \eqref{eq-definition-of-sub-cubic-for-weight-function-polytropic-gas}):

\begin{proposition}\label{prop-entropy-inequality-polytropic-gas-under-higher-energy}
Assume that $U_0$ satisfies $E_{0,2}<\infty$ and
  $$\begin{aligned}
  \E \big[\Big(\int_{\mathbb{R}} \eta_{\diamond}^*(\rho_0,m_0)\,  \d x \Big)^2\big]<\infty.
  \end{aligned}$$
Then, for any entropy pair $(\eta,q)$ defined in \eqref{eq-entropy-flux-representation} 
with convex generating 
function $\psi\in C^2(\mathbb{R})$ 
satisfying \eqref{eq-definition-of-sub-cubic-for-weight-function-polytropic-gas},
the following entropy inequality holds $\tilde{\mathbb{P}}$-{\it a.s.}{\rm :}
  $$
  \begin{aligned}
  &\int_0^T \int_{\mathbb{R}} \big( \eta(\tilde  U)\varphi_t+ q(\tilde U) \varphi_x \big) \,\d x \d t+ \int_0^T \int_{\mathbb{R}} \varphi \partial_{\tilde m} \eta (\tilde U) \Phi( \tilde U) \,\d x \d W(t) \\
  &+\frac12 \int_0^T \int_{\mathbb{R}} \varphi \partial_{\tilde m}^2 \eta(\tilde U) \sum_k a_k^2 \,\zeta_k^2(\tilde U) \,\d x \d t
  \ge  0,
  \end{aligned}
  $$
for any nonnegative $\varphi \in C_{\rm c}^{\infty}((0,T)\times\mathbb{R})$, 
where $\tilde U=(\tilde \rho, \tilde m)^\top$.
\end{proposition}

Then, result (i) in Theorem \ref{thm-better-well-posedness-for-euler-on-whole-space} 
has been proved.
Result (ii) follows from almost the same arguments as those for 
Theorem \ref{thm-compactness-of-solution-sequence}, but here we should apply 
Theorem \ref{thm-stochastic-Lp-compensated-compactness-framework-polytropic-gas} 
with $H^{-1}_{\rm loc}(\mathbb{R}^2_+)$ in 
\eqref{con-tightness-of-dissipation-measure-in-Lp-compensated-compactness-framework-polytropic-gas} replaced by $W^{-1,1}_{\rm loc}(\mathbb{R}^2_+)$, which can actually be established 
by almost the same arguments as those for 
Theorem \ref{thm-stochastic-Lp-compensated-compactness-framework}. 
This completes the proof of (ii) in 
Theorem \ref{thm-better-well-posedness-for-euler-on-whole-space}.

\section{Proof of Theorem \ref{thm-wellposedness-parabolic-approximation-R}}\label{sec-proof-of-well-posedness-of-parabolic-approximation}
To prove Theorem \ref{thm-wellposedness-parabolic-approximation-R}, we first prove the existence and uniqueness for the cut-off equations and then remove the cut-off, based on some uniform estimates with respect to the cut-off. 
From now on, we omit subscript $\varepsilon$ for simplicity whenever no confusion arises.

\subsection{Well-Posedness for the Cut-Off Equations}
Given $f\in H^3(\mathbb{R})$, we define
  $
  \Pi_R f=\frac{f}{\|f\|_{H^3(\mathbb{R})}}\min\{R, \|f\|_{H^3(\mathbb{R})}\}.
  $

Consider the following cut-off equations:
  \begin{equation}\label{eq-cut-off-parabolic-approximation}
  \begin{cases}
  \d \rho + \partial_x (\rho \Pi_R u)\,\d t=\varepsilon \partial_{x}^2 \rho\, \d t,\\[0.5mm]
  \d u +\Pi_R u \partial_x(\Pi_R u)\,\d t 
  + \frac{P^{\prime}(\rho)}{\rho}\partial_x \rho\, \d t \\
  \quad =2\varepsilon \rho^{-1}\partial_x \rho \partial_x(\Pi_R u)\,\d t
  + \varepsilon \partial_{x}^2 u\, \d t
  + \rho^{-1}\Phi^{\varepsilon}(\rho,\rho\Pi_R u)\,\d W(t)
  \end{cases}
  \end{equation}
with initial data $(\rho,u)|_{t=0}=(\rho_0,u_0)$. Inspired by the method in \cite{BFH18local-solution}, we first solve equation $\eqref{eq-cut-off-parabolic-approximation}_1$ in the pathwise sense for given $u$, obtaining the solution, which will be denoted by $\bar\rho$ in the proof. 
Then we plug $\bar \rho$ in $\eqref{eq-cut-off-parabolic-approximation}_2$ and use the Banach fixed point theorem to obtain 
the local solution $\bar{u}$ for $\eqref{eq-cut-off-parabolic-approximation}_2$,
define the map: $u\mapsto \bar{u}$, and finally prove that this map has a fixed point. Then 
this local solution can be extended to the whole time interval $[0,T]$.
The result can be stated as follows:

\begin{lemma}\label{lem-wellposedness-cut-off-parabolic-approximation-R}
Assume that $U_{0}$ satisfies the conditions in {\rm Theorem \ref{thm-wellposedness-parabolic-approximation-R}} and 
   $$
   \|\rho_0-\rho_{\infty}\|_{H^3(\mathbb{R})}<R\quad \mathbb{P}\text{-{\it a.s.}},
   \qquad u_0\in L^p(\Omega;H^4(\mathbb{R}))\quad\mbox{for $p\ge 1$}.
   $$
Then there exist a unique strong solution to the cut-off 
equations \eqref{eq-cut-off-parabolic-approximation} in the 
solution space $X_T$, which contains all predictable processes 
$(\rho,u)$ satisfying that, for $p\ge 1$,
$$
(\rho-\rho_{\infty},u)\in L^p(\Omega;C([0,T];H^3(\mathbb{R}))), \quad\,\,
\|\rho^{-1}\|_{L^{\infty}([0,T]\times \mathbb{R})}\le C_1(\varepsilon):=e^{C(T,c_0,\varepsilon,R)}
 \,\,\,\,\mbox{$\mathbb{P}$-{\it a.s.}}.
$$
\end{lemma}

\begin{proof}
We divide the proof into eight steps and only give the sketch of most of the proof.

\smallskip
\textbf{1}. Given $u\in C([0,T];H^3(\mathbb{R}))$, there exists a unique solution $\bar \rho$
of the Cauchy problem for equation $\eqref{eq-cut-off-parabolic-approximation}_1$  satisfying $\rho^{*}:=\bar \rho-\rho_{\infty} \in C([0,T];H^3(\mathbb{R}))$ $\mathbb{P}$-{\it a.s.}.

The proof can be done by the standard Banach fixed-point theorem.

\smallskip
\textbf{2}. 
Using the assumption that 
$\|\rho_0-\rho_{\infty}\|_{H^3(\mathbb{R})}<R$ $\mathbb{P}$-{\it a.s.} and by
direct energy estimate (so we omit the details), we can obtain 
\begin{equation}\label{iq-uniform-bdd-wrt-omega-for-density-in-parabolic-approximation} 
\|\bar \rho-\rho_{\infty}\|_{C([0,T];H^3(\mathbb{R}))}
\le C(R) \quad\,\, \mathbb{P}\text{-{\it a.s.}}.
  \end{equation}
  
\textbf{3}. Claim: $\|{\bar \rho}^{-1}\|_{L^{\infty}([0,T]\times \mathbb{R})}\le C_1(\varepsilon):=e^{C(T,c_0,\varepsilon,R)}$ $\mathbb{P}$-{\it a.s.}.

\smallskip
Denote the heat kernel on $\mathbb{R}$ by
  \begin{equation}\label{eq-heat-kernel-representation}
  K(t,x)=\frac{1}{\sqrt{4\pi t}}e^{-\frac{x^2}{4 t}}
  \end{equation}
and the corresponding heat semigroup by $\mathbb{S}(t)$, {\it i.e.},
  $$
  \mathbb{S}(t)f(x)=K(t)*f(x)=\int_{\mathbb{R}} K(t,x-y)f(y) \,\d y.
  $$
For $K_{\varepsilon}(t,x):=K(\varepsilon t, x)$, 
the corresponding heat semigroup is denoted by $\mathbb{S}_{\varepsilon}(t)$.

We first rewrite equation $\eqref{eq-cut-off-parabolic-approximation}_1$ as
  \begin{equation}\label{eq-rewrite-bar-rho-cut-off-parabolic-approximation-1}
  \d \bar \rho + \partial_x (\bar\rho \Pi_R u)\, \d t=\varepsilon \partial_{x}^2 \bar\rho\, \d t,
  \end{equation}
where $u\in C([0,T];H^3(\mathbb{R}))$ is given. 
Let $v=\log \bar \rho$. Then, 
from \eqref{eq-rewrite-bar-rho-cut-off-parabolic-approximation-1}, we have
  $$\begin{aligned}
  v(t,x)
  =&\int_{\mathbb{R}} K_{\varepsilon}(t,x-y)v_0(y) \,\d y \\
  &+\int_0^t \int_{\mathbb{R}} \big(\varepsilon (v_x -\frac{\Pi_R u}{2\varepsilon})^2 -\frac{(\Pi_R u)^2}{4\varepsilon}-\partial_x \Pi_R u\big)(\tau,y)K_{\varepsilon}(t-\tau ,x-y) \,\d y \d \tau.
  \end{aligned}
  $$
Since
 $
  \int_{\mathbb{R}} K_{\varepsilon}(t, x-y) \,\d y=1,
 $
we have
  $$
  \begin{aligned}
  v(t,x)&\ge \int_{\mathbb{R}} K_{\varepsilon}(t ,x-y)v_0(y)\,\d y
  +\int_0^t \int_{\mathbb{R}} \big(-\frac{(\Pi_R u)^2}{4\varepsilon}-\partial_x \Pi_R u\big)(\tau,y)K_{\varepsilon}(t-\tau, x-y) \,\d y \d \tau\\
  &\ge \log c_0 -\frac{\|\Pi_R u\|_{L^{\infty}([0,T]\times \mathbb{R})}^2}{4\varepsilon}t-\|\partial_x \Pi_R u\|_{L^{\infty}([0,T]\times \mathbb{R})}t\\
  &\ge \log c_0 -\frac{R^2}{4\varepsilon}t-Rt :=-C(t, c_0, \varepsilon, R),
  \end{aligned}
  $$
where, in the last inequality, we have used the Sobolev embedding: $H^1(\mathbb{R}) \hookrightarrow L^{\infty}(\mathbb{R})$. 
Therefore, we obtain
  \begin{equation}\label{iq-lower-bound-for-density-in-parabolic-approximation}
  \|{\bar \rho}^{-1}\|_{L^{\infty}([0,T]\times \mathbb{R})}\le C_1(\varepsilon):=e^{C(T,c_0,\varepsilon,R)}\quad  \mathbb{P}\text{-{\it a.s.}}.
  \end{equation}

\smallskip
\textbf{4}. Claim: $\bar u \in L^p(\Omega;C([0,T];H^3(\mathbb{R})))$ 
for $p\ge1$.

\smallskip
We now plug the solution $\bar \rho$ of $\eqref{eq-cut-off-parabolic-approximation}_1$ in $\eqref{eq-cut-off-parabolic-approximation}_2$ and linearize it as follows:  \begin{equation}\label{eq-linearized-cut-off-parabolic-approximation-2}
  \begin{aligned}
  &\d \bar u +\Pi_R u \partial_x(\Pi_R u)\,\d t + \frac{P^{\prime}(\bar \rho)}{\bar \rho}\partial_x \bar\rho \,\d t\\
  &\,=2\varepsilon \bar\rho^{-1}\partial_x \bar\rho \partial_x(\Pi_R u)\,\d t+ \varepsilon \partial_{x}^2 \bar u \,\d t+\frac{1}{ \bar\rho}\Phi^{\varepsilon}(\bar\rho,\bar\rho\Pi_R u)\,\d W(t),
  \end{aligned}
  \end{equation}
where $u\in L^p(\Omega;C([0,T];H^3(\mathbb{R})))$ is the 
given predictable process.

Using the mild formulation, 
we represent the solution to \eqref{eq-linearized-cut-off-parabolic-approximation-2} as 
$$\begin{aligned}
  \bar u(t,x)=&\int_{\mathbb{R}} K_{\varepsilon}(t, x-y)\bar u(0,y) \,\d y\\
  &+\int_0^t \int_{\mathbb{R}}  K_{\varepsilon}(t-\tau, x-y) \tilde F (\tau,y)\,\d y\d \tau
  +\int_0^t \int_{\mathbb{R}} K_{\varepsilon}(t-\tau, x-y) \tilde \Phi^{\varepsilon}(\tau, y) 
  \,\d y \d W(\tau),
\end{aligned}$$
where
  $$
  \begin{aligned}
  \tilde F=-\Pi_R u \partial_x(\Pi_R u) -\frac{P^{\prime}(\bar \rho)}{\bar \rho}\partial_x \bar \rho +2\varepsilon \bar \rho^{-1}\partial_x \bar \rho \partial_x(\Pi_R u),\qquad
  \tilde \Phi^{\varepsilon}=\frac{1}{\bar \rho}\Phi^{\varepsilon}(\bar \rho,\bar\rho\Pi_R u).
  \end{aligned}
  $$
Then, for any $0\le t_1 \le t_2 \le T$, we have
  $$
  \begin{aligned}
  &\bar u(t_2,x)-\bar u(t_1, x)\\
  &=\int_{\mathbb{R}} \big(K_{\varepsilon}(t_2, x-y)-K_{\varepsilon}(t_1, x-y)\big)\bar u(0,y) \,\d y+\int_{t_1}^{t_2} \int_{\mathbb{R}}  K_{\varepsilon}(t_2-\tau,x-y)\tilde F (\tau,y) \,\d y\d \tau \\
  &\quad + \int_{0}^{t_1} \int_{\mathbb{R}}  \big(K_{\varepsilon}(t_2-\tau,x-y)-K_{\varepsilon}(t_1-\tau,x-y)\big) \tilde F (\tau,y)\,\d y\d \tau\\
  &\quad + \int_{t_1}^{t_2} \int_{\mathbb{R}} K_{\varepsilon}(t_2-\tau,x-y) \tilde \Phi^{\varepsilon}(\tau,y) \,\d y \d W(\tau)\\
  &\quad +\int_{0}^{t_1} \int_{\mathbb{R}} \big(K_{\varepsilon}(t_2-\tau,x-y)-K_{\varepsilon}(t_1-\tau,x-y)\big) \tilde \Phi^{\varepsilon}(\tau,y) \,\d y \d W(\tau)\\
  &:=\mathbb{I}_1+\cdots + \mathbb{I}_5.
  \end{aligned}
  $$

We first estimate $\mathbb{I}_4$:
$$\begin{aligned}
  I_{41}&:=\E \big[\Big\|\int_{t_1}^{t_2} \int_{\mathbb{R}} K_{\varepsilon}(t_2-\tau,x-y) \tilde \Phi^{\varepsilon}(\tau,y) \,\d y \d W(\tau)\Big\|_{H^3(\mathbb{R})}^{2p}\big]\\
  &\le  \E \big[\Big(\int_{t_1}^{t_2} \sum_k \Big\|\int_{\mathbb{R}} K_{\varepsilon}(t_2-\tau,x-y) \frac{1}{\bar \rho} a_k \zeta^{\varepsilon}_k(\bar \rho,\bar\rho\Pi_R u) \,\d y \Big\|_{H^3(\mathbb{R})}^2 \,\d \tau \Big)^{p}\big]\\
  &\le  \E \big[\Big(\int_{t_1}^{t_2} \|K_{\varepsilon}(t_2-\tau)\|_{L^1(\mathbb{R})}^2 \sum_k\Big\|\frac{1}{\bar \rho} a_k\zeta^{\varepsilon}_k(\bar \rho,\bar\rho\Pi_R u) \Big\|_{H^3(\mathbb{R})}^2 \,\d \tau \Big)^{p}\big].
  \end{aligned}$$
For $\sum_k\big\|\frac{1}{\bar \rho} a_k \zeta^{\varepsilon}_k(\bar \rho,\bar\rho\Pi_R u) \big\|_{H^3(\mathbb{R})}$, 
we focus our estimates on two terms:
$\sum_k\big\|\frac{1}{\bar \rho} a_k \zeta^{\varepsilon}_k(\bar \rho,\bar\rho\Pi_R u) \big\|_{L^2(\mathbb{R})}$ and $\sum_k\big\|\partial_x^3\big(\frac{1}{\bar \rho} a_k \zeta^{\varepsilon}_k(\bar \rho,\bar\rho\Pi_R u)\big)
 \big\|_{L^2(\mathbb{R})}$, since the other terms can be treated similarly.

For $\sum_k\big\|\frac{1}{\bar \rho} a_k \zeta^{\varepsilon}_k(\bar \rho,\bar\rho\Pi_R u) \big\|_{L^2(\mathbb{R})}$, we 
use \eqref{iq-noise-coefficience-growth-conditions-for-parabolic-approximation} 
to obtain
  $$
  \begin{aligned}
  \sum_k\Big\|\frac{1}{\bar \rho} a_k \zeta^{\varepsilon}_k(\bar \rho,\bar\rho\Pi_R u) \Big\|_{L^2(\mathbb{R})}^2
  &\le B_0 \int_{\text{\rm supp}_x(\zeta^{\varepsilon}_k)} \big(C+(\Pi_R u)^2 + e(\bar \rho) \big) \,\d x\\
  &\le B_0\big( C+\|\Pi_R u\|_{L^{\infty}(\mathbb{R})}^2+\|e(\bar \rho)\|_{L^{\infty}(\mathbb{R})} \big)|\text{\rm supp}_x(\zeta^{\varepsilon}_k)|
  \le 
  C(R),
  \end{aligned}
  $$
where, in the last inequality, we have used the Sobolev 
embedding: $H^1(\mathbb{R}) \hookrightarrow L^{\infty}(\mathbb{R})$
and estimates \eqref{iq-uniform-bdd-wrt-omega-for-density-in-parabolic-approximation}
and 
\eqref{iq-lower-upper-bound-for-general-internal-energy-1}--\eqref{iq-lower-upper-bound-for-general-internal-energy-2}. 
This simple argument is used frequently 
and will not be
repeated.

For the term $\sum_k\big\|\partial_x^3\big(\frac{1}{\bar \rho} a_k \zeta^{\varepsilon}_k(\bar \rho,\bar\rho\Pi_R u)\big) \big\|_{L^2(\mathbb{R})}$, 

using 
the fact that $\zeta^{\varepsilon}_k(x, \rho, m)\in C_{\rm c}^{\infty}(\mathbb{R},\mathbb{R}^+,\mathbb{R})$, $\|\partial_x^i\partial_{\rho}^j\partial_m^l \zeta^{\varepsilon}_k\|_{L_x^{\infty}(\mathbb{R})}\le C(\varepsilon)$ for
$0\le i,j,l \le 3$, 
estimates
\eqref{iq-uniform-bdd-wrt-omega-for-density-in-parabolic-approximation}--\eqref{iq-lower-bound-for-density-in-parabolic-approximation},
the Gagliardo-Nirenberg inequality, and 
the Sobolev embedding: $H^1(\mathbb{R}) \hookrightarrow L^{\infty}(\mathbb{R})$, 
we have
  $$
  \begin{aligned}
  \sum_k\Big\|\partial_x^3\big(\frac{1}{\bar \rho} a_k \zeta^{\varepsilon}_k(\bar \rho,\bar\rho\Pi_R u)\big) \Big\|_{L^2(\mathbb{R})}^2
  \le  C(\varepsilon,C_1(\varepsilon),R)\big(C(R,\rho_{\infty})+|\text{supp}_x (\zeta_k^{\varepsilon})|^2\big).
  \end{aligned}
  $$

Therefore, we have
$$
I_{41}\le C(\varepsilon, C_1(\varepsilon),R)(t_2 - t_1)^p.
$$

\smallskip
For $\mathbb{I}_5$, similarly, we have
  $$\begin{aligned}
  I_{51}&:=\E\big[ \Big\|\int_{0}^{t_1} \int_{\mathbb{R}} \big(K_{\varepsilon}(t_2-\tau,x-y)-K_{\varepsilon}(t_1-\tau,x-y)\big) \tilde \Phi^{\varepsilon}(\tau,y) \,\d y \d W(\tau)\Big\|_{H^3(\mathbb{R})}^{2p}\big]\\
  &\le  \E \big[\Big(\int_{0}^{t_1} \|K_{\varepsilon}(t_2-\tau)-K_{\varepsilon}(t_1-\tau)\|_{L^2(\mathbb{R})}^2 \sum_k \sum_{j=0}^3\Big\|\partial_x^j \big(\frac{1}{\bar \rho} a_k \zeta^{\varepsilon}_k(\bar \rho,\bar\rho\Pi_R u)\big) \Big\|_{L^1(\mathbb{R})}^2(\tau) \,\d \tau \Big)^{p}\big],
  \end{aligned}$$
where, by Lemma \ref{lem-heat-semigroup-R}, we have
  $$
  \begin{aligned}
  &\| K_{\varepsilon}(t_2-\tau)
  - K_{\varepsilon}(t_1-\tau)\|_{L^2(\mathbb{R})}\\
  &= \Big\|\int_{t_1}^{t_2} \partial_t K_{\varepsilon}(t-\tau)\,\d t \Big\|_{L^2(\mathbb{R})}
  \le  \int_{t_1}^{t_2} \|\partial_t K_{\varepsilon}(t-\tau)\|_{L^2(\mathbb{R})} \,\d t
  \le4(t_1-\tau)^{-\frac14}-4(t_2-\tau)^{-\frac14}.
  \end{aligned}
  $$
Since
  $$
  \begin{aligned}
  \Big\|\partial_x^j \big(\frac{1}{\bar \rho} a_k \zeta^{\varepsilon}_k(\bar \rho,\bar\rho\Pi_R u)\big) \Big\|_{L^1(\mathbb{R})}^2
  \le  \Big\|\partial_x^j \big(\frac{1}{\bar \rho} a_k \zeta^{\varepsilon}_k(\bar \rho,\bar\rho\Pi_R u)\big) \Big\|_{L^2(\text{supp}_x (\zeta_k^{\varepsilon}))}^2|\text{supp}_x (\zeta_k^{\varepsilon})|,
  \end{aligned}
  $$
we use the estimates similar to $\mathbb{I}_4$ to obtain
  \begin{align}
  I_{51}\le C(\varepsilon,C_1(\varepsilon),R)
  \Big(\int_0^{t_1} \big((t_1-\tau)^{-\frac14}-(t_2-\tau)^{-\frac14}\big)^2 \,\d \tau \Big)^{p}
  \le  
  C(\varepsilon,C_1(\varepsilon),R,T)(t_2-t_1)^{\frac{p}{2}}.\label{iq-time-fractional-integral-estimate}
  \end{align}

Next, the other terms in $\mathbb{I}_1-\mathbb{I}_3$ can be estimated similarly, so we omit the details but only estimate one subterm in $\mathbb{I}_3$ which requires more careful analysis. 
We first split a subterm in $\mathbb{I}_3$ into two parts,
  $$
  \begin{aligned}
  I_{32}&:=\E \big[\Big(\Big\|\int_{0}^{t_1} \int_{\mathbb{R}} \big(\frac{P^{\prime}(\bar \rho)}{\bar \rho}\partial_y\bar \rho \big)(\tau,y) \big(K_{\varepsilon}(t_2-\tau,x-y)-K_{\varepsilon}(t_1-\tau,x-y)\big) \,\d y\d \tau\Big\|_{L^2(\mathbb{R})}\\
  &\quad\quad\,\,\,+\sum_{j=1}^3\Big\|\partial_x^j\int_{0}^{t_1} \int_{\mathbb{R}} 
  \big( \frac{P^{\prime}(\bar \rho)}{\bar \rho}\partial_y\bar \rho \big)\tau,y) \big(K_{\varepsilon}(t_2-\tau,x-y)
  -K_{\varepsilon}(t_1-\tau,x-y)\big) \,\d y\d \tau\Big\|_{L^2(\mathbb{R})} \Big)^{2p}\big]\\
  &\le  C(p)\bigg( \E\big[\Big\|\int_{0}^{t_1} \int_{\mathbb{R}} 
  \big(\frac{P^{\prime}(\bar \rho)}{\bar \rho}\partial_y\bar \rho \big)(\tau,y) \big(K_{\varepsilon}(t_2-\tau,x-y)-K_{\varepsilon}(t_1-\tau,x-y)\big) \,\d y\d \tau\Big\|_{L^2(\mathbb{R})}^{2p}\big]\\
  &\quad\quad\quad\,\,\,+\sum_{j=1}^3 \E \big[\Big\|\partial_x^j\int_{0}^{t_1} \int_{\mathbb{R}} \big(\frac{P^{\prime}(\bar \rho)}{\bar \rho}\partial_y\bar \rho \big)(\tau,y) \big(K_{\varepsilon}(t_2-\tau,x-y)
  -K_{\varepsilon}(t_1-\tau,x-y)\big) \,\d y\d \tau\Big\|_{L^2(\mathbb{R})}^{2p}\big] \bigg)\\
  &:= C(p)(I_{321}+I_{322}).
  \end{aligned}
  $$
For $I_{321}$, we rewrite it as
  $$
  \begin{aligned}
  I_{321}
  &= \E\big[\Big\|\int_{0}^{t_1} \int_{\mathbb{R}} \big( \frac{P^{\prime}(\bar \rho)}{\bar \rho}\partial_y\bar \rho \big)(\tau,y) \int_{t_1}^{t_2} \partial_t K_{\varepsilon}(t-\tau,x-y) \,\d t \d y\d \tau\Big\|_{L^2(\mathbb{R})}^{2p}\big]\\
  &= \E\big[\Big\|\int_{0}^{t_1} \int_{t_1}^{t_2} \partial_t \int_{\mathbb{R}} \big(\frac{P^{\prime}(\bar \rho)}{\bar \rho}\partial_y\bar \rho \big)(\tau,y) K_{\varepsilon}(t-\tau,x-y) \,\d y \d t\d \tau\Big\|_{L^2(\mathbb{R})}^{2p}\big].
  \end{aligned}
  $$
Note that, denoting 
$f(\tau,y)=\big( \frac{P^{\prime}(\bar \rho)}{\bar \rho}\partial_y\bar \rho \big)(\tau,y)$, 
then
$$
  \mathcal{U}(t-\tau,x):=\int_{\mathbb{R}} \big(\frac{P^{\prime}(\bar \rho)}{\bar \rho}\partial_y\bar \rho \big)(\tau,y) K_{\varepsilon}(t-\tau,x-y) \,\d y=\int_{\mathbb{R}} f(\tau,y) K_{\varepsilon}(t-\tau,x-y) \,\d y
$$
can be seen as the solution to the following heat equation with fixed $\tau$
  $$\begin{aligned}
  \begin{cases}
  \mathcal{U}_t-\Delta \mathcal{U}=0\qquad\mbox{for $x\in \mathbb{R}$},\\
  \mathcal{U}(0,y)=f(\tau,y),
  \end{cases}
  \end{aligned}$$
where $\Delta=\partial_x^2$.
Therefore, we have
  $$
  \partial_t \int_{\mathbb{R}} \big( \frac{P^{\prime}(\bar \rho)}{\bar \rho}\partial_y\bar \rho \big)(\tau,y) K_{\varepsilon}(t-\tau,x-y) \,\d y
  =\Delta \int_{\mathbb{R}} \big( \frac{P^{\prime}(\bar \rho)}{\bar \rho}\partial_y\bar \rho \big)(\tau,y) K_{\varepsilon}(t-\tau,x-y) \,\d y
  $$
and
  $$
  \begin{aligned}
  I_{321}
  &= \E\big[\Big\|\int_{0}^{t_1} \int_{t_1}^{t_2}\Delta \mathbb{S}_{\varepsilon}(t-\tau)\big(\frac{P^{\prime}(\bar \rho)}{\bar \rho} \partial_y \bar \rho \big)(\tau) \,\d t\d \tau\Big\|_{L^2(\mathbb{R})}^{2p}\big]\\
  &\le  \E\big[\Big(\int_{0}^{t_1} \int_{t_1}^{t_2} \|\Delta \mathbb{S}_{\varepsilon}(t-\tau)\big(\frac{P^{\prime}(\bar \rho)}{\bar \rho} \partial_y \bar \rho \big)(\tau)\|_{L^2(\mathbb{R})} \,\d t\d \tau\Big)^{2p} \big].
  \end{aligned}
  $$
For $b>0$ and $\nu$ satisfying $\frac{b+\nu}{2}=1$, we obtain
  $$
  \begin{aligned}
  \|\Delta \mathbb{S}_{\varepsilon}(t-\tau)\big(\frac{P^{\prime}(\bar \rho)}{\bar \rho} \partial_y \bar \rho \big)(\tau)\|_{L^2(\mathbb{R})}&=\|- (-\Delta)^{\frac{b}{2}} (-\Delta)^{\frac{\nu}{2}} \mathbb{S}_{\varepsilon}(t-\tau)\big(\frac{P^{\prime}(\bar \rho)}{\bar \rho} \partial_y \bar \rho \big)(\tau)\|_{L^2(\mathbb{R})}\\
  & \le \|(-\Delta)^{\frac{b}{2}} \mathbb{S}_{\varepsilon}(t-\tau) (I-\Delta)^{\frac{\nu}{2}}\big(\frac{P^{\prime}(\bar \rho)}{\bar \rho} \partial_y \bar \rho \big)(\tau)\|_{L^2(\mathbb{R})}\\
  & \le 
  C |t-\tau|^{-\frac{b}{2}} \|\big(\frac{P^{\prime}(\bar \rho)}{\bar \rho} \partial_y \bar \rho \big)(\tau)\|_{H^{\nu}(\mathbb{R})},
  \end{aligned}$$
where, in the first and the second inequality, we have used (i) and (ii) 
in Lemma \ref{lem-properties-of-fractional-laplace} respectively. 
Then, using the above estimate, taking $b=3$ and $\nu=-1$, and by Lemma \ref{lem-bessel-potential-space-embedding}, we have
  $$
  \begin{aligned}
  I_{321}&\le  C\E\big[\Big(\int_{0}^{t_1} \int_{t_1}^{t_2} (t-\tau)^{-\frac{3}{2}} \|\big(\frac{P^{\prime}(\bar \rho)}{\bar \rho} \partial_y \bar \rho \big)(\tau)\|_{L^2(\mathbb{R})} \,\d t\d \tau\Big)^{2p} \big]\\
  &\le  
  C(C_1(\varepsilon),R)\Big(\int_{0}^{t_1} \int_{t_1}^{t_2} (t-\tau)^{-\frac{3}{2}} \,\d t\d \tau\Big)^{2p}
  \le 
  C(C_1(\varepsilon),R)(t_2-t_1)^p,
  \end{aligned}
  $$
where, in the second inequality, we have used the fact: when $\gamma_2-2\ge 0$,
  $
  \|\bar \rho\|_{L^{\infty}(\mathbb{R})}^{\gamma_2-2}\le (\|\bar \rho-\rho_{\infty}\|_{L^{\infty}(\mathbb{R})}+\rho_{\infty})^{\gamma_2-2},
  $
combined with the Sobolev embedding: $H^1(\mathbb{R}) \hookrightarrow L^{\infty}(\mathbb{R})$, 
\eqref{iq-lower-upper-bound-for-general-pressure-1}--\eqref{iq-lower-upper-bound-for-general-pressure-2},  \eqref{iq-uniform-bdd-wrt-omega-for-density-in-parabolic-approximation}, and \eqref{iq-lower-bound-for-density-in-parabolic-approximation}. 
  $$
  \begin{aligned}
  I_{322}
  &=\sum_{j=1}^3 \E \big[\Big\|\partial_x^j\int_{0}^{t_1} \int_{\mathbb{R}} \big(\frac{P^{\prime}(\bar \rho)}{\bar \rho}\partial_y\bar \rho \big)(\tau,y) \big(K_{\varepsilon}(t_2-\tau,x-y)
  -K_{\varepsilon}(t_1-\tau,x-y)\big) \,\d y\d \tau\Big\|_{L^2(\mathbb{R})}^{2p}\big] \\
  &\le \sum_{j=1}^3 \E\big[ \Big(\int_{0}^{t_1}  \|\partial_y^{j-1}\big( \frac{P^{\prime}(\bar \rho)}{\bar \rho}\partial_y\bar \rho \big)\|_{L^2(\mathbb{R})}(\tau) \|\partial_x K_{\varepsilon}(t_2-\tau)
  -\partial_x K_{\varepsilon}(t_1-\tau)\|_{L^1(\mathbb{R})} \,\d \tau \Big)^{2p}\big].
  \end{aligned}
  $$
Then, using a similar estimate to $I_{321}$, we have
  $$
  I_{322} \le  
  C(C_1(\varepsilon),R)(t_2-t_1)^p.
  $$

Thus, combining 
estimates for $\mathbb{I}_1-\mathbb{I}_5$, 
we have
  $$
  \E \big[\|\bar u(t_2)-\bar u(t_1)\|_{H^3(\mathbb{R})}^{2p}\big] \le 
  C\big(\E [\|\bar u(0) \|_{H^4(\mathbb{R})}^{2p}],R,\varepsilon,C_1(\varepsilon),T\big)(t_2-t_1)^{\frac{p}{2}}.
  $$
Therefore, choosing $\frac{p}{2}>1$ and applying the Kolmogorov criterion,
we obtain
  $$
  \E \big[\sup_{t\in [0,T]}\|\bar u(t)\|_{H^3(\mathbb{R})}^{p}\big] \le C\big(\E \big[\|\bar u(0) \|_{H^4(\mathbb{R})}^{2p} \big],R,\varepsilon,C_1(\varepsilon),T\big)
  $$
and $\bar u \in L^p(\Omega;C([0,T];H^3(\mathbb{R})))$ for $p\ge1$.

\smallskip
\textbf{5}. The mild solution $\bar u$ is also a weak solution, 
hence also the strong solution.

Recall the representation of $\bar u$ in Step \textbf{4}.
The estimates in Step \textbf{4} give
  $$
  \E \big[\int_0^T \sum_k\Big\|\frac{1}{\bar \rho} a_k \zeta^{\varepsilon}_k(\bar \rho,\bar\rho\Pi_R u) \Big\|_{H^3(\mathbb{R})}^2 \,\d \tau\big] <\infty.
  $$
Since $\bar \rho$ and $u$ are predictable, we have $\tilde \Phi^{\varepsilon}$ is predictable. 
Hence, we can use \cite[Proposition 6.2]{dapratozabczyk14} 
to obtain that the stochastic convolution
  $
 \int_0^t \int_{\mathbb{R}} K_{\varepsilon}(t-\tau,x-y) \tilde \Phi^{\varepsilon}(\tau,y) \,\d y \d W(\tau)
  $
has a predictable version, which implies that $\bar u$ has a predictable version. 
Then, by a straightforward adaptation of 
\cite[Proposition 6.4 (i)]{dapratozabczyk14}, 
we can prove that the mild solution $\bar u$ is also a weak solution, {\it i.e.}, satisfies that, for any $t\in [0,T]$ and $\phi \in C_{\rm c}^{\infty}(\mathbb{R})$,
  $$
  \langle \bar u(t), \phi \rangle-\langle \bar u(0), \phi \rangle=\varepsilon \int_0^t \langle \bar u(\tau), \partial_x^2 \phi \rangle \,\d \tau + \int_0^t \langle \tilde F (\tau), \phi \rangle \,\d \tau + \int_0^t  \langle \phi, \tilde \Phi^{\varepsilon}(\tau)\rangle \,\d W(\tau)\quad\,\, \mathbb{P}\text{-{\it a.s.}}.
  $$
Since, by Step \textbf{4}, 
$\bar u \in L^p(\Omega;C([0,T];H^3(\mathbb{R})))$ for $p\ge1$, 
we also see that $\bar u$ is also a strong solution. 
Therefore, for any $t\in [0,T]$,
\begin{equation}\label{eq-cut-off-parabolic-approximation-linearization-strong-solution}
\bar u(t)=\bar u(0)+\int_0^t \partial_x^2 \bar u(\tau) \,\d \tau+ \int_0^t \tilde F (\tau) \,\d \tau + \int_0^t \tilde \Phi^{\varepsilon}(\tau) \,\d W(\tau)\quad \mathbb{P}\text{-{\it a.s.}}.
\end{equation}

\smallskip
\textbf{6}. 
Using the fact that given $f,g\in H^3(\mathbb{R})$, $
  \| \Pi_R f-  \Pi_R g\|_{H^3(\mathbb{R})}\le \| f- g\|_{H^3(\mathbb{R})}$ for any $R>0$, which can be easily verified, we can directly prove that
  \begin{equation}\label{iq-rho-1-rho-2-control-by-u-1-u-2}
  \|\bar\rho_1-\bar\rho_2\|_{C([0,T];H^3(\mathbb{R}))}\le C \|u_1-u_2\|_{C([0,T];H^3(\mathbb{R}))},
  \end{equation}
  so we omit the details.

\smallskip
\textbf{7}. There exists sufficiently small $T^*>0$ such that the map $\mathcal{T}(u):=\bar u$ is a contraction on $L^p(\Omega,C([0,T^*];H^3(\mathbb{R})))$, which implies that \eqref{eq-cut-off-parabolic-approximation} has a unique strong solution in $X_{T^*}$.

By the mild formulation of \eqref{eq-linearized-cut-off-parabolic-approximation-2}, we have
  $$\begin{aligned}
  \bar u_i(t,x)&=\int_{\mathbb{R}} K_{\varepsilon}(t,x-y)\bar u_i(0,y) \,\d y\\
  &\quad+\int_0^t \int_{\mathbb{R}} \tilde F_i (\tau,y) K_{\varepsilon}(t-\tau,x-y) \,\d y\d \tau +\int_0^t \int_{\mathbb{R}} K_{\varepsilon}(t-\tau,x-y) \tilde \Phi_{i \varepsilon}(\tau,y) \,\d y \d W(\tau),
  \end{aligned}$$
where
  $$\begin{aligned}
  \tilde F_i=-\Pi_R u_i \partial_x(\Pi_R u_i) -\frac{P^{\prime}(\bar \rho_i)}{\bar \rho_i}\partial_x \bar \rho_i +2\varepsilon \bar \rho_i^{-1}\partial_x \bar \rho_i \partial_x(\Pi_R u_i),
  \quad \tilde \Phi_{i\varepsilon}=\frac{1}{\bar \rho_i}\Phi^{\varepsilon}(\bar \rho_i,\bar\rho_i\Pi_R u_i).
  \end{aligned}$$
Note that
  \begin{equation}\label{eq-mean-value-theorem-for-rho}
  \begin{aligned}
  \bar \rho_1^{\alpha}-\bar \rho_2^{\alpha}=\int_0^1 \alpha \big(a\bar \rho_2+(1-a)\bar \rho_1\big)^{\alpha-1}d a(\bar \rho_2-\bar \rho_1)
  \qquad \mbox{for $\alpha\in \mathbb{R}$ with $\alpha\ne 0,1$}.
  \end{aligned}
  \end{equation}
Then, similar to the estimates for $\mathbb{I}_2$ in Step \textbf{4}, we obtain
  $$
  \begin{aligned}
  &\E \big[\sup_{t\in[0,r]}\Big\|\int_0^t \int_{\mathbb{R}} (\tilde F_1-\tilde F_2) (\tau,y) K_{\varepsilon}(t-\tau,x-y) \,\d y\d \tau\Big\|_{H^3(\mathbb{R})}^p\big]\\
  &\le 
  C\big(R,C_1(\varepsilon),T\big)r^{\frac{p}{2}} \E \big[ \sup_{t\in[0,r]}\|(\bar \rho_1-\bar \rho_2,u_1-u_2)\|_{H^3(\mathbb{R})}^p\big].
  \end{aligned}
  $$
Notice that
  \begin{equation}\label{eq-mean-value-theorem-for-noise}
  \begin{aligned}
  \zeta^{\varepsilon}_k(V_1)-\zeta^{\varepsilon}_k(V_2)=\int_0^1 (\nabla_V \zeta^{\varepsilon}_k)(a V_1+(1-a)V_2) \,\d a(V_1-V_2),
  \end{aligned}
  \end{equation}
where $V_i=(\bar \rho_i,\bar\rho_i \Pi_R u_i)$ for $i=1,2$. 
Then, 
by \eqref{eq-mean-value-theorem-for-rho}--\eqref{eq-mean-value-theorem-for-noise}, 
as the estimates for $\mathbb{I}_4-\mathbb{I}_5$ in Step \textbf{4}, 
we can use the Kolmogorov criterion to prove that the stochastic 
integral:
$$
\int_0^t \int_{\mathbb{R}} K_{\varepsilon}(t-\tau,x-y) \tilde \Phi_{i \varepsilon}(\tau,y) \,\d y \d W(\tau)
$$ 
is continuous, so that
  $$
  \begin{aligned}
  &\E\big[ \sup_{t\in[0,r]}\Big\|\int_0^t \int_{\mathbb{R}} K_{\varepsilon}(t-\tau,x-y) (\tilde \Phi_{1 \varepsilon}-\tilde \Phi_{2 \varepsilon})(\tau,y) \,\d y \d W(\tau)\Big\|_{H^3(\mathbb{R})}^p\big]\\
  &\le  
  C(\varepsilon, C_1(\varepsilon),R)r^{\frac{p}{2}}\E \big[ \sup_{t\in[0,r]}\|(\bar \rho_1-\bar \rho_2,u_1-u_2)\|_{H^3(\mathbb{R})}^p\big].
  \end{aligned}
  $$
Combining the above estimates, we obtain
  $$
  \begin{aligned}
 \E \big[\sup_{t\in[0,r]}\|\bar u_1-\bar u_2\|_{H^3(\mathbb{R})}^p(t)\big]
 &\le  
  C(\varepsilon, C_1(\varepsilon),T,R)r^{\frac{p}{2}}\E \big[ \sup_{t\in[0,r]}\|(\bar \rho_1-\bar \rho_2,u_1-u_2)\|_{H^3(\mathbb{R})}^p\big]\\
  &\le 
  C(\varepsilon, C_1(\varepsilon),T,R)\big(1+C
  \big) r^{\frac{p}{2}}\E \big[ \sup_{t\in[0,r]}\|u_1-u_2\|_{H^3(\mathbb{R})}^p\big],
  \end{aligned}
  $$
where, in the last inequality, we have 
used \eqref{iq-rho-1-rho-2-control-by-u-1-u-2}. 
Hence, choosing sufficiently small $r$, 
denoted as $T^*$, we conclude that 
the map $\mathcal{T}(u):=\bar u$ is a contraction on $L^p(\Omega,C([0,T^*];H^3(\mathbb{R})))$, which implies that
there exists a unique strong solution to
\eqref{eq-cut-off-parabolic-approximation} in  $X_{T^*}$.

\smallskip
\textbf{8}. Claim: There exists a unique global solution 
to \eqref{eq-cut-off-parabolic-approximation}.

Note that 
$T^*$ in Step \textbf{7} depends on $C_1(\varepsilon)=e^{C(T,c_0,\varepsilon,R)}$, which depends on the lower bound $c_0>0$ 
of the initial data. 
However, constant $C(T,c_0,\varepsilon,R)$ depends linearly on time $T$, 
which enables us to extend the local solution obtained in Step \textbf{7} 
to obtain the global solution. 
More specifically, we start from any time $0<t_0<T$, by \eqref{iq-lower-bound-for-density-in-parabolic-approximation}, 
we obtain that the initial density $\bar \rho(t_0)$ has a lower bound 
$e^{-C(t_0,c_0,\varepsilon,R)}$. 
Then repeating the process in Step \textbf{3} and using the semigroup property
yield that, for $t_0<t\le T$,
  $$\begin{aligned}
  v(t,x)&\ge -C(t_0,c_0,\varepsilon,R) -\frac{\|\Pi_R u\|_{L^{\infty}([0,T]\times \mathbb{R})}^2}{4\varepsilon}(t-t_0)-\|\partial_x \Pi_R u\|_{L^{\infty}([0,T]\times \mathbb{R})}(t-t_0)\\
  &\ge \log c_0 -\frac{R^2}{4\varepsilon}t_0-Rt_0 -\frac{R^2}{4\varepsilon}(t-t_0)-R(t-t_0)
  =-C(t, c_0, \varepsilon, R),
  \end{aligned}$$
which implies that the same bound on $[t_0,T]$ is obtained:
  $$
  \|{\bar \rho}^{-1}\|_{L^{\infty}([t_0,T]\times \mathbb{R})}\le C_1(\varepsilon)=e^{C(T,c_0,\varepsilon,R)}.
  $$
Therefore, we can start from $T^*$ and solve 
equation \eqref{eq-cut-off-parabolic-approximation} on $[T^*,2T^*]$. 
Repeating this process, we obtain the unique global solution 
to \eqref{eq-cut-off-parabolic-approximation}.
\end{proof}

In Theorem \ref{thm-wellposedness-cut-off-parabolic-approximation-R-relax-initial-data} below,
we relax the additional assumption 
$\|\rho_0-\rho_{\infty}\|_{H^3(\mathbb{R})}<R$ and give the well-posedness result of \eqref{eq-cut-off-parabolic-approximation} for more general initial data. The proof is direct, so we omit it.

\begin{theorem}\label{thm-wellposedness-cut-off-parabolic-approximation-R-relax-initial-data}
Assume that $U_{0}$ satisfies the conditions in {\rm Theorem \ref{thm-wellposedness-parabolic-approximation-R}} and 
  $$
   \rho_0-\rho_{\infty}\in L^{\infty}(\Omega; H^3(\mathbb{R})),\quad u_0\in L^p(\Omega;H^4(\mathbb{R}))\quad \mbox{for $p\ge 1$}.
  $$
Then there exist a unique strong solution to the cut-off equations \eqref{eq-cut-off-parabolic-approximation} in the solution space $X_T$.
\end{theorem}

\subsection{Regularity Estimates}
We now establish the following regularity estimates, which will be used to remove the cut-off in \S \ref{sec-Remove-the-cut-off-operator}. In the proof, we improve the regularity iteratively, inspired by the idea in \cite[Proposition 3.3]{BV19} where they obtained the regularity for bounded solutions. Since the cut-off is applied to $u$ only, much more careful analysis is required.

\begin{proposition}\label{prop-regularity-for-cut-off-equ}
Assume that $U_{0}$ satisfies the conditions 
in {\rm Theorem \ref{thm-wellposedness-parabolic-approximation-R}} and 
 \begin{align*}
 &\rho_0-\rho_{\infty}\in L^{\infty}(\Omega; H^3(\mathbb{R})),\qquad u_0\in L^p(\Omega;H^4(\mathbb{R}))\,\,\,\,\, \mbox{for $p\ge 1$},\\
 &(\rho_0-\rho_{\infty}, u_0)\in L^q(\Omega;H^5(\mathbb{R}))\,\,\,\,\,
 \mbox{for $q>12$}.
 \end{align*}
Then the solution $V:=(\rho,u)^{\top}$, obtained in {\rm Theorem \ref{thm-wellposedness-cut-off-parabolic-approximation-R-relax-initial-data}}, 
satisfies 
  \begin{align}
  &\sup_{r\in [0,t]}\E\big[ \|\partial_x^4 V(r)\|_{L^2(\mathbb{R})}^{2p}\big] \le C\big(\E [\| V_0^*\|_{H^5(\mathbb{R})}^{2p}],R,t,C_1(\varepsilon),\varepsilon \big)\quad\mbox{for $p>2$},
  \label{iq-improve-regularity-final-1}\\
&\sup_{r\in [0,t]}\E \big[\|\partial_x^5 V(r)\|_{L^2(\mathbb{R})}^{2p}\big]
  \le  C\big(\E \big[\| V_0^*\|_{H^5(\mathbb{R})}^{2p}\big],\E \big[\| V_0^*\|_{H^4(\mathbb{R})}^{4p}\big],R,t,C_1(\varepsilon),\varepsilon \big)
  \quad\mbox{for $p>3$},
  \end{align}
where $V_0^*(y)=(\rho-\rho_{\infty},u)^{\top}(0,y)$. 
Therefore, for $p>3$ and $t\in[0,T]$,
  \begin{equation}\label{iq-improve-regularity-final-together}
  \begin{aligned}
  \E \big[\|(\rho-\rho_{\infty},u)(t)\|_{H^5(\mathbb{R})}^{2p}\big]
  \le  C\big(\E\big[ \| V_0^*\|_{H^5(\mathbb{R})}^{2p}\big],\E \big[\| V_0^*\|_{H^4(\mathbb{R})}^{4p}\big],R,T,C_1(\varepsilon),\varepsilon \big).
  \end{aligned}
  \end{equation}
\end{proposition}

\begin{proof} We divide the proof into three steps.

\smallskip
\textbf{1}. We first rewrite \eqref{eq-cut-off-parabolic-approximation} in the following matrix form:
  \begin{equation}\label{eq-cut-off-parabolic-approximation-matrix-form}
  \d V + \hat F \,\d t =\varepsilon \partial_x^2 V \,\d t +\hat \Psi^{\varepsilon}\,\d W(t),
  \end{equation}
where $V=(\rho,u)^{\top}$, $\hat \Psi^{\varepsilon}=(0,\frac{1}{ \rho}\Phi^{\varepsilon}(\rho,\rho\Pi_R u))^{\top}$, and
  $$
  \hat F=
  (\hat F_1, \hat F_2)^\top
  =(\partial_x (\rho \Pi_R u),\,
  \Pi_R u \partial_x(\Pi_R u)+\frac{P^{\prime}(\rho)}{\rho}\partial_x \rho-2\varepsilon \rho^{-1}\partial_x \rho \partial_x(\Pi_R u))^\top.
  $$
Denote $V^*=(\rho-\rho_{\infty},u)^{\top}, V_0(y)=V(0,y)$, 
and $V_0^*(y)=V^*(0,y)$. 
Thus, we have
  $$
  \begin{aligned}
  \partial_x^3 V(r,x)&=\mathbb{S}_{\varepsilon}(r)\partial_y^3 V_0(x)
  -\int_0^r \int_{\mathbb{R}} \partial_x K_{\varepsilon}(r-\tau,x-y)\,
  \partial_y^2 \hat F (\tau,y) \,\d y\d \tau\\
  &\quad+\int_0^r \int_{\mathbb{R}} K_{\varepsilon}(r-\tau,x-y) \partial_y^3 \hat \Psi^{\varepsilon}(\tau,y) \,\d y \d W(\tau).
  \end{aligned}
  $$
For $\alpha\in(0,1)$, we consider $(I-\Delta)^{\frac{\alpha}{2}}\partial_x^3 V$. 
By Lemmas \ref{lem-heat-semigroup-R}--\ref{lem-bessel-potential-space-embedding}, 
we directly see that
  \begin{equation}\label{eq-improve-regularity-initial-data}
  \begin{aligned}
  \sup_{r\in [0,t]}\E \big[ \|(I-\Delta)^{\frac{\alpha}{2}}\mathbb{S}_{\varepsilon}(r)\partial_y^3 V_0\|_{L^2(\mathbb{R})}^{2p}\big]
  \le C \E \big[\|V_0^*\|_{H^4(\mathbb{R})}^{2p}\big].
  \end{aligned}
  \end{equation}
With this \eqref{eq-improve-regularity-initial-data}, using Lemma \ref{lem-equivalence-of-norms-on-bessel-space},  Lemma \ref{lem-properties-of-fractional-laplace} (ii)--(iii), 
and Lemmas \ref{lem-heat-semigroup-R}--\ref{lem-bessel-potential-space-embedding}, 
we can obtain by the standard estimate that, for $2p\ge 1$ and $\alpha\in(0,1)$,
  \begin{equation}\label{iq-improve-regularity-alpha}
  \sup_{r\in [0,t]}\E \big[\|\partial_x^3 V(r)\|_{H^{\alpha}(\mathbb{R})}^{2p}\big] 
  \le C(\E[ \| V_0^*\|_{H^4(\mathbb{R})}^{2p}],R,t,\alpha,C_1(\varepsilon),\varepsilon).
  \end{equation}
Additionally, repeating the above process for $\partial_x^2 V$, 
we see that, for $2p\ge 1$ and $\alpha\in(0,1)$,
  $$
  \sup_{r\in [0,t]}\E \big[\|\partial_x^2 V(r)\|_{H^{\alpha}(\mathbb{R})}^{2p}\big] \le C(\E \big[\| V_0^*\|_{H^3(\mathbb{R})}^{2p}\big],R,t,\alpha,C_1(\varepsilon),\varepsilon).
  $$

\textbf{2}.
Let $\frac12 \le \alpha <1$, and consider $(I-\Delta)^{\alpha}\partial_x^3 V$. Similar to \eqref{eq-improve-regularity-initial-data}, we have
  $$
  \begin{aligned}
  \sup_{r\in [0,t]}\E \big[ \|(I-\Delta)^{\alpha}\mathbb{S}_{\varepsilon}(r)\partial_y^3 V_0\|_{L^2(\mathbb{R})}^{2p}\big]
  \le C \E\big[ \|V_0^*\|_{H^5(\mathbb{R})}^{2p}\big].
  \end{aligned}
  $$

We now consider two cases separately: $\alpha=\frac12$ and $\frac12 < \alpha <1$.
When $\alpha=\frac12$, choosing $\alpha_1$ satisfying $\frac12 < \alpha_1 <1$, we apply the following decompositions:
  $$
  \begin{aligned}
  &(-\Delta)^{\frac{1}{2}} \mathbb{S}_{\varepsilon}(r-\tau) \partial_y^3 \big(\frac{1}{ \rho}\Phi^{\varepsilon}(\rho,\rho\Pi_R u)\big)(\tau) =(-\Delta)^{\frac{1-\alpha_1}{2}} \mathbb{S}_{\varepsilon}(r-\tau) (-\Delta)^{\frac{\alpha_1}{2}} \partial_y^3 \big(\frac{1}{ \rho}\Phi^{\varepsilon}(\rho,\rho\Pi_R u)\big)(\tau) \\
  &(-\Delta)^{\frac12}\int_0^r \int_{\mathbb{R}}  \partial_x K_{\varepsilon}(r-\tau,x-y)\partial_y^2 \hat F (\tau,y) \,\d y\d \tau\\ 
  &\,\,\,\,=\int_0^r \int_{\mathbb{R}}  (-\Delta)^{\frac{1-\alpha_1}{2}} \partial_x K_{\varepsilon}(r-\tau,x-y) (-\Delta)^{\frac{\alpha_1}{2}} \partial_y^2 \hat F (\tau,y)\,\d y\d \tau.
  \end{aligned}
  $$
Using \eqref{iq-improve-regularity-alpha}, by the mean value theorem, \eqref{iq-lower-upper-bound-for-general-pressure-1}--\eqref{iq-lower-upper-bound-for-general-pressure-2}, Lemma \ref{lem-besov-space-lipschitz-property}--\ref{lem-algebra-property-of-bessel-space}, Lemma \ref{lem-properties-of-fractional-laplace} (iii), Lemma \ref{lem-equivalence-of-norms-on-bessel-space} and Lemma \ref{lem-bessel-potential-space-embedding} (i), 
we obtain that, 
for $p>\frac{1}{\alpha_1}$,
  $$
  \begin{aligned}
  \sup_{r\in [0,t]}\E \big[\|(I-\Delta)^{\frac12}\partial_x^3 V(r)\|_{L^2(\mathbb{R})}^{2p}\big]
  \le  C(\E[\| V_0^*\|_{H^5(\mathbb{R})}^{2p}],R,t,\alpha_1,C_1(\varepsilon),\varepsilon \big).
  \end{aligned}
  $$

Similarly, when $\frac12<\alpha<1$, we obtain that, for $p>\frac{1}{1-\alpha}$,
  $$
  \begin{aligned}
  \sup_{r\in [0,t]}\E \big[\|(I-\Delta)^{\frac12}\partial_x^3 V(r)\|_{H^{2\alpha-1}(\mathbb{R})}^{2p}\big]
  \le  C(\E[\| V_0^*\|_{H^5(\mathbb{R})}^{2p}],R,t,\alpha,C_1(\varepsilon),\varepsilon \big).
  \end{aligned}
  $$
Therefore, using Lemma \ref{lem-properties-of-fractional-laplace}(iii), we see that, for $p>\frac{1}{\alpha_1}$ with $\alpha_1$ satisfying $\frac12 < \alpha_1 <1$,
  $$
  \sup_{r\in [0,t]}\E\big[ \|\partial_x^4 V(r)\|_{L^2(\mathbb{R})}^{2p}\big] \le C(\E[\| V_0^*\|_{H^5(\mathbb{R})}^{2p}],R,t,\alpha_1,C_1(\varepsilon),\varepsilon \big).
  $$
and, for $\frac12<\alpha<1$ and $p>\frac{1}{1-\alpha}$,
  $$
  \sup_{r\in [0,t]}\E\big[ \|\partial_x^4 V(r)\|_{H^{2\alpha-1}(\mathbb{R})}^{2p}\big] \le C\big(\E [\| V_0^*\|_{H^5(\mathbb{R})}^{2p}],R,t,\alpha,C_1(\varepsilon),\varepsilon \big).
  $$
Moreover, by Lemma \ref{lem-bessel-potential-space-embedding}, choosing $\alpha>\frac34$, we have
  $$
  \sup_{r\in [0,t]}\E \big[\|\partial_x^4 V(r)\|_{C^{2\alpha-1-\frac12}(\mathbb{R})}^{2p}\big] \le C(\E [ \| V_0^*\|_{H^5(\mathbb{R})}^{2p}],R,t,\alpha,C_1(\varepsilon),\varepsilon \big).
  $$

In addition, repeating the above process regarding $\frac12<\alpha<1$ for $\partial_x^2 V$, we also obtain that, for $\frac12<\alpha<1$ and $2p\ge 1$,
  $$
  \begin{aligned}
  \sup_{r\in [0,t]}\E \big[\|(I-\Delta)^{\alpha}\partial_x^2 V(r)\|_{L^2(\mathbb{R})}^{2p}\big]
  \le  C(\E[\| V_0^*\|_{H^4(\mathbb{R})}^{2p}],R,t,\alpha,C_1(\varepsilon),\varepsilon \big).
  \end{aligned}
  $$

\textbf{3}. We now consider $(I-\Delta)\partial_x^3 V$. By similar arguments as those in Step\textbf{ 2}, we obtain that, 
for $\frac12 <\alpha\le\frac23$ and $p>\frac{1}{2\alpha-1}$, 
or for $\frac23 <\alpha<1$ and $p>\frac{1}{1-\alpha}$, we have
  $$
  \begin{aligned}
  \sup_{r\in [0,t]}\E\big[ \|(I-\Delta)\partial_x^3 V(r)\|_{L^2(\mathbb{R})}^{2p}\big]
  \le  C(\E[\| V_0^*\|_{H^5(\mathbb{R})}^{2p}],
  \E[\| V_0^*\|_{H^4(\mathbb{R})}^{4p}],R,t,\alpha,C_1(\varepsilon),\varepsilon \big).
  V_0^*\|_{H^4(\mathbb{R})}^{4p},R,\rho_{\infty},t,\alpha,\gamma,C_1(\varepsilon),\varepsilon,|\text{supp}_x (\zeta_k^{\varepsilon})|).
  \end{aligned}
  $$
Then, by Lemma \ref{lem-properties-of-fractional-laplace}(iii), we obtain that, 
for $\frac12 <\alpha\le\frac23$ and $p>\frac{1}{2\alpha-1}$, 
or for $\frac23 <\alpha<1$ and $p>\frac{1}{1-\alpha}$,
  $$
  \begin{aligned}
  \sup_{r\in [0,t]}\E \big[\|\partial_x^5 V(r)\|_{L^2(\mathbb{R})}^{2p}\big]
  \le  C\big(\E[ \| V_0^*\|_{H^5(\mathbb{R})}^{2p}],
  \E[ \| V_0^*\|_{H^4(\mathbb{R})}^{4p}],R,t,\alpha,C_1(\varepsilon),\varepsilon \big).
  \end{aligned}
  $$
Taking $\alpha=\frac23$, then the proof is completed.
\end{proof}

\subsection{Remove the Cut-Off}\label{sec-Remove-the-cut-off-operator}
We denote the solution obtained 
in Theorem \ref{thm-wellposedness-cut-off-parabolic-approximation-R-relax-initial-data} by $(\bar \rho_R,\bar u_R)$ 
and define $\bar m_R:=\bar \rho_R \bar u_R$. 
Then we pass to the limit $R\to \infty$ to obtain 
the solution to \eqref{eq-parabolic-approximation}, {\it i.e.}, remove the cut-off. Toward this end, we define the stopping time
 $$
 \tau_R =\inf\big\{ t\in [0,T]\,:\,\|\bar u_R\|_{H^3(\mathbb{R})}(t)\ge R\big\}
 $$
with the convention that $\inf \emptyset =T$. 
Note that $\tau_R$ is increasing and $(\bar \rho_R,\bar u_R)=(\bar \rho_{R+1},\bar u_{R+1})$ on $[0,\tau_R]$ by uniqueness. 
Therefore, we can define $\tau^*=\lim_{R\to \infty} \tau_R$ and $(\tilde \rho,\tilde u):=(\bar \rho_R,\bar u_R)$ on $[0,\tau_R]$, which is the solution to \eqref{eq-parabolic-approximation} on $[0,\tau_R]$. 
Then the next step is to prove that $\mathbb{P}(\tau^*=T)=1$, 
which implies that $(\tilde \rho,\tilde u)$ is the global solution to \eqref{eq-parabolic-approximation}. 
To prove that $\mathbb{P}(\tau^*=T)=1$, 
we need some uniform
estimates with respect to $R$.

Before proving the uniform estimates, we first use the Sobolev embedding to obtain
\begin{lemma}\label{lem-parabolic-approximation-classical-sense}
Assume that $U_0$ satisfies the conditions in {\rm Theorem \ref{thm-wellposedness-cut-off-parabolic-approximation-R-relax-initial-data}}. Then the solution $(\bar \rho_R,\bar u_R)$ satisfies $(\partial_x^2 \bar \rho_R, \partial_x^2 \bar u_R)(t,\cdot) \in C(\mathbb{R})$, so that
the cut-off equations \eqref{eq-cut-off-parabolic-approximation} hold in the classical strong sense that, for any fixed $x$, equations \eqref{eq-cut-off-parabolic-approximation} hold as a finite-dimensional stochastic differential equations.
\end{lemma}

We now derive the 
lower bound for density $\rho$.

\begin{proposition}\label{prop-a-priori-lower-bound-for-density}
Assume that $U_0$ satisfies the conditions in 
{\rm Theorem \ref{thm-wellposedness-cut-off-parabolic-approximation-R-relax-initial-data}}. 
Then the density has lower bound{\rm :} 
$\|{\bar \rho_R}^{-1}\|_{L^{\infty}(\mathbb{R}\times[0,T])}
\le e^{C(T, c_0, \varepsilon, \| \bar u_R\|_{L^{\infty}([0,T]\times \mathbb{R})})}
\,\,\,\mathbb{P}$-{\it a.s.}.
\end{proposition}

\begin{proof}
The proof is similar to that of \eqref{iq-lower-bound-for-density-in-parabolic-approximation} but here we use the integration by parts such that the lower bound depends on $\| \bar u_R\|_{L^{\infty}([0,T]\times \mathbb{R})}$ only, {\it i.e.}, is independent of 
$\| \partial_x \bar u_R\|_{L^{\infty}([0,T]\times \mathbb{R})}$. So we omit the details.
\end{proof}

We now derive the 
uniform $L^{\infty}$ estimate with respect to $R$.

\begin{proposition}\label{prop-L-infty-estimates-by-invariant-region}
Assume that $U_0$ satisfies the conditions in 
{\rm Theorem \ref{thm-wellposedness-cut-off-parabolic-approximation-R-relax-initial-data}}.
Then the solution $(\bar \rho_R,\bar u_R)$ on $[0,\tau_R]$ is bounded, {\it i.e.}, $(\bar \rho_R,\bar u_R)\in \Gamma_{\mathcal{H}^{\varepsilon}}$
for all $t\in [0,\tau_R]$ $\mathbb{P}$-{\it a.s.}. 
Therefore, $\| \bar u_R\|_{L^{\infty}([0,\tau_R]\times \mathbb{R})}\le \mathcal{H}^{\varepsilon}$ and $\| \bar \rho_R\|_{L^{\infty}([0,\tau_R]\times \mathbb{R})}\le (\mathcal{H}^{\varepsilon})^{\frac{1}{\theta_2}}$ $\mathbb{P}$-{\it a.s.}.
\end{proposition}

\begin{proof}
We use the idea of the invariant region method 
(for the invariant region method in the deterministic case, see \cite[\S 4]{Diperna83} 
and \cite{CCS77} for more details) 
and take the Riemann invariant 
$w_2(\rho,m)=\frac{m}{\rho}+\mathcal{K}(\rho)$ for an example. 
Another Riemann invariant $w_1
$ can be dealt with similarly.

Notice that, on $[0,\tau_R]$, the cut-off equations \eqref{eq-parabolic-approximation} and \eqref{eq-cut-off-parabolic-approximation}
are coincide. 
Therefore, restricting ourselves on $\left[0,\tau_R\right)$,
by Lemma \ref{lem-parabolic-approximation-classical-sense}, 
we apply the It\^o formula to derive the equation for $w_2
$:
  \begin{align}\label{eq-riemann-invariant}
  \d w_2 &= \varepsilon \partial_x^2 w_2 \,\d t 
  + \big(2\varepsilon \frac{ \partial_x \bar \rho_R }{\bar \rho_R}  - \lambda_2 \big) \partial_x w_2\,\d t 
  -\frac{\varepsilon (\partial_x\bar \rho_R)^2}{2(\bar \rho_R)^2 \sqrt{P^{\prime}(\bar \rho_R)}}\big(2P^{\prime}(\bar \rho_R)+\bar \rho_R P^{\prime \prime}(\bar \rho_R)\big)\,\d t\nonumber\\
  &\quad\, +\sum_k \frac{a_k}{\bar \rho_R} \zeta_k^{\varepsilon} \,\d \beta_k(t),
  \end{align}
where $\lambda_2=\frac{ \bar m_R}{\bar \rho_R}+ \sqrt{P^{\prime}(\bar \rho_R)}$.

Since $(\bar \rho_R-\rho_{\infty},\bar u_R)\in C([0,T];H^3(\mathbb{R}))$ $\P$-{\it a.s.}, by the Sobolev embedding, $(\bar \rho_R, \bar u_R)$ is $\mathbb{P}$-{\it a.s.} 
continuous with respect to $(t,x)$. 
Thus, $w_2(\bar \rho_R,  \bar m_R)=\frac{ \bar m_R}{ \bar \rho_R}+ \mathcal{K}(\bar \rho_R)$ is also $\mathbb{P}$-{\it a.s.} continuous 
with respect to $(t,x)$. 
Moreover,  $\|\bar \rho_R-\rho_{\infty}\|_{C([0,T],L^{\infty}(\mathbb{R}))}<\infty$ and $\| \bar u_R \|_{C([0,T],L^{\infty}(\mathbb{R}))}<\infty$ $\mathbb{P}$-{\it a.s.}, which implies that $\|w_2\|_{C([0,\tau_R],L^{\infty}(\mathbb{R}))}<\infty$ $\mathbb{P}$-{\it a.s.}. 
Additionally, since $( \bar \rho_R-\rho_{\infty}, \bar u_R)\in C([0,T],H^1(\mathbb{R}))$ $\mathbb{P}$-{\it a.s.}, we obtain that, 
for each $t\in [0,T]$, $( \bar \rho_R-\rho_{\infty}, \bar u_R)(t,x) \to (0,0)$ as $|x|\to \infty$. 
Therefore, for each $t\in [0,T]$, for any $0<\delta\le1$, 
there exists $L(t,\delta)>0$ such that,
for $|x|\ge L(t,\delta)$, 
$| \bar \rho_R-\rho_{\infty}|\le \delta$ and $|\bar u_R|\le \delta$. 
Recall that $\mathcal{H}^{\varepsilon}\ge 1+(\rho_{\infty}+1)^{\theta_2}$ (see \eqref{eq-invariant-region-H-varepsilon}). 
Hence, for $|x|\ge L(t,\delta)$, we see that
 $w_2( \bar \rho_R, \bar m_R)(t,x) \le \delta+(\rho_{\infty}+\delta)^{\theta_2}\le 1+(\rho_{\infty}+1)^{\theta_2} \le \mathcal{H}^{\varepsilon}$, where we have 
 used \eqref{iq-lower-upper-bound-for-general-pressure-2} for the general pressure law case. For each $t\in [0,T]$, define
  $$
  \tilde L(t,1)
  :=\inf\big\{ L(t,1)\,:\, w_2(t,x)\le 1+(\rho_{\infty}+1)^{\theta_2}\,\,\,\mbox{for any 
  $|x|\ge L(t,1)$}\big\}.
  $$

Suppose that the set
  $$
  \Sigma:=\big\{(t,x)\in \left(0,\tau_R \right]\times \mathbb{R}\,:\, w_2(t,x)>\mathcal{H}^{\varepsilon} \big\}
  $$
is nonempty. Then, by continuity of $w_2$, 
this set $\Sigma$ is an open set.
Also, $\Sigma \subset \Sigma_b$, where
  $$
  \Sigma_b:=\big\{(t,x)\,:\, t\in [0,\tau_R],\ |x| \le \tilde L(t,1) \big\}.
  $$
Moreover, on $\Sigma$, 
by \eqref{condition-compact-support-of-noise-coefficient}, 
equation \eqref{eq-riemann-invariant} turns into a deterministic equation:
  $$
  \d w_2 = \varepsilon \partial_x^2 w_2 \,\d t 
  + \big(2\varepsilon \frac{ \partial_x \bar \rho_R }{\bar \rho_R}  - \lambda_2 \big) \partial_x w_2\,\d t -\frac{\varepsilon (\partial_x\bar \rho_R)^2}{2(\bar \rho_R)^2 \sqrt{P^{\prime}(\bar \rho_R)}}\big(2P^{\prime}(\bar \rho_R)+\bar \rho_R P^{\prime \prime}(\bar \rho_R)\big)\,\d t
  $$
which implies
  $$
  \partial_t w_2 \le \varepsilon \partial_x^2 w_2 + \big(2\varepsilon \frac{\partial_x\bar \rho_R}{\bar \rho_R}  - \lambda_2 \big) \partial_x w_2.
  $$
Since $\|w_2\|_{C([0,\tau_R],L^{\infty}(\mathbb{R}))}<\infty$, 
we find that $\mathcal{H}^{\varepsilon}<M:=\sup_{(t,x)\in \Sigma}w_2(t,x) <\infty$. Consider $w_2^{\beta}:=w_2-\beta t$ for $\beta>0$. Then
  $$
  \partial_t w_2^{\beta} \le \varepsilon \partial_x^2 w_2^{\beta}  + \big(2\varepsilon \frac{\partial_x\bar \rho_R}{\bar \rho_R}  - \lambda_2 \big) \partial_x w_2^{\beta}-\beta.
  $$
Also, $M^{\beta}:=\sup_{(t,x)\in \Sigma}w_2^{\beta}(t,x) <\infty$. 
If the supremum $M^{\beta}$ is attained at an interior point $(t_0,x_0)\in \Sigma$
with $|x_0|< \tilde L(t_0,1)$, then, at this point $(t_0,x_0)$ (recall that $\Sigma$ is an open set), we have
  $$
  \partial_t w_2^{\beta}\ge0,\quad \partial_x w_2^{\beta}=0,\quad \partial_x^2 w_2^{\beta}\le 0.
  $$
Thus, we obtain that $0\le \partial_t w_2^{\beta}\le -\beta <0$,
which is a contradiction. 
Therefore, $M^{\beta}=\sup_{(t,x)\in \partial\Sigma}w_2^{\beta}(t,x)$. 
Taking $\beta\to 0$ gives $M=\sup_{(t,x)\in \partial\Sigma}w_2(t,x)$.
However, it follows from the definition of $\Sigma$
that $w_2(t,x)|_{(t,x)\in \partial\Sigma}=\mathcal{H}^{\varepsilon}<M$, 
which is a contradiction again. Therefore, the set $\Sigma$ is an empty set and $w_2(t,x) \le \mathcal{H}^{\varepsilon}$ for any $(t,x)\in [0,\tau_R]\times \mathbb{R}$. 
This completes the proof.
\end{proof}

Propositions 
\ref{prop-a-priori-lower-bound-for-density}--\ref{prop-L-infty-estimates-by-invariant-region} lead to the following simple corollary.

\begin{corollary}\label{cor-lower-bound-for-density}
Assume that $U_0$ satisfies the conditions in 
{\rm Theorem \ref{thm-wellposedness-cut-off-parabolic-approximation-R-relax-initial-data}}. 
Then the density has a lower bound{\rm :} $\|{\bar \rho_R}^{-1}\|_{L^{\infty}(\mathbb{R}\times[0,\tau_R])}
\le e^{C(T, c_0, \varepsilon, \mathcal{H}^{\varepsilon})}$ $\mathbb{P}$-{\it a.s.}.
\end{corollary}

Next, we derive the 
uniform energy estimate.

\begin{proposition}\label{prop-energy-estimates}
Assume that $U_0$ satisfies the conditions in 
{\rm Theorem \ref{thm-wellposedness-cut-off-parabolic-approximation-R-relax-initial-data}}. Then the solution $(\bar \rho_R,\bar u_R)$ of the cut-off equations \eqref{eq-cut-off-parabolic-approximation} satisfies energy estimate{\rm :}
    \begin{equation}\label{iq-energy-estimate-stopping-time-a}
    \begin{aligned}
    &\mathbb{E} \big[\Big(\sup_{r\in [0,t\land \tau_R]}\int_{\mathbb{R}} \big(\frac12 \bar \rho_R (\bar u_R)^2 +e^*(\bar \rho_R,\rho_{\infty})\big) \,\d x \\
    &\,\,\,\, \quad+ \varepsilon \int_0^{t\land \tau_R} \int_{\mathbb{R}} \Big(\big(\bar \rho_R e(\bar \rho_R)\big)^{\prime \prime}(\partial_x\bar \rho_R)^2 + \bar \rho_R (\partial_x \bar u_R)^2\Big) \,\d x \d r \Big)^p\big]
    \le C(p,T,E_{0,p}),
    \end{aligned}
    \end{equation}
where 
$C(p,T,E_{0,p})$ is independent of $\varepsilon$ and $R$.
\end{proposition}

\begin{proof}
By Lemma \ref{lem-parabolic-approximation-classical-sense},
for any fixed $x$, equations \eqref{eq-cut-off-parabolic-approximation} hold as a finite-dimensional stochastic differential system. 
Since  the cutoff equations \eqref{eq-cut-off-parabolic-approximation} and \eqref{eq-parabolic-approximation} coincide when restricted to $[0,\tau_R]$,
we apply the It\^o formula to the relative mechanical energy $\eta^*_E$ 
({\it cf.} \eqref{eq-relative-mechanical-energy}) to obtain from \eqref{eq-parabolic-approximation} that
$$
\d \eta^*_E(U)
=-\partial_x q_E^* (U) \,\d t + \varepsilon \nabla \eta_E^*(U)\partial_x^2 U\,\d t
   +\nabla \eta_E^*(U)\Psi^{\varepsilon}(U) \,\d W(t) 
   + \frac12 \partial_{m}^2 \eta_E^*(U) \big(\mathfrak{G}_i^{\varepsilon} \big)^2 \,\d t,
$$
where $(\eta^*_E,q_E^*)$ is an entropy pair with 
$q_E^*(\rho,m)=\frac12 \frac{m^3}{\rho^2}
+m\big( \rho e(\rho)\big)^{\prime}-\big( \rho e(\rho)\big)^{\prime}|_{\rho=\rho_\infty}m$ and $\mathfrak{G}_i^{\varepsilon}, i=1,2$, 
are defined in the beginning of \S \ref{sec-parabolic-approximation}. 
Then, by integration by parts, we see that, for $i=1,2$,
$$
  \begin{aligned}
  &\int_{\mathbb{R}} \eta_E^* (U)(\tau)\,\d x 
  +\varepsilon \int_0^{\tau} \int_{\mathbb{R}} (\partial_x U)^{\top}\,(\eta_E^{*})^{\prime \prime}(U)\,
  \partial_x U \,\d x\d r \\
  &= \int_{\mathbb{R}} \eta_E^* (U_0) \,\d x+ \int_0^{\tau} \int_{\mathbb{R}} \nabla \eta_E^*(U)\Psi^{\varepsilon}(U) \,\d x\d W(r) + \int_0^{\tau} \int_{\mathbb{R}} \frac12 \partial_{m}^2 \eta_E^*(U) \big(\mathfrak{G}_i^{\varepsilon} \big)^2 \,\d x \d r.
  \end{aligned}
$$

Since the solution satisfies $(\rho-\rho_{\infty},u)\in C([0,T];H^3(\mathbb{R}))$ $\P$-{\it a.s.}, 
using the Sobolev embedding: $H^1(\mathbb{R})\hookrightarrow L^{\infty}(\mathbb{R})$ and the fact that 
$\zeta_k^{\varepsilon}$ has compact support, 
we can directly prove that the stochastic integral 
$G^{\varepsilon}(\tau):=\int_0^{\tau} \int_{\mathbb{R}} \nabla \eta_E^*(U)\Psi^{\varepsilon}(U) \,\d x\d W(r)$ is a local martingale. 
Hence, by the Burkholder-Davis-Gundy inequality, we obtain
  $$
  \begin{aligned}
  \E\big[ \sup_{\tau\in [0,t\land \tau_R]}|G^{\varepsilon}(\tau)|^p\big]
  &=\E \big[ \sup_{\tau\in [0,t\land \tau_R]}\Big|\int_0^{\tau} \sum_k \int_{{\rm supp}_x \zeta_k^{\varepsilon}} u a_k \zeta_k^{\varepsilon} \,\d x\d \beta_k(r)  \Big|^p\big]\\
  &\le C\, \E \big[\Big( \int_0^{t\land \tau_R} \sum_k \Big( \int_{{\rm supp}_x \zeta_k^{\varepsilon}} u a_k \zeta_k^{\varepsilon} \,\d x \Big)^2 \d r \Big)^{\frac{p}{2}}\big]\\
  &\le  \delta \,\E\big[\Big(\sup_{r\in[0,t\land \tau_R]} \int_{\mathbb{R}} \rho u^2 \,\d x\Big)^p\big] 
  + C_{\delta}\E\big[\Big( \int_0^{t\land \tau_R} \sum_k  \int_{{\rm supp}_x \zeta_k^{\varepsilon}} \frac{1}{\rho}a_k^2 ( \zeta_k^{\varepsilon})^2 \,\d x \d r \Big)^p\big].
  \end{aligned}
  $$
Recall that we have estimated the second term on the right-hand side of the above inequality in \eqref{iq-estimate-for-noise-term-1/rho--in-energy-estimate}. 

Therefore, we conclude
$$
\begin{aligned}
  \E \big[\sup_{\tau\in [0,t\land \tau_R]}|G^{\varepsilon}(\tau)|^p\big]
  &\le \delta\,\E\big[\Big(\sup_{r\in[0,t\land \tau_R]} \int_{\mathbb{R}} \rho u^2 \,\d x\Big)^p\big] + 
  C_{\delta}C(p,t)\\
  &\quad\,\,+ C_{\delta}
  C(p)\,\E \big[\Big(\int_0^{t\land \tau_R}  \int_{{\rm supp}_x (\zeta_k^{\varepsilon})} 
  \big(\rho u^2 +e^*(\rho, \rho_{\infty})\big) \,\d x \d r \Big)^p\big].
\end{aligned}
$$

Notice that $\partial_{m}^2 \eta_E^*(U)=\frac{1}{\rho}$, by the same arguments as above, we obtain that,
for $i=1,2$,
$$
\begin{aligned}
&\E\big[ \sup_{\tau\in [0,t\land \tau_R]}\Big|\int_0^{\tau} \int_{\mathbb{R}} \frac12 \partial_{m}^2 \eta_E^*(U) 
\big(\mathfrak{G}_i^{\varepsilon} \big)^2 \,\d x \d r \Big|^p\big]\\
  &\le 
  C(p,t)
  +
  C(p)\E \big[\Big(\int_0^{t\land \tau_R}  \int_{{\rm supp}_x (\zeta_k^{\varepsilon})} 
  \big(\rho u^2 +e^*(\rho, \rho_{\infty})\big) \,\d x \d r \Big)^p\big].
\end{aligned}
$$

Combining the above estimates, we have
$$
\begin{aligned}
  &\mathbb{E} \big[\Big(\sup_{\tau\in [0,t\land \tau_R]}\int_{\mathbb{R}} \big(\frac12 \rho u^2 +e^*(\rho,\rho_{\infty})\big) \,\d x
  + \varepsilon \int_0^{t\land \tau_R} \int_{\mathbb{R}} \big((\rho e(\rho))^{\prime \prime}(\partial_x\rho)^2
    + \rho (\partial_x u)^2\big) \,\d x \d r \Big)^p\big]\\
  &\le  \delta\, \E \big[\Big(\sup_{r\in[0,t\land \tau_R]} \int_{\mathbb{R}} \rho u^2 \,\d x\Big)^p\big] 
  +C(p)\mathbb{E}\big[ \Big(\int_{\mathbb{R}} \big(\frac12 \rho_0 u_0^2 +e^*(\rho_{\varepsilon 0 },\rho_{\infty})\big) \,\d x \Big)^p\big] \\
  &\quad+
  C_{\delta}C(p,t)
  + C_{\delta}
  C(p)\,\E \big[\Big(\int_0^{t\land \tau_R}  \int_{{\rm supp}_x \zeta_k^{\varepsilon}} \big(\rho u^2 +e^*(\rho, \rho_{\infty})\big) \,\d x \d r \Big)^p\big].
  \end{aligned}$$
Then the proof is completed by taking $\delta$ sufficiently small and using the Gr\"onwall inequality.
\end{proof}

We are now ready to prove the uniform estimates with respect to $R$.

\begin{proposition}\label{prop-uniform-H3-estimates-with-respect-to-R}
Assume that $U_{0}$ satisfies the conditions in 
{\rm Theorem \ref{thm-wellposedness-cut-off-parabolic-approximation-R-relax-initial-data}} and 
  $$
   (\rho_0-\rho_{\infty}, u_0)\in L^{p_0}(\Omega;H^5(\mathbb{R}))\qquad \text{for some $p_0>12$.}
  $$
Then the solution $(\bar \rho_R, \bar u_R)$ of the cut-off equations \eqref{eq-cut-off-parabolic-approximation}, obtained 
in {\rm Theorem \ref{thm-wellposedness-cut-off-parabolic-approximation-R-relax-initial-data}}, obeys the following uniform estimates{\rm :}
   \begin{equation}\label{iq-total-3-order-derivative-estimate-for-cut-off-uniform-in-R}
   \begin{aligned}
  \E \big[\sup_{r\in [0,\,t\land \tau_R]}\|(\bar \rho_R-\rho_{\infty},\bar u_R)  \|_{H^3(\mathbb{R})}^2+   \int_0^{t\land \tau_R}\|(\partial_x \bar \rho_R, \partial_x \bar u_R)\|_{H^3(\mathbb{R})}^2 \,\d r\big] \le C,
  \end{aligned}
  \end{equation}
where $C>0$ is independent of $R$.
\end{proposition}

For the proof of Proposition \ref{prop-uniform-H3-estimates-with-respect-to-R}, notice that, on $[0,\tau_R]$, 
the cut-off equations \eqref{eq-parabolic-approximation} and 
\eqref{eq-cut-off-parabolic-approximation} are coincide. 
Additionally, $U_0$ and $(\bar \rho_R, \bar u_R)$ satisfy the condition in Propositions 
\ref{prop-a-priori-lower-bound-for-density}--\ref{prop-energy-estimates}. Therefore, restricting ourselves on $\left[0,\tau_R\right)$,
we can work on equations \eqref{eq-parabolic-approximation} and apply the estimates in Propositions 
\ref{prop-a-priori-lower-bound-for-density}--\ref{prop-energy-estimates}, using which the proof can be done by standard careful estimates 
for parabolic equations. 
We omit the details and only point out that, by the regularity estimates in Proposition \ref{prop-regularity-for-cut-off-equ}, we can use the infinite-dimensional It\^o formula in \cite[Theorem 4.32]{dapratozabczyk14} for the functional
  \begin{equation}\label{eq-L2-functional-for-ito-formula}
  \begin{aligned}
  F:&\, L^2(\mathbb{R})\to \mathbb{R},\\
  &f(t)\mapsto  \|f(t)\|_{L^2(\mathbb{R})}^2\qquad \mbox{for any $t\in[0,T]$},
  \end{aligned}
  \end{equation}
with $f(t)=\partial^i_x(\bar{\rho}_R-\rho_\infty)$ or $\partial^i_x\bar{u}_R$ for 
$i=0,1,2,3$.

\smallskip
Now we remove the cut-off and obtain the global solution to \eqref{eq-parabolic-approximation}. By Proposition \ref{prop-uniform-H3-estimates-with-respect-to-R}, we have
$$\begin{aligned}
\P \{\tau^* < T\} &\le \P \{\tau_R < T\}\\
&=\P \{\|\bar u_R\|_{C([0,\tau_R];H^3(\mathbb{R}))}\ge R \}\\
&\le \frac{1}{R^2} \E\big[\sup_{t\in[0,\,\tau_R]}\|\bar u_R\|_{H^3(\mathbb{R})}^2\big]\\
&\le \frac{C}{R^2} \to 0\quad \text{as}\ R\to \infty,
\end{aligned}$$
which implies that $\P(\tau^* = T)=1$. 
Thus, $(\tilde \rho,\tilde u)$, defined by $(\tilde \rho,\tilde u):=(\bar \rho_R,\bar u_R)$ on $[0,\tau_R]$, is the global solution to \eqref{eq-parabolic-approximation} when $U_{0}$ satisfies the conditions in Theorem \ref{thm-wellposedness-cut-off-parabolic-approximation-R-relax-initial-data}.

\subsection{Completion of the Proof of Theorem \ref{thm-wellposedness-parabolic-approximation-R}}

In \S \ref{sec-Remove-the-cut-off-operator}, we have obtained the global solution $(\tilde \rho,\tilde u)$ to \eqref{eq-parabolic-approximation} under the assumptions that $U_{0}$ satisfies the conditions 
in Theorem \ref{thm-wellposedness-parabolic-approximation-R} and 
\begin{equation}\label{eq-additional-assumption-initial-data}
\rho_0-\rho_{\infty}\in L^{\infty}(\Omega; H^3(\mathbb{R})),\quad u_0\in L^p(\Omega;H^4(\mathbb{R}))\,\,\mbox{for}\,\, p\ge 1,\quad (\rho_0-\rho_{\infty}, u_0)\in L^{p_0}(\Omega;H^5(\mathbb{R}))
\end{equation}
for some $p_0>12$. We now remove these additional assumptions. To do this, we need the following uniqueness result.

\begin{proposition}\label{prop-uniqueness-of-strong-solution-to-parabolic-approximation}
Assume that $U_{0}$ satisfies the conditions 
in {\rm Theorem \ref{thm-wellposedness-parabolic-approximation-R}} and \eqref{eq-additional-assumption-initial-data}.
Then the strong solution of \eqref{eq-parabolic-approximation} satisfying conditions {\rm (i)}--{\rm (ii)} 
in {\rm Theorem \ref{thm-wellposedness-parabolic-approximation-R}} is unique.
\end{proposition}

We give a sketch of the proof for Proposition \ref{prop-uniqueness-of-strong-solution-to-parabolic-approximation} here. To show the pathwise uniqueness, we consider two strong solutions $U_1$ and $U_2$ to \eqref{eq-parabolic-approximation}. We then define the stopping time for $U_i, i=1,2$:
 $$
 \tau_R^i =\inf\big\{ t\in [0,T]\,:\, \|\rho_i-\rho_{\infty}\|_{H^3(\mathbb{R})}(t)+\|m_i\|_{H^3(\mathbb{R})}(t)\ge R\big\}
 $$
with the convention that $\inf \emptyset =T$ and 
$\tilde \tau_R=\tau_R^1 \land \tau_R^2$. 
Then $\tilde \tau_R$ is increasing.
By \eqref{iq-total-3-order-derivative-estimate-for-cut-off-uniform-in-R}, 
we have
$$\begin{aligned}
  \P \{\sup_{R\in \mathbb{N}}\tilde \tau_R < T\} &\le \P \{\tilde \tau_R < T\}\\
  &\le \sum_{i=1,2}\P \{\|\rho_i-\rho_{\infty}\|_{C([0, \tau_R^i];H^3(\mathbb{R}))}+\|m_i\|_{C([0, \tau_R^i];H^3(\mathbb{R}))}\ge R  \}\\
  &\le \frac{C}{R^2} \to 0\qquad\, \text{as $R\to \infty$}.
  \end{aligned}$$
Restricting to $[0, \tilde \tau_R]$, 
we can prove the uniqueness by the standard estimate, with the aid of 
(ii) in Theorem \ref{thm-wellposedness-parabolic-approximation-R}. 
Then, sending $R\to \infty$, we can complete the proof.

\bigskip
We now remove the additional conditions of the initial data and 
prove Theorem \ref{thm-wellposedness-parabolic-approximation-R}. 
More specifically, we consider the general initial data $(\rho_0,u_0)$ 
satisfying the conditions in 
Theorem \ref{thm-wellposedness-parabolic-approximation-R}. 
For any given $N>0$, we define
  $
  A_N=\{(\rho,u)\,:\,\ \|(\rho-\rho_{\infty},u)\|_{H^5(\mathbb{R})}\le N   \}.
  $
Then, for the initial data $(\rho_0,u_0)1_{(\rho_0,u_0)\in A_N}$, 
the argument in \S \ref{sec-Remove-the-cut-off-operator} gives us the unique 
global strong solution $(\tilde \rho_N,\tilde u_N)$ 
to \eqref{eq-parabolic-approximation}. 
Moreover, by the uniqueness result in 
Proposition \ref{prop-uniqueness-of-strong-solution-to-parabolic-approximation}, 
we obtain that, when $N\ge M$,
  $
  (\tilde \rho_N,\tilde u_N)1_{(\rho_0,u_0)\in A_M}=(\tilde \rho_M,\tilde u_M)1_{(\rho_0,u_0)\in A_M}.
  $
Since $(\rho_0-\rho_{\infty},u_0)\in H^5(\mathbb{R})$ $\P$-{\it a.s.}, we see that 
$\P \{(\rho_0,u_0)\in \cup_{N=1}^{\infty}A_N\}=1$. 
Define
  $$
  (\rho,u)=(\tilde \rho_N,\tilde u_N)\qquad \text{when $(\rho_0,u_0)\in A_N$}.
  $$
We find that $(\rho,u)$ is the unique global strong solution to \eqref{eq-parabolic-approximation} satisfying $(\rho-\rho_{\infty}, u) \in C([0,T];H^3(\mathbb{R}))\cap L^2([0,T];H^4(\mathbb{R}))$ $\mathbb{P}$-{\it a.s.}, {\it i.e.} condition (i) 
in Theorem \ref{thm-wellposedness-parabolic-approximation-R}. 
Additionally, condition (ii) in 
Theorem \ref{thm-wellposedness-parabolic-approximation-R} follows from Proposition \ref{prop-L-infty-estimates-by-invariant-region} and Corollary \ref{cor-lower-bound-for-density}, conditions (iii) and (iv) 
in Theorem \ref{thm-wellposedness-parabolic-approximation-R} follow 
from Lemma \ref{lem-parabolic-approximation-classical-sense} and Proposition \ref{prop-energy-estimates}, respectively. This completes the proof 
of Theorem \ref{thm-wellposedness-parabolic-approximation-R}.

\appendix
\section{Appendix}
In this appendix, we recall some analytic results used in this paper, 
present a stochastic version of the generalized Murat lemma 
and a generalized div-curl lemma, 
and provide the proofs of Proposition 4.1.

\subsection{Some Analytic Results}
\renewcommand\theequation{A.\arabic{equation}}
\subsubsection{Some properties related to the heat semigroup}

Using the same notation as in the proof 
of Lemma \ref{lem-wellposedness-cut-off-parabolic-approximation-R}, 
we denote $\mathbb{S}_{\varepsilon}(t)$ as the heat semigroup and $K_{\varepsilon}(t,x)$ 
as the heat kernel corresponding to the following heat equation:
  $$
  \partial_t f-\varepsilon \Delta f =g \qquad \mbox{for $x\in \mathbb{R}$ and $t\in [0,T]$}.
  $$
where and hereafter $\Delta=\partial_{x}^2$. 
In this appendix, we consider the case $\varepsilon =1$ and simply
use the notation $\mathbb{S}(t)$ and $K(t,x)$. 
By the representation of the heat kernel \eqref{eq-heat-kernel-representation}, 
we can directly prove the following well-known properties:

\begin{lemma}\label{lem-heat-semigroup-R}
For all $p\in [1,\infty], \,l,m\in \mathbb{N}$, and $t>0$,
  $$
  \|\partial_t^l \partial_x^m K(t)\|_{L^p(\mathbb{R})} \le C(l,m,p)t^{-l-\frac{m}{2}-\frac{1}{2P^{\prime}}},
  $$
where $\frac{1}{p}+\frac{1}{p^{\prime}}=1$, and $C(l,m,p)>0$ is a constant depending on $l,m$,and $p$.
\end{lemma}

We now recall the definition and an embedding theorem for the Bessel potential spaces:
\begin{definition}\label{def-bessel-potential-space}
For $\ell\in \mathbb{R}$ and $p\in (1,\infty)$, 
the Bessel potential space $H_p^{\ell}(\mathbb{R})$ is the space of all distributions 
$u$ such that $(I-\Delta)^{\frac{\ell}{2}}u \in L^p(\mathbb{R})$.
For $u\in H_p^{\ell}(\mathbb{R})$, the norm is defined as
  $$
  \|u\|_{H_p^{\ell}(\mathbb{R})}=\|(I-\Delta)^{\frac{\ell}{2}}u\|_{L^p(\mathbb{R})}.
  $$
\end{definition}

When $p=2$, $H_2^{\ell}=W^{\ell,2}$, {\it i.e.}, $H_2^{\ell}$
coincides with the Sobolev space $W^{\ell,2}$, so 
we simply denote $H_2^{\ell}$ by $H^{\ell}$.

\begin{lemma}\label{lem-bessel-potential-space-embedding}
The following embeddings hold{\rm :}
\begin{enumerate}
\item[\rm (i)] If ${\ell_1,\ell_2} 
\in \mathbb{R}$ satisfying {$\ell_1 \le \ell_2$},
then $H_p^{\ell_2} \hookrightarrow H_p^{\ell_1}${\rm ;}
\item[\rm (ii)] If $p {\ell}>1$ and $\ell-\frac{1}{p}$ is not an integer, 
then $H_p^{\ell}\hookrightarrow C^{\ell-\frac{1}{p}}$.
\end{enumerate}
\end{lemma}

The proof of (i)--(ii) can be found in \cite[Corollary 13.3.9]{krylov08ellipticparabolic} 
and \cite[Theorem 13.8.1]{krylov08ellipticparabolic}, respectively.

\medskip
We also recall the definition of the Besov space and some of its properties. 
By \cite[Theorem 1.2.5]{Triebel92function}, we define the Besov space:
\begin{definition}\label{def-besov-space}
Let $0<\ell<1, \,1<p<\infty$, and $1\le q\le \infty$. The Besov space is defined by
  $$
  B_{pq}^{\ell}(\mathbb{R})
  =\Big\{ f\in L^p(\mathbb{R})\,:\, \|f\|_{B_{pq}^{\ell}(\mathbb{R})}=\|f\|_{L^p(\mathbb{R})}
  +\Big(\int_{\mathbb{R}}|h|^{-{\ell}q}\|\Delta_h f\|_{L^p(\mathbb{R})}^q \frac{\d h}{|h|}\Big)^{\frac{1}{q}}<\infty\Big\},
  $$
where $\Delta_h f(x):=f(x+h)-f(x)$.
\end{definition}

Using Definition \ref{def-besov-space}, we can check the following property (see also \cite[(B.12)]{BV19}):

\begin{lemma}\label{lem-besov-space-lipschitz-property}
For any locally Lipschitz $f:\mathbb{R}\to \mathbb{R}$ satisfying $f(0)=0$, 
then
  $$
  \|f(u)\|_{B_{pq}^{\ell}(\mathbb{R})}\le L_M(f)\|u\|_{B_{pq}^{\ell}(\mathbb{R})}
  \qquad\mbox{for any $u\in B_{pq}^{\ell}(\mathbb{R}) \cap L^{\infty}(\mathbb{R})$},
  $$
where $L_M$ denotes the Lipschitz constant of $f$ on interval $[-M,M]$ with $M:=\|u\|_{L^{\infty}(\mathbb{R})}$.
\end{lemma}

From \cite[Theorem 1.3.2 (iv)]{Triebel92function}, we have the following relation between 
the Besov space and the Bessel potential space.

\begin{lemma}\label{lem-relation-between-besov-and-bessel-potential-space}
For $0<\ell<1$, $H_2^{\ell}(\mathbb{R})=B_{22}^{\ell}(\mathbb{R})$.
\end{lemma}

The next lemma states the algebraic property of the Bessel potential space.

\begin{lemma}{\rm \cite[Proposition 1.1]{tao01algebra}}\label{lem-algebra-property-of-bessel-space}
\,\, If $\ell>\frac12$ and $f,g\in H^{\ell}(\mathbb{R})$, 
then
$$
\|fg\|_{H^{\ell}(\mathbb{R})}\le C\|f\|_{H^{\ell}(\mathbb{R})}\|g\|_{H^{\ell}(\mathbb{R})}.
$$
\end{lemma}

\begin{lemma}\label{lem-properties-of-fractional-laplace}
If $\alpha>0$, then
\begin{itemize}
\item[\rm (i)] $(-\Delta)^{\frac{\alpha}{2}}(I-\Delta)^{-\frac{\alpha}{2}}$ is bounded in $L^p(\mathbb{R})$ 
for $1\le p \le \infty${\rm ;}
\item[\rm (ii)] $\|(-\Delta)^{\alpha}\mathbb{S}(t)u\|_{L^2(\mathbb{R})}\le C(\alpha)t^{-\alpha}\|u\|_{L^2(\mathbb{R})}${\rm ;}
\item[\rm (iii)] $(\partial_x)(-\Delta)^{-\frac12}$ and $(\partial_x)(I-\Delta)^{-\frac12}$ are bounded in $L^2(\mathbb{R})$.
\end{itemize}
\end{lemma}

\begin{proof}
For (i), see \cite[Lemma V.3.2]{stein70singular} for the details. 
For (ii), it follows from \cite[Theorem 14.11]{krasnoselskii1976integral}. 

For (iii), using the Fourier transform and the definition of the Riesz potential, as well as (i), it can be easily proved. 
\end{proof}

\begin{lemma}\label{lem-equivalence-of-norms-on-bessel-space}
For $\ell>0$, the two norms $\|\cdot\|_{H_p^{\ell}}$ 
and $\|\cdot\|_{L^p}+\|(-\Delta)^{\frac{\ell}{2}}\cdot\|_{L^p}$ are equivalent.
\end{lemma}
Lemma \ref{lem-equivalence-of-norms-on-bessel-space} follows from \cite[Lemma V.3.2]{stein70singular} and \cite[Example 6.2.9]{grafakosGTM249}.

\subsubsection{A compactness lemma}
The following lemma can be found in \cite[Corollary B.2]{ondrejat10}.
\begin{lemma}\label{lem-compactness-result-from-ondrejat}
Let $a=\{a_n\}$ be a sequence of positive constants, 
let $B_n \subset \mathbb{R}^d$ be an open ball with center 
at the origin and radius $n$, 
and let $\alpha>0,\, 1<r,p < \infty$, and $-\infty < \ell \le k$ 
satisfy $\frac{1}{p}-\frac{k}{d}\le \frac{1}{r}-\frac{\ell}{d}$. Then the set
 $$
 A(a):=\big\{g\in C_{\rm w}([0,T];W^{k,p}_{\rm loc}(\mathbb{R}^d))\,:\,\|g\|_{L^{\infty}([0,T];W^{k,p}(B_n))}+\|g\|_{C^{\alpha}([0,T];W^{\ell,r}(B_n))} \le a_n,\ n\in \mathbb{N}\big\}
 $$
is a metrizable compact set in $C_{\rm w}([0,T];W^{k,p}_{\rm loc}(\mathbb{R}^d))$, {\it i.e.}, the embedding{\rm :}
  $$
  L^{\infty}([0,T];W^{k,p}_{\rm loc}(\mathbb{R}^d)) \cap C^{\alpha}([0,T];W^{\ell,r}_{\rm loc}(\mathbb{R}^d)) \hookrightarrow C_{\rm w}([0,T];W^{k,p}_{\rm loc}(\mathbb{R}^d))
  $$
is compact.
\end{lemma}

\subsection{Stochastic Version of Murat's Lemmas}
The following lemma is a generalized version of \cite[Lemma A.3]{feng-Nualart08} in $\mathbb{R}^2$; also see
\cite{chen1986convergence,ding1985convergence}.
\begin{lemma}\label{lem-stochastic-version-murat-lemma}
Assume that $\mathcal{O}\subset \mathbb{R}^2$ is a bounded open set with smooth $C^{\infty}$ boundary. Suppose that $\phi^{\varepsilon}=\psi^{\varepsilon}_1+\psi^{\varepsilon}_2$ satisfies
\begin{itemize}
  \item[\rm (i)] $\{\psi^{\varepsilon}_1\}_{\varepsilon >0}$ is tight in $W^{-1,p_1}(\mathcal{O})$ with $p_1>1$,
  \item[\rm (ii)] $\{\psi^{\varepsilon}_2\}_{\varepsilon >0}$ is tight in $W^{-1,p_2}(\mathcal{O})$ with $p_2\in (1,2)$,
  \item[\rm (iii)] $\{\phi^{\varepsilon}\}_{\varepsilon >0}$ is stochastically bounded in $W^{-1,p}(\mathcal{O})$ 
   for some $p> \min\{p_1, p_2\}$, {\it i.e.}, for any $\delta >0$, there exists $C_{\delta} >0$ 
  such that $\sup_{\varepsilon >0} \mathbb{P}(\|\phi^{\varepsilon}\|_{W^{-1,p}} > C_{\delta}) < \delta$.
\end{itemize}
Then $\{\phi^{\varepsilon}\}_{\varepsilon >0}$ is tight in $W^{-1,q}(\mathcal{O})$ with $\min\{p_1, p_2\} \le q < p$.
\end{lemma}

\begin{proof}
The proof is almost similar to 
that of Lemma A.3 in \cite{feng-Nualart08}. 
In fact, it suffices to notice the following two points: 
The first point is that, in the proof of Lemma A.3 in \cite{feng-Nualart08}, 
it is used that the uniform boundedness of a sequence in measures 
implies its pre-compactness in $W^{-1,q}(\mathcal{O})$ with $q\in (1,2)$. 
Therefore, actually, we only need the tightness of $\psi^{\varepsilon}_2$ 
in $W^{-1,q}(\mathcal{O})$ with $q\in (1,2)$. 
The second point is that, in the proof of Lemma A.3 in \cite{feng-Nualart08}, 
the final conclusion that $\{\phi^{\varepsilon}\}_{\varepsilon >0}$ is tight 
in $H^{-1}(\mathcal{O})$ comes from the classical interpolation ({\it i.e.}, Murat's lemma). 
Here to conclude, 
we should use the compact interpolation theorem
(see \cite{chen1986convergence,ding1985convergence}): 
For $1<\tilde p_2 < \tilde p_1 \le \infty$ and $\tilde p_2 \le \tilde p_0 <\tilde p_1$,
  $$
  \big( \text{compact set of $W^{-1,\tilde p_2}_{\rm loc}(\mathbb{R}^2_+)$} \big)\cap 
  \big( \text{bounded set of $W^{-1,\tilde p_1}_{\rm loc}(\mathbb{R}^2_+)$} \big) 
  \subset \big( \text{compact set of $W^{-1,\tilde p_0}_{\rm loc}(\mathbb{R}^2_+)$} \big).
  $$
\end{proof}

We next establish another stochastic-version Murat's lemma used in the proof of the compactness of 
the martingale entropy solutions, {\it i.e.}, Lemma \ref{lem-stochastic-version-murat-lemma-81} 
(for the deterministic one, see \cite{murat81}),
before which we state a lemma from \cite{murat81} without proof.

\begin{lemma}\label{lem-deterministic-lemma-used-in-stochastic-murat-lemma-81}
Let $\mathcal{O} \subset \mathbb{R}^n$ be a bounded domain with Lipschitz 
boundary $\partial \mathcal{O}$. 
Then, for any $\alpha >0$ small enough and $1< p_0 < q_0 < \infty$, 
there exists $\psi^{\alpha} \in \mathcal{D}(\mathcal{O})$ such that, 
for any $f \in W^{1,q_0}_0(\mathcal{O})$, 
\begin{equation}\label{iq-1-psi-alpha-varphi-le-alpha-varphi}
\|(1-\psi^{\alpha})\varphi\|_{W^{1,p_0}_0(\mathcal{O})} 
\le \alpha \|\varphi \|_{W^{1,q_0}_0(\mathcal{O})}.
\end{equation}
\end{lemma}

\begin{lemma}\label{lem-stochastic-version-murat-lemma-81}
Suppose that $\mathcal{O}\subset \mathbb{R}^n$ is an open bounded domain with Lipschitz boundary, $1< p< \infty$, 
and $\frac{1}{p}+ \frac{1}{p^{\prime}}=1$. If $\{f^{\delta}\}_{\delta>0}$ satisfies 
\begin{itemize}
\item[\rm (i)] $\{f^{\delta}\}_{\delta>0}\subset W^{-1,p}(\mathcal{O})$ is stochastically bounded in $W^{-1,p}(\mathcal{O})$, 
{\it i.e.}, for any $\beta>0$ there exists $C_{\beta}>0$ such that 
$\sup_{\delta>0} \mathbb{P}\{\|f^{\delta}\|_{W^{-1,p}(\mathcal{O})}>C_{\beta}   \} < \beta${\rm ;}
\item[\rm (ii)] $\mathbb{P}$-{\it a.s.}, $f^{\delta}\le 0$ in the sense of distribution, {\it i.e.},  $\langle f^{\delta}(\omega), \varphi \rangle \le 0$
for any $\varphi \in \mathcal{D}(\mathcal{O})$ with $\varphi \ge 0$, 
where $\langle \cdot, \cdot \rangle$ denotes the pairing between $W^{-1,p}$ and $W^{1,p^{\prime}}_0$.
\end{itemize}
Then $\{f^{\delta}\}_{\delta>0}$ is tight in $W^{-1,p_0}(\mathcal{O})$ for any $1<p_0 < p$.
\end{lemma}

\begin{proof}
The proof is similar to that in the deterministic case.

\smallskip
\textbf{1}. Claim: For any $K\Subset \mathcal{O}$, there exists $C_{\beta}(K)$ such that,
for any $\varphi \in \mathcal{D}(\mathcal{O})$ with ${\rm supp}(\varphi)\subset K$ and any $\delta >0$, 
$\{ \|f^{\delta}\|_{W^{-1,p}(\mathcal{O})}\le C_{\beta}\} \subset \{|\langle f^{\delta}, \varphi \rangle|\le C_{\beta}(K)\|\varphi\|_{L^{\infty}}\}$.

In fact, for any $K\Subset \mathcal{O}$, there exists $\Phi_K \in \mathcal{D}(\mathcal{O})$ such that $\Phi_K\equiv 1$ on $K$ and $\Phi_K \ge 0$ in $\mathcal{O}$. 
Then, for any $\varphi \in \mathcal{D}(\mathcal{O})$ with ${\rm supp}(\varphi)\subset K$,
$-\|\varphi\|_{L^{\infty}(\mathcal{O})}\Phi_K(x) \le \varphi(x) \le \|\varphi\|_{L^{\infty}(\mathcal{O})}\Phi_K(x)$ in $\mathcal{O}$, {\it i.e.}, $\varphi(x) + \|\varphi\|_{L^{\infty}(\mathcal{O})}\Phi_K(x), \|\varphi\|_{L^{\infty}(\mathcal{O})}\Phi_K(x) - \varphi(x) \ge 0$ in $\mathcal{O}$. By (ii), we see that 
$\langle f^{\delta}(\omega), \varphi+\|\varphi\|_{L^{\infty}(\mathcal{O})}\Phi_K\rangle\le 0$
and $\langle f^{\delta}(\omega),  \|\varphi\|_{L^{\infty}(\mathcal{O})}\Phi_K - \varphi \rangle \le 0$, which imply that 
  $$
  |\langle f^{\delta}, \,\varphi \rangle|
  \le -\langle f^{\delta}, \,\|\varphi\|_{L^{\infty}(\mathcal{O})}\Phi_K  \rangle \le C_{\beta}(K)\|\varphi\|_{L^{\infty}(\mathcal{O})}.
  $$
Thus, $\{ \|f^{\delta}\|_{W^{-1,p}(\mathcal{O})}\le C_{\beta}\} \subset \{|\langle f^{\delta}, \varphi \rangle|\le C_{\beta}(K)\|\varphi\|_{L^{\infty}}\}$. 
Then, by (i), $\sup_{\delta>0} \mathbb{P}\{|\langle f^{\delta}, \varphi \rangle|\le C_{\beta}(K)\|\varphi\|_{L^{\infty}}\} > 1-\beta$.

\smallskip
\textbf{2}. Suppose  that $\Phi \in \mathcal{D}(\mathcal{O})$
with ${\rm supp} (\Phi) \subset \mathcal{O}$. Note that
  $$
  \|\Phi f^{\delta}\|_{W^{-1,p}(\mathcal{O})} = \sup_{\varphi \in W^{1,P^{\prime}}, \|\varphi\|_{W^{1,P^{\prime}}}\le 1}|\langle \Phi f^{\delta}, \varphi \rangle | \le C({\rm supp} \Phi) \|f^{\delta}\|_{W^{-1,p}},
  $$
which implies that there exists $C_{\beta,1}({\rm supp} (\Phi))>0$ such that 
$$
\{ \|f^{\delta}\|_{W^{-1,p}(\mathcal{O})}\le C_{\beta}\} \subset \{ \|\Phi f^{\delta}\|_{W^{-1,p}(\mathcal{O})}\le C_{\beta,1}({\rm supp} (\Phi)) \}.
$$ 
Additionally, by Step \textbf{1}, there exists $C_{\beta,2}({\rm supp} (\Phi))$ 
such that 
$$
\{ \|f^{\delta}\|_{W^{-1,p}(\mathcal{O})}\le C_{\beta}\} \subset \{ |\langle \Phi f^{\delta}, \varphi \rangle| \le C_{\beta,2}({\rm supp} (\Phi))\|\varphi\|_{L^{\infty}} \},
$$ 
which further implies that $\{ \|f^{\delta}\|_{W^{-1,p}(\mathcal{O})}\le C_{\beta}\} \subset \{ \| \Phi f^{\delta}\|_{\big( C_0(\bar{\mathcal{O}}) \big)^*} \le C_{\beta,2}({\rm supp}(\Phi)) \}$.

\smallskip
\textbf{3}. 
By the Sobolev embedding theorem and the duality principle, 
$\big( C_0(\bar{\mathcal{O}}) \big)^* \hookrightarrow W^{-1,r^{\prime}}(\mathcal{O})$ is compact, for $1\le r^{\prime}=\frac{r}{r-1}\le \frac{n}{n-1}$.
Then there exists a deterministic compact set $B_{1,\beta} \Subset W^{-1,r^{\prime}}(\mathcal{O})$ such that 
$\{ \|f^{\delta}\|_{W^{-1,p}(\mathcal{O})}\le C_{\beta}\} \subset \{ \| \Phi f^{\delta}\|_{\big( C_0(\bar{\mathcal{O}}) \big)^*} 
\le C_{\beta,2}({\rm supp} (\Phi)) \} \subset \{ \omega: \Phi f^{\delta} \in B_{1,\beta}\}$.

\textbf{4}. 
By interpolation between $W^{-1,p}(\mathcal{O})$ and $W^{-1,r^{\prime}}(\mathcal{O})$, there exists 
$B_{2,\beta}\Subset W^{-1,p_0}(\mathcal{O}), 1< r^{\prime} \le p_0 < p$ such that 
  $$
  \{ \|f^{\delta}\|_{W^{-1,p}(\mathcal{O})}\le C_{\beta}\} \subset  \{ \omega: \Phi f^{\delta} \in B_{1,\beta}\} \cap \{ \|\Phi f^{\delta}\|_{W^{-1,p}(\mathcal{O})}
  \le C_{\beta,1}({\rm supp} \Phi) \} \subset \{ \omega: \Phi f^{\delta} \in B_{2,\beta} \}.
  $$

\smallskip
\textbf{5}. For any $\alpha >0$ small enough, let $\psi^{\alpha} \in \mathcal{D}(\mathcal{O})$ be the function obtained 
in Lemma \ref{lem-deterministic-lemma-used-in-stochastic-murat-lemma-81} satisfying \eqref{iq-1-psi-alpha-varphi-le-alpha-varphi} 
and $(1-\psi^{\alpha})f^{\delta}\in W^{-1,p}(\mathcal{O})$. Then it follows from 
\eqref{iq-1-psi-alpha-varphi-le-alpha-varphi} that, 
for any $\varphi \in \mathcal{D}(\mathcal{O})$,
  $$
  |\langle (1-\psi^{\alpha})f^{\delta}, \varphi \rangle|\le \alpha \|f^{\delta}\|_{W^{-1,p}(\mathcal{O})}\|\varphi\|_{W_0^{1,q^{\prime}}}\qquad\mbox{for any $q^{\prime} > p^{\prime}$},
  $$
which implies that $\|(1-\psi^{\alpha})f^{\delta}\|_{W^{-1,q}(\mathcal{O})}\le  \alpha \|f^{\delta}\|_{W^{-1,p}(\mathcal{O})}$. Therefore, $\{ \|f^{\delta}\|_{W^{-1,p}(\mathcal{O})}\le C_{\beta}\} \subset \{ \|(1-\psi^{\alpha})f^{\delta}\|_{W^{-1,q}(\mathcal{O})}\le  \alpha C_{\beta} $. On the other hand, by Step \textbf{4}, 
we see that $\{ \|f^{\delta}\|_{W^{-1,p}(\mathcal{O})}\le C_{\beta}\} \subset \{ f^{\delta}-(1- \psi^{\alpha})f^{\delta} \in B_{2,\beta}\}$. Combining together gives 
$$
\{ \|f^{\delta}\|_{W^{-1,p}(\mathcal{O})}\le C_{\beta}\} \subset \{ f^{\delta}-(1- \psi^{\alpha})f^{\delta} \in B_{2,\beta}\}\cap \{ \|(1-\psi^{\alpha})f^{\delta}\|_{W^{-1,q}(\mathcal{O})}\le  \alpha C_{\beta} \}.
$$ 
Thus, letting $\alpha\to 0$, we obtain that $\{ \|f^{\delta}\|_{W^{-1,p}(\mathcal{O})}\le C_{\beta}\} \subset \{ f^{\delta} \in B_{2,\beta}\}$. The proof is completed.
\end{proof}

\subsection{Generalized Div-Curl Lemma}
For the general pressure law case, 
an improvement of the classical div-curl lemma in \cite{murat78} is needed. 

\begin{lemma}[\hspace{-0.2mm}\cite{ContiDolzmannMuller11div-curl}]\label{lem-generalized-div-curl-lemma}
Let $\mathcal{O}\subset \mathbb{R}^n$ be an open bounded set 
and $p,q\in (1,\infty)$ with $\frac{1}{p}+\frac{1}{q}=1$. 
Suppose that $\mathbf{v}^{\varepsilon}$ and $\mathbf{w}^{\varepsilon}$ are sequences of vector fields satisfying
  $$
  \mathbf{v}^{\varepsilon} \rightharpoonup \mathbf{v}\ \text{in $L^p(\mathcal{O};\mathbb{R}^n)$},
  \quad \mathbf{w}^{\varepsilon} \rightharpoonup \mathbf{w}\ \text{in $L^q(\mathcal{O};\mathbb{R}^n)$}
  \qquad\,\, {\text{as $\varepsilon \to 0$}}.
$$
Suppose further that $\mathbf{v}^{\varepsilon}\cdot \mathbf{w}^{\varepsilon}$ is equi-integrable and
  $$\begin{aligned}
  &{\rm div}\,\mathbf{v}^{\varepsilon}\qquad\,\,\, \text{is {\rm (}pre-{\rm )}compact in $W^{-1,1}_{\rm loc}(\mathcal{O})$},\\
  &{\rm curl}\,\mathbf{w}^{\varepsilon}\qquad \text{is {\rm (}pre-{\rm )}compact in $W^{-1,1}_{\rm loc}(\mathcal{O};\mathbb{R}^{n\times n})$}.\\
  \end{aligned}$$
Then $\mathbf{v}^{\varepsilon}\cdot \mathbf{w}^{\varepsilon}$ converges to $\mathbf{v}\cdot \mathbf{w}$ 
in the distributional sense{\rm :}
 $$
  \mathbf{v}^{\varepsilon}\cdot \mathbf{w}^{\varepsilon}\, \rightharpoonup\, 
  \mathbf{v}\cdot \mathbf{w}\qquad\,\, \text{in $\mathcal{D}^{\prime}$}.
  $$
\end{lemma}
In Lemma \ref{lem-generalized-div-curl-lemma} above, $W^{-1,1}(\mathcal{O})$ is the dual of $W^{1,\infty}_0(\mathcal{O})$.

\subsection{Proof of Proposition \ref{prop-compactness-criteria-for-young-measure}}\label{appendix-some-proofs}
The proof is in the spirit of Ball \cite{J.Ball89}. 
Since $L^{\infty}_{\rm w} (\mathbb{R}^2_+; \mathcal{P}(\mathbb{R}^2_+))$ is relatively compact 
in $L^{\infty}_{\rm w} (\mathbb{R}^2_+; \mathcal{M}(\mathbb{R}^2_+))$, it remains to prove that $\Lambda_C$ is closed.

Let $\{\mu_n\}_n \subset \Lambda_C$ such that
$\mu_n \overset{\ast}{\rightharpoonup}
\mu$. Define the cut-off functions $\chi_k\in C_0(\mathbb{R}_+^2;\mathbb{R}_+)$ for $k>0$ by
  $$
  \chi_k (\mathbf{f})=\begin{cases}
  1\ &\quad \text{if}\ |\mathbf{f}|\le k,\\
  1+k-|\mathbf{f}| &\quad\text{if}\  k\le |\mathbf{f}|\le k+1,\\
  0 &\quad\text{if}\  |\mathbf{f}|> k+1.
  \end{cases}
  $$
Then $\varsigma \chi_k \in C_0(\mathbb{R}_+^2;\mathbb{R}_+)$. By the monotone convergence theorem, we have
  $$
  \begin{aligned}
  \int_{[0,T]\times \hat{K}} \int_{\mathbb{R}^2_+}\varsigma (\mathbf{f}) \,\d \mu(\mathbf{x},\mathbf{f})\,\d \mathbf{x}
  &= \lim_{k\to \infty} \int_{[0,T]\times \hat{K}} \int_{\mathbb{R}^2_+}\varsigma \chi_k(\mathbf{f}) \,\d \mu(\mathbf{x},\mathbf{f})\,\d \mathbf{x}\\
  &=\lim_{k \to \infty}\lim_{n\to \infty} \int_{[0,T]\times \hat{K}} \int_{\mathbb{R}^2_+}\varsigma \chi_k(\mathbf{f}) \,\d \mu_n(\mathbf{x},\mathbf{f})\,\d \mathbf{x}\\
  &\le  \lim_{k \to \infty}\lim_{n\to \infty} \int_{[0,T]\times \hat{K}} \int_{\mathbb{R}^2_+}\varsigma (\mathbf{f}) \,\d \mu_n(\mathbf{x},\mathbf{f})\,\d \mathbf{x} \le C.
  \end{aligned}$$
Furthermore, by the weak-* lower semicontinuity of the norm, we have
  $$
  \|\mu\|_{\infty,\mathcal{M}}\le \lim_{n\to \infty}\|\mu_n\|_{\infty,\mathcal{M}}=1.
  $$
  
Next, we prove that $\|\mu\|_{\infty,\mathcal{M}}=1$. Since $\mu_n \overset{\ast}{\rightharpoonup}\mu$ 
and $\chi_k\in C_0(\mathbb{R}_+^2;\mathbb{R}_+)$, we have
  $$
  \begin{aligned}
  \lim_{n\to \infty}\int_{[0,T]\times \hat{K}}\int_{\mathbb{R}^2_+}\chi_k (\mathbf{f}) \,\d \mu_n(\mathbf{x},\mathbf{f})\d \mathbf{x} &= \int_{[0,T]\times \hat{K}}\int_{\mathbb{R}^2_+}\chi_k (\mathbf{f}) \,\d \mu(\mathbf{x},\mathbf{f})\d \mathbf{x}
  \le  \int_{[0,T]\times \hat{K}}\|\mu(\mathbf{x})\|_{\mathcal{M}}\,\d \mathbf{x}.
  \end{aligned}
  $$
On the other hand, denoting $N(k):=\inf_{|\mathbf{f}|>k} \varsigma(\mathbf{f})$ satisfying that
$N(k) \to \infty$ as $k\to \infty$, we have
  $$\begin{aligned}
  0&\le  \frac{1}{|\hat K|T}\int_{[0,T]\times \hat{K}}\int_{\mathbb{R}^2_+} \big(1-\chi_k (\mathbf{f})\big) \,\d \mu_n(\mathbf{x},\mathbf{f})\d \mathbf{x}\\
  &= \frac{1}{|\hat K|T}\int_{[0,T]\times \hat{K}}\int_{\{|\mathbf{f}|>k\}}  \,\d \mu_n(\mathbf{x},\mathbf{f})\d \mathbf{x}\\
  &\le   \frac{1}{N(k)}\frac{1}{|\hat K|T}\int_{[0,T]\times \hat{K}}\int_{\{|\mathbf{f}|>k\}} \varsigma(\mathbf{f}) \,\d \mu_n(\mathbf{x},\mathbf{f})\d \mathbf{x} \le \frac{C}{N(k)}\frac{1}{T\times |\hat K|},
  \end{aligned}$$
which implies
  $$
  1- \frac{C}{N(k)}\frac{1}{|\hat K|T} 
  \le \frac{1}{|\hat K|T}\int_{[0,T]\times \hat{K}}\int_{\mathbb{R}^2_+} \chi_k (\mathbf{f}) \,\d \mu_n(\mathbf{x},\mathbf{f})\d \mathbf{x}.
  $$
Letting $n\to \infty$ first and then letting $k\to \infty$, we obtain
  $$
  \frac{1}{|\hat K|T}\int_{[0,T]\times \hat{K}} \|\mu(\mathbf{x})\|_{\mathcal{M}} \,\d \mathbf{x} \ge 1,
  $$
which implies that $\|\mu(\mathbf{x})\|_{\mathcal{M}}=1$ almost everywhere, since $T$ and $\hat K$ are arbitrary and $\|\mu(\mathbf{x})\|_{\mathcal{M}} \le 1$ almost everywhere. Therefore, $\|\mu\|_{\infty,\mathcal{M}}=1$ and $\mu \in \Lambda_C$. This completes the proof of Proposition \ref{prop-compactness-criteria-for-young-measure}.

\bigskip
\noindent\textbf{Acknowledgement.}
The research of Gui-Qiang G. Chen was supported in part by the UK Engineering and Physical Sciences Research
Council Award EP/L015811/1, EP/V008854, and EP/V051121/1. 
The research of Feimin Huang was partially supported by the National Key R\&D Program of
China, grant No. 2021YFA1000800, and the National Natural Sciences Foundation of China, grant
No. 12288201. 
The research of Danli Wang was supported in part by the China Scholarship Council Scholarship (No. 202204910435).

\end{document}